\documentclass[a4paper,11pt,oneside,english]{amsart}
\usepackage{babel}
\usepackage[latin1]{inputenc}
\usepackage[T1]{fontenc}
\usepackage{amssymb,amsfonts,amsmath,amsthm}
\usepackage{color}
\usepackage{setspace}
\usepackage{enumitem}
\usepackage{epsfig,graphicx,pinlabel}
\usepackage[stretch=10,shrink=10,step=2,kerning=true,protrusion=true,final]{microtype}

\usepackage[colorlinks=true, linkcolor=blue, citecolor=red,
urlcolor=blue]{hyperref}

\newcounter{notes}%[page]   %Le 2eme argument fait reinitialiser les numeros de notes a chaque page
\newtheorem{theorem}{Theorem}
\newtheorem{proposition}[theorem]{Proposition}
\newtheorem{corollary}[theorem]{Corollary}
\newtheorem{lemma}[theorem]{Lemma}
\newtheorem{observation}[theorem]{Observation}
\newtheorem{claim}[theorem]{Claim}

\newtheorem{fact}[theorem]{Fact}

\theoremstyle{definition}
\newtheorem{definition}[theorem]{Definition}

\theoremstyle{remark}
\newtheorem{remark}[theorem]{Remark}
\newtheorem{remarks}[theorem]{Remarks}
\newtheorem{example}[theorem]{Example}

\newtheoremstyle{theoremwithref}{}{}{\itshape}{}{\bfseries}{.}{.5em}{#1 #2 #3}
\theoremstyle{theoremwithref}

\newcommand{\ie}{i.e.\ }
\newcommand{\eg}{e.g.\ }
\newcommand{\resp}{resp.\ }
\newcommand{\PP}{\mathbf{P}}
\newcommand{\CC}{\mathbf{C}}
\newcommand{\bH}{\mathbf{H}}
\newcommand{\RR}{\mathbf{R}}
\newcommand{\RRp}{\mathbf{R}_+}
\newcommand{\QQ}{\mathbf{Q}}
\newcommand{\ZZ}{\mathbf{Z}}
\newcommand{\NN}{\mathbf{N}}
\newcommand{\KK}{\mathbf{K}}
\newcommand{\GL}{\mathrm{GL}}
\newcommand{\SL}{\mathrm{SL}}
\newcommand{\PSL}{\mathrm{PSL}}
\newcommand{\SO}{\mathrm{SO}}
\newcommand{\OO}{\mathrm{O}}
\newcommand{\SU}{\mathrm{SU}}
\newcommand{\U}{\mathrm{U}}
\newcommand{\Sp}{\mathrm{Sp}}
\newcommand{\Spin}{\mathrm{Spin}}

\newcommand{\g}{\mathfrak{g}}
\newcommand{\h}{\mathfrak{h}}
\newcommand{\aaa}{\mathfrak{a}}
\newcommand{\athetaplus}{\overline{\mathfrak{a}}^{+}_{\theta}}
\newcommand{\z}{\mathfrak{z}}
\newcommand{\s}{\mathfrak{s}}
\newcommand{\gl}{\mathfrak{gl}}

\newcommand{\Hom}{\mathrm{Hom}}
\newcommand{\Aut}{\mathrm{Aut}}
\newcommand{\HH}{\mathbb{H}}
\newcommand{\Isom}{\mathrm{Isom}}
\newcommand{\Ker}{\mathrm{Ker}}

\newcommand{\Diag}{\mathrm{Diag}}
\newcommand{\Ad}{\mathrm{Ad}}
\newcommand{\ad}{\mathrm{ad}}

\newcommand{\CM}{C_{\!\!_\mathcal{M}}}

\newcommand{\pr}{\mathrm{pr}}

\newcommand{\dtheta}{d_{\mu_{\theta}}}

\newcommand{\ellFF}[2]{|#1|_{#2}}

\newcommand{\ellGamma}[1]{|#1|_{\Gamma}}
\newcommand{\ellGammaf}{\ellGamma{\cdot{}}}

\newcommand{\ellinfty}[1]{|#1|_{\infty{}}}
\newcommand{\ellinftyf}{\ellinfty{\cdot{}}}

\newcommand{\trlf}{\ell_{\Gamma}{}}
\newcommand{\trl}[1]{\trlf(#1)}

\newcommand{\vt}{{T}^{\mathrm{v}}}

\newcommand*{\longhookrightarrow}{\ensuremath{\lhook\joinrel\relbar\joinrel\rightarrow}}

\newcommand{\dist}{\mathrm{dist}}

\title{Anosov representations and proper actions}

\author[F. Gu\'eritaud]{Fran\c{c}ois Gu\'eritaud}
\address{CNRS and Universit\'e Lille 1, Laboratoire Paul Painlev\'e, 59655 Villeneuve d'Ascq Cedex, France 
\newline Wolfgang-Pauli Institute, University of Vienna, CNRS-UMI 2842, Austria}\email{francois.gueritaud@math.univ-lille1.fr}

\author[O. Guichard]{Olivier Guichard}
\address{Universit\'e de Strasbourg, IRMA, 7 rue  Descartes, 67000 Strasbourg, France}
\email{olivier.guichard@math.unistra.fr}

\author[F. Kassel]{Fanny Kassel}
\address{CNRS and Universit\'e Lille 1, Laboratoire Paul Painlev\'e, 59655 Villeneuve d'Ascq Cedex, France}
\email{fanny.kassel@math.univ-lille1.fr}

\author[A. Wienhard]{Anna Wienhard}
\address{Ruprecht-Karls Universit\"at Heidelberg, Mathematisches Institut, Im Neuenheimer Feld~288, 69120 Heidelberg, Germany
\newline HITS gGmbH, Heidelberg Institute for Theoretical Studies, Schloss-Wolfs\-brunnen\-weg 35, 69118 Heidelberg, Germany }
\email{wienhard@mathi.uni-heidelberg.de}

\thanks{FG, OG, and FK were partially supported by the Agence Nationale de la Recherche under the grants ETTT (ANR-09-BLAN-0116-01) and DiscGroup (ANR-11-BS01-013), as well as through the Labex CEMPI (ANR-11-LABX-0007-01).
OG also received funding from the European Research Council under the European Community's seventh Framework Programme (FP7/2007-2013)/ERC grant agreement FP7-246918.
AW was partially supported by the National Science Foundation under agreements DMS-1065919 and 0846408, by the Sloan Foundation, by the Deutsche Forschungsgemeinschaft, by the European Research Council under ERC-Consolidator grant 614733, and by the Klaus Tschira Foundation.
This paper was completed while OG, FK, and AW were in residence at the MSRI in Berkeley, California, supported by the National Science Foundation under grant 0932078~000.}

\begin{document}

\numberwithin{theorem}{section}
\numberwithin{equation}{section}

\begin{abstract}
We establish several characterizations of Anosov representations of word hyperbolic groups into real reductive Lie groups, in terms of a Cartan projection or Lyapunov projection of the Lie group. 
Using a properness criterion of Benoist and Kobayashi, we derive applications to proper actions on homogeneous spaces of reductive groups.
\end{abstract}

\maketitle

%%%%%%%%%%%%%%%%%%%%%%%%%%%%%%%%%%%%%%%%%%%%%%%%%%%
\section{Introduction}

Anosov representations of word hyperbolic groups into real Lie groups were first introduced by~Labou\-rie \cite{Labourie_anosov}.
They provide an interesting class of discrete subgroups of semisimple or reductive Lie groups, with a rich structure theory.
In many respects they generalize, to a higher-rank setting, the convex cocompact representations into rank-one simple groups \cite{Guichard_Wienhard_DoD, KapovichLeebPortiv2, KapovichLeebPorti14, KapovichLeebPorti14_2}.
They also play an important role in the context of higher Teichm\"uller spaces. 

The original definition of Anosov representations from \cite{Labourie_anosov} involves the flow space of a word hyperbolic group, whose construction is not completely straightforward.
In this paper, we establish several characterizations of Anosov
representations that do not involve the flow space. 
A central role in our characterizations is played by the Cartan projection of~$G$ (associated with a fixed Cartan decomposition), which measures dynamical properties of diverging sequences in~$G$.
We apply our characterizations to the study of proper actions on homogeneous spaces by establishing a direct link between the properties, for a representation $\rho : \Gamma\to G$ to be Anosov, and for~$\Gamma$ to act properly discontinuously via~$\rho$ on certain homogeneous spaces of~$G$.

We now describe our results in more detail.

%%%%%%%%%%%%%%%%%%%%%%%%%
\subsection{Existence of continuous boundary maps}\label{subsec:intro-constr-xi}

Since the foundational work of Furstenberg and the celebrated rigidity theorems of Mostow and Margulis, the existence of boundary maps has been playing a crucial role in the study of discrete subgroups of Lie groups. 
Given a representation $\rho :\nolinebreak\Gamma\to G$ of a finitely generated group $\Gamma$ into a reductive (\eg semisimple) Lie group~$G$, {\emph{measurable}} $\rho$-equivariant boundary maps from a Poisson boundary of~$\Gamma$ to a boundary of~$G$ exist under rather weak assumptions, \eg Zariski density of $\rho(\Gamma)$ in~$G$.
However, if $\Gamma$ comes with some geometric or topological boundary, obtaining a \emph{continuous} $\rho$-equivariant boundary map is in general difficult. 

Anosov representations of word hyperbolic groups come, by definition, with a pair of continuous equivariant boundary maps.
More precisely, let $\rho :\nolinebreak\Gamma\to G$ be a $P_\theta$-Anosov representation.
(We always assume $G$ to be noncompact and linear, and use the notation $P_{\theta}, P_{\theta}^-$ for its standard parabolic subgroups with the convention that $P_{\emptyset}=G$, see Section~\ref{subsec:some-structure-semi}.) 
Then there exist $\rho$-equivariant boundary maps $\xi^+ : \partial_{\infty}\Gamma\to G/P_{\theta}$ and $\xi^- : \partial_{\infty}\Gamma\to G/P_{\theta}^-$ that are continuous. 
These boundary maps have additional remarkable properties: they are {\em transverse}, \ie for any distinct points $\eta , \eta' \in \partial_{\infty} \Gamma$ the images $\xi^+ (\eta) \in G/P_{\theta}$ and $\xi^-(\eta') \in G/P_{\theta}^-$ are in general position, and they are {\em dynamics-preserving}, which means that for any $\gamma\in\Gamma$ of infinite order with attracting fixed point $\eta^+_{\gamma} \in \partial_{\infty} \Gamma$, the point $\xi^+(\eta^{+}_{\gamma})$ (\resp $\xi^-(\eta^{+}_{\gamma})$) is an attracting fixed point for the action of $\rho(\gamma)$ on $G/P_{\theta}$ (\resp $G/P_{\theta}^-$).
Furthermore, these maps satisfy an \emph{exponential contraction} property involving certain bundles over the flow space of~$\Gamma$ (see Section~\ref{subsec:anos-repr}).

In this paper, given a word hyperbolic group $\Gamma$ and a representation $\rho : \Gamma\to\nolinebreak G$, we construct, under some growth assumption for the Cartan projection of~$G$ restricted to~$\rho(\Gamma)$ (Theorem~\ref{thm:constr-xi}.\eqref{item:xi-exist_1}), an explicit pair $(\xi^+, \xi^-)$ of continuous, $\rho$-equiva\-riant boundary maps.
(Recall that the Cartan projection $\mu: G \to \overline{\aaa}^+$, defined from a Cartan decomposition $G=K(\exp\overline{\aaa}^+)K$, is a continuous, proper, surjective map to the closed Weyl chamber $\overline{\aaa}^+$; see Section~\ref{subsubsec:Cartan-proj}, and Example~\ref{ex:mu} for $G=\GL_d(\RR)$.)
We give a sufficient condition for these maps to be dynamics-preserving (Theorem~\ref{thm:constr-xi}.\eqref{item:xi-exist_2}).
Under an additional assumption on the growth of~$\mu$ along geodesic rays in~$\Gamma$ (Theorem~\ref{thm:constr-xi}.\eqref{item:xi-transv_1}), we prove that the pair of maps $(\xi^+, \xi^-)$ is also transverse and that $\rho$ is Anosov. 
This assumption involves the following notion: we say that a sequence $(x_n)\in(\RRp)^{\NN}$ is \emph{CLI} (\ie has \emph{coarsely linear increments}) if there exist $\kappa,\kappa',\kappa'',\kappa'''>0$ such that for all $n,m\in\NN$,
\begin{equation} \label{eqn:clik}
\kappa m - \kappa' \leq x_{n+m} - x_n \leq \kappa'' m + \kappa'''.
\end{equation}
In other words, $n\mapsto x_n$ is a quasi-isometric embedding of $\NN$ into~$\RRp$.

\begin{theorem} \label{thm:constr-xi}
Let $\Gamma$ be a word hyperbolic group and $\ellGammaf : \Gamma\to\NN$ its
word length function with respect to some fixed finite generating subset of~$\Gamma$.
Let $G$ be a real reductive Lie group and $\rho : \Gamma \to G$ a representation.
Fix a nonempty subset $\theta \subset \Delta$ of the simple restricted roots of~$G$ (see Section~\ref{subsubsec:lie-algebr-decomp}), and let $\Sigma^+_{\theta}\subset\aaa^*$ be the set of positive roots that do \emph{not} belong to the span of $\Delta\smallsetminus\theta$.
\begin{enumerate}
  \item\label{item:xi-exist_1} If there is a constant $C>0$ such that for any $\alpha\in\theta$, 
  \[ {\langle\alpha,\mu(\rho(\gamma))\rangle} \geq {2\log \ellGamma {\gamma}} -C, \]
then there exist continuous, $\rho$-equivariant boun\-dary maps $\xi^+ : \partial_{\infty}\Gamma\to G/P_{\theta}$ and $\xi^- : \partial_{\infty}\Gamma\to\nolinebreak G/P_{\theta}^-$.
  \item\label{item:xi-exist_2} If moreover for any $\alpha\in\theta$ and any $\gamma\in\Gamma$, 
  \[ {\langle\alpha,\mu(\rho(\gamma^n))\rangle} - {2\log |n|} \underset{|n|\to +\infty}{\longrightarrow} +\infty ,\]
then $\xi^+$ and~$\xi^-$ are dynamics-preserving.
  \item\label{item:xi-transv_1} If moreover for any $\alpha\in\Sigma^+_{\theta}$ and any geodesic ray $(\gamma_n)_{n\in\NN}$ in the Cayley graph of~$\Gamma$, the sequence $\big(\langle \alpha,\, \mu( \rho(\gamma_n))\rangle\big)_{n\in\NN}$ is CLI, then $\xi^+$ and~$\xi^-$ are transverse and $\rho$ is $P_{\theta}$-Anosov.
\end{enumerate}
\end{theorem}

Let us briefly discuss the meaning of the assumptions of Theorem~\ref{thm:constr-xi}.
Let $\Vert\cdot\Vert$ be a Euclidean norm on the Cartan subspace~$\aaa$ which is invariant under the restricted Weyl group of $\aaa$ in~$G$.
For any $\alpha\in\Delta$, the function $\langle\alpha,\cdot\rangle : \overline{\aaa}^+\to\RRp$ is proportional to the distance function to the wall $\Ker(\alpha)$, with respect to~$\Vert\cdot\nolinebreak\Vert$.
Thus the assumption of Theorem~\ref{thm:constr-xi}.\eqref{item:xi-exist_1} means that the set $\mu(\rho(\Gamma))$ ``drifts away at infinity'' from the union of walls $\bigcup_{\alpha\in\theta} \Ker(\alpha)$ in~$\overline{\aaa}^+$, at speed at least $2\log$ in the word length.
The CLI assumption in Theorem~\ref{thm:constr-xi}.\eqref{item:xi-transv_1} means that the image under $\mu\circ\rho$ of any geodesic ray in the Cayley graph of~$\Gamma$ drifts away ``forever linearly'' from the hyperplane $\Ker(\alpha)$ for every $\alpha\in\Sigma_{\theta}^+$ (see Section~\ref{subsubsec:Cartan-proj-proper}).
Inside the Riemannian symmetric space $G/K$ of~$G$, the function $\Vert\mu(\cdot)\Vert : G\to\RRp$ gives the distance between the basepoint $x_0=eK\in G/K$ and its image under an element of~$G$ (see \eqref{eqn:mu-dist}).
The functions $\langle\alpha,\mu(\cdot)\rangle$ can be thought of as refinements of $\Vert\mu(\cdot)\Vert$.
The CLI assumption in Theorem~\ref{thm:constr-xi}.\eqref{item:xi-transv_1} means that $\langle\alpha,\mu \circ \rho(\cdot)\rangle : \Gamma\to\RRp$ restricts to a quasi-isometric embedding on any geodesic ray in the Cayley graph of~$\Gamma$, for $\alpha \in \Sigma_{\theta}^+$.

\begin{remarks} \label{rem:constr-xi}
\begin{enumerate}[label=(\alph*),ref=\alph*]
  \item Theorem~\ref{thm:constr-xi}.\eqref{item:xi-exist_1} extends a rank-one result of Floyd \cite{Floyd80} to the setting of higher real rank.
  As in \cite{Floyd80}, the image of~$\xi^+$ in Theorem~\ref{thm:constr-xi} is the limit set of $\Gamma$ in $G/P_{\theta}$ (Definition~\ref{defi:limit-set}), see Theorem~\ref{thm:constr-xi-explicit}.
  \item Boundary maps are commonly constructed only on some large subset of $\partial_{\infty}\Gamma$ (dense or of full measure).
  By contrast, Theorem~\ref{thm:constr-xi}.\eqref{item:xi-exist_1} is based on an explicit construction of the boundary maps at every point: see Theorem~\ref{thm:constr-xi-explicit}.
  \item Theorem~\ref{thm:constr-xi-explicit} refines Theorem~\ref{thm:constr-xi} by providing weaker conditions for the existence of $\rho$-equivariant boundary maps with various properties.
  The subtleties in dropping one assumption among continuity, dynamics-preservation, or transversality are illustrated by Examples \ref{ex:log-growth} and~\ref{ex:xi-not-dyn-preserv}.
  We expect that the methods of the proof of Theorem~\ref{thm:constr-xi-explicit} will have applications in broader contexts where $\Gamma$ is not necessarily word hyperbolic.
\item The growth condition in Theorem~\ref{thm:constr-xi}.\eqref{item:xi-exist_2} is optimal (see Lemma~\ref{lem:prox} and Remark~\ref{rem:compatible}.\eqref{item:Ano-prox}). 
  \item \label{item:cli-non-unif} In Theorem~\ref{thm:constr-xi}.\eqref{item:xi-transv_1} the CLI constants are not required to be uniform. 
  \item The CLI assumption in Theorem~\ref{thm:constr-xi}.\eqref{item:xi-transv_1} can be restricted to the set of geodesic rays starting at the identity element $e\in\Gamma$.
    In fact, we only need the CLI assumption for one quasi-geodesic representative per point in the boundary at infinity $\partial_{\infty}\Gamma$ (see Proposition~\ref{prop:transv-SL}).
    If $\theta=\Delta$ (\ie $P_{\theta}$ is a minimal parabolic subgroup of~$G$), then the CLI assumption for all $\alpha\in\Sigma_{\theta}^+$ is equivalent to the CLI assumption for all $\alpha\in\theta$, because in this particular case $\Sigma_{\theta}^+\subset\mathrm{span}(\theta)$.
  \item For a quasi-geodesic ray $(\gamma_n)_{n \in \NN}$, under the hypothesis of Theorem~\ref{thm:constr-xi}.\eqref{item:xi-transv_1}, the sequence $(\langle \alpha,\, \mu( \rho(\gamma_n))\rangle)_{n\in\NN}$ is always upper CLI: see \eqref{eqn:mu-leq-klength} and Fact~\ref{fact:mu-subadditive}.
\end{enumerate}
\end{remarks}

(Here we use the terminology \emph{$(\kappa'',\kappa''')$-upper CLI} for a sequence $(x_n)\in(\RRp)^{\NN}$ satisfying the right-hand inequality of \eqref{eqn:clik} for all $n,m\in\NN$; we say $(x_n)_{n\in\NN}$ is upper CLI if it is $(\kappa'',\kappa''')$-upper CLI for some $\kappa'',\kappa'''>0$.
Similarly, we shall use below the terminology \emph{$(\kappa,\kappa')$-lower CLI} for a sequence $(x_n)\in(\RRp)^{\NN}$ satisfying the left-hand inequality of \eqref{eqn:clik} for all $n,m\in\NN$.)

%%%%%%%%%%%%%%%%%%%%%%%%%
\subsection{Characterizations of Anosov representations in terms of the Cartan projection}

Theorem~\ref{thm:constr-xi} provides sufficient conditions for a representation $\rho : \Gamma\to\nolinebreak G$ to be Anosov in terms of the Cartan projection~$\mu$.
Conversely, we prove that any Anosov representation satisfies these conditions.
This yields the following characterizations.

\begin{theorem}\label{thm:char_ano}
Let $\Gamma$ be a word hyperbolic group, $G$ a real reductive Lie group, and $\theta \subset \Delta$ a nonempty subset of the simple restricted roots of~$G$.
For any representation $\rho : \Gamma \to G$, the following conditions are equivalent:
\begin{enumerate}
  \item\label{item:Ano} $\rho$ is $P_\theta$-Anosov;
  \item\label{item:away-from-walls} There exist continuous, $\rho$-equivariant, dynamics-preserving, and transverse maps $\xi^+ : \partial_{\infty}\Gamma\to G/P_{\theta}$ and $\xi^- : \partial_{\infty}\Gamma\to G/P_{\theta}^-$, and for any $\alpha\in\theta$ we have $\langle\alpha,\mu(\rho(\gamma))\rangle\to+\infty$ as $\gamma\to\infty$ in~$\Gamma$;
  \item\label{item:lin-away-from-walls} There exist continuous, $\rho$-equivariant, dynamics-preserving, and transverse maps $\xi^+ : \partial_{\infty}\Gamma\to G/P_{\theta}$ and $\xi^- : \partial_{\infty}\Gamma\to G/P_{\theta}^-$, and constants $c,C>0$ such that $\langle\alpha,\mu(\rho(\gamma))\rangle\geq c\,\ellGamma{\gamma}-C$ for all $\alpha\in\theta$ and $\gamma\in\Gamma$;
  \item\label{item:cli} There exist $\kappa,\kappa'>0$ such that for any $\alpha\in\Sigma^+_{\theta}$ and any geodesic ray $(\gamma_n)_{n\in\NN}$ with $\gamma_0=e$ in the Cayley graph of~$\Gamma$, the sequence $(\langle \alpha,\, \mu( \rho(\gamma_n))\rangle)_{n\in\NN}$ is $(\kappa,\kappa')$-lower CLI.
\end{enumerate}
\end{theorem}

By $\gamma\to\infty$ we mean that $\gamma$ leaves every finite subset of~$\Gamma$, or equivalently that the word length $\ellGamma{\gamma}$ of~$\gamma$ goes to $+\infty$.

\begin{remark}
For Zariski-dense representations the existence of continuous, $\rho$-equiva\-riant, dynamics-preserving, and transverse maps $\xi^+ : \partial_{\infty}\Gamma\to G/P_{\theta}$ and $\xi^- : \partial_{\infty}\Gamma\to G/P_{\theta}^-$ is sufficient for $\rho$ to be $P_{\theta}$-Anosov  \cite[Th.\,4.11]{Guichard_Wienhard_DoD}.
However, we observe that this is \emph{not} true in general, even when $\rho$ is semisimple: see Example~\ref{ex:nice-maps-not-Ano}.
\end{remark}

\begin{remark}\label{rem:intro-char-Ano}
  From Theorem~\ref{thm:char_ano}.\eqref{item:lin-away-from-walls} we recover the fact \cite{Labourie_anosov}, \cite[Th.\,5.3]{Guichard_Wienhard_DoD} that any Anosov representation is a quasi-isometric embedding. 
    For a semi\-simple Lie group $G$ of real rank one, \ie when \(|\Delta|=1\), being a quasi-isometric embedding is equivalent to being $P_{\Delta}$-Anosov (Remark~\ref{rem:rank1ccquasi}), and in particular is an open property. 
    In higher real rank this is not true: being a quasi-isometric embedding is \emph{not} an open property (see Appendix~\ref{app:discr-quasi-isom}), whereas being Anosov is. 
  In higher real rank it is more difficult to find natural constraints on quasi-isometric embeddings $\rho: \Gamma \to G$  that define an open subset of $\Hom(\Gamma, G)$; characterization~\eqref{item:cli} of Theorem~\ref{thm:char_ano} provides one answer to this problem.
\end{remark}

The characterization of Anosov representations given by Theorem~\ref{thm:char_ano}.\eqref{item:cli} does not involve the boundary $\partial_{\infty}\Gamma$, but only the behavior of the Cartan projection along geodesic rays.
Here are some consequences.

\begin{remarks}\label{rem:intro-char-Ano-bis}
\begin{enumerate}[label=(\alph*),ref=\alph*]
  \item\label{item:p-adic} Theorem~\ref{thm:char_ano}.\eqref{item:cli} provides a notion of Anosov representations of word hyperbolic groups into $p$-adic Lie groups.
  Indeed, the Cartan projection~$\mu$, with values in a convex cone inside some Euclidean space, is also well defined, with similar properties, when $G$ is a reductive group over~$\QQ_p$ (or more generally a non-Archimedean local field).
  \item\label{item:prim_stable} Theorem~\ref{thm:char_ano}.\eqref{item:cli} can be used to define new classes of representations into real Lie groups.
  For instance, for a free group~$\Gamma$, requiring condition~\eqref{item:cli} only for certain ``primitive'' geodesic rays gives rise to a notion of primitive stable representations into higher-rank Lie groups (generalizing the notion introduced by Minsky \cite{Minsky} for $G=\PSL_2(\CC)$).
\end{enumerate}
\end{remarks}

%%%%%%%%%%%%%%%%%%%%%%%%%
\subsection{Relation to the work of Kapovich, Leeb, and Porti} \label{subsec:KLP}

There is overlap between Theorem~\ref{thm:char_ano}
(which contains a weaker version of Theorem~\ref{thm:constr-xi}.\eqref{item:xi-transv_1}, see Remark~\ref{rem:constr-xi}.\ref{item:cli-non-unif}) and results of Kapovich, Leeb, and Porti, as we now~describe.

In a series of three papers \cite{KapovichLeebPortiv2, KapovichLeebPorti14, KapovichLeebPorti14_2} (see also \cite{KapovichLeebPorti15}), Kapovich, Leeb, and Porti develop a theory of discrete groups of isometries of higher-rank Riemannian symmetric spaces with nice geometric, dynamical, and topological properties, generalizing some of the characterizations of convex cocompactness in rank one.
This theory depends on a choice of a face $\tau_{mod}$ of the model Weyl chamber $\sigma_{mod}$ associated with the Riemannian symmetric space.
In \cite{KapovichLeebPortiv2,KapovichLeebPorti14}, they prove the equivalence of several properties for $\tau_{mod}$-nonelementary finitely generated discrete groups of isometries~$\Gamma$, namely
\begin{enumerate}[label=(\roman*),ref=\roman*]
  \item\label{item:KLP-RCA} $\tau_{mod}$-RCA (regularity, conicality, antipodality);
  \item\label{item:KLP-expand} $\tau_{mod}$-CEA (convergence, expansion, antipodality);
  \item\label{item:KLP-bound-map} word hyperbolicity, $\tau_{mod}$-regularity, and $\tau_{mod}$-asymptotical embeddedness;
  \item\label{item:KLP-Morse} word hyperbolicity and a $\tau_{mod}$-Morse property;
  \item\label{item:KLP-Ano} word hyperbolicity and a $\tau_{mod}$-Anosov property.
\end{enumerate}
In \cite{KapovichLeebPorti14_2}, using asymptotic cones, they prove that these properties are equivalent to
\begin{enumerate}[label=(\roman*),ref=\roman*]
\setcounter{enumi}{5}
  \item\label{item:KLP-URU} nondistorsion and asymptotic uniform $\tau_{mod}$-regularity.
\end{enumerate}

Let $G$ be the full isometry group of a Riemannian symmetric space.
In our language and with our notation, the choice of $\tau_{mod}$ is equivalent to the choice of a subset $\theta\subset\Delta$ of the simple restricted roots of~$G$, via the identification $\theta\mapsto \bigcap_{\alpha\in\Delta\smallsetminus\theta} \Ker(\alpha) \cap \overline{\aaa}^+$; in turn, this is equivalent to the choice of a standard parabolic subgroup $P_{\theta}$ of~$G$ (see Section~\ref{subsec:some-structure-semi}).
The $\tau_{mod}$-Anosov property in~\eqref{item:KLP-Ano} is the $P_{\theta}$-Anosov property of the present paper (see Definition~\ref{defi:ano1}). 
Condition~\eqref{item:KLP-URU} means, in our language, that there exist $c,C>0$ such that $\langle\alpha,\mu(\rho(\gamma))\rangle\geq c\,\ellGamma{\gamma}-C$ for all $\alpha\in\theta$ and $\gamma\in\Gamma$.
The implications $\eqref{item:Ano}\Leftrightarrow\eqref{item:away-from-walls}\Rightarrow\eqref{item:lin-away-from-walls}$ in Theorem~\ref{thm:char_ano} are analogous to the implications~\eqref{item:KLP-bound-map}\,$\Leftrightarrow$\,\eqref{item:KLP-Ano}$\Rightarrow$\,\eqref{item:KLP-URU} above, after observing that the $\tau_{mod}$-conicality requirement in $\tau_{mod}$-asymptotic embeddedness in \eqref{item:KLP-bound-map} is always satisfied when the continuous equivariant boundary maps are dynamics-preserving and transverse.
The implication $\eqref{item:cli}\Rightarrow\eqref{item:Ano}$ in Theorem~\ref{thm:char_ano} follows from \eqref{item:KLP-URU}\,$\Rightarrow$\,\eqref{item:KLP-Ano} above.
We do not see any direct link between the implication $\eqref{item:Ano}\Rightarrow\eqref{item:cli}$ and the characterizations above in general, since \eqref{item:cli} involves $\Sigma^+_{\theta}$ and not~$\theta$. 

We note that characterizations \eqref{item:KLP-RCA}, \eqref{item:KLP-expand}, \eqref{item:KLP-URU} of Kapovich--Leeb--Porti do not assume the discrete group to be hyperbolic a priori.

In \cite{KapovichLeebPorti14_2} the authors also develop a notion of Morse action on Euclidean buildings, in particular Bruhat--Tits buildings of $p$-adic groups; compare our Remark~\ref{rem:intro-char-Ano-bis}.\eqref{item:p-adic}.

%%%%%%%%%%%%%%%%%%%%%%%%%
\subsection{Characterizations of Anosov representations in terms of the Lyapunov projection}

We also establish new characterizations of Anosov representations that
are analogous to Theorem~\ref{thm:char_ano} but involve the
\emph{Lyapunov projection} $\lambda: g \mapsto \lim_n \mu(g^n)/n$
associated with the Jordan decomposition in~$G$ (see Section~\ref{subsec:proximal}),
and the \emph{stable length} \(\gamma \mapsto \ellinfty{\gamma} = \lim_n \ellGamma{\gamma^n}/n\) (see \eqref{eqn:stablelength}).

\begin{theorem}\label{thm:char_ano_lambda_intro}
Let $\Gamma$ be a word hyperbolic group, $G$ a real reductive Lie group, and $\theta \subset \Delta$ a nonempty subset of the simple restricted roots of~$G$.
For any representation $\rho : \Gamma \to G$, the following conditions are equivalent:
\begin{enumerate}
  \item\label{item:Ano-lambda} $\rho$ is $P_\theta$-Anosov;
  \item\label{item:away-from-walls-lambda} There exist continuous, $\rho$-equivariant, dynamics-preserving, and transverse maps $\xi^+ : \partial_{\infty}\Gamma\to G/P_{\theta}$ and $\xi^- : \partial_{\infty}\Gamma\to G/P_{\theta}^-$, and for any $\alpha\in\theta$ we have $\langle\alpha,\lambda(\rho(\gamma))\rangle\to+\infty$ as $\ellinfty{\gamma} \to +\infty$;
  \item\label{item:lin-away-from-walls-lambda} There exist continuous, $\rho$-equivariant, dynamics-preserving, and transverse maps $\xi^+ : \partial_{\infty}\Gamma\to G/P_{\theta}$ and $\xi^- : \partial_{\infty}\Gamma\to G/P_{\theta}^-$, and a constant $c>0$ such that $\langle\alpha,\lambda(\rho(\gamma))\rangle\geq c\,\ellinfty{\gamma}$ for all $\alpha\in\theta$ and $\gamma\in\Gamma$.
\end{enumerate}
\end{theorem}

A $P_{\theta}$-Anosov representation $\rho : \Gamma\to G$ is not necessarily semisimple, even when $P_{\theta}$ is a minimal parabolic subgroup of~$G$ (Remark~\ref{rem:Ano-not-ss}).
In the course of the proof of Theorem~\ref{thm:char_ano_lambda_intro} we establish the following result, which is of independent interest.
(See Section~\ref{subsubsec:rho-non-ss} for a definition of the semisimplification.)

\begin{proposition}\label{prop:Ano_iff_ssAno}
  Let $\Gamma$ be a word hyperbolic group, $G$ a real reductive Lie group, and $\theta \subset \Delta$ a nonempty subset of the simple restricted roots of~$G$.
Let  $\rho : \Gamma \to G$ be a representation and $\rho^{ss}$ its semisimplification.
Then
  \begin{center}
    $\rho$ is $P_\theta$-Anosov $\Longleftrightarrow$ $\rho^{ss}$ is $P_\theta$-Anosov.
  \end{center}
\end{proposition}

Since the character variety of $\Gamma$ in~$G$ can be viewed as the quotient of \(\Hom(\Gamma,G)\) by the relation ``having the same semisimplification'', Proposition~\ref{prop:Ano_iff_ssAno} means that the notion of being $P_{\theta}$-Anosov is well defined in the character variety.

%%%%%%%%%%%%%%%%%%%%%%%%%
\subsection{Anosov representations and proper actions}\label{subsec:intro-Ano-implies-proper}

The first applications of Anosov representations to proper actions on homogeneous spaces were investigated in \cite{Guichard_Wienhard_DoD} through constructions of domains of discontinuity.
Theorem~\ref{thm:char_ano} now provides a more direct link via the following properness criterion of Benoist and Kobayashi.

\smallskip 
\noindent
{\bf Properness criterion} \cite{Benoist_properness, TKobayashi_proper}: 
{\em Let $G$ be a reductive Lie group and $H,\Gamma$ two closed subgroups of~$G$.
Then $\Gamma$ acts properly on $G/H$ if and only if for any compact subset $\mathcal{C}$ of~$\aaa$ the intersection $(\mu(\Gamma) + \mathcal{C}) \cap \mu (H) \subset \aaa$ is compact. }

\smallskip

In other words, $\Gamma$ acts properly on $G/H$ if and only if the set $\mu(\Gamma)$ ``drifts away at infinity'' from~$\mu(H)$.
In this case the quotient $\Gamma \backslash G/H$ is an orbifold, sometimes called a \emph{Clifford--Klein form} of $G/H$.

Based on the properness criterion, a strengthening of the notion of proper discontinuity was introduced in \cite{KasselKobayashi}: a discrete subgroup $\Gamma < G$ is said to act \emph{sharply} (or \emph{strongly properly discontinuously}) on $G/H$ if the set $\mu(\Gamma)$ drifts away from $\mu(H)$ at infinity ``with a nonzero angle'', \ie there are constants $c, C >0$ such that for all $\gamma\in\Gamma$,
\begin{equation}\label{eqn:sharp}
d_{\aaa}(\mu(\gamma),\, \mu(H)) \geq c\,\Vert\mu(\gamma)\Vert - C,
\end{equation}
where $d_{\aaa}$ denotes the metric on $\aaa$ induced by $\Vert\cdot\Vert$.
The quotient $\Gamma \backslash G/H$ is said to be a sharp Clifford--Klein form. 
Many (but not all) properly discontinuous actions are sharp; the sharpness constants $(c,C)$ give a way to quantify this proper discontinuity.
Sharp actions are interesting for several reasons.
Firstly, 
they tend to be stable under small deformations, which is not true for general properly discontinuous actions.
Secondly, there are applications to spectral theory in the setting of affine symmetric spaces $G/H$: by \cite{KasselKobayashi}, if the discrete spectrum of the Laplacian on $G/H$ is nonempty (which is~equi\-valent to the rank condition $\mathrm{rank}\,G/H=\mathrm{rank}\,K/(K\cap H)$), then the discrete spectrum of the Laplacian is infinite on any sharp Clifford--Klein form $\Gamma\backslash G/H$.

Here is an immediate consequence of the implication $\eqref{item:Ano}\Rightarrow\eqref{item:lin-away-from-walls}$ of Theorem~\ref{thm:char_ano} and of \eqref{eqn:mu-leq-klength} below (which expresses the subadditivity of $\Vert\mu\Vert$). 

\begin{corollary}\label{cor:link-proper-Ano}
Let $\Gamma$ be a word hyperbolic group, $G$ a real reductive Lie group, and $\theta \subset \Delta$ a nonempty subset of the simple restricted roots of~$G$.
For any $P_{\theta}$-Anosov representation $\rho : \Gamma\to G$, the group $\rho(\Gamma)$ acts sharply (in particular, properly discontinuously) on $G/H$ for any closed subgroup $H$ of~$G$ such that $\mu(H)\subset\bigcup_{\alpha\in\theta} \Ker(\alpha)$.
\end{corollary}

Since the set of $P_{\theta}$-Anosov representations is open in $\Hom(\Gamma,G)$ \cite{Labourie_anosov,Guichard_Wienhard_DoD}, Corollary~\ref{cor:link-proper-Ano} provides sharp actions that remain sharp under any small deformation.

This corollary applies for instance to Hitchin representations of surface groups (which are Anosov with respect to a minimal parabolic
subgroup, \ie $\theta=\Delta$) and to maximal representations (which are Anosov with respect to a specific maximal proper parabolic subgroup, \ie $|\theta| = 1$).
We refer to Sections \ref{sec:hitch-repr} and~\ref{sec:maxim-repr} for more explanation.

\begin{corollary}\label{cor:Hitchin}
Let $(G,H)$ be a pair in Table~\ref{table1}.
For any Hitchin representation $\rho: \pi_1(\Sigma) \to G$, the group $\pi_1(\Sigma)$ acts sharply on $G/H$.

\medskip

\begin{table}[h!] 
\centering
\begin{tabular}{|c|c|c|c|}
\hline
& $G$ & $H$ & \rm{Conditions} \tabularnewline
\hline
{\rm (i)} & $\SL_d(\RR)$ & $\SL_k(\RR)$ & $k<d-1$\tabularnewline
{\rm (ii)} & $\SL_d(\RR)$ & $\SO(d-k,k)$ & $|d-2k|>1$\tabularnewline
{\rm (iii)} & $\SL_{2d}(\RR)$ & $\SL_d(\CC)\times\U(1)$ & \tabularnewline
{\rm (iv)} & $\SO(d,d)$ & $\SO(k,\ell)\times\SO(d-k,d-\ell)$ & $|k-\ell|>1$ \tabularnewline
{\rm (v)} & $\SO(d,d+1)$ & $\SO(k,\ell)\times\SO(d-k,d+1-\ell)$ & $\ell \notin \{k,k+1\}$\tabularnewline
{\rm (vi)} & $\SO(d,d)$ & $\GL_k(\RR)$ & $k<d-1$ \tabularnewline
{\rm (vii)} & $\SO(d,d+1)$ & $\GL_k(\RR)$ & $k<d$ \tabularnewline
{\rm (viii)} & $\SO(2d,2d)$ & $\U(d,d)$ & \tabularnewline
{\rm (ix)} & $\Sp(2d,\RR)$ & $\U(d-k,k)$ & \tabularnewline
{\rm (x)} & $\Sp(2d,\RR)$ & $\Sp(2k,\RR)$ & $k<d$ \tabularnewline
{\rm (xi)} & $\Sp(4d,\RR)$ & $\Sp(2d,\CC)$ & \tabularnewline
\hline
\end{tabular}
\vspace{0.2cm}
\caption{In these examples, \(0\leq k, \ell \leq d\) are any integers with \(d\geq
  2\), satisfying the specified conditions.}
\label{table1}
\end{table}
\end{corollary}

\begin{corollary}\label{cor:max}
Let $(G,H)$ be a pair in Table~\ref{table2}.
For any maximal representation $\rho: \pi_1(\Sigma) \to G$, the group $\pi_1(\Sigma)$ acts sharply on $G/H$.\end{corollary}

\begin{table}[h!] 
\centering
\begin{tabular}{|c|c|c|c|}
\hline
& $G$ & $H$ & \rm{Conditions} \tabularnewline
\hline
{\rm (i)} & $\SO(2,d)$ & $\U(1,k)$ & $k\leq d/2$ \tabularnewline
{\rm (ii)} & $\Sp(2d,\RR)$ & $\U(d-k,k)$ & $d \neq 2k$\tabularnewline
{\rm (iii)} & $\Sp(2d,\RR)$ & $\Sp(2k,\RR)$ & $k<d$ \tabularnewline
{\rm (iv)} & $\Sp(4d,\RR)$ & $\Sp(2d-2,\CC)$ & \tabularnewline
{\rm (v)} & $\SU(2d+1,2d+1)$ & $\SO^\ast(4d+2)$ & \tabularnewline
{\rm (vi)} & $\SU(p,q)$ & $\SU(k,\ell) \times \SU(p-k,q-\ell)$ & $(p-q)(k-\ell)<0$ \tabularnewline
{\rm (vii)} & $\SO^\ast(2d)$ & $\U(d-k,k)\times\U(1)$ & \tabularnewline
{\rm (viii)} & $\SO^\ast(4d)$ & $\SO^\ast(4d-2)$ & \tabularnewline
{\rm (ix)} & $E_{6(-14)}$ & $F_{4(-20)}$ & \tabularnewline
{\rm (x)} & $E_{7(-25)}$ & $E_{6(-14)}$ & \tabularnewline
{\rm (xi)} & $E_{7(-25)}$ & $\SU(6,2)$ & \tabularnewline
\hline
\end{tabular}
\vspace{0.2cm}
\caption{In these examples, $ k,\ell,d,p,q\in\NN$ are any integers with $d\geq 2$ and $k,\ell \leq d$ (as well as \(k\leq p\) and \(\ell\leq q\) in~\textrm{(vi)}), satisfying the specified conditions.}
\label{table2}
\end{table}

Applying Corollary~\ref{cor:link-proper-Ano}, it is easy to find many other examples with similar properties.

Conversely to Corollary~\ref{cor:link-proper-Ano}, we prove that certain properly discontinuous actions give rise to Anosov representations, using the implication $\eqref{item:cli}\Rightarrow\eqref{item:Ano}$ of Theorem~\ref{thm:char_ano}.
This works well, for instance, in the so-called \emph{standard} case, namely when the discrete group $\Gamma$ lies inside some Lie subgroup $G_1$ of~$G$ that itself acts properly on $G/H$; in this case the action of $\Gamma$ is automatically properly discontinuous, and even sharp \cite[Ex.\,4.10]{KasselKobayashi} (see Section~\ref{subsec:proof-cor-intro}).

\begin{corollary}\label{cor:standard-Ano}
Let $G$ be a real reductive Lie group, $\theta \subset \Delta$ a nonempty subset of the simple restricted roots of~$G$, and $H$ a closed subgroup of~$G$ such that $\mu(H) \supset \bigcup_{\alpha\in\theta} \Ker(\alpha) \cap \overline{\aaa}^+$.
Let $G_1$ be a reductive subgroup of~$G$, of real rank~$1$, acting properly on $G/H$.
Then for any convex cocompact subgroup $\Gamma$ of~$G_1$, the inclusion of $\Gamma$ into~$G$ is $P_{\theta}$-Anosov.
\end{corollary}

Recall that $\Gamma$ being convex cocompact in~$G_1$ means that there is a non\-empty, $\Gamma$-invariant, closed, convex subset $\mathcal{C}$ of the Riemannian symmetric space of~$G_1$ such that $\Gamma\backslash\mathcal{C}$ is compact.
Since $G_1$ has real rank~$1$, this is equivalent to $\Gamma$ being finitely generated and quasi-isometrically embedded in~$G_1$ (Remark~\ref{rem:rank1ccquasi}).

Corollary~\ref{cor:standard-Ano} applies in particular to the examples in Table~\ref{table3} below.
For more examples, including exceptional groups, see \cite{Kobayashi_Yoshino}.
In examples (i) to~(iv) when $k=d/2$, in example (v) when $\ell=d/4$, and in example (vii), the group $G_1$ acts cocompactly on $G/H$, hence compact Clifford--Klein forms of $G/H$ can be obtained by taking $\Gamma$ to be a uniform lattice in~$G_1$.
We refer to \cite{Okuda13, Bochenski_Jastrzebski_Tralle} for many more examples to which Corollary~\ref{cor:standard-Ano} applies, where $G_1$ is locally isomorphic to $\SL_2(\RR)$.

\begin{table}[h!]
\centering
\begin{tabular}{|c|c|c|c|c|}
\hline
& $G$ & $\theta$ & $H$ & $G_1$\tabularnewline
\hline
(i) & $\SO(2,d)$ & $\alpha_1$& $\SO(1,d)$ & $\U(1,k)$\tabularnewline
(ii) & $\SO(2,d)$ &$ \alpha _0$& $\U(1,k)$ & $\SO(1,d)$\tabularnewline
(iii) & $\U(2,d)$ & $\alpha_1$& $\U(1,d)$ & $\Sp(1,k)$\tabularnewline
(iv) & $\U(2,d)$ & $ \alpha _0$& $\Sp(1,k)$ & $\U(1,d)$\tabularnewline
(v) & $\SO(4,d)$ & $\alpha_1$& $\SO(3,d)$ & $\Sp(1,\ell)$\tabularnewline
(vi) & $\SO(4,d)$ &$ \alpha_0$& $\Sp(1,\ell)$ & $\SO(1,d)$\tabularnewline
(vii) & $\SO(8,8)$ & $\alpha_1$& $\SO(7,8)$ & $\Spin(8,1)$\tabularnewline
(viii) & $\SO(8,8)$ & $\alpha_0$& $\Spin(8,1)$ & $\SO(1,8)$\tabularnewline
\hline
\end{tabular}
\vspace{0.2cm}
\caption{In these examples, $d,k,\ell$ are any integers with $0<~\!\!\!k\leq d/2$ and $0<\ell\leq d/4$. We denote by $\alpha_0$ the simple root of~$G$ such that $P_{\alpha_0}$ is the stabilizer of an isotropic line, and by $\alpha_1$ the simple root of~$G$ such that $P_{\alpha_1}$ is the stabilizer of a maximal isotropic subspace.}
\label{table3}
\end{table}

The geometric construction of domains of discontinuity from \cite{Guichard_Wienhard_DoD} can be applied in many of the cases of Table~\ref{table3} to furthermore obtain compactifications of the corresponding Clifford--Klein forms $\Gamma \backslash G/H$.
These compactifications generalize for instance the conformal compactifications of Fuchsian and quasi-Fuchsian groups.  
This is the object of \cite{GGKW_compact}.

\begin{remarks}
\begin{enumerate}[label=(\alph*),ref=\alph*]
  \item The properness criterion of Benoist and Kobayashi also applies when $G$ is a reductive group over a non-Archimedean local field (\eg $\QQ_p$): see \cite{Benoist_properness}.
Corollaries \ref{cor:link-proper-Ano} and~\ref{cor:standard-Ano} also hold in this setting (see Remark~\ref{rem:intro-char-Ano-bis}.\eqref{item:p-adic}).
  \item From Corollaries \ref{cor:link-proper-Ano} and~\ref{cor:standard-Ano}, we recover the main result of \cite{Kassel_deformation}: in the setting of Corollary~\ref{cor:standard-Ano} (over $\RR$ or~$\QQ_p$), there is a neighborhood $\mathcal{U}\subset\Hom(\Gamma,G)$ of the natural inclusion such that for any $\varphi\in\mathcal{U}$ the group $\varphi(\Gamma)$ is discrete in~$G$ and acts properly discontinuously on $G/H$.
\end{enumerate}
\end{remarks}

%%%%%%%%%%%%%%%%%%%%%%%%%
\subsection{Proper actions on group manifolds}

For a Lie group~$G$, let $\Diag(G)$ be the diagonal of $G\times G$.
The homogeneous space $(G\times G)/\Diag(G)$ identifies with $G$ endowed with the transitive action of $G\times G$ by left and right translation, and is 
 called a \emph{group manifold}.
Using the full equivalence $\eqref{item:Ano}\Leftrightarrow\eqref{item:lin-away-from-walls}$ of Theorem~\ref{thm:char_ano}, as well as Theorem~\ref{thm:char_ano_lambda_intro}, we obtain a particularly satisfying characterization of quasi-isometrically embedded groups acting properly on $(G\times G)/\Diag(G)$ when $G$ is semisimple of real rank~$1$.
This covers in particular the cases of anti-de Sitter $3$-manifolds ($G=\PSL_2(\RR)$) and of Riemannian holomorphic complex $3$-manifolds with constant nonzero sectional curvature ($G=\PSL_2(\CC)$).

\begin{theorem}\label{thm:proper-GxG-rk1}
Let $G$ be a semisimple Lie group of real rank~$1$ and $\Gamma$ a finitely generated subgroup of $G\times G$.
Then the following are equivalent:
  \begin{enumerate}
  \item\label{item:proper-GxG-rk1} $\Gamma$ acts properly discontinuously on $(G\times G)
    / \Diag(G)$ and the inclusion $\Gamma \hookrightarrow G\times G$ is a quasi-isometric embedding,
    \item\label{item:sharp-GxG-rk1} $\Gamma$ acts sharply on $(G\times G)
    / \Diag(G)$ and the inclusion $\Gamma \hookrightarrow G\times G$ is a quasi-isometric embedding,
  \item\label{item:Gamma-j-rho} $\Gamma$ is word hyperbolic, of the form
        \[\Gamma = \{ ( \rho_L( \gamma), \rho_R( \gamma)) \mid \gamma \in \Gamma_0\} ,\]
 where $\rho_L,\rho_R : \Gamma_0\to G$ are representations and, up to switching the two factors of $G\times G$, the representation $\rho_L$ is convex cocompact and uniformly dominates~$\rho_R$.
  \end{enumerate}
\end{theorem}
Here we say that $\rho_L$ \emph{uniformly dominates}~$\rho_R$ if there exists $c<1$ such that for all $\gamma\in\Gamma_0$,
 \[ \lambda(\rho_R(\gamma)) \leq c \, \lambda(\rho_L(\gamma)), \]
where $\lambda : G\to\RRp$ is the translation length function in the Riemannian symmetric space $G/K$ of~$G$, given by $\lambda(g)=\inf_{x\in G/K} d(x,g\cdot x)$ for all $g\in G$. 

\begin{remark}
Theorem~\ref{thm:proper-GxG-rk1}, together with Corollary~\ref{cor:proper-open-GxG} below, was first established in \cite{Kassel_these} for $G = \PSL_2(\RR)\simeq\SO(1,2)_0$, then in \cite{Gueritaud_Kassel} for $G=\SO(1,d)$ with $d\geq 2$.
(For a $p$-adic version, with $G$ of relative rank~$1$ over a non-Archimedean local field, see \cite{Kassel_padic}.)
The fact that for a general Lie group $G$ of real rank~$1$, any discrete subgroup of $G\times G$ acting properly discontinuously on $(G\times G)/\Diag(G)$ is of the form $(\rho_L,\rho_R)(\Gamma_0)$ where $\rho_L$ or~$\rho_R$ is discrete with finite kernel, was proved in \cite{Kassel_corank1}: see Theorem~\ref{thm:kas-cork1} for a precise statement.
\end{remark}

To prove Theorem~\ref{thm:proper-GxG-rk1}, we relate conditions \eqref{item:proper-GxG-rk1}, \eqref{item:sharp-GxG-rk1}, \eqref{item:Gamma-j-rho} to the fact that $\Gamma$ is word hyperbolic and its natural inclusion inside some larger group containing $G\times\nolinebreak G$ is Anosov.
Such a relationship also exists, in a weaker form, when $G$ has higher real rank: a general statement is given in Theorem~\ref{thm:complete-proper-GxG} below.
Here we explain this relationship when $G=\Aut_{\KK}(b)$ is the group of automorphisms of a vector space over $\KK=\RR$ or \(\CC\) preserving a nondegenerate bilinear (symmetric or symplectic) form~$b$, or the group of automorphisms of a vector space over $\KK=\CC$ or $\bH$ (quaternions) preserving a nondegenerate (Hermitian or anti-Hermitian) form~$b$ --- a situation that includes all classical simple groups of real rank~$1$, namely $\SO(1,d)$, $\SU(1,d)$, $\Sp(1,d)$ (Example~\ref{ex:SO(1,d)}). 

\begin{theorem} \label{thm:Anosov=proper}
For $\KK = \RR$, $\CC$, or~$\bH$, let $V$ be a $\KK$-vector space and $b : V \otimes_{\RR} V \to \KK$ a nondegenerate $\RR$-bilinear form which is symmetric, antisymmetric, Hermitian, or anti-Hermitian over~$\KK$, with $G:=\Aut_{\KK}(b)$ noncompact.
Let $Q_0(b\oplus b)$ be the stabilizer in $\Aut_{\KK}(b \oplus b)$ of a $(b\oplus b)$-isotropic line in $V\oplus V$, and similarly for $b\oplus (-b)$.
For a discrete subgroup $\Gamma$ of $G\times G$, the following three conditions are equivalent:
\begin{enumerate}
\setcounter{enumi}{2}
  \item \label{item:Gamma-j-rho-Ano}
$\Gamma$ is word hyperbolic, of the form \[\Gamma = \{ ( \rho_L( \gamma), \rho_R( \gamma)) \mid \gamma \in \Gamma_0\} ,\]  where $\rho_L,\rho_R : \Gamma_0\to G$ are representations and, up to switching the two factors of $G\times G$, the representation $\rho_L$ is $Q_0(b)$-Anosov and uniformly $Q_0(b)$-dominates~$\rho_R$ (see Definition~\ref{def:dominate});
  \item \label{item:Q0bb-Ano} $\Gamma$ is word hyperbolic and the natural inclusion
  \[ \Gamma \longhookrightarrow G \times G = \Aut_{\KK}(b) \times \Aut_{\KK}(b) \longhookrightarrow \Aut_{\KK}(b \oplus b) \]
  is $Q_0(b\oplus b)$-Anosov;
  \item \label{item:Q0b-b-Ano} $\Gamma$ is word hyperbolic and the natural inclusion
  \[ \Gamma \longhookrightarrow G \times G = \Aut_{\KK}(b) \times \Aut_{\KK}(-b) \longhookrightarrow \Aut_{\KK}(b \oplus (-b)) \]
  is $Q_0(b\oplus (-b))$-Anosov.
\end{enumerate}
If \eqref{item:Gamma-j-rho-Ano}, \eqref{item:Q0bb-Ano}, or \eqref{item:Q0b-b-Ano} holds, then \eqref{item:proper-GxG-rk1} and \eqref{item:sharp-GxG-rk1} of Theorem~\ref{thm:proper-GxG-rk1} hold.
The converse is true if and only if $G$ has real rank $1$. 
\end{theorem}

We refer to Remark~\ref{rem:GxG-false-higher-rk} for an explanation of why \eqref{item:sharp-GxG-rk1} does not imply \eqref{item:Q0bb-Ano} when $G$ has higher real rank.

Even though $\Aut_{\KK}(b) = \Aut_{\KK}(-b)$, the embeddings in \eqref{item:Q0bb-Ano} and \eqref{item:Q0b-b-Ano} are in general quite different.
For instance, for $\Aut_{\KK}(b) = \OO(1,d)$,  these embeddings are $\Gamma \hookrightarrow \OO(1,d) \times \OO(1,d) \hookrightarrow \OO(2,2d)$ and $\Gamma \hookrightarrow \OO(1,d) \times \OO(1,d) \simeq \OO(1,d) \times \OO(d,1) \hookrightarrow \OO(d+1, d+1)$.

Here are two consequences of Theorems~\ref{thm:proper-GxG-rk1} and~\ref{thm:Anosov=proper} (and their refinement, Theorem~\ref{thm:complete-proper-GxG}); the second one uses the fact that being Anosov is an open property.

\begin{corollary}\label{cor:sharpness-conj-GxG-rk1}
Let $G$ be a semisimple Lie group of real rank~$1$ and $\Gamma$ a discrete subgroup of $G\times G$.
If the action of $\Gamma$ on $(G\times G) / \Diag(G)$ is properly discontinuous and cocompact, then it is in fact sharp. 
\end{corollary}

\begin{corollary}\label{cor:proper-open-GxG}
Let $G$ be a semisimple Lie group of real rank~$1$ and $\Gamma$ a finitely generated quasi-isometrically embedded subgroup of $G\times G$.
If $\Gamma$ acts properly discontinuously on $(G\times G)/\Diag(G)$, then there is a neighborhood $\mathcal{U}\subset\Hom(\Gamma,G\times G)$ of the natural inclusion such that any $\rho\in\mathcal{U}$ is a quasi-isometric embedding from $\Gamma$ to $G\times G$, and $\Gamma$ acts properly discontinuously on $(G\times G)/\Diag(G)$ via~$\rho$. 
If moreover the action of $\Gamma$ on $(G\times G)/\Diag(G)$ is cocompact, then $\Gamma$ also acts cocompactly on $(G\times G)/\Diag(G)$ via~$\rho$.
\end{corollary}

\begin{remark} \label{rem:complete-struct}
For $G$ semisimple of real rank~$1$, Corollary~\ref{cor:proper-open-GxG} together with \cite[Th.\,3]{Tholozan} implies that the space of \emph{complete} $(G\times G, (G\times G)/\Diag(G))$-structures on a compact manifold $M$ is a union of connec\-ted components of the space of $(G\times G,\linebreak (G\times G)/\Diag(G))$-structures on~$M$.
\end{remark}

%%%%%%%%%%%%%%%%%%%%%%%%%
\subsection*{Conventions}

In the whole paper, we assume the reductive group $G$ to be noncompact, equal to a finite union of connected components (for the real topology) of $\mathbf{G}(\RR)$ for some algebraic group~$\mathbf{G}$.
We set $\RRp:=[0,+\infty)$, as well as $\NN:=\ZZ\cap\RRp$ and $\NN^{\ast}:=\NN\smallsetminus\{ 0\}$.

%%%%%%%%%%%%%%%%%%%%%%%%%
\subsection*{Organization of the paper}

In Section~\ref{sec:preliminaries} we review some background material on word hyperbolic groups, the structure of reductive Lie groups, proximality, and Anosov representations, and establish some basic preliminary results.
In Section~\ref{sec:repr-semis-lie} we explain how one can always reduce to Anosov representations into $\GL(V)$.
In Section~\ref{sec:anos-repr-cart} we prove the equivalences $\eqref{item:Ano}\Leftrightarrow\eqref{item:away-from-walls}\Leftrightarrow\eqref{item:lin-away-from-walls}$ of Theorems \ref{thm:char_ano} and~\ref{thm:char_ano_lambda_intro} (characterizations of Anosov representations assuming the existence of boundary maps).
In Section~\ref{sec:boundary} we give a point-by-point construction of boundary maps (proving Theorem~\ref{thm:constr-xi}) and establish the equivalence $\eqref{item:Ano}\Leftrightarrow\eqref{item:cli}$ of Theorem~\ref{thm:char_ano}.
In Section~\ref{sec:gen-proper} we provide short proofs of Corollaries \ref{cor:Hitchin}, \ref{cor:max}, and~\ref{cor:standard-Ano}.
The link between Anosov representations and proper actions on group manifolds is established in Section~\ref{sec:proper}, where we prove Theorems \ref{thm:proper-GxG-rk1} and~\ref{thm:Anosov=proper} as well as Corollaries \ref{cor:sharpness-conj-GxG-rk1} and~\ref{cor:proper-open-GxG}.

%%%%%%%%%%%%%%%%%%%%%%%%%
\subsection*{Acknowledgements}

This work started during a visit of F.G.\ and F.K.\ to Princeton in April 2011, and continued through several meetings at the IHP (Paris), Princeton University, CIRM (Luminy), Universit\'e de Strasbourg, Universit\'e Lille~1, University of Maryland, Universit\"at Heidelberg, Caltech, and MSRI (Berkeley).
We are grateful to these institutions for their hospitality. 
We also thank the research network ``Geometric structures and representation varieties'' (GEAR), funded by the NSF under the grants DMS 1107452, 1107263, and 1107367, for providing several opportunities for us to meet.

We would like to thank Yves Benoist, Marc Burger, Dick Canary, Misha Kapovich, Fran\c{c}ois Labourie, Sara Maloni, and Joan Porti for interesting discussions, as well as an anonymous referee for many valuable comments.

%%%%%%%%%%%%%%%%%%%%%%%%%%%%%%%%%%%%%%%%%%%%%%%%%%%
\section{Preliminaries on Anosov representations and the structure of reductive Lie groups}
\label{sec:preliminaries}

In this section we set up notation and recall some definitions and useful facts about word hyperbolic groups~$\Gamma$, real reductive Lie groups~$G$, and Anosov representations $\rho : \Gamma\to G$.

%%%%%%%%%%%%%%%%%%%%%%%%%
\subsection{Word hyperbolic groups and their boundary at infinity}
\label{subsec:hyperb-groups-their}

Recall that a finitely generated group $\Gamma$, with finite generating set $S\subset\Gamma$, is said to be \emph{word hyperbolic} if its Cayley graph $\mathcal{C}(\Gamma, S)$, equipped with the natural graph metric, is Gromov hyperbolic.
The induced metric on~\(\Gamma\) is the one coming from the word~length~\(|\cdot\nolinebreak|_\Gamma\).

%%%%%%%%
\subsubsection{The boundary at infinity} \label{subsubsec:boundary-Gamma}
 
Let $c,C>0$.
A map \(f: (X,d) \to (X', d')\) between metric spaces is a \emph{\((c,C)\)-quasi-isometric embedding} if for all \(x,y\in X\), 
\[c^{-1} d(x,y) -C \leq d'(f(x), f(y)) \leq c \, d(x,y) + C.\]
 It is a \emph{quasi-isometry} if furthermore there exists \(R\geq 0\) such that for any \(x'\in X'\) we can find \(x\in X\) with \(d'(x', f(x))\leq R\).
When \(X=\NN\), a \((c,C)\)-quasi-isometric embedding is called a \emph{quasi-geodesic ray}. 
When \(X'\) is the Cayley graph of~\(\Gamma\), a sequence $(\gamma_n) \in \Gamma^{\NN}$ defines a \emph{$(c,C)$-quasi-geodesic ray} in the Cayley graph of~$\Gamma$ if for all $n,m\in\NN$,
\[ c^{-1} |n-m| -C \leq \ellGamma{\gamma_{n}^{-1} \gamma_{m}} \leq c \, |n-m| + C. \]

If $X$ and~$X'$ are geodesic metric spaces and if $f: X \to X'$ is a quasi-isometry, then $X$ is Gromov hyperbolic if and only if $X'$ is, and in this case $f$ induces a homeomorphism $\partial_{\infty} f : \partial_{\infty} X \to \partial_{\infty} X'$ between the visual boundaries \cite[Ch.\,III, Th.\,2.2]{Coornaert_Delzant_Papadopoulos}.
This fundamental fact has the following consequences:
\begin{enumerate}[label=(\roman*),ref=\roman*]
  \item the word hyperbolicity of the group~$\Gamma$ does not depend on the choice of finite generating set~$S$;
  \item the boundary at infinity $\partial_{\infty} \Gamma = \partial_{\infty} \mathcal{C}(\Gamma,S)$ is well defined and $\Gamma$ acts on it by homeomorphisms.
\end{enumerate}

The word hyperbolic group $\Gamma$ acts on $\partial_{\infty} \Gamma$ as a \emph{uniform convergence group} (see \eg \cite{Bowditch_tophyp}), which means that it acts properly discontinuously and cocompactly on the set of triples of pairwise distinct elements of $\partial_{\infty} \Gamma$.
As a con\-sequence, it satisfies the following dynamical properties:
\begin{fact}\label{fact:dyn_at_infty} 
  \begin{enumerate}
  \item\label{item:4} For any sequence $(\gamma_n)\in \Gamma^\NN$
    going to infinity, there exist
    $\eta, \eta' \in \partial_{\infty} \Gamma$ (possibly equal) and a
    subsequence $(\gamma_{\phi(n)})_{n\in\NN}$ such that
    $\gamma_{\phi(n)} |_{\partial_{\infty} \Gamma \smallsetminus \{
      \eta'\}}$
    converges, in the compact-open topology, to the constant map with
    image~$\{\eta\}$.
  \item \label{item:1} For any $\gamma \in \Gamma$ of infinite order,
    there exist $\eta^+_{\gamma} \neq \eta^-_{\gamma}$ in
    $\partial_{\infty} \Gamma$ such that
    $\lim_{n\to +\infty} \gamma^n \cdot \eta = \eta^+_{\gamma}$ for
    all $\eta \neq \eta^-_{\gamma}$ and
    $\lim_{n\to +\infty} \gamma^{-n} \cdot \eta = \eta^-_{\gamma}$ for
    all $\eta \neq \eta^+_{\gamma} $.
  \item \label{item:3} The pairs $(\eta^{+}_{\gamma}, \eta^{-}_{\gamma})$ of attracting and repelling fixed points of elements $\gamma\in\Gamma$ of infinite order form a dense subset of $(\partial_{\infty} \Gamma \times \partial_{\infty} \Gamma) \smallsetminus \mathrm{Diag}(\partial_\infty\Gamma)$.
  \item \label{item:9} If $\Gamma$ is nonelementary (\ie if $\# \partial_{\infty} \Gamma >2$, \ie if $\Gamma$ is not virtually cyclic), then the action of $\Gamma$ on $\partial_{\infty}\Gamma$ is \emph{minimal} (\ie every nonempty $\Gamma$-invariant subset is dense).
  \end{enumerate}
\end{fact}

%%%%%%%%
\subsubsection{Word length, stable length, and translation length}
\label{subsubsec:word-length-stable}

Associated with the word length function $\ellGammaf : \Gamma \to \NN$ is the \emph{stable length} function $\ellinfty{\cdot} : \Gamma \to \RR$, given by
\begin{equation} \label{eqn:stablelength}
\ellinfty{\gamma} = \lim_{n\to +\infty}\, \frac{1}{n} \, \ellGamma {\gamma^n}
 \end{equation}
for all $\gamma\in\Gamma$.
It is easily seen to be invariant under conjugation: $\ellinfty{\beta\gamma \beta^{-1}} = \ellinfty{\gamma}$ for all $\beta,\gamma\in\Gamma$.
Moreover, it is related as follows to the translation length function on the Cayley graph
\begin{equation} \label{eqn:translationlength}
\gamma \longmapsto \trl{\gamma} = \inf_{ \beta \in \Gamma} \ellGamma {\beta \gamma \beta^{-1}}.
\end{equation}

\begin{proposition}[{\cite[Ch.\,X, Prop.\,6.4]{Coornaert_Delzant_Papadopoulos}}]\label{prop:word-length-stable}
  If the group $\Gamma$ is $\delta$-hyperbolic, then $\trl{\gamma} -16 \delta \leq \ellinfty{\gamma}
  \leq \trl{\gamma}$ for all $\gamma\in\Gamma$.
\end{proposition}

%%%%%%%%
\subsubsection{The flow space}
\label{subsubsec:flow-space-an}

An important object for the definition of an Anosov representation below and for some proofs in this paper is the \emph{flow space} of the word hyperbolic group~$\Gamma$.
It is a proper metric space $\mathcal{G}_\Gamma$ with the following properties:
\begin{enumerate}
  \item $\mathcal{G}_\Gamma$ is Gromov hyperbolic.
  \item $\mathcal{G}_\Gamma$ is equipped with a properly discontinuous and cocompact action of~$\Gamma$ by isometries.
  In particular, any orbit map $\gamma \mapsto \gamma \cdot v$ from $\Gamma$ to $\mathcal{G}_\Gamma$ is a quasi-isometry, and $\partial_{\infty} \Gamma$ is equivariantly homeomorphic to $\partial_{\infty} \mathcal{G}_\Gamma$.
\item\label{item:10} $\mathcal{G}_\Gamma$ is equipped with a flow $\{\varphi_t\}_{t\in \RR}$ (\ie a continuous $\RR$-action) which commutes with the $\Gamma$-action and for which there exist $c, C >0$ such that any orbit $\RR \to \mathcal{G}_\Gamma$ of the flow is a $(c,C)$-quasi-isometric embedding.
  This implies the existence of two continuous maps
  \begin{align*}
    \varphi_{\pm \infty} :  \mathcal{G}_\Gamma & \longrightarrow \hspace{0.6cm} \partial_{\infty} \Gamma \\
     v  & \longmapsto  \lim_{t\to\pm\infty} \varphi_t\cdot v
  \end{align*}
  associating to $v \in \mathcal{G}_\Gamma$ the endpoints of its orbit.
\item $\mathcal{G}_\Gamma$ is equipped with an isometric $\ZZ/2\ZZ$-action commuting with
  $\Gamma$ and anticommuting with $\RR$.
\item The natural map
    \[(\varphi_{+\infty}, \varphi_{-\infty}) :\ \RR\backslash\mathcal{G}_\Gamma \longrightarrow (\partial_{\infty} \Gamma \times \partial_{\infty} \Gamma) \smallsetminus \mathrm{Diag}(\partial_\infty\Gamma) \]
  is a homeomorphism.
\end{enumerate}
The flow space was constructed by Gromov \cite[Th.\,8.3.C]{Gromov_hyp}, and more details were provided by Champetier \cite[\S\,4]{Champetier}.
Mineyev \cite{Mineyev_flow} introduced a different construction of the flow
space of a hyperbolic graph with bounded valency (not necessarily coming with a group action).
It is based on the existence of a hyperbolic metric $\hat{d}$ on the graph satisfying some subtle properties (see \cite[Th.\,26]{Mineyev_flow} and \cite[Th.\,17]{MineyevYu}): when applied to a Cayley graph of $\Gamma$ these yield a  space $\mathcal{G}_\Gamma$ as above.
In Mineyev's version the $\RR$-orbits are geodesics and not only quasi-geodesics.
There is also a uniqueness statement for the flow space $\mathcal{G}_\Gamma$ as a $\Gamma \times ( \RR \rtimes \ZZ/2\ZZ)$-space up to quasi-isometry and up to reparameterization of the $\RR$-orbits, but we shall not need~it.

\begin{remark}\label{rem:flow-space-dense}
It follows from Fact~\ref{fact:dyn_at_infty}.\eqref{item:3} that the union of the periodic geodesics of the flow $\{\varphi_t\}_{t\in\RR}$ is dense in~$\mathcal{G}_{\Gamma}$.
\end{remark}

If \(\Gamma_1\to \Gamma_2\) is a homomorphism with finite kernel and finite-index image and if \(\Gamma_2\) is word hyperbolic with flow space \(\mathcal{G}_{\Gamma_2}\), then \(\Gamma_1\) is finitely generated and word hyperbolic and a flow space for \(\Gamma_1\) is \(\mathcal{G}_{\Gamma_2}\) with the action of \(\Gamma_1\) induced by the homomorphism \(\Gamma_1\to \Gamma_2\).

%%%%%%%%
\subsubsection{Negatively curved Riemannian manifolds}
\label{subsubsec:negat-curv-riem}

For a large class of word hyperbolic groups~$\Gamma$ (including the fundamental groups of closed negatively-curved Riemannian manifolds), the flow space $\mathcal{G}_{\Gamma}$ has a simple geometric interpretation:

\begin{fact}\label{fact:flow}
  Let $X$ be a simply connected Riemannian manifold with sectional curvature bounded above by $-a^2$, for some $a \neq 0$, and let $\rho:\Gamma \to \Isom(X)$ be a homomorphism with finite kernel and convex cocompact image.
  Then \(\Gamma\) is finitely generated, word hyperbolic, and a flow space of~$\Gamma$ is given by
  \begin{align*}
    \mathcal{G}_\Gamma & = \{ v \in T^1(X) \mid (\varphi_{+\infty},\varphi_{-\infty})(v) \in \Lambda_{\rho(\Gamma)} \times \Lambda_{\rho(\Gamma)} \}\\
    & = \{v \in T^1(X) \mid \forall t \in \RR, \ \pi(\varphi_t\cdot v) \in \mathcal{C}_{\rho(\Gamma)}\}
  \end{align*}
  with its natural $\Gamma \times (\RR \rtimes \ZZ / 2 \ZZ)$-action.
\end{fact}

Here $\Isom(X)$ is the group of isometries of~$X$ and $\pi : T^1(X)\to X$ the natural projection.
We denote by $\Lambda_{\rho(\Gamma)}$ the limit set of $\rho(\Gamma)$ in~$\partial_{\infty} X$, which is by definition the closure in~$\partial_{\infty} X$ of any $\Gamma$-orbit in~$X$.
By convex cocompact we mean that $\Gamma$ acts properly discontinuously and cocompactly, via~$\rho$, on the convex hull $\mathcal{C}_{\rho(\Gamma)}\subset X$ of $\Lambda_{\rho(\Gamma)}$.
In this case $\partial_{\infty}\Gamma$ is homeomorphic to the limit set~$\Lambda_{\rho(\Gamma)}$.

  This example illustrates the nonuniqueness of the flow space as a metric space, since a given convex cocompact subgroup of $\Isom(X)$ can have nontrivial deformations.

\begin{remark}
In Corollaries \ref{cor:Hitchin}, \ref{cor:max}, \ref{cor:standard-Ano}, Theorem~\ref{thm:proper-GxG-rk1}, and Corollaries \ref{cor:sharpness-conj-GxG-rk1}, \ref{cor:proper-open-GxG}, the group~$\Gamma$ falls in the setting of Fact~\ref{fact:flow}.
\end{remark}

%%%%%%%%
\subsubsection{Geodesics in $\Gamma$ and in its flow space}
\label{subsubsec:geodesics-gamma-flow}

We make the following definition.

\begin{definition}\label{defi:cli}
A sequence $(x_n)\in\RR^{\NN}$ is \emph{CLI} (\ie has \emph{coarsely linear increments}) if $n\mapsto x_n$ is a quasi-isometric embedding 
 of $\NN$ into $[a,+\infty)$ for some $a\in\RR$, \ie there exist $\kappa,\kappa',\kappa^{\prime \prime},\kappa^{\prime \prime \prime}>0$ such that for all $n,m\in\NN$,
\[ \kappa m - \kappa' \leq x_{n+m} - x_n \leq \kappa^{\prime \prime} m + \kappa^{\prime \prime\prime}.\]
In this case we say that $(x_n)_{n\in\NN}$ is \emph{$(\kappa, \kappa')$-lower CLI}.
\end{definition}

\begin{remark}
For a positive sequence $(x_n)_{n\in\NN}$, the property of being CLI is stronger than the property of growing linearly as $n\to +\infty$.
For instance, if \(f : \NN \to (\RRp)^2\) is a quasi-isometric embedding whose image zigzags vertically and horizontally between the lines \(y=x\) and \(y=2x\), then the composition of~$f$ with either of the two projections $(\RRp)^2\to\RRp$ grows linearly but is not CLI. 
\end{remark}

The following result will be used several times throughout the paper.

\begin{proposition}\label{prop:QIflowbis}
Let $\Gamma$ be a word hyperbolic group with flow space~$\mathcal{G}_\Gamma$.
For any $c,C>0$, there exist a compact subset $\mathcal{D}$ of~$\mathcal{G}_\Gamma$ and constants $\kappa, \kappa' >0$ with the following property: for any $(c,C)$-quasi-geodesic ray $(\gamma_n)_{n\in\NN}$ with $\gamma_0=e$ in the Cayley graph of~$\Gamma$, there exist $v \in \mathcal{D}$ and a $(\kappa, \kappa')$-lower CLI sequence $(t_n)\in\RR^{\NN}$ such that $\varphi_{t_n}\cdot v \in \gamma_n\cdot\mathcal{D}$ for all $n\in\NN$. 
\end{proposition}

\begin{proof}
Any $(c,C)$-quasi-geodesic ray $(\gamma_n)_{n\in\NN}$ with $\gamma_0=e$ can be extended to a full uniform quasi-geodesic $(\gamma_n)_{n\in\ZZ}$ in the Cayley graph of~$\Gamma$.
Let $\psi: \Gamma \rightarrow \mathcal{G}_\Gamma$ be an orbit map; it is a quasi-isometry.
By hyperbolicity, $(\psi(\gamma_n))_{n\in\ZZ}$ lies within uniformly bounded Hausdorff distance $R>0$ from the $\RR$-orbit in~$\mathcal{G}_\Gamma$ with the same endpoints at infinity.
Let us write this $\RR$-orbit as $(\varphi_t\cdot v)_{t\in\RR}$ where $v$ lies at distance $\leq R$ from $\psi(e)$.  
For any $n\in\NN$, the point $\psi(\gamma_n)\in\mathcal{G}_\Gamma$ lies at distance $\leq R$ from $\varphi_{t_n}\cdot v$ for some $t_n\in\RR$, and the sequence $(t_n)_{n\in\NN}$ is CLI because $\psi$ is a quasi-isometry and the \(\RR\)-orbits are quasi-isometric embeddings.
The lower CLI constants of $(t_n)_{n\in\NN}$ depend only on $(c,C)$ and on the quasi-isometry constants of~$\psi$.
\end{proof}

\begin{corollary}\label{cor:geo_seg_ell}
  Let $\Gamma$ be a word hyperbolic group.
  Then there exist a compact subset $\mathcal{D}$ of~$\mathcal{G}_\Gamma$ and constants $c_1,c_2>0$ with the following property: for any $\gamma\in\Gamma$ there exist $v \in \mathcal{D}$ and $t\geq 0$ such that $\varphi_t \cdot v \in \gamma \cdot \mathcal{D}$ and $t\geq c_1 \ellGamma{\gamma} -c_2$.
\end{corollary}

\begin{proof}
Any $\gamma\in\Gamma$ belongs to a uniform quasi-geodesic $(\gamma_n)_{n\in \NN}$ with $\gamma_0=e$ and $\gamma_{\ellGamma{\gamma}} =\gamma$.
We conclude using Proposition~\ref{prop:QIflowbis}.
\end{proof}

%%%%%%%%%%%%%%%%%%%%%%%%%
\subsection{Parabolic subgroups of reductive Lie groups} \label{subsec:some-structure-semi}

We now recall the necessary Lie-theoretic background. 
Let $G$ be a noncompact real reductive Lie group.
We assume that $G$ is a finite union of connected components (for the real topology) of $\mathbf{G}(\RR)$ for some algebraic group~$\mathbf{G}$.
For simplicity, we assume that the adjoint action of \(G\) on its Lie algebra \(\g\) is by inner automorphisms, \ie \(\Ad(G)\subset\Aut(\g)_0\); this is the case for instance if $G$ is connected.
Recall that $G$ is the almost product of $Z(G)_0$ and $G_s$, where $Z(G)_0$ is the identity component (for the real topology) of the center $Z(G)$ of~$G$, and $G_s=D(G)$ is the derived subgroup of~$G$, which is semisimple.

%%%%%%%%
\subsubsection{Parabolic subgroups}\label{subsubsec:parabolic-subgroups}

By definition, a \emph{parabolic subgroup} of~$G$ is a subgroup of the form $P=G\cap\mathbf{P}(\RR)$ for some algebraic subgroup $\mathbf{P}$ of~$\mathbf{G}$ with $\mathbf{G}(\RR)/\mathbf{P}(\RR)$ compact.

\begin{definition}\label{defi:transverse-compatible}
Two parabolic subgroups $P$ and $Q$ are said to be
\begin{itemize}
  \item \emph{transverse} (or \emph{opposite}) if their intersection is a reductive subgroup;
  \item \emph{compatible} (or in \emph{singular position}) if their intersection is a parabolic subgroup.
\end{itemize}
\end{definition}

When $G$ has real rank~$1$, two proper parabolic subgroups are either transverse (\ie distinct) or compatible (\ie equal), but when $G$ has higher real rank there are other cases between these two extremes.

\begin{remark} \label{rem:pnorm}
Any parabolic subgroup $P$ is its own normalizer in~$G$, hence $G/P$ identifies (as a $G$-set) with the set of conjugates of $P$ in~$G$.
In the sequel, we shall make no distinction between elements of $G/P$ and parabolic subgroups.
In particular, the terminology \emph{transverse} and \emph{compatible} will be used for elements of $G/P \times G/Q$.
\end{remark}

\begin{remark}
Let $X=G/K$ be the Riemannian symmetric space of~$G$; it has nonpositive curvature and its visual boundary $\partial_{\infty} X$ is a sphere.
Geome\-trically, a proper parabolic subgroup of~$G$ is the stabilizer in~$G$ of a (not necessarily unique) point $\xi \in \partial_{\infty} X$.
Two proper parabolic subgroups $P$ and~$Q$ are transverse if and only if there is a bi-infinite geodesic $c : \RR\to X$ such that $P= \mathrm{Stab}_G (\lim_{+\infty} c)$ and $Q= \mathrm{Stab}_G (\lim_{-\infty}c)$.
\end{remark}

\begin{example}\label{ex:transv-GL}
Let $\KK$ be $\RR$, $\CC$, or the ring $\bH$ of quaternions, and let $G$ be $\GL_{\KK}(V)$ for some (right) $\KK$-vector space~$V$.
Any parabolic subgroup of~$G$ is the stabilizer in~$G$ of a partial flag of $\KK$-subspaces of~$V$.
Two parabolic subgroups are transverse if and only if the corresponding flags $\{0\}=V_0\subsetneq \dots\subsetneq V_r=V$ and $\{0\}=W_0\subsetneq \dots\subsetneq W_{s}=V$ 
satisfy $r=s$ and $V=V_i\oplus W_{r-i}$ for all $0\leq i\leq r$.
\end{example}

%%%%%%%%
\subsubsection{Lie algebra decompositions}
\label{subsubsec:lie-algebr-decomp}

Let $\z(\g)$ (\resp $\g_s$) be the Lie algebra of the center $Z(G)$ (\resp of the derived group~$G_s$).
Then $\g=\z(\g)\oplus\g_s$, and this decomposition is orthogonal with respect to the Killing form of~$\g$, whose restriction to $\z(\g)$ (\resp\nolinebreak $\g_s$) is zero (\resp nondegenerate).
Here are some algebraic and combinatorial objects needed to give a more comprehensive description of the parabolic subgroups of~$G$:
\begin{itemize}
  \item $K$: a maximal compact subgroup of~$G$, with Lie algebra~$\mathfrak{k}$;
  \item $\g = \mathfrak{k} \oplus \mathfrak{k}^{\perp}$: the induced orthogonal decomposition of~$\g$ for the Killing form;
  \item $\aaa \subset \mathfrak{k}^{\perp}$: a Cartan subspace of~$\g$, \ie a maximal abelian subspace of~$\mathfrak{k}^{\perp}$; it is the direct sum of $\mathfrak{k}^{\perp}\cap \z(\g)$ and of a maximal abelian subspace $\aaa_s$ of $\mathfrak{k}^{\perp}\cap \g_s$ (unique up to the $\Ad(K)$-action);
  \item $\g = \g_0 \oplus \bigoplus_{\alpha \in \Sigma} \g_\alpha$: the decomposition of~$\g$ into $\ad(\aaa)$-eigenspaces. By definition,
  \[ (\ad\,Y)(Y')=\langle \alpha,Y\rangle Y' \]
  for all $Y\in\aaa$ and $Y'\in\g_\alpha$.
  The eigenspace $\g_0$ is the centralizer of $\aaa$ in~$\g$; it is the direct sum of $\z(\g)$ and of the centralizer of $\aaa$ in $\g_s$.
  The set $\Sigma \subset \aaa^* = \Hom_{\RR}(\aaa,\RR)$ projects to a (possibly nonreduced) root system of~$\aaa_s^*$, and each $\alpha \in \Sigma$ is called a \emph{restricted root} of $\aaa$ in~$\g$;
  \item  $\Delta\subset \Sigma$: a simple system (see \cite[\S\,II.6, p.\,164]{Knapp_LieGrp}), \ie a subset such that any root is expressed uniquely as a linear combination of elements of~\(\Delta\) with coefficients all of the same sign; the elements of~\(\Delta\) are called the \emph{simple} roots;  
  \item $\Sigma^+ \subset \Sigma$: the set of \emph{positive} roots, \ie roots that are nonnegative linear combinations of elements of \(\Delta\); then $\Sigma = \Sigma^+ \cup (- \Sigma^+)$.
\end{itemize}
Note that \(\Delta\) projects to a basis of the vector space~$\aaa_s^*$.
The \emph{real rank} of~$G$ is by definition the dimension of~$\aaa$.
Let 
\[ \overline{\aaa}^+ := \{Y \in \aaa \mid \langle{\alpha,Y}\rangle \geq
0\;\; \forall \alpha \in \Sigma^+\} = \{Y \in \aaa \mid \langle{\alpha,Y}\rangle \geq
0\;\; \forall \alpha \in \Delta\} \]
be the closed positive Weyl chamber of~$\aaa$ associated
with~$\Sigma^+$.

Given a subset $\theta \subset \Delta$, we define $P_\theta$ (\resp $P_{\theta}^-$) to be the normalizer in~$G$ of the Lie algebra
\begin{equation} \label{eqn:T-G/P} \mathfrak{u}_\theta = \bigoplus_{\alpha \in \Sigma_{\theta}^{+}} \g_\alpha \quad\quad \text{\bigg(resp.}~\mathfrak{u}_{\theta}^{-} = \bigoplus_{\alpha \in \Sigma_{\theta}^{+}} \g_{-\alpha} \text{\bigg)},\end{equation}
where $\Sigma_{\theta}^{+} = \Sigma^+ \smallsetminus  \mathrm{span}(\Delta \smallsetminus \theta)$ is the set of positive roots that do \emph{not} belong to the span of $\Delta \smallsetminus \theta$.
The group $P_\theta$ (\resp $P_{\theta}^-$) is a parabolic subgroup of~$G$, equal to the semidirect product of its unipotent radical $U_\theta := \exp(\mathfrak{u}_\theta)$ (\resp $\exp(\mathfrak{u}_{-\theta})$) and of the Levi subgroup
\begin{equation}
\label{eqn:levi} L_\theta := P_\theta \cap P_{\theta}^-.
\end{equation}  
Explicitly, 
\begin{equation} \label{eqn:decompose} \mathrm{Lie}(P_{\theta}) = \g_0 \oplus \bigoplus_{\alpha\in\Sigma^+} \g_{\alpha} \oplus \bigoplus_{\alpha\in\Sigma^+\smallsetminus \Sigma_\theta^+} \g_{-\alpha}.\end{equation}
In particular, $P_\emptyset = G$ and $P_\Delta$ is a minimal parabolic subgroup of~$G$.

\begin{example} \label{ex:roots}
Let $\KK$ be $\RR$, $\CC$, or the ring $\bH$ of quaternions and let $G$ be $\GL_d(\KK)$, seen as a real Lie group.
Its derived group is $G_s = D(G) = \SL_d(\KK)$.
If $\KK=\RR$ (\resp $\CC$, \resp $\bH$), then we can take $K$ to be $\OO(d)$ (\resp $\U(d)$, \resp $\Sp(d)$), and in all cases we can take $\aaa \subset \gl_d(\KK)$ to be the set of real diagonal matrices of size $d\times d$.
For $1\leq i\leq d$, let $\varepsilon_i\in\aaa^{\ast}$ be the evaluation of the $i$-th diagonal entry.
Then $\aaa = \z(\g) \oplus \aaa_s$, where $\z(\g) = \bigcap_{1\leq i,j\leq d} \Ker(\varepsilon_i - \varepsilon_j)$ is the set of real scalar matrices and $\aaa_s = \Ker(\varepsilon_1 + \dots + \varepsilon_d)$ the set of traceless real diagonal matrices.
The set of restricted roots of $\aaa$ in~$G$ is
\[ \Sigma = \{ \varepsilon_i - \varepsilon_j ~|~ 1\leq i\neq j\leq d\} .\]
We can take $\Delta = \{ \varepsilon_i - \varepsilon_{i+1} ~|~ 1\leq i\leq d-1\}$, so that
\[ \Sigma^+ = \{ \varepsilon_i - \varepsilon_j ~|~ 1\leq i<j\leq d\} \]
and $\overline{\aaa}^+$ is the set of the elements of~$\aaa$ whose entries are in nonincreasing~order.
For $\theta=\{ \varepsilon_{n_1}-\varepsilon_{n_1+1}, \dots, \varepsilon_{n_m}-\varepsilon_{n_m+1}\}$ with $1\leq n_1<\dots<\nolinebreak n_m\leq\nolinebreak d-\nolinebreak 1$, the parabolic subgroup $P_{\theta}$ (\resp $P_{\theta}^-$) is the set of block upper (\resp lower) triangular matrices in $\GL_d(\KK)$ with square diagonal blocks of sizes $n_1,n_2-n_1,\dots,n_m-n_{m-1},d-n_m$.
In particular, $P_{\Delta}$ is the set of upper triangular matrices in $\GL_d(\KK)$.
\end{example}

The following classical fact will be used in Section~\ref{sec:repr-semis-lie}.

\begin{fact}[see \eg {\cite[Prop.\,7.76]{Knapp_LieGrp}}]\label{fact:para-lie-subalgebra}
  Any Lie subalgebra of \(\g\) containing \(\mathrm{Lie}(P_\Delta)\) is of the form \(\mathrm{Lie}(P_\theta)\) for a unique \(\theta\subset\Delta\).
\end{fact}

%%%%%%%%
\subsubsection{Conjugacy classes of parabolic subgroups and invariant distributions on $G/L_{\theta}$}
\label{subsubsec:conj-class-parab}

Recall that any parabolic subgroup is conjugate to $P_\theta$ for some $\theta \subset \Delta$, and any pair of opposite parabolic subgroups is conjugate to $(P_\theta, P_\theta^-)$ for some $\theta \subset \Delta$; the set $\theta$ is unique since $\Ad(G)$ is assumed to act on $\mathfrak{g}$ by inner automorphisms (see \cite[\S\,5]{Borel_Tits_ihes}).
Since the stabilizer in~$G$ of $(P_{\theta},P_{\theta}^-)$ is $L_{\theta}=P_{\theta}\cap P_{\theta}^-$, the set of pairs $(P,Q)$ of transverse parabolic subgroups of~$G$ identifies, as a $G$-set, with the disjoint union of the $G/L_\theta$ for $\theta \subset \Delta$.
More precisely, with the identification of Remark~\ref{rem:pnorm},
\begin{equation} \label{eqn:orbits}
\{(P,Q) \in G/ P_\theta \times G/P_{\theta}^- \mid P, Q \text{ transverse}\} \simeq
 G/L_\theta, 
 \end{equation}
and $G/L_\theta$ is the unique open $G$-orbit in $G/ P_\theta \times G/P_{\theta}^-$.
From this the tangent bundle $T(G/L_\theta)$ inherits a decomposition
\begin{equation} \label{eqn:E+E-}
T(G/L_\theta) = E^+ \oplus E^-. 
\end{equation}
This decomposition is $G$-invariant, and so for any bundle with fiber $G/L_\theta$ there is a corresponding decomposition of the vertical tangent space.

%%%%%%%%%%%%%%%%%%%%%%%%%
\subsection{The Cartan projection}\label{subsec:Cartan-Lyap}

A central role in this paper is played by the Cartan projection~$\mu$, which can be used to measure dynamical properties of diverging sequences in~$G$.

%%%%%%%%
\subsubsection{Basics on the Cartan projection}\label{subsubsec:Cartan-proj}

Recall that, with the notation of Section~\ref{subsubsec:lie-algebr-decomp}, the Cartan decomposition $G=K(\exp\overline{\aaa}^+)K$ holds: any $g\in G$ may be written $g=k(\exp\mu(g))k'$ for some $k,k'\in K$ and a unique $\mu(g)\in\overline{\aaa}^+$ (see \cite[Ch.\,IX, Th.\,1.1]{Helgason}).
This defines a map
\begin{align*}
  \mu :\ G & \longrightarrow \hspace{0.2cm} \overline{\aaa}^+ \\
  g \hspace{0.1cm} & \longmapsto \mu(g) ,
\end{align*}
called the \emph{Cartan projection}, inducing a homeomorphism $K \backslash G / K \simeq \overline{\aaa}^+$.

\begin{example}\label{ex:mu}
For $\KK=\RR$ or~$\CC$, let $G=\GL_d(\KK)$, and let $K\subset G$ and~$\overline{\aaa}^+$ be as in Example~\ref{ex:roots}.
Then the diagonal entries of $\mu(g)$ are the logarithms of the singular values of~$g$ (\ie of the square roots of the eigenvalues of $^t\bar{g}g$, where $\bar{g}$ is the complex conjugate of~$g$), in nonincreasing order.
\end{example}

The \emph{(restricted) Weyl group} of $\aaa$ in~$\g$ is the group $W=N_K(\aaa)/Z_K(\aaa)$, where $N_K(\aaa)$ (\resp $Z_K(\aaa)$) is the normalizer (\resp centralizer) of $\aaa$ in~$K$.
We now fix a $W$-invariant Euclidean norm $\|\cdot\|$ on~$\aaa$.
By a little abuse of notation, we shall use the same symbol for the induced norm on the dual space~$\aaa^{\ast}$.
If $G$ is simple, then $\|\cdot\|$ is unique up to scale: it derives from the restriction to~$\aaa$ of the Killing form of~$\g$.
In general, $\|\cdot\|$ is not unique, but any choice will do.
This choice determines the Riemannian metric $d_{G/K}$ on the symmetric space $G/K$, and for any $g\in G$ we have
\begin{equation} \label{eqn:mu-dist}
\Vert \mu(g) \Vert = d_{G/K}(x_0, g\cdot x_0),
\end{equation}
where $x_0:=eK\in G/K$.

Seen as a subgroup of \(\GL_{\RR}(\aaa)\), the Weyl group $W$ is a finite Coxeter group.
A system of generators of $W$ is given by the orthogonal reflections $s_{\alpha}$ in the hyperplanes $\Ker(\alpha)\subset\aaa$, for $\alpha\in\Delta$.
The group $W$ acts simply transitively on the set of connected components of \(\aaa \smallsetminus \bigcup_{\alpha\in\Sigma}\,\Ker(\alpha)\) (open Weyl chambers).
Therefore there is a unique element $w_0 \in W$ such that $w_0\cdot (-\overline{\aaa}^+)=\overline{\aaa}^+$; it is the longest element with respect to the generating set \(\{s_\alpha\}_{\alpha \in \Delta}\).
The involution of~$\aaa$ defined by $Y\mapsto -w_0\cdot Y$
is called the \emph{opposition involution}\footnote{This involution is nontrivial only if the restricted root system $\Sigma$ is of type $A_n$, $D_{2n+1}$, or~$E_6$, where $n\geq 2$.}; it sends $\mu(g)$ to $\mu(g^{-1})$ for any $g\in G$.
The corresponding dual linear map preserves~$\Sigma$.
We shall denote it by
\begin{align}\label{eqn:opp-inv}
  \aaa^{\ast} & \longrightarrow  \aaa^{\ast}\\
  \alpha\, & \longmapsto  \alpha^{\star} = -w_0\cdot \alpha. \notag
\end{align}
By definition, for any $\alpha\in\Sigma$ and any $g\in G$,
\begin{equation} \label{eqn:opp-inv-mu}
\langle\alpha,\mu(g)\rangle = \langle\alpha^{\star},\mu(g^{-1})\rangle .
\end{equation}

\begin{example}\label{ex:opp-inv}
Take $G=\GL_d(\KK)$ with $\KK$, $K$, and~$\overline{\aaa}^+$ as in Example~\ref{ex:roots}.
The Weyl group $W$ is the symmetric group $\mathfrak{S}_d$, which permutes the diagonal entries of the elements of~$\aaa$.
The longest element $w_0$ of~$W$ is the permutation of $\{1, \dots, d\}$ taking $i$ to $d+1-i$.
For $\alpha=\varepsilon_i-\varepsilon_{i+1}\in\Delta$, we have $\alpha^{\star}=\varepsilon_{d-i}-\varepsilon_{d-i+1}$.
\end{example}

Here are some useful properties (see for instance \cite[Lem.\,2.3]{Kassel_corank1}), expressing that the map $\mu$ is ``strongly subadditive''.

\begin{fact}\label{fact:mu-subadditive}
For any $g,g_1,g_2,g_3\in G$,
\begin{enumerate}
\item $\| \mu(g) \| = \| \mu( g^{-1} )\|$;
\item \label{item:16} $\| \mu(g_1g_2) -\mu(g_1)\| \leq \| \mu(g_2)\|$;
\item \label{item:mu-strong-subadd} in particular, $\| \mu(g_1 g_2 g_3) -\mu(g_2)\| \leq \| \mu(g_1)\| + \|
  \mu(g_3)\|$.
\end{enumerate}
\end{fact}

As a consequence, for any representation $\rho: \Gamma \to G$, there exists $k>0$ such that for any $\gamma\in\Gamma$,
\begin{equation} \label{eqn:mu-leq-klength}
  \| \mu( \rho(\gamma))\| \leq k \, \ellGamma{\gamma}.
\end{equation}
Indeed, we can take $k:=\max_{s\in S} \| \mu\circ\rho(s) \|$ where $S$ is the finite generating set of~$\Gamma$ defining the word length~$\ellGammaf$.

%%%%%%%%
\subsubsection{Properness of the Cartan projection and consequences}\label{subsubsec:Cartan-proj-proper}

A crucial point is that the map $\mu : G\to\overline{\aaa}^+$ is proper, by compactness of~$K$.
This implies the following.

\begin{remark}\label{rem:qimu}
Let $\Gamma$ be a finitely generated discrete group and $\rho:\Gamma\rightarrow G$ a representation.
The representation $\rho$ has finite kernel in~$\Gamma$ and discrete image in~$G$ if and only if there is a function $f : \NN\rightarrow\RR$ with $\lim_{+\infty} f = +\infty$ such that for all $\gamma\in\Gamma$,
  \[ \Vert \mu(\rho(\gamma)) \Vert \geq f(\ellGamma{\gamma}). \]
  The map $\rho$ is a quasi-isometric embedding if and only if $f$ can be taken to be affine.
\end{remark}

In particular, if $\rho$ is a quasi-isometric embedding, then for any positive root $\alpha\in\Sigma^+$ the following two conditions are equivalent:
\begin{enumerate}[label=(\roman*),ref=\roman*]
  \item\label{item:avoid-cone} There exist $c,C>0$ such that $\langle\alpha,\mu(\rho(\gamma))\rangle\geq c\,\Vert\mu(\rho(\gamma))\Vert - C$ for all $\gamma\in\Gamma$;
  \item There exist $c,C>0$ such that $\langle\alpha,\mu(\rho(\gamma))\rangle\geq c\,\ellGamma{\gamma} - C$ for all $\gamma\in\Gamma$.
\end{enumerate}
Condition~\eqref{item:avoid-cone} means that the set $\mu(\rho(\Gamma))$ avoids some translate $\mathfrak{C}'$ of the cone $\mathfrak{C}:=\{ x\in\aaa~|~\langle\alpha,x\rangle < c\Vert x\Vert\}$ in the Euclidean space~$\aaa$ (see Figure~\ref{CLI}, left panel).

In Theorem~\ref{thm:constr-xi}.\eqref{item:xi-transv_1} we consider the following slightly different condition: for any geodesic ray $\mathcal{R} = (\gamma_n)_{n\in\NN}$ in the Cayley graph of~$\Gamma$, the sequence $(\langle \alpha,\, \mu( \rho(\gamma_n))\rangle)_{n\in\NN}\in\nolinebreak\RR^{\NN}$ is \emph{lower CLI}, \ie there exist $\kappa_{\mathcal{R}},\kappa^{\prime}_{\mathcal{R}}>0$ such that for all $n,m\in\NN$,
\[ \langle \alpha,\, \mu(\rho(\gamma_{n+m})) - \mu(\rho(\gamma_n))\rangle \geq \kappa_{\mathcal{R}} \, m - \kappa^{\prime}_{\mathcal{R}}.\]
This means that there is a translate $\mathfrak{C}'_{\mathcal{R}}$ in~$\aaa$ of the cone
\[\mathfrak{C}_{\mathcal{R}} := \{ x\in\aaa~|~\langle\alpha,x\rangle < \kappa_{\mathcal{R}}\Vert x\Vert\} \]
such that for any $n\in\NN$, the sequence $(\mu(\rho(\gamma_{n+m})))_{m\in\NN}$ avoids $\mu(\rho(\gamma_n))+\mathfrak{C}'_{\mathcal{R}}$ (see Figure~\ref{CLI}, right panel).
This ``nested cone'' property is what we mean when we say (in the introduction) that the sequence $(\mu(\rho(\gamma_n)))_{n\in\NN}\in (\overline{\aaa}^+)^{\NN}$ drifts away ``forever linearly'' from $\Ker(\alpha)$.

\begin{figure}[h!]
\centering
\labellist
\small\hair 2pt
\pinlabel $\overline{\mathfrak{a}}^+$  at 60 50
\pinlabel $\mu(\rho(\Gamma))$  at 40 30
\pinlabel $\mathfrak{C}$  at 62 24
\pinlabel $\mathfrak{C}'$  at 63 15
\pinlabel {$\color{red} \varepsilon_1-\varepsilon_2$}  at 35 1
\pinlabel {$\color{red} \varepsilon_2-\varepsilon_3$}  at 11 29
\pinlabel $\mu(\rho(\gamma_n))$  at 114 14
\pinlabel $\mu(\rho(\gamma_n))+\mathfrak{C}'_{\mathcal{R}}$  at 150 21
\pinlabel $\mu(\rho(\gamma_{n'}))$  at 138 34
\pinlabel $\mu(\rho(\gamma_{n'}))+\mathfrak{C}'_{\mathcal{R}}$  at 178 40
\pinlabel {$\color{red} \varepsilon_1-\varepsilon_2$}  at 114 1
\pinlabel {$\color{red} \varepsilon_2-\varepsilon_3$}  at 90 29
\endlabellist
\includegraphics[width=12.5cm]{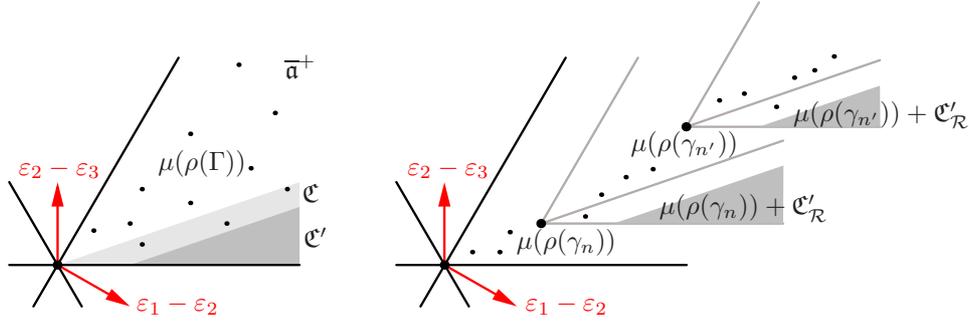}
\caption{Here $G=\SL_3(\RR)$ and $\theta = \{ \varepsilon_2 - \varepsilon_3\}$. Left panel: Condition~\eqref{item:avoid-cone} above. Right panel: The CLI condition of Theorem~\ref{thm:constr-xi}.\eqref{item:xi-transv_1} for a geodesic ray $\mathcal{R} = (\gamma_n)_{n\in\NN}$.}
\label{CLI}
\end{figure} 

If $G$ has real rank~$1$, then Proposition~\ref{prop:QIflowbis} implies the following strengthening of Remark~\ref{rem:qimu} (which yields the implication $\eqref{item:Ano}\Rightarrow \eqref{item:cli}$ of Theorem~\ref{thm:char_ano} in that case):

\begin{corollary}\label{cor:CLI-rk1}
Let $\Gamma$ be a finitely generated discrete group and $G$ a semisimple Lie group of real rank~$1$.
If $\rho : \Gamma\to G$ is a quasi-isometric embedding, then for any geodesic ray $(\gamma_n)_{n\in\NN}$ in the Cayley graph of~$\Gamma$, the sequence $(\Vert\mu(\rho(\gamma_n))\Vert)_{n\in\NN}$ is CLI; moreover, the CLI constants are uniform over all geodesic rays $(\gamma_n)_{n\in\NN}$ with $\gamma_0=e$.
\end{corollary}

In other words, in real rank one the fact that $\rho$ is a quasi-isometric embedding implies that the sequence $(\langle\alpha,\mu(\rho(\gamma_n))\rangle)_{n\in\NN}$ is CLI for all $\alpha\in\Sigma^+$.
This is not true when $G$ has higher real rank: see \eg the representation $\rho_0 : \Gamma\to\SL_2(\RR)\times\SL_2(\RR)$ constructed in the proof of Proposition~\ref{prop:an-unstable-quasi}.

\begin{proof}[Proof of Corollary~\ref{cor:CLI-rk1}]
If $\rho:\Gamma \to G$ is a quasi-isometric embedding, then $\rho$ has finite kernel and the group $\rho(\Gamma)$ is convex cocompact in~$G$ (see \cite{Bourdon_conforme,BowditchGFV}). 
By Fact~\ref{fact:flow}, $\Gamma$ 
admits a flow space $\mathcal{G}_{\Gamma}$ which isometrically, $\rho$-equivariantly, and $(\varphi_t)$-equivariantly embeds into the unit tangent bundle $T^1(X)$.
We conclude using Proposition~\ref{prop:QIflowbis}.
\end{proof}

%%%%%%%%
\subsubsection{The Cartan projection for $L_\theta$ and the $\theta$-coset distance map}
\label{subsubsec:cart-proj-l_th}

Let $\theta\subset \Delta$ be a nonempty subset of the simple restricted roots of~$G$.
Recall the Levi subgroup $L_\theta$ of~$P_{\theta}$ from \eqref{eqn:levi}.
The group $K_\theta := K \cap L_\theta$ is a maximal compact subgroup of~$L_\theta$, and $L_\theta$ admits the Cartan decomposition $L_{\theta} = K_{\theta} (\exp\athetaplus) K_{\theta}$ where
\begin{equation} \label{eqn:thetachamber}
\athetaplus :=\{ Y\in \aaa \mid \langle\alpha,Y\rangle \geq 0 \quad
\forall \alpha \in \Delta \smallsetminus \theta\}.
\end{equation}
We denote by
\begin{equation} \label{eqn:defmuth}
\mu_\theta : L_\theta \longrightarrow  \athetaplus
\end{equation}
the corresponding Cartan projection of~$L_{\theta}$.
As in Section~\ref{subsubsec:Cartan-proj}, the map $\mu_{\theta}$ induces a homeomorphism $K_\theta \backslash L_\theta / K_\theta \simeq \athetaplus$.
The Weyl chamber $\athetaplus$ for~$L_{\theta}$ is convex and is a union of  $W$-translates of the Weyl chamber $\overline{\aaa}^+$ for~$G$.
We will sometimes use the following observation.

\begin{remark}\label{rem:mu-theta-equal-mu-l-in-L}
  For $l\in L_\theta$, if $\mu_\theta(l) \in \overline{\aaa}^+$ (equivalently if \(\langle \alpha, \mu_\theta(l) \rangle \geq 0\) for all \(\alpha\in \theta\)), then $\mu_\theta(l) = \mu(l)$. 
\end{remark}

\begin{example}\label{ex:a+theta}
Take $G=\GL_d(\KK)$ with $\KK$, $K$ and~$\overline{\aaa}^+$ as in Example~\ref{ex:roots}.
For $\theta=\{ \varepsilon_i-\varepsilon_{i+1}\}$, the set $\athetaplus$ consists of elements $\mathrm{diag}(t_1,\dots,t_d)\in\aaa$ with $t_j\geq t_{j+1}$ for all $j\in\{ 1,\dots,d-1\}\smallsetminus\{ i\}$.
\end{example}

The Cartan projection $\mu_{\theta}$ induces a Weyl-chamber-valued metric on the Riemannian symmetric space of~$L_\theta$:
\begin{equation*}
  \begin{array}{ccc}
  \dtheta :\ L_\theta/K_{\theta} \times L_\theta/K_{\theta} & \longrightarrow &
  \athetaplus \\
  \hspace{1cm} (g K_\theta, h K_\theta) & \longmapsto & \mu_\theta(g^{-1}h).
  \end{array}
\end{equation*}
This extends to a {\em coset distance} map on the set of pairs of elements of $G/ K_\theta$ projecting to the same element in $G/ L_\theta$.

\begin{definition}\label{defi:mutheta}
The \emph{$\theta$-coset distance} map is
\[ \begin{array}{ccc}
\dtheta :\ \{(gK_\theta, h K_\theta) \in G/K_\theta \times
  G/K_\theta \mid g L_\theta = h L_\theta\} & \longrightarrow &  \athetaplus\\
  (gK_\theta, hK_\theta)  & \longmapsto & \mu_\theta (g^{-1}h).
  \end{array} \]
\end{definition}

This map was introduced in the study of Anosov representations in \cite{Guichard_Wienhard_DoD}; it will play a crucial role in Section~\ref{sec:anos-repr-cart}.

%%%%%%%%%%%%%%%%%%%%%%%%%
\subsection{The Lyapunov projection and proximality in \(G/P_\theta\)}\label{subsec:proximal}

The natural projection associated with the Jordan decomposition is called the \emph{Lyapunov projection}; we denote it by $\lambda: G \to \overline{\aaa}^+$. 
Explicitly, any $g\in G$ can be written uniquely as the commuting product $g = g_h g_e g_u$ of a hyperbolic, an elliptic, and a unipotent element (see \eg \cite[Th.\,2.19.24]{Eberlein}). 
The conjugacy class of $g_h$ intersects $\exp(\overline{\aaa}^+)$ in a unique element $\exp(\lambda(g))$.
This projection can also be defined as a limit: for any $g\in G$,
\begin{equation} \label{eqn:lambda-lim-mu}
\lambda(g) = \lim_{n \to +\infty} \frac{1}{n} \mu( g^n).
\end{equation}
Note that by definition of the opposition involution, and similarly to \eqref{eqn:opp-inv-mu},
\begin{equation} \label{eqn:opp-inv-lambda}
\langle\alpha,\lambda(g)\rangle = \langle\alpha^{\star},\lambda(g^{-1})\rangle
\end{equation}
for all $\alpha\in\Sigma$ and $g\in G$.

\begin{example}
For $\KK=\RR$ or~$\CC$, let $G=\GL_d(\KK)$, and let $K\subset G$ and~$\overline{\aaa}^+$ be as in Example~\ref{ex:roots}.
Then the diagonal entries of $\lambda(g)$ are the logarithms of the moduli of the complex eigenvalues of~$g$, in nonincreasing order.
\end{example}

Let $\theta \subset \Delta$ be a nonempty subset of the simple restricted roots of~$G$.
We shall use the following terminology.

\begin{definition} \label{defi:prox-G/P}
An element $g\in G$ is \emph{proximal in $G/P_{\theta}$} if it satisfies any of the following two equivalent properties:
\begin{enumerate}
  \item\label{item:prox} $g$ has a fixed point $\xi^+_g\in G/P_{\theta}$ which is attracting, in the sense that the derivative at $\xi^+_g\in G/P_{\theta}$ of the action of $g$ on $G/P_{\theta}$ has spectral radius $<1$;
  \item\label{item:lambda>0} $\langle\alpha,\lambda(g)\rangle>0$ for all $\alpha\in\theta$.
\end{enumerate}
\end{definition}

The equivalence between \eqref{item:prox} and \eqref{item:lambda>0} is well known, and implicitly contained in \cite{Benoist}.
Since we could not find it explicitly in the literature, we shall provide a proof of it, as well as of the uniqueness of the attracting fixed point $\xi^+_g\in G/P_{\theta}$, in Proposition~\ref{prop:theta_emb}.\eqref{item:c_thtmb} below.
The basin of attraction of a proximal element $g$ in $G/P_{\theta}$ is described as follows (see Section~\ref{subsec:proof-theta-emb} for a proof).

\begin{lemma} \label{lem:attr-basin}
If $g\in G$ is proximal in $G/P_{\theta}$, then $g^{-1}$ is proximal in $G/P_{\theta}^-$ and $\lim_{n\to +\infty} g^n\cdot x=\xi^+_g$ for all $x\in G/P_{\theta}$ transverse to the attracting fixed point $\xi^-_{g^{-1}}$ of $g^{-1}$ in $G/P_{\theta}^-$.
\end{lemma}

We now prove the following.

\begin{lemma}\label{lem:prox}
An element $g\in G$ is proximal in $G/P_{\theta}$ if and only if for any $\alpha\in\theta$,
\begin{equation}\label{eqn:growth>2log}
\langle\alpha,\mu(g^n)\rangle -2\log n \underset{n\to +\infty}{\longrightarrow} +\infty. 
\end{equation}
\end{lemma}

Lemma~\ref{lem:prox} is based on the following claim.

\begin{claim} \label{cla:unip-growth}
For any unipotent element $u\in G$, there is a constant $C_u>0$ such that for all $\alpha\in\Delta$ and $n\in\NN^{\ast}$,
\begin{equation*} 
\langle\alpha,\mu(u^n)\rangle \leq 2 \log n + C_u.
\end{equation*}
\end{claim}

\begin{proof}[Proof of Claim~\ref{cla:unip-growth}]
Consider the following two elements of $\mathfrak{sl}_2(\RR)$:
\[ x = \begin{pmatrix} 1/2 & 0\\ 0 & -1/2 \end{pmatrix}, \quad\quad e_+ = \begin{pmatrix} 0 & 1\\ 0 & 0 \end{pmatrix}. \]
Let $\mu_{\SL_2(\RR)} : \SL_2(\RR)\to\RRp x$ be the Cartan projection of $\SL_2(\RR)$ with respect to the Cartan decomposition $\SL_2(\RR) = \SO(2) (\exp\RRp x) \SO(2)$.
For any $n\in\NN$, an elementary computation shows that $\mu_{\SL_2(\RR)}(\exp n e_+) = t_n x$ where $t_n = 2\mathrm{argsinh} (\frac{n}{2}) \leq2\log (n+1)$.

Let $u\in G$ be unipotent.
By the Jacobson-Morozov theorem (see \eg \cite[Ch.\,VIII, \S\,11, Prop.\,2]{Bourbaki_Lie_79}), there is a homomorphism $\tau : \SL_2(\RR)\rightarrow G$ such that $\tau(\exp e_+)=u$.
Up to conjugating in~\(G\) (which only changes $\mu$ by a bounded additive amount, see Fact~\ref{fact:mu-subadditive}), we may assume that \(\mathrm{d}_e\tau(x)\in\overline{\aaa}^+\) and that $\tau(\SO(2))\subset K$.
Then $\mu(u^n) = t_n \mathrm{d}_e\tau(x)$, and so
\[ \langle\alpha,\mu(u^n)\rangle = t_n \, \langle \alpha, \mathrm{d}_e\tau(x)\rangle \leq 2\log(n+1) \, \langle \alpha, \mathrm{d}_e\tau(x)\rangle \]
for all \(\alpha\in\Delta\) and $n\in\NN$.
We conclude using the fact \cite[Lem.\,5.1]{Kostantsl2} that \(\langle \alpha, \mathrm{d}_e\tau(x)\rangle\in\{ 0, 1/2, 1\}\) for all \(\alpha\in\Delta\).
\end{proof}

\begin{proof}[Proof of Lemma~\ref{lem:prox}]
Let $g\in G$.
If $g$ is proximal in $G/P_{\theta}$, then \eqref{eqn:growth>2log} holds by \eqref{eqn:lambda-lim-mu}.
Conversely, suppose \(g\) satisfies \eqref{eqn:growth>2log}, and let
\[ \theta' := \{ \alpha\in\Delta~|~\langle\alpha,\lambda(g)\rangle>0\}. \]
It is sufficient to prove the existence of a constant $C>0$ such that for all $\beta\in\Delta\smallsetminus\theta'$ and $n\in\NN$,
\begin{equation} \label{eq:beta_logn}
 \langle\beta,\mu(g^n)\rangle \leq 2\log n + C.
\end{equation}
Indeed, then \eqref{eqn:growth>2log} shows that \(\theta\) and \(\Delta \smallsetminus \theta'\) do not intersect, hence \(\theta \subset \theta'\) and \(g\) is proximal in $G/P_{\theta}$ by definition of~\(\theta'\).

To prove \eqref{eq:beta_logn}, consider the Jordan decomposition \(g= g_h g_e g_u\) of~\(g\).
By Fact~\ref{fact:mu-subadditive}, $\Vert\mu(g^n)-\mu(g_h^n g_u^n)\Vert\leq\Vert\mu(g_e^n)\Vert$ is bounded, and so we may assume $g_e=1$.
Up to conjugation (which only changes $\mu$ by a bounded additive amount, by Fact~\ref{fact:mu-subadditive} again), we may assume that $g_h=\exp(\lambda(g))\in\overline{\aaa}^+$ and that $g_u\in\exp(\bigoplus_{\beta\in\Sigma^+} \g_{\beta})$.
Then $g_h$ belongs to the center of the group $L_{\theta'}$ of \eqref{eqn:levi}; the element $g_u$ commutes with~$g_h$, hence belongs to
\[ \exp\bigg(\bigoplus_{\beta\in\Sigma^+\cap\mathrm{span}(\Delta\smallsetminus\theta')} \g_{\beta}\bigg) \subset L_{\theta'} \,; \]
we have $\mu_{\theta'}(g^n) = n\,\lambda(g) + \mu_{\theta'}(g_u^n)$ for all $n\in\NN$.
By Fact~\ref{fact:mu-subadditive}, for any $\alpha\in\theta'$ and $n\in\NN$,
\[ |\langle\alpha, \mu_{\theta'}(g^n) - \mu_{\theta'}(g_h^n)\rangle| \leq \Vert\alpha\Vert \, \Vert\mu_{\theta'}(g^n) - \mu_{\theta'}(g_h^n)\Vert \leq \Vert\alpha\Vert \, \Vert\mu_{\theta'}(g_u^n)\Vert. \]
By Claim~\ref{cla:unip-growth} applied to $g_u\in L_{\theta'}$, the right-hand side grows logarithmically with~$n$, while $\langle\alpha, \mu_{\theta'}(g_h^n)\rangle = n\,\langle\alpha,\lambda(g)\rangle$ grows linearly; therefore, $\langle\alpha, \mu_{\theta'}(g^n)\rangle \geq 0$ for all large enough $n\in\NN$.
This holds for all $\alpha\in\theta'$, and so for all large enough $n\in\NN$ we have $\mu_{\theta'}(g^n)\in\overline{\aaa}^+$, hence $\mu(g^n)=\mu_{\theta'}(g^n)$ (see Remark~\ref{rem:mu-theta-equal-mu-l-in-L}).
In particular, for any $\beta\in\Delta\smallsetminus\theta'$ and any large enough $n\in\NN$,
\[ \langle\beta,\mu(g^n)\rangle = \langle\beta,\mu_{\theta'}(g^n)\rangle = \langle\beta, n\,\lambda(g) + \mu_{\theta'}(g_u^n)\rangle = \langle\beta,\mu_{\theta'}(g_u^n)\rangle. \]
By Claim~\ref{cla:unip-growth} again, there is a constant $C_{g_u}>0$ such that the right-hand side is $\leq 2\log n+C_{g_u}$ for all $n\in\NN^*$.
This proves \eqref{eq:beta_logn} and completes the proof.
\end{proof}

%%%%%%%%%%%%%%%%%%%%%%%%%
\subsection{Anosov representations}\label{subsec:anos-repr}

Finally, in this section we recall the definition and some properties of Anosov representations.
For more details and proofs we refer to \cite{Guichard_Wienhard_DoD}.

\begin{remark}
In this paper our convention for the notation of parabolic subgroups is different from the one adopted in \cite{Guichard_Wienhard_DoD}: definitions and statements involving $\theta \subset \Delta$ should be changed to $\Delta \smallsetminus \theta$ when compared with their versions in \cite{Guichard_Wienhard_DoD}.
\end{remark}

We now fix a word hyperbolic group~$\Gamma$ with flow space~$\mathcal{G}_\Gamma$, a nonempty subset $\theta\subset\Delta$ of the simple restricted roots of our reductive Lie group~$G$, and a representation $\rho : \Gamma \to G$.

%%%%%%%%
\subsubsection{The dynamical definition}
\label{subsubsec:dynamical-defi}

The space
\[ \mathcal{E}(\rho) := \Gamma \backslash (\mathcal{G}_\Gamma \times G/L_\theta) \]
is a $G/L_\theta$-bundle over $\Gamma \backslash \mathcal{G}_\Gamma$.
The subbundles $E^+, E^-$ of $T(G/L_\theta)$ defined in \eqref{eqn:E+E-} induce vector bundles (still denoted by $E^+, E^-$) over $\mathcal{E}(\rho)$.
In fact, $E^+$ and $E^-$ are subbundles of the vertical tangent bundle 
\[ \vt \mathcal{E}(\rho):=\Gamma \backslash (\mathcal{G}_\Gamma \times T (G/L_\theta)). \]
The geodesic flow $\{\varphi_t\}_{t\in \RR}$ naturally acts on the products $\mathcal{G}_\Gamma \times G/L_\theta$ and $\mathcal{G}_\Gamma \times T (G/L_\theta)$ (leaving the second coordinate unchanged), hence on their quotients $\mathcal{E}(\rho) $ and $\vt \mathcal{E}(\rho)$; the subbundles $E^\pm$ are flow-invariant.

\begin{definition} \label{defi:ano1}
  The representation $\rho : \Gamma\to G$ is $P_\theta$\emph{-Anosov} if there is a continuous section $\sigma : \Gamma \backslash \mathcal{G}_\Gamma \to \mathcal{E}(\rho)$ with the following properties:
  \begin{enumerate}[label=(\roman*),ref=\roman*]
  \item $\sigma$ is flow-equivariant (\ie its image $F = \sigma(\Gamma \backslash \mathcal{G}_\Gamma)$ is flow-invariant);
  \item\label{item:E+} the action of the flow on the vector bundle $E^+ |_F\subset \vt \mathcal{E}(\rho)$ is dilating;
  \item\label{item:E-} the action of the flow on the vector bundle $E^- |_F\subset \vt \mathcal{E}(\rho)$ is contracting.
  \end{enumerate}
\end{definition}

It is known that \eqref{item:E+} implies \eqref{item:E-} \cite[Prop.\,3.16]{Guichard_Wienhard_DoD} and that
the section~$\sigma$ is unique.
This definition is useful to determine certain properties of Anosov representations, such as openness, or to define natural metrics on spaces of Anosov representations \cite{Bridgeman_Canary_Labourie_Sambarino}.

%%%%%%%%
\subsubsection{Another equivalent definition}
\label{subsubsec:anoth-equiv-defin}

The following notions will be needed for various characterizations of Anosov representations.

\begin{definition}\label{defi:comp_trsvrs_maps}
  Two maps $\xi^+: \partial_{\infty}\Gamma \to G/ P_\theta$ and $\xi^-:\partial_{\infty}\Gamma \to G/ P_\theta^{-}$ are said~to~be
\begin{itemize}
  \item \emph{transverse} if for any $\eta \neq \eta'$ in~$\partial_{\infty}\Gamma$, the points $\xi^+(\eta)$ and $\xi^-(\eta')$ are transverse in the sense of Definition~\ref{defi:transverse-compatible}; 
  \item {\emph{dynamics-preserving}} for $\rho: \Gamma \to G$ if for any $\gamma\in\Gamma$ of infinite order with attracting fixed point $\eta^{+}_{\gamma} \in \partial_{\infty} \Gamma$, the point $\xi^+(\eta^{+}_{\gamma})$ (\resp $\xi^-(\eta^{+}_{\gamma})$) is an attracting fixed point
  for the action of $\rho(\gamma)$ on $G/P_{\theta}$ (\resp $G/P_{\theta}^-$).
\end{itemize}
\end{definition}

\begin{remarks}\label{rem:compatible}
\begin{enumerate}[label=(\alph*),ref=\alph*]
    \item \label{item:Ano-prox} If there exists a map $\xi^+ : \partial_{\infty}\Gamma\to G/P_{\theta}$ which is dynamics-preserving for~$\rho$, then $\rho(\gamma)$ is proximal in $G/P_{\theta}$ 
  for any $\gamma\in\Gamma$ of infinite order.
  \item \label{item:unique-bdry-maps} Continuous, dynamics-preserving boundary maps $\xi^+ : \partial_{\infty}\Gamma \to G/ P_\theta$ and $\xi^- : \partial_{\infty}\Gamma \to G/ P_\theta^{-}$ for~$\rho$, if they exist, are entirely determined on the dense subset consisting of the attracting fixed points $\eta^{+}_{\gamma}$ for $\gamma \in \Gamma$ of infinite order.
  As a consequence, such maps are necessarily unique and $\rho$-equivariant.
  Moreover, for any $\eta \in \partial_{\infty}\Gamma$, the points $\xi^+(\eta)\in G/P_{\theta}$ and $\xi^-(\eta)\in G/P_{\theta}^-$ are compatible in the sense of Definition~\ref{defi:transverse-compatible}: indeed, the attracting fixed points $\xi^+(\eta^{+}_{\gamma})$ and $\xi^-(\eta^{+}_{\gamma})$ are always compatible for $\gamma\in\Gamma$ of infinite order.
  \item \label{item:compat-transv->inj} If $\xi^+$ and~$\xi^-$ are continuous, dynamics-preserving for~$\rho$, and transverse, then they are both injective.
  Indeed, for any \(\eta\neq \eta'\) in \(\partial_\infty \Gamma\), the points $\xi^+(\eta)\in G/P_{\theta}$ and $\xi^-(\eta)\in G/P_{\theta}^-$ are compatible while the points $\xi^+(\eta')\in G/P_{\theta}$ and $\xi^-(\eta)\in G/P_{\theta}^-$ are transverse.
\end{enumerate}
\end{remarks}

A section $\sigma$ as in Definition~\ref{defi:ano1} is equivalent to a $\rho$-equivariant map $\tilde{\sigma} : \mathcal{G}_\Gamma \to G/L_\theta$.
Working out the properties of~$\tilde{\sigma}$, one obtains the following equivalent definition of Anosov representations (see \cite[Def.\,2.10]{Guichard_Wienhard_DoD}), which still makes use of the flow space $\mathcal{G}_\Gamma$ of~$\Gamma$ but avoids the language of bundles.
Recall the maps \(\varphi_{\pm\infty}\) from point~\eqref{item:10} of Section~\ref{subsubsec:flow-space-an} and the coset distance map \(\dtheta\) from Definition~\ref{defi:mutheta}.

\begin{definition}\label{defi:ano2}
  A representation $\rho: \Gamma \to G$ is $P_\theta$-Anosov if there exist continuous, $\rho$-equivariant maps $\xi^+: \partial_{\infty} \Gamma \to G/ P_\theta$ and $\xi^- : \partial_{\infty} \Gamma \to G/ P_\theta^{-}$ with the following properties:
  \begin{enumerate}[label=(\roman*),ref=\roman*]
  \item \label{item:ano-i}
 $\xi^+$ and $\xi^-$ are transverse.
  This implies that $\xi^+$ and $\xi^-$ combine, via \eqref{eqn:orbits}, to a continuous, $\Gamma$-equivariant, flow-invariant map
\[ \begin{array}{rcc}
    \tilde{\sigma} :\ \mathcal{G}_\Gamma & \longrightarrow & G/L_\theta\\
    v \hspace{0.15cm} & \longmapsto & \bigl( \xi^+\circ\varphi_{+\infty}(v),\ \xi^-\circ\varphi_{-\infty}(v) \bigr).
  \end{array} \]
  Such a map $\tilde{\sigma}$ always admits a continuous, $\Gamma$-equivariant lift
  \[\tilde{\beta} : \mathcal{G}_\Gamma \longrightarrow G/K_\theta,\]
  (this follows from the contractibility of $L_\theta/ K_\theta$), \ie $\pr \circ \tilde{\beta} = \tilde{\sigma}$.
  
    \item\label{item:ano-ii} 
  There exist $c,C > 0$ such that for all $\alpha \in \theta$, all $v \in \mathcal{G}_\Gamma$, and all $t \in \RR$,
    \[\bigl\langle \alpha,\, \dtheta \bigl(\tilde{\beta}(v),
    \tilde{\beta}(\varphi_t\cdot v)\bigr) \bigr\rangle \geq ct -C. \]
  \end{enumerate}
\end{definition}

This last inequality expresses the exponential contraction in Definition~\ref{defi:ano1}.
Note that the left-hand side of the inequality is nonnegative for $\alpha \in \Delta \smallsetminus \theta$ as well, by definition \eqref{eqn:thetachamber} of the range~$\athetaplus$ of \(\dtheta\).

The maps $\xi^+, \xi^-$ of an Anosov representation are always dynamics-preser\-ving (see \cite[Lem.\,3.1]{Guichard_Wienhard_DoD}); in particular, they are unique (Remark~\ref{rem:compatible}.\eqref{item:unique-bdry-maps}).

Let $\theta^{\star}\subset\Delta$ be the image of~$\theta$ under the opposition involution \eqref{eqn:opp-inv}.
By definition, the group $P_{\theta}^-$ is conjugate to $P_{\theta^{\star}}$, hence $G/P_{\theta}^-$ identifies with $G/P_{\theta^{\star}}$.
It is sometimes useful to reduce to the case that $\theta=\theta^{\star}$: this is always possible by the following fact.

\begin{fact}\cite[Lem.\,3.18]{Guichard_Wienhard_DoD} \label{fact:Anosov_opp}
A representation $\rho : \Gamma\to G$ is $P_{\theta}$-Anosov if and only if it is $P_{\theta\cup\theta^{\star}}$-Anosov.
\end{fact}

\begin{remark}\label{rem:xipm}
If $\theta=\theta^{\star}$, then the maps $\xi^+$ and~$\xi^-$ associated with an Anosov representation are equal (after identifying $G/P_{\theta}$ and $G/P_{\theta}^-$ with the set of parabolic subgroups of~$G$ conjugate to~$P_{\theta}$, see Remark~\ref{rem:pnorm}).
More generally, any continuous, dynamics-preserving boundary maps $\xi^+ : \partial_{\infty}\Gamma\to G/P_{\theta}$ and $\xi^- : \partial_{\infty}\Gamma\to G/P_{\theta}^-$ for an arbitrary representation $\rho : \Gamma\to G$ are equal, by uniqueness (Remark~\ref{rem:compatible}.\eqref{item:unique-bdry-maps}).
\end{remark}

%%%%%%%%
\subsubsection{Examples and properties}
\label{subsubsec:examples-properties}

Examples of Anosov representations include:
\begin{enumerate}[label=(\alph*),ref=\alph*]
  \item\label{item:2} The inclusion of convex cocompact subgroups in semisimple Lie groups $G$ of real rank~$1$ \cite{Labourie_anosov, Guichard_Wienhard_DoD} (here $|\Delta|=1$, and so $P_{\Delta}$ is the only proper parabolic subgroup of~$G$ up to conjugacy);
  \item Representations of surface groups belonging to the Hitchin component, when $G$ is a split real semisimple Lie group \cite{Labourie_anosov}, \cite[Th.\,1.15]{Fock_Goncharov};
    \item Maximal representations of surface groups, when the Riemannian symmetric space of~$G$ is Hermitian \cite{Burger_Iozzi_Labourie_Wienhard, Burger_Iozzi_Wienhard_anosov};
  \item The inclusion of quasi-Fuchsian subgroups in $\SO(2,d)$ \cite{Barbot_Merigot_fusionPubli, BarbotDefAdSQF};
  \item Holonomies of compact convex $\PP^n(\RR)$-manifolds whose fundamental group is word hyperbolic \cite{Benoist_CD1}.
\end{enumerate}

\noindent Here are some basic properties of Anosov representations \cite{Labourie_anosov, Guichard_Wienhard_DoD}:
\begin{enumerate}
  \item The set of $P_\theta$-Anosov representations is open in $\Hom(\Gamma,G)$, invariant under conjugation by~$G$ at the target, and the map sending $\rho$ to its pair of boundary maps $(\xi^+,\xi^-)=:(\xi^{+}_{\rho},\xi^{-}_{\rho})$ is continuous.
  \item Any $P_\theta$-Anosov representation $\rho : \Gamma\to G$ is a quasi-isometric embedding (see Remark~\ref{rem:qimu} for an interpretation in terms of~$\mu$).
  \item \label{item:f-i-subgroup} Let $\Gamma'$ be a finite-index subgroup of~$\Gamma$.
  A representation $\rho : \Gamma\to G$ is $P_{\theta}$-Anosov if and only if its restriction to $\Gamma'$ is.\end{enumerate}
  
\begin{remark}\label{rem:rank1ccquasi}
Let $\Gamma$ be a finitely generated discrete group and $\rho : \Gamma\to G$ a representation.
If $G$ has real rank~$1$, then the following are actually equivalent (see \cite{Bourdon_conforme,BowditchGFV}):
\begin{itemize}
  \item \(\Gamma\) is word hyperbolic and $\rho$ is Anosov;
  \item the kernel of~$\rho$ is finite and the image of~$\rho$ is convex cocompact in~$G$;
  \item \(\Gamma\) is finitely generated and $\rho$ is a quasi-isometric embedding.
\end{itemize}
\end{remark}

\begin{remark}\label{rem:reductive-ss}
To check the Anosov property, it is always possible to restrict to a semisimple Lie group instead of a reductive one.
Indeed, if $G$ is reductive with center $Z(G)$, then the group $G':=G/Z(G)$ is semisimple, and $\Delta$ identifies with a simple system for~$G'$.
For $\theta\subset\Delta$, let $P'_{\theta}$ be the corresponding parabolic subgroup of~$G'$.
Then $G/P_{\theta}$ identifies with $G'/P'_{\theta}$.
A representation $\rho : \Gamma\to G$ is $P_{\theta}$-Anosov if and only if the induced representation $\rho' : \Gamma\to G'$ is $P'_{\theta}$-Anosov.
Similarly, \(\rho\) admits equivariant (\resp continuous, \resp transverse, \resp dynamics-preserving) boundary maps if and only if \(\rho'\) does.
\end{remark}

%%%%%%%%
\subsubsection{Semisimple representations of discrete groups}\label{subsubsec:rho-non-ss}

Let $\rho : \Gamma\to G$ be a representation of the discrete group $\Gamma$ into the reductive Lie group~$G$.
Recall that $\rho$ is called \emph{semisimple} if the Zariski closure of $\rho(\Gamma)$ in~$G$ is reductive
or, equivalently, if for any linear representation $\tau : G\to\GL(V)$ we can write $V$ as a direct sum of irreducible $(\tau\circ\rho)(\Gamma)$-modules.

\begin{remark}
  \label{rem:semisimple_finiteindex}
  If \(\Gamma'\) is a finite-index subgroup of~\(\Gamma\), then \(\rho\) is semisimple if and only if \(\rho|_{\Gamma'}\) is semisimple. 
\end{remark}

For a general representation $\rho : \Gamma\to G$, we define the semisimplification of~$\rho$ as follows.
Let $H$ be the Zariski closure of $\rho(\Gamma)$ in~$G$.
Choose a Levi decomposition $H=L\ltimes\nolinebreak R_u(H)$, where $R_u(H)$ is the unipotent radical of~$H$.
The composition of~$\rho$ with the projection onto~$L$ does not depend, up to conjugation by $R_u(H)$, on the choice of the Levi factor~$L$.
We shall call this representation the \emph{semisimplification} of~$\rho$, denoted by~$\rho^{ss}$.
The $G$-orbit of~$\rho^{ss}$ in \(\Hom(\Gamma, G)\) (for the action of~$G$ by conjugation at the target) is the unique closed orbit in the closure of the $G$-orbit of~$\rho$.

As a consequence of the openness of the set of Anosov representations, a representation is Anosov as soon as its semisimplification is.
The converse is also true, and will be proved in Section~\ref{subsec:char-terms-lyap}.

\begin{proposition}\label{prop:semisimplification_anosov}
  Let $\Gamma$ be a word hyperbolic group, $G$ a real reductive Lie group, $\theta \subset \Delta$ a nonempty subset of the simple restricted roots of~$G$, and $\rho: \Gamma \to G$ a representation.
  If a representation $\rho'$ belonging to the closure of the $G$-orbit of~$\rho$ in $\Hom(\Gamma,G)$ is $P_{\theta}$-Anosov, then the representation $\rho$ itself is $P_{\theta}$-Anosov.
  In particular, if the semisimplification of $\rho$ is $P_{\theta}$-Anosov, then $\rho$ is $P_{\theta}$-Anosov.
\end{proposition}

\begin{proof}
Recall from Section~\ref{subsubsec:examples-properties} that being $P_{\theta}$-Anosov is an open property which is invariant under the action of $G$ on $\Hom(\Gamma,G)$.
Let $\rho'\in\Hom(\Gamma,G)$ be $P_{\theta}$-Anosov.
There is a neighborhood $\mathcal{U}\subset\Hom(\Gamma,G)$ of~$\rho'$ consisting of $P_{\theta}$-Anosov representations.
If $\rho'$ belongs to the closure of the $G$-orbit of~$\rho$, then $\rho$ admits a conjugate in~$\mathcal{U}$, hence $\rho$ is $P_{\theta}$-Anosov.
\end{proof}

Finally, semisimplification does not change the values taken by the Lyapunov projection $\lambda$ of Section~\ref{subsec:proximal}:

\begin{lemma}\label{lem:lyapu_not_changed}
  Let $\rho: \Gamma \to G$ be a homomorphism from a group $\Gamma$ into a reductive Lie group~$G$, and let $\lambda : G\to\overline{\aaa}^+$ be a Lyapunov projection for~$G$.
  Then the semisimplification $\rho^{ss}$ of~$\rho$ satisfies, for all $\gamma\in\Gamma$,
  \[\lambda( \rho^{ss}(\gamma)) = \lambda( \rho(\gamma)).\]
\end{lemma}

\begin{proof}
  There is a sequence $(\rho_n)_{n\in \NN}$ of conjugates of $\rho$ converging to $\rho^{ss}$.
  The map~$\lambda$ is invariant under conjugation and continuous.
\end{proof}

\begin{remark}\label{rem:Ano-not-ss}
There exist $P_{\theta}$-Anosov representations that are not semisimple.
For instance, let $\Gamma$ be a free group and $\rho : \Gamma\to\SL_2(\RR)$ a convex cocompact representation.
By embedding $\SL_2(\RR)$ into the upper left corner of $\SL_3(\RR)$, we see $\rho$ as a representation of $\Gamma$ into $G = \SL_3(\RR)$.
It easily follows from Theorem~\ref{thm:char_ano} that $\rho : \Gamma\to G$ is $P_{\Delta}$-Anosov. 
By embedding $\RR^2$ into the upper right corner of $\SL_3(\RR)$, any $\rho$-cocycle $\zeta : \Gamma\to\RR^2$ defines a representation $\rho_{\zeta} : \Gamma\to G=\SL_3(\RR)$ with semisimplification $\rho_{\zeta}^{ss} = \rho$.
This representation is $P_{\theta}$-Anosov by Proposition~\ref{prop:semisimplification_anosov}; it is semisimple only if \(\zeta\) is a \(\rho\)-coboundary.
\end{remark}

%%%%%%%%%%%%%%%%%%%%%%%%%%%%%%%%%%%%%%%%%%%%%%%%%%%
\section{Reducing to Anosov representations into $\GL_{\KK}(V)$}\label{sec:repr-semis-lie}

Let $G$ be a reductive Lie group and $\theta\subset\Delta$ a nonempty subset of the simple restricted roots of~$G$.
In this section we explain that there exist (infinitely many) finite-dimensional linear representations $(\tau,V)$ of~$G$ over $\KK=\RR$, $\CC$, or~$\bH$ with the property that a homomorphism $\rho:\Gamma\rightarrow G$ is $P_{\theta}$-Anosov if and only if the composed homomorphism $\tau\circ\rho : \Gamma \rightarrow\GL_{\KK}(V)$ is Anosov with respect to the stabilizer of a line (Lemma~\ref{lem:theta-comp-exists} and Proposition~\ref{prop:theta-comp-Anosov}).
This will be used with $\KK=\RR$ in the proofs of Section~\ref{sec:boundary}: it will make computations simpler by reducing them to the group $\GL_{\RR}(V)$.
Certain technical lemmas, and the possibility of working with \(\KK=\CC\) or~\(\bH\), will also be used in Section~\ref{sec:proper}.

The section is organized as follows. 
In Section~\ref{subsec:weights} we introduce some notation.
In Section~\ref{subsec:theta-compat} we introduce the notion of \emph{$\theta$-proximal} linear representation of~$G$; we state that irreducible $\theta$-proximal representations $(\tau,V)$ exist in abundance and make the link with Anosov representations.
The core of the proofs lies in Section~\ref{subsec:proof-theta-emb}.

\emph{All linear representations in this paper are understood to be finite-di\-men\-sional.}

%%%%%%%%%%%%%%%%%%%%%%%%%
\subsection{Notation: Restricted weights of linear representations of~$G$}
\label{subsec:weights}

Let $G$ be a real reductive Lie group as in Section~\ref{subsec:some-structure-semi}.
We use the notation of Section~\ref{sec:preliminaries}:
in particular, we denote by $(\cdot,\cdot)$ a $W$-invariant scalar product on~$\aaa$ and by $\Vert\cdot\Vert$ the induced Euclidean norm on~$\aaa$; we use the same symbols for the induced scalar product and norm on~$\aaa^{\ast}$.
Any linear representation $(\tau,V)$ of~$G$ decomposes under the action of $\aaa$; the joint eigenvalues (elements of~$\aaa^*$) are called the \emph{restricted weights} of $(\tau,V)$.
The union of the restricted weights of all linear representations of~$G$ is the set
\[\Phi = \left \{ \alpha \in \aaa^* \: \middle | \: 2\frac{(\alpha, \beta)}{(\beta, \beta)} \in \ZZ \quad \forall \beta \in \Sigma \right \} \]
which projects to a lattice of~$\aaa_{s}^{\ast}$; the set $\Phi$ is discrete if and only if $G$ is semisimple.
Let $\z(\g)^0\subset\aaa^*$ be the annihilator of~$\z(\g)$, \ie the subspace of~$\aaa^*$ consisting of linear forms vanishing on~$\z(\g)$; it identifies with the dual $\aaa_{s}^{*}$ of $\aaa_s$.
Similarly, let $\aaa_{s}^0\subset\aaa^*$ be the annihilator of~$\aaa_s$.
For any $\alpha\in\Delta$, let $\omega_{\alpha}\in \z(\g)^0\subset\aaa^*$ be the \emph{fundamental weight} associated with~$\alpha$, defined by 
\begin{equation}\label{eqn:9}
2\frac{( \omega_{\alpha}, \beta)}{(\beta, \beta)} = \delta_{\alpha, \beta} \quad \text{for all} \
\beta \in \Delta
\end{equation}
where $\delta_{\cdot, \cdot}$ is the Kronecker symbol.
Then 
\[\Phi = \aaa_s^0 + \sum_{\alpha \in \Delta} \ZZ \omega_\alpha, \]
and $(\omega_\alpha)_{\alpha \in \Delta}$ projects to a basis of~$\aaa_s^*$.
The set of \emph{dominant} weights is the semigroup $\Phi^+ = \aaa_s^0 + \sum_{\alpha \in \Delta} \NN \omega_\alpha$.
The cone generated by the positive roots (or, equivalently, by the simple roots) determines a partial ordering on~\(\aaa^*\), given by
\begin{equation*}
\nu \leq \nu' \quad \Longleftrightarrow \quad \nu' - \nu \in \sum_{\alpha \in \Sigma^+} \RRp \alpha = \sum_{\alpha \in \Delta} \RRp \alpha.
\end{equation*}
Given an \emph{irreducible} linear representation $(\tau,V)$ of~$G$, the set of restricted weights of $\tau$ admits, for that ordering, a unique maximal element (see \eg \cite[Cor.\,3.2.3]{GoodmanWallach}), which is a dominant weight called the \emph{highest weight} of~$\tau$; we denote it by~$\chi_{\tau}$.

%%%%%%%%%%%%%%%%%%%%%%%%%
\subsection{Compatible and proximal linear representations of~$G$}\label{subsec:theta-compat}

For $\KK=\RR$, $\CC$ or~$\bH$, let $V$ be a finite-dimensional (right) \(\KK\)-vector space.
Recall that an endomorphism $g \in \GL_{\KK}(V)$ is said to be \emph{proximal in} $\PP_\KK(V) = (V-\{0\})/\KK^*$ if it has a unique eigenvalue of maximal modulus and if the corresponding eigenspace is one-dimensional.
This coincides with both Definitions~\ref{defi:prox-G/P}.\eqref{item:prox} and~\ref{defi:prox-G/P}.\eqref{item:lambda>0} for $G=\GL_\KK(V)$ and $P_{\theta}$ the stabilizer of a line in~$V$.
The eigenspace corresponding to the highest eigenvalue then gives rise to a unique attracting fixed point $\xi_g^+\in\PP_\KK(V)$ for the action of $g$ on $\PP_\KK(V)$.
There is a unique complementary hyperplane $H_g^-$ which is stable under~$g$, and $\lim_{n \to +\infty} g^n\cdot x = \xi_g^+$ for all $x \in \PP_\KK(V)\smallsetminus\nolinebreak\PP_\KK(H_g^-)$. 

Let $G$ be a reductive group as above.
We shall say that a linear representation $\tau : G \to \GL_\KK(V)$ is \emph{proximal} if the group $\tau(G) \subset \GL_\KK(V)$ contains an element which is proximal in $\PP_\KK(V)$.
For irreducible~$\tau$, this is equivalent to the highest-weight space $V^{\chi_{\tau}}\subset V$ being a line.

We introduce the following notions. 

\begin{definition} \label{defi:theta-compatible}
Let $\theta \subset \Delta$ be a nonempty subset of the simple restricted roots of~$G$.
An irreducible representation $\tau : G\to\GL_{\KK}(V)$ with highest weight~$\chi_{\tau}$ is
\begin{enumerate}
   \item \emph{$\theta$-compatible} if
  \[ \{\alpha \in \Delta \mid (\chi_\tau, \alpha) > 0\} = \theta \, ;\]
\item \emph{\(\theta\)-proximal} if it is proximal and \(\theta\)-compatible.
\end{enumerate}
\end{definition}

Since the highest weight $\chi_{\tau}$ of any irreducible representation $(\tau,V)$ belongs to $\Phi_+ = \aaa_s^0 + \sum_{\alpha \in \Delta} \NN \omega_\alpha$, we have that $(\tau,V)$ is $\theta$-compatible if and only if
\[ \chi_\tau \in \aaa_s^0 + \sum_{\alpha\in\theta} \NN^{\ast} \, \omega_{\alpha} . \]
We shall use the following fact.

\begin{lemma}\label{lem:theta-comp-exists}
For any real reductive Lie group $G$, there is an integer \(N\geq 1\) such that any \(\chi \in N \sum_{\alpha \in \Delta} \NN \omega_\alpha\) is the highest weight of some irreducible proximal linear representation \((\tau,V)\) of~$G$.
By definition, such a representation \((\tau,V)\) is \(\theta\)-compatible (for some nonempty subset \(\theta\) of \(\Delta\)) if and only if \(\chi \in \sum_{\alpha \in \theta} \NN^{\ast} \omega_\alpha\).
\end{lemma}

\begin{proof}
The image \(H := \Ad(G_0) \subset \GL_{\RR}(\g)\) of the identity component \(G_0\) of~$G$ under the adjoint representation is a connected, semisimple, linear Lie group whose Lie algebra $\h$ is isomorphic to~\(\g_s\).
Any weight \(\chi\) for the group~$G$ induces a weight for the group~\(H\).
By results of Abels--Margulis--Soifer \cite[Th.\,6.3]{Abels_Margulis_Soifer_Prox} and Helgason \cite[Ch.\,V, Th.\,4.1]{HelgasonGGA} (see \cite[\S\,2.3]{Benoist_autcon}), any weight $\chi_1 \in 2  \sum_{\alpha \in \Delta} \NN \omega_\alpha$ is the highest weight of an irreducible proximal representation $(\tau_1, V_1)$ of~\(H\), hence of~\(G_0\). 
The induced representation \(V_2= \mathrm{Ind}_{G_0}^{G} V_1\) is an irreducible representation of~\(G\); its highest weight is again \(\chi_1\) but the weight space \(V_{2}^{\chi_1}\) is now \(p\)-dimensional, where $p:=[G:G_0]$ is the number of connected components of~$G$.
The factor \((\tau,V)\) in \(\Lambda^p(V_2)\) generated by \(\Lambda^p(V_2^{\chi_1})\) is then an irreducible proximal representation of~\(G\) with highest weight \(\chi= p \chi_1\).
Thus \(N=2p\) has the desired property.
\end{proof}

The relevance of the notion of $\theta$-proximality lies in the following two propositions.

\begin{proposition}\label{prop:theta_emb}
  Let $(\tau,V)$ be an irreducible, $\theta$-proximal linear representation of~$G$ with highest weight~$\chi_{\tau}$.
  Let $V^{\chi_{\tau}}$ be the weight space corresponding to $\chi_{\tau}$  in~$V$ and \(V_{< \chi_\tau}\) the sum of all other weight spaces. 
  \begin{enumerate}[label=(\alph*),ref=\alph*]
  \item\label{item:a_thtmb} The stabilizer in~$G$ of \(V^{\chi_\tau}\) (\resp \(V_{< \chi_\tau}\)) is the parabolic subgroup \(P_\theta\) (\resp\nolinebreak \(P^{-}_{\theta}\)).
  \item \label{item:b_thtmb} The maps $g\mapsto\tau(g)V^{\chi_{\tau}}$ and $g\mapsto\tau(g) V_{<\chi_{\tau}}$ induce $\tau$-equivariant embeddings
    \[ \iota^+: G/P_{\theta} \to \PP_\KK(V) \quad\mathrm{and}\quad \iota^-: G/P^-_{\theta} \to \PP_\KK(V^*).\]
    Two parabolic subgroups $P\in G/P_\theta$ and $Q\in G/P^{-}_{\theta}$ of~$G$ are transverse if and only if $\iota^+(P)$ and $\iota^-(Q)$ are transverse.
  \item \label{item:c_thtmb} For an element \(g\in G\), the following conditions are equivalent: 
  \begin{enumerate}[label=(\roman*),ref=\roman*]
    \item\label{item:prox-G} $g$ has an attracting fixed point $\xi^+_g\in G/P_{\theta}$;
    \item\label{item:lambda>0-G} $\langle\alpha,\lambda(g)\rangle>0$ for all $\alpha\in\theta$;
    \item\label{item:prox-P(V)} $\tau(g)$ is proximal in $\PP_{\KK}(V)$.
  \end{enumerate}
  In this case the attracting fixed point \(\xi^{+}_{g}\in G/P_\theta\) of~$g$ is unique, and its image \(\iota^+(\xi^+_g) \in \PP_\KK(V)\) is the unique fixed point of \(\tau(g)\).
  \item \label{item:d_thtmb}
  A similar statement holds after replacing $(\theta,\iota^+,V)$ with $(\theta^\star,\iota^-,V^*)$.
\end{enumerate}
\end{proposition}

Here we use the identification of Remark~\ref{rem:pnorm}.
Recall also Example~\ref{ex:transv-GL} characterizing transversality in $\PP_\KK(V)$.
Proposition~\ref{prop:theta_emb} will be proved in Section~\ref{subsec:proof-theta-emb} just below.

\begin{remark}
For $G=\GL_d(\RR)$, for $\theta=\{ \varepsilon_i-\varepsilon_{i+1}\}$, and for $V=\Lambda^i\RR^d$, the space $G/P_{\theta}$ is the Grassmannian of $i$-dimensional planes of~$\RR^d$ and the map $\iota^+$ of Proposition~\ref{prop:theta_emb} is the \emph{Pl\"ucker embedding}.
\end{remark}

\begin{proposition} \label{prop:theta-comp-Anosov}
Let $(\tau,V)$ be an irreducible, $\theta$-proximal linear representation of~$G$ over $\KK=\RR$, $\CC$ or~$\bH$.
Let $\Gamma$ be a word hyperbolic group and $\theta\subset\Delta$ a nonempty subset of the simple restricted roots of~$G$.
Then 
\begin{enumerate}
  \item\label{item:theta-comp-xi} there exist continuous, $\rho$-equivariant, dynamics-preserving, transverse boundary maps $\xi^+ : \partial_{\infty}\Gamma\to G/P_{\theta}$ and $\xi^- : \partial_{\infty}\Gamma\to G/P_{\theta}^-$ if and only if there exist continuous, $(\tau\circ\rho)$-equivariant, dynamics-preserving, transverse boundary maps $\xi_V^+ : \partial_{\infty}\Gamma\to\PP_{\KK}(V)$ and $\xi_V^- : \partial_{\infty}\Gamma\to\PP_{\KK}(V^*)$;
  \item\label{item:theta-comp-Ano} a representation $\rho: \Gamma \to G$ is $P_{\theta}$-Anosov if and only if $\tau \circ \rho : \Gamma \to\nolinebreak \GL_{\KK}(V)$ is \(P_{\varepsilon_1-\varepsilon_2}\)-Anosov (\ie Anosov with respect to the stabilizer of a line, see Example~\ref{ex:roots}).
\end{enumerate}
\end{proposition}

Proposition~\ref{prop:theta-comp-Anosov}.\eqref{item:theta-comp-Ano} was proved in \cite[\S\,4]{Guichard_Wienhard_DoD} for $\KK=\RR$, using the characterization of $P_{\theta}$-Anosov representations in terms of the coset distance map $d_{\mu_{\theta}}$ (Definition~\ref{defi:ano2}).
We could use the same characterization to prove Proposition~\ref{prop:theta-comp-Anosov}.\eqref{item:theta-comp-Ano} in general.
However, here we shall instead use the characterization of Anosov representations given by Theorem~\ref{thm:char_ano}.\eqref{item:away-from-walls}, which is simpler.
This characterization is established in Sections \ref{subsec:prelim-result}--\ref{subsec:when-cart-proj}; therefore, we postpone the proof of Proposition~\ref{prop:theta-comp-Anosov}.\eqref{item:theta-comp-Ano} to Section~\ref{sec:anos-repr-embedd}.

%%%%%%%%%%%%%%%%%%%%%%%%%
\subsection{Embeddings into projective spaces}\label{subsec:proof-theta-emb}

This subsection is devoted to the proof of Propositions \ref{prop:theta_emb} and~\ref{prop:theta-comp-Anosov}.\eqref{item:theta-comp-xi}, as well as Lemma~\ref{lem:attr-basin}.
As in Theorem~\ref{thm:constr-xi}, we denote by $\Sigma^+_{\theta}$ the set of positive roots that do \emph{not} belong to the span of $\Delta\smallsetminus\theta$.

\begin{remark}
  \label{rem:Cartan_comp_rep}
  For a linear representation \((\tau, V)\) of~\(G\), we will always choose a Cartan decomposition of \(\GL_\KK(V)\) compatible with that of~\(G\).
  This means that the basis \((e_1, \dots, e_d)\) of~\(V\) providing the isomorphism \(\GL_\KK(V) \simeq \GL_d(\KK)\) is a basis of eigenvectors for the action of~\(\mathrm{d}_e\tau(\aaa)\), and that the group \(\tau(K)\) is included in \(\OO(d)\) or \(\U(d)\) or \(\Sp(d)\), depending on whether \(\KK=\RR\) or \(\CC\) or \(\bH\).
  We shall always assume that \(e_1 \in V^{\chi_\tau}\), so that \(\langle \chi_\tau, Y\rangle = \langle \varepsilon_1, \mathrm{d}_e\tau(Y)\rangle\) for all \(Y\in \aaa\).
  The Cartan projection for \(\GL_\KK(V)\) will be denoted by \(\mu_{\GL_\KK(V)}\) and the Lyapunov projection by \(\lambda_{\GL_\KK(V)}\).
\end{remark}

Proposition~\ref{prop:theta_emb} relies on the following fact, which will also be used in Section~\ref{sec:proper}.

\begin{lemma}\label{lem:tau}
Suppose $\tau : G\to\GL_{\KK}(V)$ is irreducible and $\theta$-compatible.
Then 
\begin{enumerate}  
  \item \label{item:nonhighest-weights}
  for any weight $\chi$ of~$\tau$, we have $\chi_{\tau} - \chi \in \sum_{\alpha\in\Sigma^+_{\theta}} \NN\alpha$;
  \item \label{item:diff-weights-theta}
  for any $\alpha\in\theta$, the element $\chi_{\tau}-\alpha\in\aaa^*$ is a weight of~$\tau$.
\item \label{item:eps1-eps2} If \(\tau\) is \(\theta\)-proximal, then \(\langle \varepsilon_1 - \varepsilon_2, \mu_{\GL_\KK(V)}( \tau(g))\rangle = \min_{\alpha \in \theta} \langle \alpha, \mu(g)\rangle\) and\linebreak \(\langle \varepsilon_1 - \varepsilon_2, \lambda_{\GL_\KK(V)}( \tau(g))\rangle = \min_{\alpha \in \theta} \langle \alpha, \lambda(g)\rangle\) for all \(g\in G\).
\end{enumerate}
\end{lemma}

\begin{proof}
Let $\Phi':=\sum_{\alpha\in\Delta} \ZZ\alpha\subset\Phi$ be the root lattice of~$G$.
The set of weights of~$\tau$ is the intersection of $\chi_{\tau}+\Phi'$ with the convex hull of the $W$-orbit of $\chi_{\tau}$ in~$\aaa^{\ast}$ (see \cite[Prop.\,3.2.10]{GoodmanWallach}).
Therefore, in order to prove \eqref{item:nonhighest-weights}, it is sufficient to prove that $\chi_{\tau} - w\cdot\chi_{\tau} \in \sum_{\alpha\in\Sigma^+_{\theta}}\RRp \alpha$ for all $w\in W$, or in other words that $\chi_{\tau} - w\cdot\chi_{\tau} \notin \mathrm{span}(\Delta\smallsetminus\theta)$ for all $w\in W$ with $w\cdot\chi_{\tau}\neq\chi_{\tau}$.
By definition of $\theta$-compatibility, $\chi_{\tau}$ belongs to the orthogonal of $\mathrm{span}(\Delta\smallsetminus\theta)$.
Therefore, for any $w\in W$, if $\chi_{\tau} - w\cdot\chi_{\tau} \in \mathrm{span}(\Delta\smallsetminus\theta)$, then
\[ \Vert w\cdot\chi_{\tau}\Vert^2 = \Vert\chi_{\tau}\Vert^2 + \Vert\chi_{\tau}-w\cdot\chi_{\tau}\Vert^2. \]
But this is also equal to $\Vert\chi_{\tau}\Vert^2$ by $W$-invariance of the norm, and so \(\chi_{\tau}=w\cdot\chi_{\tau}\).
This proves \eqref{item:nonhighest-weights}.
For any $\alpha\in\Delta$, the orthogonal reflection $s_{\alpha}\in W$ in the hyperplane $\Ker(\alpha)$ satisfies
\[ s_{\alpha}\cdot\chi_{\tau} = \chi_{\tau} - 2\frac{(\chi_{\tau},\alpha)}{(\alpha,\alpha)}\alpha, \]
and $s_{\alpha}\cdot\chi_{\tau}$ is a weight of~$\tau$.
Therefore the intersection of $\chi_{\tau}+\ZZ\alpha$ with the segment $[\chi_{\tau}, \chi_{\tau} - 2\frac{(\chi_{\tau},\alpha)}{(\alpha,\alpha)}\alpha]$ consists of weights of~$\tau$.
In particular, $\chi_{\tau}-\alpha$ is a weight of~$\tau$ since $2\frac{(\chi_{\tau},\alpha)}{(\alpha,\alpha)}\geq 1$ and $\alpha\in\Phi$.
This proves~\eqref{item:diff-weights-theta}.

Point~\eqref{item:eps1-eps2} follows from~\eqref{item:nonhighest-weights}--\eqref{item:diff-weights-theta} and from the definition of $\Sigma^+_{\theta}$.
\end{proof}

\begin{proof}[Proof of Proposition~\ref{prop:theta_emb}]
\eqref{item:a_thtmb} Let us first prove that the stabilizer of $V^{\chi_{\tau}}$ in~$\g$ is $\mathrm{Lie}(P_{\theta})$.
To simplify notation, we abusively write $\tau$ for the derivative $\mathrm{d}_e\tau : \g\to\gl(V)$.
We use the notation \(\g_0\), \(\g_\alpha\) of Section~\ref{subsubsec:lie-algebr-decomp}.
The fact that this derivative $\tau : \g\to\gl(V)$ is a Lie algebra homomorphism implies that for any root $\alpha\in\Sigma$ and any weight $\chi$~of~$\tau$,
\begin{equation}\label{eqn:alpha-V-chi}
\tau(\g_\alpha) V^\chi \subset V^{\chi+\alpha}.
\end{equation}
In particular, \(\tau(\g_0) V^{\chi_\tau} \subset V^{\chi_\tau}\), and  \(\tau(\g_\alpha) V^{\chi_\tau} = \{0\}\) for all \(\alpha\in \Sigma^+\) by \eqref{eqn:alpha-V-chi} and maximality of~\(\chi_\tau\) for the partial order of Section~\ref{subsec:weights}.
This means that the stabilizer of \(V^{\chi_\tau}\) in~$\g$ contains the Lie algebra $\mathrm{Lie}(P_\Delta)$, hence it is of the form \(\mathrm{Lie}(P_{\theta'})\) for a unique subset \(\theta'\) of \(\Delta\) (see Fact~\ref{fact:para-lie-subalgebra}).
By \eqref{eqn:alpha-V-chi} and Lemma~\ref{lem:tau}, for any \(\alpha \in \Delta \smallsetminus \theta\) we have  
\(\tau( \g_{-\alpha}) V^{\chi_\tau} = \{ 0\}~\subset V^{\chi_\tau}\); moreover, for any \(\alpha \in \theta\) we have $\{ 0\} \neq \tau( \g_{-\alpha}) V^{\chi_\tau} \subset V^{\chi_{\tau}-\alpha}$ (see \eg \cite[proof of Lem.\,3.2.9]{GoodmanWallach}), hence $\tau( \g_{-\alpha}) V^{\chi_\tau}\neq V^{\chi_{\tau}}$.
Therefore, $\theta'=\theta$, \ie the stabilizer of $V^{\chi_{\tau}}$ in~$\g$ is $\mathrm{Lie}(P_{\theta})$.

In particular, the stabilizer of $V^{\chi_\tau}$ in~$G$ is contained in~$P_\theta$.
For the reverse inclusion, note that for any $g\in P_{\theta}$ the line \(\tau(g) V^{\chi_\tau}\) is stable under \(\Ad(g)\,\mathrm{Lie}(P_\theta) = \mathrm{Lie}(P_\theta)\), hence in particular under \(\mathrm{Lie}(P_\Delta)\).
But $V^{\chi_{\tau}}$ is the only \(\mathrm{Lie}(P_\Delta)\)-invariant line in~$V$, by arguments similar to the above (any \(\mathrm{Lie}(P_\Delta)\)-invariant line \(L\) is \(\aaa\)-invariant, hence contained in a weight space \(V^{\chi}\), and for \(\chi\neq \chi_\tau\) there always exists \(\alpha\in \Delta\) with \(\{0\} \neq \tau( \g_\alpha) L \subset V^{\chi+\alpha}\)).
This proves that the stabilizer of $V^{\chi_\tau}$ in~$G$ is exactly~$P_\theta$.

Similarly, the stabilizer in~$G$ of the hyperplane \(V_{<\chi_\tau}\) is~\(P_{\theta}^{-}\).

\eqref{item:b_thtmb} The map $\iota^+ : G/P_{\theta} \to \PP_{\KK}(V)$ is well defined and injective by \eqref{item:a_thtmb}; it is clearly a $\tau$-equivariant embedding.
Similarly, $\iota^- : G/P_{\theta}^- \to \PP_{\KK}(V^*)$ is a $\tau$-equivariant embedding.

Let us show that two parabolic subgroups $P \in G/P_{\theta}$ and $Q \in G/P_{\theta}^-$ of~$G$ are transverse if and only if $\iota^+(P)$ and $\iota^-(Q)$ are transverse.
Note that transversality is invariant under the $G$-action, both in $G/P_{\theta}\times G/P_{\theta}^-$ and in $\PP_{\KK}(V)\times\PP_{\KK}(V^*)$ (by $\tau$-equivariance of $\iota^+$ and~$\iota^-$).
We can write $P = g P_{\theta} g^{-1}$ and \(Q = h P_{\theta}^- h^{-1}\) where $g,h\in G$.
By the Bruhat decomposition (see \cite[Th.\,5.15]{Borel_Tits_ihes}), there exist \((p,p') \in P_{\theta}^- \times P_{\theta}\) and \(w\in W=N_K(\aaa)/Z_K(\aaa)\) such that \(h^{-1}g = p w p'\).
Up to conjugating both $P$ and~$Q$ by $p^{-1} h^{-1}$, we may assume that $(P,Q) = (wP_{\theta}w^{-1}, P_{\theta}^-)$, so that $\iota^+(P) = V^{w\cdot \chi_\tau}$ and $\iota^-(Q) = V^{< \chi_\tau}$.
Then $\iota^+(P)$ and~$\iota^-(Q)$ are transverse if and only if $w\cdot\chi_\tau = \chi_{\tau}$.
By \cite[Ch.\,V, \S\,3, Prop.\,1]{Bourbaki}, this happens if and only if \(w\) belongs to the subgroup \(W_{\Delta\smallsetminus \theta}\) of~$W$ generated by the reflections $s_\alpha$ for $\alpha\in \Delta\smallsetminus \theta$.
Therefore, it is sufficient to prove that $P = wP_{\theta}w^{-1}$ and $Q = P_{\theta}^-$ are trans\-verse if and only if $w \in W_{\Delta\smallsetminus\theta}$.
If \( w \in W_{\Delta\smallsetminus\theta}\), it is not difficult to see that
\[ \Ad(w) \mathrm{Lie}(P_\theta) =\mathrm{Lie}(P_\theta), \]
and so $P$ and~$Q$ are transverse.
Conversely, suppose $w\notin W_{\Delta\smallsetminus\theta}$; let us prove that \(P = w P_\theta w^{-1}\) and \(Q = P^{-}_{\theta}\) are \emph{not} transverse.
Let \(s_{\alpha_1} \cdots s_{\alpha_q}\) be a reduced expression of~$w$ and \(i\in \{1,\dots,q\}\) the smallest index such that \(\alpha_i\) belongs to~\(\theta\).
By \cite[Ch.\,VI, \S\,1, Cor.\,2]{Bourbaki}, the root \(\beta = s_{\alpha_q} \cdots s_{\alpha_{i+1}}(\alpha_i)\) is positive, hence \(\g_\beta \subset \mathrm{Lie}(P_\theta)\).
Its image under~$w$ is \(-\beta'=- s_{\alpha_1}\cdots s_{\alpha_{i-1}}(\alpha_i)\),
which satisfies
\[ \g_{-\beta'} = \Ad(w) \g_\beta \subset \Ad(w)\mathrm{Lie}( P_\theta) = \mathrm{Lie}(w P_\theta w^{-1}). \]
Since \(\omega_{\alpha_i}\) is invariant under the reflections $s_\alpha$ for $\alpha \neq \alpha_i$ and since the scalar product $(\cdot,\cdot)$ is $W$-invariant,
\[ ( \beta', \omega_{\alpha_i}) = ( s_{\alpha_1}\cdots s_{\alpha_{i-1}}(\alpha_i), s_{\alpha_1}\cdots s_{\alpha_{i-1}}(\omega_{\alpha_i})) = (\alpha_i, \omega_{\alpha_i}) >0. \]
This forces \(\beta'\) to belong to \(\Sigma^{+}_{\theta}\); this means that \(\g_{-\beta'} \in \mathfrak{u}^{-}_{\theta}\) and \(\mathrm{Lie}( wP_\theta w^{-1})\) intersects nontrivially the unipotent radical of \(\mathrm{Lie}(P^{-}_{\theta})\).
Thus \(P = w P_\theta w^{-1}\) and \(Q = P_{\theta}^{-}\) are not transverse.

\eqref{item:c_thtmb} It follows from the definition that $\tau(g)\in\GL_{\KK}(V)$ is proximal in $\PP_{\KK}(V)$ if and only if \( \langle \varepsilon_1 - \varepsilon_2 , \lambda_{\GL_\KK(V)}(g)\rangle >0\).
This is equivalent to \(\min_{\alpha \in \theta} \langle \alpha, \lambda(g)\rangle >0\) by Lemma~\ref{lem:tau}.\eqref{item:eps1-eps2}.
Thus the equivalence $\eqref{item:lambda>0-G}\Leftrightarrow\eqref{item:prox-P(V)}$ holds.

We claim that if \eqref{item:prox-P(V)} holds, \ie if $\tau(g)$ admits an attracting fixed point $\xi_{\tau(g)}^+$ in $\PP_{\KK}(V)$, then $g$ admits a fixed point in $G/P_{\theta}$.
Indeed, let $g=g_h g_e g_u$ be the Jordan decomposition of~$g$, so that $\tau(g)=\tau(g_h)\tau(g_e)\tau(g_u)$ is the Jordan decomposition of~$\tau(g)$.
If $\xi_{\tau(g)}^+$ is an attracting fixed point for $\tau(g)$, then it is also an attracting fixed point for $\tau(g_h)$; by construction, it is then equal to $\iota^+(m\cdot x_{\theta})$, where $m\in G$ satisfies $g_h\in m\exp(\overline{\aaa}^+)m^{-1}$ and $x_{\theta}=eP_{\theta}\in G/P_{\theta}$.
Then $m\cdot x_{\theta}$ is a fixed point of~$g$, by equivariance and injectivity of~$\iota^+$.
Thus, if either of \eqref{item:prox-G}, \eqref{item:lambda>0-G}, or \eqref{item:prox-P(V)} holds, then $g$ admits a fixed point in $G/P_{\theta}$.

We now prove the equivalence $\eqref{item:prox-G}\Leftrightarrow\eqref{item:lambda>0-G}$ for $g$ admitting a fixed point in $G/P_{\theta}$.
Up to replacing $g$ with a conjugate (which does not change $\lambda(g)$), we may assume that this fixed point is~$x_{\theta}$.
The tangent space $T_{x_{\theta}}(G/P_{\theta})$ with the derivative of the action of~$g$ on $G/P_{\theta}$ identifies with $\g/\mathrm{Lie}(P_{\theta})$ with the adjoint action of~$g$.
Further identifying $\g/\mathrm{Lie}(P_{\theta})$ with $\bigoplus_{\alpha\in\Sigma_{\theta}^+} \g_{-\alpha}$ using \eqref{eqn:decompose}, we see that the eigenvalues of the derivative at~$x_{\theta}$ of the action of $g$ on $G/P_{\theta}$ are the $e^{-\langle\alpha,\lambda(g)\rangle}$ for $\alpha\in\Sigma_{\theta}^+$.
By definition of~$\Sigma_{\theta}^+$, these eigenvalues are all $<1$ (\ie \eqref{item:prox-G} holds) if and only if $\langle\alpha,\lambda(g)\rangle>0$ for all $\alpha\in\theta$ (\ie \eqref{item:lambda>0-G} holds).
In this case $V^{\chi_{\tau}}=\iota^+(x_{\theta})$ is an attracting fixed point of $\tau(g)$ in $\PP_{\KK}(V)$ since \( \langle \varepsilon_1 - \varepsilon_2 , \lambda_{\GL_\KK(V)}(g)\rangle >0\).
This proves that in general, if $\xi_g^+$ is an attracting fixed point of~$g$ in \(G/P_\theta\), then $\iota^+(\xi_g^+)$ is an attracting fixed point of \(\tau(g)\) in \(\PP_\KK(V)\); the uniqueness of $\xi_g^+$ follows from the uniqueness of $\xi_{\tau(g)}^+$ and from the injectivity of~$\iota^+$.
\end{proof}

\begin{proof}[Proof of Proposition~\ref{prop:theta-comp-Anosov}.\eqref{item:theta-comp-xi}]
Let $\iota^+: G/P_{\theta} \to \PP_\KK(V)$ and $\iota^-: G/P^-_{\theta} \to \PP_\KK(V^*)$ be the $\tau$-equivariant embeddings given by Proposition~\ref{prop:theta_emb}.\eqref{item:b_thtmb}.

Suppose there exist continuous, $\rho$-equivariant, transverse, dynamics-pre\-serv\-ing boundary maps \(\xi^+: \partial_\infty \Gamma \to G/ P_\theta\) and \(\xi^- : \partial_\infty \Gamma \to G/ P_{\theta}^{-}\).
  The maps \(\xi_V^+ := \iota^+ \circ \xi^+ : \partial_{\infty}\Gamma\to\PP_{\RR}(V)\) and \( \xi_V^- := \iota^- \circ \xi^- : \partial_{\infty}\Gamma\to\PP_{\RR}(V^*)\) are continuous and \((\tau\circ \rho)\)-equivariant.
  They are transverse by Proposition~\ref{prop:theta_emb}.\eqref{item:b_thtmb} and dynamics-pre\-serv\-ing by Proposition~\ref{prop:theta_emb}.\eqref{item:c_thtmb}--\eqref{item:d_thtmb}.

Conversely, suppose there exist continuous, $(\tau\circ\rho)$-equivariant, transverse, dynamics-pre\-serv\-ing boundary maps \(\xi_V^+ : \partial_{\infty}\Gamma\to\PP_{\RR}(V)\) and \( \xi_V^- : \partial_{\infty}\Gamma\to\PP_{\RR}(V^*)\).
For any \(\gamma\in\Gamma\) of infinite order with attracting fixed point \(\eta^{+}_{\gamma}\in \partial_\infty \Gamma\), the element \(\tau\circ \rho(\gamma)\) is proximal in \(\PP_\KK(V)\) with attracting fixed point \(\xi^{+}_{V}( \eta^{+}_{\gamma})\).
By Proposition~\ref{prop:theta_emb}.\eqref{item:c_thtmb}, the element \(\rho(\gamma)\in G\) is proximal in \(G/P_\theta\) and its attracting fixed point \(\xi^{+}_{g}\in G/P_\theta\) satisfies \(\iota^+(\xi^+_g) = \xi^{+}_{V}( \eta^{+}_{\gamma})\).
The set
\[ \{ \eta \in \partial_\infty \Gamma \mid \xi^{+}_{V}( \eta) \in \iota^+ ( G/P_\theta)\} \]
is closed and contains the dense set \(\{ \eta^{+}_{\gamma} \mid \gamma \in \Gamma \text{ of infinite order}\}\), hence it is equal to \(\partial_\infty \Gamma\).
Therefore there is a map \(\xi^+ : \partial_\infty \Gamma \to G/P_\theta\) such that \(\xi^{+}_{V} = \iota^+ \circ \xi^+\).
Similarly there is a map \(\xi^- : \partial_\infty \Gamma \to G/P_\theta^-\) such that \(\xi^{-}_{V} = \iota^- \circ \xi^-\).
The maps \(\xi^+\) and \(\xi^-\) are continuous and \(\rho\)-equivariant.
They are transverse by Proposition~\ref{prop:theta_emb}.\eqref{item:b_thtmb} and dynamics-preserving by Proposition~\ref{prop:theta_emb}.\eqref{item:c_thtmb}--\eqref{item:d_thtmb}.
\end{proof}

\begin{proof}[Proof of Lemma~\ref{lem:attr-basin}]
By Lemma~\ref{lem:theta-comp-exists}, there exists an irreducible, $\theta$-proximal, linear representation $(\tau,V)$ of~$G$.
For any $g\in G$, if $\tau(g)\in\GL_{\RR}(V)$ is proximal in $\PP_{\RR}(V)$, then $\tau(g^{-1})$ is proximal in $\PP_{\RR}(V^*)$ by definition of the action on~$V^*$.
The attracting fixed point $\xi^+_{\tau(g)}$ of $\tau(g)$ in $\PP_{\RR}(V)$ is the eigenline of $\tau(g)$ corresponding to the maximal eigenvalue, and the attracting fixed point $\xi^-_{\tau(g^{-1})}$ of $\tau(g^{-1})$ in $\PP_{\RR}(V^*)$ is the hyperplane of~$V$ which is the sum of the eigenspaces of $\tau(g)$ corresponding to nonmaximal eigenvalues.
In particular, $\tau(g)^n\cdot x\underset{\scriptscriptstyle n\to +\infty}{\longrightarrow} \xi^+_{\tau(g)}$ for all $x\in\PP_{\RR}(V)\smallsetminus\xi^-_{\tau(g^{-1})}$.
We conclude using Proposition~\ref{prop:theta_emb}.
\end{proof}

%%%%%%%%%%%%%%%%%%%%%%%%%%%%%%%%%%%%%%%%%%%%%%%%%%%
\section{Anosov representations in terms of boundary maps and Cartan or Lyapunov projections}\label{sec:anos-repr-cart}

In this section we prove the equivalences $\eqref{item:Ano}\Leftrightarrow\eqref{item:away-from-walls}\Leftrightarrow\eqref{item:lin-away-from-walls}$ of Theorems \ref{thm:char_ano} and~\ref{thm:char_ano_lambda_intro}.
(The equivalence $\eqref{item:Ano}\Leftrightarrow\eqref{item:cli}$ of Theorem~\ref{thm:char_ano} will be proved in Section~\ref{sec:boundary}.)
More precisely, we prove the following refinements.
Recall the word length function $\ellGammaf:\Gamma\rightarrow \NN$ and the stable length function $\ellinftyf:\Gamma\rightarrow \RRp$ from Section~\ref{subsubsec:word-length-stable}.

\begin{theorem}\label{thm:char_ano_maps}
  Let $\Gamma$ be a word hyperbolic group, $G$ a real reductive Lie group, and $\theta \subset \Delta$ a nonempty subset of the simple restricted roots of~$G$.
  Let $\rho : \Gamma \to G$ be a representation, and suppose there exist continuous, $\rho$-equivariant, and transverse maps $\xi^+ : \partial_{\infty}\Gamma \to G/ P_\theta$ and $\xi^- : \partial_{\infty} \Gamma \to G/P_{\theta}^-$.
    Then the following conditions are equivalent:
  \begin{enumerate}
  \item\label{item:5} The maps $\xi^+, \xi^-$ lift to a map
    $\tilde{\beta}:\mathcal{G}_\Gamma \rightarrow G/K_\theta$
    satisfying the contraction property \eqref{item:ano-ii} of Definition~\ref{defi:ano2}, \ie $\rho$ is $P_\theta$-Anosov;
  \item\label{item:8} The maps $\xi^+, \xi^-$ are dynamics-preserving for~$\rho$ and 
    \[ \exists c,C>0,\ \forall \alpha \in \theta,\ \forall \gamma \in \Gamma, \quad
    \bigl\langle \alpha,\, \mu( \rho(\gamma)) \bigr\rangle \geq c \, \ellGamma{\gamma}-C \,;\]
  \item\label{item:6} The maps $\xi^+, \xi^-$ are dynamics-preserving for~$\rho$ and
    \[\exists c,C>0,\ \forall \alpha \in \theta,\ \forall \gamma \in \Gamma, \quad
    \bigl\langle \alpha,\, \mu( \rho(\gamma)) \bigr\rangle \geq c \, \| \mu(\rho(\gamma))\|-C \,;\]
  \item\label{item:7} The maps $\xi^+, \xi^-$ are dynamics-preserving for~$\rho$ and
    \[\forall \alpha \in \theta, \quad \lim_{\gamma \to \infty} \bigl\langle \alpha,\, \mu( \rho(\gamma))
    \bigr\rangle = +\infty. \]
\end{enumerate}
\end{theorem}

\begin{theorem}\label{thm:char_ano_lambda}
Let $\Gamma$, $G$, $\rho$, $\theta$ and $\xi^\pm$ be as in Theorem~\ref{thm:char_ano_maps}.
Then the following conditions are equivalent:
  \begin{enumerate}
  \item\label{item:lambda1}  $\rho$ is $P_\theta$-Anosov;
  \item\label{item:lambda2} The maps $\xi^+, \xi^-$ are dynamics-preserving for~$\rho$ and 
    \[ \exists c>0,\ \forall \alpha \in \theta,\ \forall \gamma \in \Gamma, \quad
    \bigl\langle \alpha,\, \lambda( \rho(\gamma)) \bigr\rangle \geq c \, \ellinfty{\gamma} \,;\]
  \item\label{item:lambda3} The maps $\xi^+, \xi^-$ are dynamics-preserving for~$\rho$ and 
    \[\exists c>0,\ \forall \alpha \in \theta,\ \forall \gamma \in \Gamma, \quad    \bigl\langle \alpha,\, \lambda( \rho(\gamma)) \bigr\rangle \geq c \, \| \lambda(\rho(\gamma))\| \,;\]
  \item\label{item:lambda4} The maps $\xi^+, \xi^-$ are dynamics-preserving for~$\rho$ and 
    \[\forall \alpha \in \theta, \quad \lim_{\ellinfty{\gamma} \to +\infty}\, \bigl\langle \alpha,\, \lambda( \rho(\gamma))
    \bigr\rangle = +\infty. \]
  \end{enumerate}
\end{theorem}

\begin{remarks} \label{rem:combine}
  \begin{enumerate}[label=(\alph*),ref=\alph*]
  \item In Theorem~\ref{thm:char_ano_maps}, conditions \eqref{item:8}, \eqref{item:6}, \eqref{item:7} use the Cartan projection~$\mu$, whereas condition~\eqref{item:5} uses the coset distance map $\dtheta$ as in Definition~\ref{defi:ano2}.\eqref{item:ano-ii} of the Anosov property.
    Recall that the range of $\dtheta$ is a finite union of copies of the range of~$\mu$ (see Section~\ref{subsubsec:cart-proj-l_th}). 
  \item \label{item:compatible} If $\Gamma$ is nonelementary, then the existence of a continuous, injective, $\rho$-equivariant map $\xi^+ : \partial_{\infty}\Gamma\to G/P_{\theta}$ as in Theorems \ref{thm:char_ano_maps} and~\ref{thm:char_ano_lambda} (see Remark~\ref{rem:compatible}.\eqref{item:compat-transv->inj}) already implies that $\rho$ has finite kernel and discrete image. 
    Indeed, consider a sequence $(\gamma_n)\in\Gamma^{\NN}$ such that \((\rho(\gamma_n))_{n\in\NN}\) is bounded in~\(G\).
    Seen as a sequence of homeomorphisms of \(G/P_\theta\), the family \((\rho(\gamma_n))_{n\in\NN}\) is equicontinuous.
    Since \(\xi^+\) is injective and continuous and since \(\partial_\infty \Gamma\) is compact, the map \(\xi^+\) is a homeomorphism onto its image.
    Therefore, seen as a sequence of homeomorphisms of \(\partial_\infty \Gamma\), the family \((\gamma_n)_{n\in\NN}\) is equicontinuous. 
    By Fact~\ref{fact:dyn_at_infty}.\eqref{item:4}, no subsequence of \((\gamma_n)_{n\in\NN}\) can diverge to infinity, \ie \((\gamma_n)_{n\in\NN}\) is bounded in~\(\Gamma\).
    \item By Proposition~\ref{prop:word-length-stable}, condition 
   \eqref{item:lambda2} of Theorem~\ref{thm:char_ano_lambda} is also equivalent to:
   
   \medskip
  \begin{itemize}
    \item[\emph{(2')}] \emph{The maps $\xi^+, \xi^-$ are dynamics-preserving for~$\rho$ and
    \[ \exists c,C>0,\ \forall \alpha \in \theta,\ \forall \gamma \in \Gamma, \quad
    \bigl\langle \alpha,\, \lambda( \rho(\gamma)) \bigr\rangle \geq c \, \trl{\gamma} - C ,\]
    where $\trlf : \Gamma\to\NN$ is the translation length function \eqref{eqn:translationlength} on the Cayley graph of~$\Gamma$.}
  \end{itemize}
  
  \medskip
  
  \item \label{item:theta-symmetric} By Fact~\ref{fact:Anosov_opp}, as well as \eqref{eqn:opp-inv-mu} and \eqref{eqn:opp-inv-lambda}, all conditions of Theorems \ref{thm:char_ano_maps} and~\ref{thm:char_ano_lambda} are invariant under replacing $\theta$ with $\theta^{\star}$ or with $\theta\cup\theta^{\star}$.
  Thus, without loss of generality, we may assume that $\theta$ is symmetric (\ie $\theta=\nolinebreak\theta^{\star}$).
\end{enumerate}
\end{remarks}

The implication $\eqref{item:8} \Rightarrow \eqref{item:6}$ of Theorem~\ref{thm:char_ano_maps} follows from the subadditivity of $\Vert\mu\Vert$, see \eqref{eqn:mu-leq-klength}.
The implication $\eqref{item:6} \Rightarrow \eqref{item:7}$ follows from Remark~\ref{rem:combine}.\eqref{item:compatible} above and from the properness of~\(\mu\). 
We start by proving the implications $\eqref{item:5} \Rightarrow \eqref{item:8}$ and $\eqref{item:7} \Rightarrow \eqref{item:5}$ of Theorem~\ref{thm:char_ano_maps} in Sections \ref{subsec:contr-cart-proj} and~\ref{subsec:when-cart-proj}, respectively.
Then Theorem~\ref{thm:char_ano_lambda} is proved in Sections~\ref{subsec:abels-marg-souif} and~\ref{subsec:char-terms-lyap}.

%%%%%%%%%%%%%%%%%%%%%%%%%
\subsection{A preliminary result}
\label{subsec:prelim-result}

In the setting of Theorem~\ref{thm:char_ano_maps}, the maps $\xi^+, \xi^-$ always combine as in Definition~\ref{defi:ano2}.\eqref{item:ano-i} into a continuous, $\rho$-equivariant, flow-invariant map $\tilde{\sigma}:\mathcal{G}_\Gamma\rightarrow G/L_\theta$, of which we choose a continuous lift $\tilde{\beta}:\mathcal{G}_\Gamma\rightarrow G/K_\theta$ (with $L_\theta, K_\theta$ defined as in Section~\ref{subsubsec:cart-proj-l_th}).
We further choose a $\rho$-equivariant set-theoretic lift $\breve{\beta}:\mathcal{G}_\Gamma\rightarrow G$ of $\tilde{\beta}$.
Then for any $(t,v)\in \RR \times \mathcal{G}_\Gamma$, since $\breve{\beta}(\varphi_t \cdot v) L_\theta = \tilde{\sigma}(\varphi_t \cdot v) = \tilde{\sigma}(v) = \breve{\beta}(v) L_\theta$, we have
\[ \breve{\beta}(\varphi_t \cdot v) = \breve{\beta}(v) \, l_{t,v}  \]
for some $l_{t,v} \in L_\theta$ satisfying the following $\Gamma$-invariance property: for all $\gamma\in\Gamma$ and \(v\in \mathcal{G}_\Gamma\),
\begin{equation}\label{eqn:l-inv}
l_{t,\gamma\cdot v} = l_{t,v},
\end{equation}
since \(\breve{\beta}(\gamma\cdot v)^{-1} \breve{\beta}(\varphi_t \cdot \gamma\cdot v) = \breve{\beta}(v)^{-1}\rho(\gamma)^{-1}\rho(\gamma)\breve{\beta}(\varphi_t\cdot v) \).
By definition, $\mu_\theta (l_{t,v}) = \dtheta\big( \tilde{\beta}(v),\tilde{\beta}(\varphi_t \cdot v)\big)$, where $\dtheta$ is the $\theta$-coset distance map (Definition~\ref{defi:mutheta}).
The following cocycle condition is satisfied: for all $t,s\in\RR$ and $v\in\mathcal{G}_{\Gamma}$,
  \begin{equation} \label{eqn:l-cocycle}
l_{t+s,v} = l_{t, \varphi_{s}\cdot v} \, l_{s,v} .
  \end{equation}
We fix this map $(t,v)\mapsto l_{t,v}$ for the remainder of Section~\ref{sec:anos-repr-cart}.

The following lemma will be used in Sections \ref{subsec:contr-cart-proj} and~\ref{subsec:when-cart-proj} to prove the implications $\eqref{item:5}\Rightarrow\eqref{item:8}$ and $\eqref{item:7}\Rightarrow\eqref{item:5}$ of Theorem~\ref{thm:char_ano_maps}. 

\begin{lemma}\label{lem:mu_mutheta}
For any compact subset $\mathcal{D}$ of $\mathcal{G}_\Gamma$, the set $\breve{\beta}(\mathcal{D})$ is relatively compact in~$G$; in particular, $\mathcal{K} := 2 \sup_{v\in \mathcal{D}} \|\mu(\breve{\beta}(v))\|$ is finite.
Moreover, for any $v \in \mathcal{D}$, any $t\in \RR$, and any $\gamma \in \Gamma$ such that $\varphi_t\cdot v \in \gamma\cdot \mathcal{D}$,
\begin{equation*}
\| \mu(l_{t,v}) - \mu(\rho(\gamma))  \| \leq \mathcal{K}.
\end{equation*}
\end{lemma}

\begin{proof}
The set $\breve{\beta}(\mathcal{D})$ is mapped onto the compact set $\tilde{\beta}( \mathcal{D})$ by the proper map $G \to G/K_\theta$, hence it is relatively compact in~$G$.
In particular, the continuous function $\| \mu\|$ is bounded on $\breve{\beta}(\mathcal{D})$, \ie $\mathcal{K}<+\infty$.
  Consider $(v,t,\gamma) \in \mathcal{D}\times\RR\times\Gamma$ such that $\varphi_t \cdot v \in \gamma\cdot \mathcal{D}$.
  Then
\[ \breve{\beta}(v) l_{t,v} = \breve{\beta}(\varphi_t \cdot v)  = \rho(\gamma) \breve{\beta}(\gamma^{-1}\cdot \varphi_t\cdot v )\]
and $\gamma^{-1}\cdot \varphi_t\cdot v \in \mathcal{D}$.
  Therefore, $\rho(\gamma) = g_1 l_{t,v} g_2^{-1}$ for some $g_1$, $g_2$ in $\breve{\beta}(\mathcal{D})$.
  The bound follows by Fact~\ref{fact:mu-subadditive}.
\end{proof}

%%%%%%%%%%%%%%%%%%%%%%%%%
\subsection{The Cartan projection on an Anosov representation} 
\label{subsec:contr-cart-proj}

We first prove the implication $\eqref{item:5} \Rightarrow \eqref{item:8}$ of Theorem~\ref{thm:char_ano_maps}.

Let $\Gamma$ be a word hyperbolic group, $G$ a real reductive Lie group, $\theta \subset \Delta$ a nonempty subset of the simple restricted roots of~$G$, and $\rho: \Gamma\to G$ a representation.
Suppose condition \eqref{item:5} of Theorem~\ref{thm:char_ano_maps} is satisfied, \ie $\rho$ is $P_\theta$-Anosov.
Then the exponential contraction property \eqref{item:ano-ii} in Definition~\ref{defi:ano2} is satisfied: there are constants $c, C>0$ such that for all $\alpha \in \theta$, all $v \in \mathcal{G}_\Gamma$, and all $t\in\RR$,
\begin{equation*}
\bigl\langle \alpha,\, \dtheta( \tilde{\beta}(v), \tilde{\beta}(\varphi_t \cdot v)) \bigr\rangle \geq ct -C.
\end{equation*}
For any $t\in\RR$ and $v\in\mathcal{G}_\Gamma$, let $l_{t,v}\in L_{\theta}$ be defined as in Section~\ref{subsec:prelim-result}.
This equation can be rephrased as 
\begin{equation}\label{eqn:exp-contract}
\langle\alpha,\mu_{\theta}(l_{t,v})\rangle\geq ct-C
\end{equation}
for all $\alpha \in \theta$, all $v \in \mathcal{G}_\Gamma$, and all $t\in\RR$.
We claim that if $t\geq C/c$, then
\begin{equation}\label{eqn:mu=mu-theta}
\mu( l_{t,v}) = \mu_\theta(l_{t,v}).
\end{equation}
Indeed, we always have $\langle \alpha,\, \mu_\theta(l_{t,v})\rangle\geq 0$ for $\alpha \in \Delta \smallsetminus \theta$ by definition \eqref{eqn:thetachamber} of the range $\athetaplus$ of~$\mu_{\theta}$, and if $t\geq C/c$ then $\langle \alpha,\, \mu_\theta(l_{t,v})\rangle\geq 0$ for $\alpha \in \theta$ by \eqref{eqn:exp-contract}.
Therefore, if $t\geq C/c$, then $\mu_\theta(l_{t,v}) \in \overline{\aaa}^+$ for all $v\in\mathcal{G}_\Gamma$, hence $\mu( l_{t,v}) = \mu_\theta(l_{t,v})$ (Remark~\ref{rem:mu-theta-equal-mu-l-in-L}).

By Corollary~\ref{cor:geo_seg_ell}, there are a compact subset $\mathcal{D}$ of~$\mathcal{G}_\Gamma$ and constants $c_1,c_2>0$ such that for any $\gamma\in\Gamma$ there exists $(t,v) \in \RR \times \mathcal{D}$ with $\varphi_t \cdot v \in \gamma \cdot \mathcal{D}$ and $t\geq c_1 \ellGamma{\gamma} -c_2$.
In particular, if $\ellGamma{\gamma}\geq n_0:=(C+cc_2)/cc_1$, then $t\geq C/c$ and so \eqref{eqn:mu=mu-theta} holds. Lemma~\ref{lem:mu_mutheta} applied to the compact set~$\mathcal{D}$ gives $\mathcal{K}>0$ such that
\begin{align*}
  \langle \alpha ,\, \mu(\rho(\gamma))\rangle & \geq \langle \alpha,
  \mu( l_{t,v}) \rangle - \mathcal{K} \| \alpha\|\\
  &= \langle \alpha,
  \mu_\theta( l_{t,v}) \rangle - \mathcal{K} \| \alpha\|\\
  & \geq cc_1 \ellGamma{\gamma} - C',
\end{align*}
where $C'=cc_2 + C + \mathcal{K} \| \alpha\|$.
Up to adjusting the additive constant $C'$, the same inequality also holds for the (finitely many) $\gamma\in\Gamma$ such that $\ellGamma{\gamma}<n_0$, hence condition \eqref{item:8} of Theorem~\ref{thm:char_ano_maps} holds.

This completes the proof of the implication $\eqref{item:5} \Rightarrow \eqref{item:8}$ of Theorem~\ref{thm:char_ano_maps}.

%%%%%%%%%%%%%%%%%%%%%%%%%
\subsection{Weak contraction and Anosov representations}
\label{subsec:weak-contr-anos}

In the course of proving $\eqref{item:7}\Rightarrow\nolinebreak\eqref{item:5}$ of Theorem~\ref{thm:char_ano_maps}, we will need the following characterization of Anosov representations.

\begin{proposition}\label{prop:weakcontr_ano}
  Let $\Gamma$ be a word hyperbolic group, $G$ a real reductive Lie group,~$\theta\subset\nolinebreak\Delta$ a nonempty subset of the simple restricted roots of~$G$, and $\rho:\Gamma \to\nolinebreak G$ a representation.
  Suppose there exist continuous, $\rho$-equivariant, transverse maps $\xi^+: \partial_{\infty} \Gamma \to G/P_\theta$ and $\xi^-: \partial_{\infty} \Gamma \to G/P_{\theta}^-$, and let $\tilde{\sigma} : \mathcal{G}_\Gamma \to G/L_\theta$ be the induced equivariant map and $\tilde{\beta}: \mathcal{G}_\Gamma \to G/K_\theta$ a lift, as in Definition~\ref{defi:ano2}.\eqref{item:ano-i}.
  Suppose also that for any $\alpha \in \theta$,
  \begin{equation}\label{eqn:beta-v-infty}
  \lim_{t \to +\infty}\, \inf_{v \in \mathcal{G}_\Gamma} \big\langle \alpha,\, \dtheta ( \tilde{\beta}(v), \tilde{\beta}(\varphi_t \cdot v)) \big\rangle = +\infty.
  \end{equation}
  Then the representation is $P_\theta$-Anosov (Definition~\ref{defi:ano2}): there exist $c,C>\nolinebreak 0$ such that
  \[\forall \alpha \in \theta,\ \forall t \in \RR,\ \forall v \in
    \mathcal{G}_\Gamma, \quad \big\langle \alpha,\, \dtheta (\tilde{\beta}(v), \tilde{\beta}(\varphi_t \cdot v)) \big\rangle\geq ct -C.\]
\end{proposition}

\begin{remark}\label{rem:inf_on_cpt}
  For fixed $t\in\RR$, by the \(\rho\)-equivariance of \(\tilde{\beta}\) and the invariance of~\(\dtheta\) under the diagonal action of $G$ on $G/K_{\theta}\times G/K_{\theta}$ (see Definition~\ref{defi:mutheta}), we have
  \[\inf_{v \in
    \mathcal{G}_\Gamma} \big\langle \alpha,\, \dtheta ( \tilde{\beta}(v),
  \tilde{\beta}(\varphi_t \cdot v)) \big\rangle = \inf_{v \in
    \mathcal{D}} \big\langle \alpha,\, \dtheta ( \tilde{\beta}(v),
  \tilde{\beta}(\varphi_t \cdot v)) \big\rangle\]
  for any subset $\mathcal{D}$ of~$\mathcal{G}_\Gamma$ with $\Gamma \cdot \mathcal{D} = \mathcal{G}_\Gamma$. 
  Choosing $\mathcal{D}$ to be compact, we see that the infimum is in fact a minimum.
\end{remark}

For the proof it will be useful to relate the contraction on $G/P_\theta$ with a lower bound on $\langle\alpha,\mu_{\theta}(\cdot)\rangle$ for $\alpha \in \theta$.
We set $x_{\theta}:=eP_{\theta}\in G/P_{\theta}$.

\begin{lemma}\label{lem:contrctGmodP}
  Let $G$ be a real reductive Lie group and $\theta\subset \Delta$ a nonempty subset of the simple restricted roots of~$G$.
  Fix a $K_\theta$-invariant norm on the tangent space $T_{x_\theta} (G/ P_\theta) \simeq \mathfrak{u}_{\theta}^{-}$. Then there is a constant $C>0$ with the following properties:
  \begin{enumerate}[label=(\roman*),ref=\roman*]
  \item\label{item:27} For any $t \in\RR$, if the left multiplication action of $l \in L_\theta$ on $T_{x_\theta} (G/P_\theta)$ is $e^{-t}$-Lipschitz, then $\langle \alpha,\, \mu_\theta(l)\rangle \geq t$ for all $\alpha\in \theta$.
  \item\label{item:28} For any nonnegative $t \in \RR$ and any $l \in L_\theta$, if $\langle \alpha,\, \mu_\theta(l)\rangle \geq t$ for all $\alpha\in \theta$, then the action of $l$ on $T_{x_\theta} (G/P_\theta)$ is $Ce^{-t}$-Lipschitz.
  \end{enumerate}
\end{lemma}

\begin{proof}[Proof of Lemma~\ref{lem:contrctGmodP}]
The action of $l\in L_{\theta}$ on $T_{x_\theta} (G/P_\theta)$ identifies with the adjoint action of $l$ on $\mathfrak{u}_{\theta}^{-} = \bigoplus_{\alpha \in \Sigma_{\theta}^{+}} \g_{-\alpha}$.
Write $l\in K_{\theta}(\exp \mu_{\theta}(l))K_{\theta}$ where $\mu_{\theta}(l)\in\athetaplus$.
By definition, any $a\in\exp(\aaa)$ acts on~$\g_{-\alpha}$ by multiplication by the scalar $e^{-\langle\alpha,\, \log a\rangle}$.
Applying this to $\log a=\mu_{\theta}(l)$ and using the $K_{\theta}$-invariance of the norm, we obtain~\eqref{item:27}.
For \eqref{item:28}, note that if $\langle \alpha,\, \mu_\theta(l)\rangle \geq t$ for all $\alpha\in \theta$, then the scalars $e^{-\langle\alpha,\, \mu_{\theta}(l)\rangle}$ for $\alpha\in\Sigma_\theta^+$ are all in $(0,e^{-t}]$, since $\langle \alpha,\, \mu_\theta(l)\rangle\geq 0$ for $\alpha \in \Delta \smallsetminus \theta$ by definition \eqref{eqn:thetachamber} of the range $\athetaplus$ of~$\mu_{\theta}$.
The constant $C$ comes from the fact that the weight spaces $\g_{-\alpha}$ may not be orthogonal for the chosen $K_{\theta}$-invariant norm on $T_{x_\theta} (G/P_\theta) \simeq \mathfrak{u}_{\theta}^{-}$.
\end{proof}

\begin{proof}[Proof of Proposition~\ref{prop:weakcontr_ano}]
  As in Section~\ref{subsec:prelim-result}, we choose a $\rho$-equivariant set-theoretic lift $\breve{\beta} :\nolinebreak\mathcal{G}_\Gamma\to\nolinebreak G$ of~$\tilde{\beta}$, and for any $(t,v)\in \RR \times \mathcal{G}_\Gamma$ we set $l_{t,v} := \breve{\beta}(v)^{-1} \breve{\beta}(\varphi_t \cdot v) \in L_\theta$, so that $\dtheta (\tilde{\beta}(v), \tilde{\beta}(\varphi_t \cdot v))  = \mu_\theta (l_{t,v})$.
  Suppose \eqref{eqn:beta-v-infty} holds for all $\alpha\in\theta$, \ie
  \[\inf_{\alpha\in \theta,\ v\in \mathcal{G}_\Gamma}\ \langle \alpha,\mu_\theta(l_{t,v})\rangle \underset{t\to +\infty}{\longrightarrow} +\infty. \]
By Lemma~\ref{lem:contrctGmodP}.\eqref{item:28}, there exists $\tau>0$ such that $l_{\tau,v}$ is $e^{-1}$-contracting on $T_{x_\theta} (G/P_\theta)$ for all $v\in\mathcal{G}_\Gamma$, where $x_{\theta}=eP_{\theta}\in G/P_{\theta}$.
Let $M\geq 0$ be a bound for the Lipschitz constant of $l_{s,w}$ acting on $T_{x_\theta} (G/P_\theta)$ for all $s \in [0,\tau]$ and $w \in \mathcal{G}_\Gamma$.
  By \eqref{eqn:l-cocycle}, for any $t\geq 0$ and $v\in\mathcal{G}_\Gamma$,
  \[l_{t,v} = l_{t-n\tau, \varphi_{n\tau}\cdot v} \, \left ( l_{\tau, \varphi_{(n-1)\tau} \cdot v} \, \cdots \, l_{\tau,v} \right ) \quad \text{where} \ n = \Bigl\lfloor \frac{t}{\tau} \Bigr\rfloor, \]
  hence $l_{t,v}$ is $M e^{-n}$-Lipschitz on $T_{x_\theta}(G/P_\theta)$.
  By Lemma~\ref{lem:contrctGmodP}.\eqref{item:27}, for any~$\alpha\in\nolinebreak\theta$,
  \[\langle \alpha,\, \mu_\theta(l_{t,v}) \rangle \geq \Bigl\lfloor
    \frac{t}{\tau} \Bigr\rfloor - \log M \geq \frac{1}{\tau} t
  - \log M-1.\qedhere \]
\end{proof}

%%%%%%%%%%%%%%%%%%%%%%%%%
\subsection{When the Cartan projection drifts away from the $\theta$-walls}
\label{subsec:when-cart-proj}

We now address the implication $\eqref{item:7}\Rightarrow\nolinebreak\eqref{item:5}$ of Theorem~\ref{thm:char_ano_maps}.
We first fix an arbitrary representation $\rho : \Gamma \to G$ of the word hyperbolic group $\Gamma$ into the reductive group~$G$ and establish the following.

\begin{proposition}\label{prop:mu-theta-mu} 
If condition~\eqref{item:7} of Theorem~\ref{thm:char_ano_maps} holds, then there exists $T \geq 0$ such that for all $t\geq T$ and $v \in \mathcal{G}_\Gamma$, 
\begin{equation} \label{eqn:finite}
\mu_\theta(l_{t,v}) = \mu(l_{t,v}). 
\end{equation}
\end{proposition}

Before proving\ Proposition~\ref{prop:mu-theta-mu}, we establish Lemmas \ref{lem:aRplus_connec}, \ref{lem:l_tgamma_in_aplus}, and~\ref{lem:on_axis_in_aplus} below.

Recall that condition~\eqref{item:7} of Theorem~\ref{thm:char_ano_maps} is the existence of continuous, $\rho$-equiva\-riant, transverse, dynamics-preserving maps $\xi^+ : \partial_{\infty}\Gamma \to G/P_{\theta}$ and $\xi^- : \partial_{\infty}\Gamma \to\nolinebreak G/P_{\theta}^-$ and the fact that for any $\alpha \in \theta$,
\begin{equation}\label{eqn:5}
\lim_{\gamma \to \infty} \bigl\langle
  \alpha,\, \mu(\rho(\gamma)) \bigr\rangle = + \infty.
\end{equation}
Consider maps $\tilde{\sigma}$, $\tilde{\beta}$, $\breve{\beta}$ and $l_{t,v}$ as in Section~\ref{subsec:prelim-result} and let
\[ \mathcal{K} := 2 \sup_{v\in \mathcal{D}} \|\mu(\breve{\beta}(v))\| < +\infty ,\]
where $\mathcal{D}$ is a compact fundamental domain of~$\mathcal{G}_{\Gamma}$ for the action of~$\Gamma$.
By \eqref{eqn:5}, there is a finite subset $F$ of~$\Gamma$ such that for any $\alpha \in \theta$ and $\gamma \in \Gamma \smallsetminus F$,
\[ \bigl\langle \alpha,\, \mu(\rho(\gamma)) \bigr\rangle \geq 1 + \Vert\alpha\Vert \, \mathcal{K}. \]
By Lemma~\ref{lem:mu_mutheta}, we then have $\langle\alpha,\mu(l_{t,v})\rangle\geq 1$ for all $\alpha\in\theta$, all $t\in\RR$, and all $v\in\mathcal{D}$ with $\varphi_t\cdot v\in (\Gamma\smallsetminus F)\cdot\mathcal{D}$.
Note that the set of $t\in\RR$ with $\varphi_t\cdot\mathcal{D}\cap F\cdot\mathcal{D}\neq\emptyset$ is compact.
Therefore there exists $T>0$ such that for any $v\in\mathcal{D}$ we have
\begin{equation} \label{eqn:insulate}
\forall\alpha \in \theta,\ \forall t\geq T,\quad \langle \alpha ,\, \mu(l_{t,v}) \rangle \geq 1.
\end{equation}
This actually holds for any $v\in\mathcal{G}_{\Gamma}$ by \eqref{eqn:l-inv}.

Using the Weyl group~$W$, we define 
\begin{align*}
\s_{\theta} & :=  \big\{Y \in \overline{\aaa}^+ \mid \langle \alpha,Y \rangle \geq 1 \text{ for all } \alpha \in \theta\big\} ,\\ 
\s'_{\theta} & :=  \athetaplus \cap (W \cdot \s_{\theta}).
\end{align*}

\begin{lemma}\label{lem:aRplus_connec}
  The set $\s_{\theta}$ is a connected component of~$\s'_{\theta}$.
\end{lemma}

\begin{figure}[h!]
\centering
\labellist
\small\hair 2pt
\pinlabel $\s_{\theta}$  at 70 26
\pinlabel $\s'_{\theta}\smallsetminus \s_\theta$  at 31 31
\pinlabel $\athetaplus$  at -2 34
\pinlabel $\overline{\aaa}^+$  at 88 40
\pinlabel {$\color{red} \varepsilon_1-\varepsilon_2$}  at 75 9
\pinlabel {$\color{red} \varepsilon_2-\varepsilon_3$}  at 49 40
\endlabellist
\includegraphics[width=6cm]{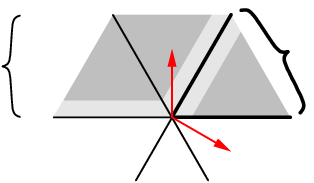}
\caption{Illustration of Lemma~\ref{lem:aRplus_connec} when $G=\SL_3(\RR)$ and $\theta=\{\varepsilon_1-\varepsilon_2\}$: the set $\s'_{\theta}=\s_{\theta}\cup (\s'_{\theta}\smallsetminus\s_{\theta})$ is shown in dark gray.}
\label{chambers}
\end{figure} 

\begin{proof}
The set $\s_{\theta}$ is clearly connected (in fact, convex) and contained in~$\s'_{\theta}$.
Since $\s'_{\theta}$ is a union of $W$-translates of~$\s_{\theta}$, it is sufficient to prove that  
for any nontrivial $w\in W$, if $w\cdot\s_{\theta} \cap \s_{\theta} \neq \emptyset$, then $w\cdot\s_{\theta} \not\subset \s'_{\theta}$ (see Figure~\ref{chambers}).
Note that for any nontrivial $w\in W$ there exists $\alpha\in\Delta$ such that $\langle\alpha,Y\rangle\leq 0$ for all $Y\in w\cdot\overline{\aaa}^+$.
If $\s_{\theta}\cap w\cdot\s_{\theta}\neq\emptyset$, then any such $\alpha$ belongs to $\Delta\smallsetminus\theta$ by definition of~$\s_{\theta}$, and so $w\cdot\s_{\theta}\not\subset\s'_{\theta}$  
by definition \eqref{eqn:thetachamber} of~$\athetaplus$.
\end{proof}

Let $\gamma \in \Gamma$ be an element of infinite order.
The invariant axis  
 $\mathcal{A}_{\gamma} \subset \mathcal{G}_\Gamma$ of~$\gamma$ is the set of $v\in\mathcal{G}_\Gamma$ such that $(\varphi_{+\infty},\varphi_{-\infty})(v) = (\eta^{+}_{\gamma}, \eta^{-}_{\gamma})$, where $\eta^+_{\gamma}$ and $\eta^-_{\gamma}$ are respectively the attracting and the repelling fixed point of $\gamma$ in $\partial_{\infty}\Gamma$.
  There is a constant $T_\gamma >0$ (the \emph{period} of~$\gamma$) such that $\gamma \cdot v = \varphi_{T_\gamma}\cdot v$ for all $v \in \mathcal{A}_{\gamma}$.
The assumption from Theorem~\ref{thm:char_ano_maps}.\eqref{item:7} that the maps $\xi^+, \xi^-$ are dynamics-preserving for~$\rho$ implies the following:

\begin{lemma}
  \label{lem:l_tgamma_in_aplus}
  Suppose condition~\eqref{item:7} of Theorem~\ref{thm:char_ano_maps} holds.
  Let $\gamma \in \Gamma$ be an element of infinite order, with invariant axis $\mathcal{A}_{\gamma} \subset \mathcal{G}_\Gamma$.
  Then for any $v\in\mathcal{A}_{\gamma}$,
  \[\mu_\theta (l_{T_\gamma,v})\in \overline{\aaa}^+, \]
  \ie $\langle \alpha,\, \mu_\theta(l_{T_\gamma, v}) \rangle \geq 0$ for all $\alpha \in \theta$ (for $\alpha\in \Delta\smallsetminus \theta$ the same inequality is true a priori by definition \eqref{eqn:thetachamber} of the range $\athetaplus$ of~$\mu_\theta$).
\end{lemma}

\begin{proof}
  Let $v\in\mathcal{A}_{\gamma}$.
  Then \(\tilde{\sigma}(v) = \breve{\beta}(v) L_{\theta}\) and $\xi^{+}(\varphi_{+\infty}(v)) =\xi^+(\eta^{+}_{\gamma}) = \breve{\beta}(v)P_\theta$.
  Also,
  \[ \breve{\beta}( v) \, l_{T_\gamma,v} = \breve{\beta}(\varphi_{T_\gamma} \cdot v) = \breve{\beta}(\gamma\cdot v) = \rho(\gamma) \, \breve{\beta}( v) \]
  (from the construction of~$l_{t,v}$), hence $l_{T_\gamma,v}$ and $\rho(\gamma)$ are conjugate by \(\breve{\beta}(v)\). 
In particular, since the attracting fixed point of $\rho(\gamma)$ in $G/P_\theta$ is $\breve{\beta}(v) P_\theta$, the attracting fixed point for $l_{T_\gamma,v}$ is \(P_{\theta}\).
  By Lemma~\ref{lem:contrctGmodP}.\eqref{item:27},
  \[\langle \alpha ,\, \mu_\theta (l_{T_\gamma,v}) \rangle >0\]
  for all $\alpha \in \theta$, hence $\mu_\theta(l_{T_\gamma,v}) \in \overline{\aaa}^+$.
\end{proof}

\begin{lemma}\label{lem:on_axis_in_aplus}
  Suppose condition~\eqref{item:7} of Theorem~\ref{thm:char_ano_maps} holds.
Recall the constant $T$ from \eqref{eqn:insulate}.
Let $\gamma \in \Gamma$ be an element of infinite order with invariant axis $\mathcal{A}_{\gamma}\subset\mathcal{G}_{\Gamma}$.
For any $v\in\mathcal{A}_{\gamma}$ and any $t \geq T$,
\[\mu_\theta (l_{t,v}) \in \s_{\theta}.\]
\end{lemma}

\begin{proof}
By \eqref{eqn:insulate}, the map $\psi : [T, + \infty ) \rightarrow \overline{\aaa}^+_{\theta}$ sending $t$ to $\mu_\theta (l_{t,v})$ takes values in~$\s'_{\theta}$.
This map is continuous since the map $(t,v)\mapsto K_\theta l_{t,v} K_\theta$ and the (bi-$K_\theta$-invariant) Cartan projection $\mu_\theta$ are both continuous.
Applying Lemma~\ref{lem:l_tgamma_in_aplus} to an appropriate power of~$\gamma$, we see that $\psi(t) \in \s_{\theta}$ for some $t\geq T$.
Lemma~\ref{lem:aRplus_connec} then implies $\psi(t) \in \s_{\theta}$ for all $t\geq T$.
\end{proof}

\begin{proof}[Proof of Proposition~\ref{prop:mu-theta-mu}]
Using Lemma~\ref{lem:on_axis_in_aplus}, we see that for any $t \geq T$ the closed set
\[\{v\in \mathcal{G}_\Gamma \mid \mu_\theta (l_{t,v}) \in
\s_{\theta}\}\]
 contains the invariant axis $\mathcal{A}_{\gamma}$ of any element $\gamma\in\Gamma$ of infinite order.
By density of the union of those axes (Remark~\ref{rem:flow-space-dense}), this set is all of $\mathcal{G}_\Gamma$. Using again Remark~\ref{rem:mu-theta-equal-mu-l-in-L}, we see that $\mu_\theta( l_{t,v}) = \mu(l_{t,v})$ for all $t \geq T$ and all $v\in\nolinebreak\mathcal{G}_\Gamma$.
\end{proof}

\begin{proof}[Proof of the implication $\eqref{item:7}\Rightarrow\eqref{item:5}$ of Theorem~\ref{thm:char_ano_maps}]
Suppose that \eqref{item:7} holds, \ie there exist continuous, $\rho$-equivariant, transverse, dynamics-preserving maps $\xi^+,\xi^-$ and \eqref{eqn:5} holds.
By Proposition~\ref{prop:weakcontr_ano}, in order to show that $\rho$ is $P_{\theta}$-Anosov, it is enough to prove that
\begin{equation} \label{eqn:6}
  \forall \alpha \in \theta, \quad \lim_{t \to +\infty} \inf_{v \in
    \mathcal{G}_\Gamma} \bigl\langle 
  \alpha,\, \mu_\theta(l_{t,v}) \bigr\rangle = + \infty.  
\end{equation}
By Remark~\ref{rem:inf_on_cpt}, the infimum needs to be taken only on a compact set $\mathcal{D}$ such that $\Gamma\cdot \mathcal{D} = \mathcal{G}_\Gamma$.
For any $t\in \RR$ and any $v \in \mathcal{D}$ there is an element $\gamma_{t,v}\in\Gamma$ such that $\varphi_t \cdot v \in \gamma_{t,v} \cdot
\mathcal{D}$. 
By Lemma~\ref{lem:mu_mutheta} and Proposition~\ref{prop:mu-theta-mu}, there are constants $\mathcal{K}>0$ and $T\geq 0$ (independent of $t$ and~$v$) such that if \(t\geq T\), then
\[ \Vert \mu_\theta(l_{t,v}) - \mu(\rho(\gamma_{t,v})) \Vert \leq \mathcal{K}. \]
Since any orbit map $\Gamma \to \mathcal{G}_\Gamma$ is a quasi-isometry, we obtain that $\gamma_{t,v} \to \infty$ as $t \to +\infty$ (uniformly in \(v\in\mathcal{D}\)), and so the assumption \eqref{eqn:5} readily implies \eqref{eqn:6}.
\end{proof}

%%%%%%%%%%%%%%%%%%%%%%%%%
\subsection{Anosov representations and embeddings into general linear groups}
\label{sec:anos-repr-embedd}

Proposition~\ref{prop:theta-comp-Anosov}.\eqref{item:theta-comp-Ano} is a consequence of the results of Section~\ref{subsec:proof-theta-emb} and of the equivalence $\eqref{item:5}\Leftrightarrow\eqref{item:7}$ of Theorem~\ref{thm:char_ano_maps} (or equivalence $\eqref{item:Ano}\Leftrightarrow\eqref{item:away-from-walls}$ of Theorem~\ref{thm:char_ano}) which we have just established.

\begin{proof}[Proof of Proposition~\ref{prop:theta-comp-Anosov}.\eqref{item:theta-comp-Ano}]
By Proposition~\ref{prop:theta-comp-Anosov}.\eqref{item:theta-comp-xi}, there exist continuous, $\rho$-equiva\-riant, dynamics-preserving, transverse boundary maps $\xi^+ : \partial_{\infty}\Gamma\to G/P_{\theta}$ and\linebreak $\xi^- : \partial_{\infty}\Gamma\to G/P_{\theta}^-$ if and only if there exist continuous, $(\tau\circ\rho)$-equivariant, dynamics-preserving, transverse boundary maps $\xi_V^+ : \partial_{\infty}\Gamma\to\PP_{\KK}(V)$ and $\xi_V^- : \partial_{\infty}\Gamma\to\PP_{\KK}(V^*)$.
By Lemma~\ref{lem:tau}.\eqref{item:eps1-eps2}, we have $\langle \alpha, \mu (\rho(\gamma))\rangle \to +\infty$ as $\gamma \to \infty$ for all \(\alpha\in\theta\) if and only if $\langle\varepsilon_1 - \varepsilon_2, \mu_{\GL_\KK(V)} (\tau\circ \rho( \gamma))\rangle \to +\infty$ as $\gamma \to \infty$.
Therefore, condition \eqref{item:7} of Theorem~\ref{thm:char_ano_maps} holds for~$\rho$ if and only if it holds for $\tau\circ\rho$.
Since \eqref{item:7} is equivalent to \eqref{item:5} in Theorem~\ref{thm:char_ano_maps}, we conclude that $\rho$ is $P_{\theta}$-Anosov if and only if $\tau\circ\rho$ is.
\end{proof}

%%%%%%%%%%%%%%%%%%%%%%%%%
\subsection{Approaching the Cartan projection by the Lyapunov projection}
\label{subsec:abels-marg-souif}

The implications $\eqref{item:lambda1}\Rightarrow\eqref{item:lambda2}\Rightarrow\eqref{item:lambda3}$ of Theorem~\ref{thm:char_ano_lambda} rely on \eqref{eqn:stablelength} and on \eqref{eqn:lambda-lim-mu}, which expresses the fact that the Lyapunov projection can be approached by the Cartan projection.

\begin{proof}[Proof of $\eqref{item:lambda1}\Rightarrow\eqref{item:lambda2}$ in Theorem~\ref{thm:char_ano_lambda}]
Suppose $\rho$ is $P_{\theta}$-Anosov.
By Theorem~\ref{thm:char_ano_maps}.\eqref{item:8}, there exist $c,C>0$ such that $\bigr\langle \alpha, \mu( \rho(\gamma) )\bigr\rangle \geq c \, \ellGamma{\gamma} -C$ for all $\alpha\in\theta$ and $\gamma\in\Gamma$.
Using \eqref{eqn:stablelength} and \eqref{eqn:lambda-lim-mu}, we find that for all $\alpha\in\theta$ and $\gamma\in\Gamma$,
\[ \bigl\langle \alpha, \lambda( \rho(\gamma) )\bigr\rangle =  \lim_{n\to +\infty} \dfrac{1}{n} \bigl\langle \alpha, \mu( \rho(\gamma^n) )\bigr\rangle \geq \lim_{n\to +\infty} \dfrac{1}{n} \big(c \, \ellGamma{\gamma^n} -C\big) = c \, \ellinfty{\gamma}. \qedhere \]
\end{proof}

\begin{proof}[Proof of $\eqref{item:lambda2}\Rightarrow\eqref{item:lambda3}$ in Theorem~\ref{thm:char_ano_lambda}]
By \eqref{eqn:mu-leq-klength} there exists $k>0$ such that $\Vert\mu(\rho(\gamma))\Vert\leq k\,\ellGamma{\gamma}$ for all $\gamma\in\Gamma$.
Using \eqref{eqn:stablelength} and \eqref{eqn:lambda-lim-mu}, we find that for all $\gamma\in\Gamma$,
\[ \Vert \lambda(\rho(\gamma)) \Vert = \lim_{n\to +\infty} \frac{1}{n} \, \Vert \mu(\rho(\gamma^n)) \Vert \leq k \lim_{n\to +\infty} \frac{1}{n} \, \ellGamma{\gamma^n} = k \, \ellinfty{\gamma}. \qedhere \]
\end{proof}

The implication $\eqref{item:lambda2}\Rightarrow \eqref{item:lambda4}$ of Theorem~\ref{thm:char_ano_lambda} is obvious.
It only remains to prove $\eqref{item:lambda3}\Rightarrow \eqref{item:lambda1}$ and $\eqref{item:lambda4}\Rightarrow \eqref{item:lambda1}$.

For semisimple~$\rho$, these implications 
both rely on the following result, which states that the Cartan projection can be approached by the Lyapunov projection.
This was established by Benoist~\cite{Benoist}, using a theorem of Abels--Margulis--Soifer \cite{Abels_Margulis_Soifer_Prox}.
Since the result is not explicitly stated as such in \cite{Benoist}, we recall a sketch of the proof.

\begin{theorem}[\cite{Benoist}] \label{thm:abels-marg-soif}
Let $\Gamma$ be a discrete group, $G$ a real reductive Lie group, and $\rho: \Gamma \to G$ a semisimple representation.
Then there exist a finite set $F\subset\Gamma$ and a constant $C_\rho>0$ such that for any $\gamma \in \Gamma$ we can find $f\in F$ with
\[\|\lambda(\rho(\gamma f)) - \mu(\rho(\gamma))\| \leq C_\rho.\] 
\end{theorem}

\begin{proof}
By subadditivity of~$\mu$ (Fact~\ref{fact:mu-subadditive}), it is sufficient to prove the existence of a finite set $F\subset\Gamma$ and a constant $C>0$ such that for any $\gamma \in \Gamma$ we can find $f\in F$ with $\|\lambda(\rho(\gamma f)) - \mu(\rho(\gamma f))\| \leq C$. 
Since $\rho$ is semisimple by assumption, \ie the Zariski closure of $\rho(\Gamma)$ in~$G$ is reductive, we may assume without loss of generality that $\rho(\Gamma)$ is actually Zariski-dense in~$G$.
It is then known that the set $\langle\alpha,\mu(\rho(\Gamma))\rangle\subset\RRp$ is unbounded for all $\alpha\in\Delta$ (see \cite{Goldsheid_Margulis,Guivarch_Raugi,Benoist_Labourie,Prasad}).
By Lemma~\ref{lem:theta-comp-exists}, for any $\alpha\in\Delta$ there is an irreducible proximal linear representation $(\tau_{\alpha},V_{\alpha})$ whose highest weight $\chi_{\tau_{\alpha}}$ is a multiple of the fundamental weight~$\omega_{\alpha}$ associated with~$\alpha$.
By \cite[Th.\,5.17]{Abels_Margulis_Soifer_Prox}, there are a finite set $F\subset\Gamma$ and a constant $\varepsilon>0$ with the following property: for any $\gamma \in \Gamma$ we can find $f\in F$ such that for any $\alpha\in\Delta$, the element $\tau_{\alpha}\circ\rho(\gamma f)\in\GL_{\RR}(V_{\alpha})$ is $\varepsilon$-proximal in $\PP_{\RR}(V_{\alpha})$ in the sense of \cite[\S\,2.2]{Benoist}.
We conclude as in \cite[Lem.\,4.5]{Benoist}: on the one hand, for any $\alpha\in\Delta$, there is a constant $C_{\alpha}>0$ such that for any $g\in G$, if $\tau_{\alpha}(g)$ is $\varepsilon$-proximal in $\PP_{\RR}(V_{\alpha})$, then $|\langle\omega_{\alpha},\mu(g)-\lambda(g)\rangle|\leq C_{\alpha}$ \cite[Lem.\,2.2.5]{Benoist}; on the other hand, the fundamental weights $\omega_{\alpha}$, for $\alpha\in\Delta$, form a basis of~$\aaa_{s}^{\ast}$, and for any $g\in G$ the elements $\lambda(g)$ and $\mu(g)$ of~$\aaa$ have the same projection to $\z(\g)$.
\end{proof}

\begin{proof}[Proof of $\eqref{item:lambda3}\Rightarrow\eqref{item:lambda1}$ and $\eqref{item:lambda4}\Rightarrow\eqref{item:lambda1}$ in Theorem~\ref{thm:char_ano_lambda} for semisimple~$\rho$]
Let $F$ be the finite subset of~$\Gamma$ and $C_\rho>0$ the constant given by Theorem~\ref{thm:abels-marg-soif}.
Let $(\gamma_n)\in\Gamma^{\NN}$ be a sequence of elements going to infinity in~$\Gamma$.
The existence of $\rho$-equivariant, continuous, transverse maps $\xi^+,\xi^-$ implies that $\rho$ has finite kernel and discrete image (Remark~\ref{rem:combine}.\eqref{item:compatible}), and so $\Vert\mu(\rho(\gamma_n))\Vert\to +\infty$ by properness of the map~$\mu$.
By Theorem~\ref{thm:abels-marg-soif}, for any $n\in\NN$ there exists $f_n\in F$ such that $\| \lambda(\rho(\gamma_n f_n)) - \mu(\rho(\gamma_n))\|\leq \nolinebreak C_\rho$.
In particular, $\Vert\lambda(\rho(\gamma_n f_n))\Vert\to +\infty$. 
\begin{itemize}
  \item If condition~\eqref{item:lambda3} of Theorem~\ref{thm:char_ano_lambda} holds, then for any $\alpha\in \theta$ we have $\langle \alpha, \lambda(\rho(\gamma_n f_n)) \rangle\linebreak \to +\infty$, hence $\langle \alpha, \mu(\rho(\gamma_n)) \rangle \to\nolinebreak +\infty$.
This holds for any sequence $(\gamma_n)\in\Gamma^{\NN}$ going to infinity in~$\Gamma$, hence condition~\eqref{item:7} of Theorem~\ref{thm:char_ano_maps} is satisfied.
Therefore $\rho$ is $P_{\theta}$-Anosov by the implication $\eqref{item:7}\Rightarrow\eqref{item:5}$ of Theorem~\ref{thm:char_ano_maps}.
  \item Note that $\Vert\lambda(\rho(\gamma_n f_n))\Vert\to +\infty$ implies that the conjugacy classes of the $\gamma_n f_n$ diverge, \ie $\trl{\gamma_n f_n}\to +\infty$, and so $\ellinfty{\gamma_n f_n}\to +\infty$ by Proposition~\ref{prop:word-length-stable}.
  If condition~\eqref{item:lambda4} of Theorem~\ref{thm:char_ano_lambda} holds, then we obtain $\langle \alpha, \lambda(\rho(\gamma_n f_n)) \rangle \to +\infty$ for all $\alpha\in \theta$, and we conclude as above.\qedhere
\end{itemize}
\end{proof}

%%%%%%%%%%%%%%%%%%%%%%%%%
\subsection{When the Lyapunov projection drifts away from the $\theta$-walls}
\label{subsec:char-terms-lyap}

It remains to establish the implications $\eqref{item:lambda3} \Rightarrow \eqref{item:lambda1}$ and $\eqref{item:lambda4} \Rightarrow \eqref{item:lambda1}$ of Theorem~\ref{thm:char_ano_lambda} in the general case, when $\rho$ is not necessarily semisimple.
We first prove the following proposition.

\begin{proposition} \label{prop:xi-for-rho-ss}
Let $\Gamma$ be a word hyperbolic group, $G$ a real reductive Lie group, and $\theta \subset \Delta$ a nonempty subset of the simple restricted roots of~$G$.
Consider a representation $\rho : \Gamma \to G$ with semisimplification $\rho^{ss} : \Gamma\to G$ (see Section~\ref{subsubsec:rho-non-ss}).
  If there exist continuous, dynamics-preserving, transverse boundary maps $\xi^+ : \partial_{\infty}\Gamma \to G/ P_\theta$ and $\xi^- : \partial_{\infty} \Gamma \to G/P_{\theta}^-$ for~$\rho$, then there exist continuous, dynamics-preserving, transverse boundary maps $\xi^{\prime +} : \partial_{\infty}\Gamma \to G/ P_\theta$ and $\xi^{\prime -} : \partial_{\infty} \Gamma \to G/P_{\theta}^-$ for~$\rho^{ss}$.
\end{proposition}

\begin{proof}
For any irreducible linear representation $(\tau,V)$ of~$G$, the representation $\tau\circ\rho^{ss}$ is a semisimplification of $(\tau\circ\rho)$.
Therefore, by Proposition~\ref{prop:theta-comp-Anosov}.\eqref{item:theta-comp-xi}, up to postcomposing $\rho$ with some irreducible, $\theta$-proximal representation $\tau : G\to\GL_{\RR}(V)$ (which exists by Lemma~\ref{lem:theta-comp-exists}), we may assume that $G=\GL_{\RR}(V)$ and $\theta=\{\varepsilon_1-\varepsilon_2\}$, so that $G/P_{\theta}=\PP_{\RR}(V)$ and $G/P_{\theta}^-=\PP_{\RR}(V^*)$.

Viewing $\PP_{\RR}(V)$ and $\PP_{\RR}(V^\ast)$ as the sets of lines and hyperplanes of~$V$, respectively, we define the linear subspaces
\[ U := \sum_{\eta \in \partial_{\infty}\Gamma} \xi^+(\eta) \quad\quad\mathrm{and}\quad\quad V_1 := U \cap \bigcap_{\eta \in \partial_{\infty}\Gamma} \xi^-(\eta) \]
of~$V$; they are $\rho(\Gamma)$-invariant.
Let $V_2$ be a complementary subspace of $V_1$ in~$U$, and $V_3$ a complementary subspace of $U$ in~$V$.

Under the decomposition $V= V_1 \oplus V_2 \oplus V_3$, the representation $\rho : \Gamma\to\GL_{\RR}(V)$ is block upper triangular: there are representations $\rho_i : \Gamma\to\GL_{\RR}(V_i)$ and maps $a_{i,j} : \Gamma\to\Hom_{\RR}(V_i,V_j)$, for $1\leq i,j\leq 3$, such that for any $\gamma \in \Gamma$,
 \[ \rho(\gamma) = \left(
  \begin{matrix}
    \rho_1(\gamma) & a_{1,2}(\gamma) & a_{1,3}(\gamma) \\ & \rho_2(\gamma) & a_{2,3}(\gamma) \\ & & \rho_3(\gamma)
  \end{matrix}\right).
\]
Then 
\[ \rho^{ss}(\gamma) = \left(
  \begin{matrix}
    \rho_{1}^{ss}(\gamma) & &  \\ & \rho_2^{ss}(\gamma) &  \\ & & \rho_{3}^{ss}(\gamma) \end{matrix}\right)
\]
for all $\gamma\in\Gamma$, where $\rho_{i}^{ss}$ is a semisimplification of~$\rho_i$.
We now construct continuous, dynamics-preserving, transverse boundary maps $\xi^{\prime +} : \partial_{\infty}\Gamma \to \PP_{\RR}(V)$ and $\xi^{\prime -} :\nolinebreak \partial_{\infty}\Gamma \to \PP_{\RR}(V^*)$ for~$\rho^{ss}$.

Since the maps $\xi^+$ and~$\xi^-$ are transverse, $\xi^+$ takes values in $\PP_\RR(V_1\oplus V_2)\smallsetminus \PP_\RR(V_1)$ and $\xi^-$ takes values in $\PP_\RR(V_1^0)\smallsetminus \PP_\RR(V_1\oplus V_2)^0$, where we denote by $V_1^0\subset V^*$ the annihilator of~$V_1$
and by $(V_1\oplus V_2)^0\subset V^*$ the annihilator of $V_1\oplus V_2$.
Let 
\[ p : V_1\oplus V_2 \longrightarrow V_2 \]
be the linear projection onto~$V_2$ with kernel~$V_1$ and $\PP p : \PP_\RR(V_1\oplus V_2) \smallsetminus\nolinebreak \PP_\RR( V_1)\to\nolinebreak\PP_\RR(V_2)$ the induced map.
Similarly, let
\[ p^* : V_1^0 \longrightarrow (V_1\oplus V_3)^0 \]
be the linear projection onto $(V_1\oplus V_3)^0$ with kernel $(V_1\oplus V_2)^0$ and\linebreak $\PP p^* : \PP_\RR(V_1^0) \smallsetminus \PP_\RR((V_1\oplus V_2)^0)\to\PP_\RR((V_1\oplus V_3)^0)$ the induced map.
We set
\[ \left\{\begin{array}{ccl}
\xi^{\prime +} := \PP p\circ\xi^+ & : & \partial_{\infty}\Gamma\to\PP_\RR(V_2) \subset \PP_\RR(V),\\
\xi^{\prime -} := \PP p^*\circ\xi^- & : & \partial_{\infty}\Gamma\to\PP_\RR((V_1\oplus V_3)^0) \subset \PP_\RR(V^{\ast}).
\end{array}\right. \]
These maps are continuous and transverse by construction, as well as dyna\-mics-preserving for $\rho_{1}^{ss} \oplus \rho_2 \oplus \rho_{3}^{ss} : \Gamma\to\GL_{\RR}(V_1 \oplus V_2\oplus V_3)$.
To see that they are dynamics-preserving for $\rho^{ss} : \Gamma\to\GL_{\RR}(V)$, it is sufficient to prove that $\rho_2 = \rho_2^{ss}$.

If $\Gamma$ is elementary, \ie if \(\Gamma\) is virtually the infinite cyclic group generated by a nontorsion element~\(\gamma\) and \(\partial_\infty \Gamma =\{ \eta^{+}_{\gamma}, \eta^{-}_{\gamma}\}\), then \(V_1=\{0\}\) and \(U=V_2= \xi^+(\eta^{+}_{\gamma}) \oplus \xi^+(\eta^{-}_{\gamma})\), thus \(\rho_2|_{\langle\gamma\rangle}\) is semisimple and $\rho_2 =  \rho_{2}^{ss}$ by Remark~\ref{rem:semisimple_finiteindex}.
We now assume that $\Gamma$ is nonelementary.
We claim that $\rho_2 : \Gamma\to\GL_{\RR}(V_2)$ is then irreducible.
Indeed, let $R$ be a $\rho_2(\Gamma)$-invariant subspace of~$V_2$.
Let $\gamma\in\Gamma$ be an element of infinite order with attracting (\resp repelling) fixed point $\eta^+_{\gamma}$ (\resp $\eta^-_{\gamma}$) in $\partial_{\infty}\Gamma$.
Since $\xi^{\prime +}$ and~$\xi^{\prime -}$ are dynamics-preserving for $\rho_1^{ss}\oplus\rho_2\oplus\rho_3^{ss}$ and take values in $\PP_{\RR}(V_2)$ and $\PP_\RR((V_1\oplus V_3)^0)\simeq\PP_{\RR}(V_2^*)$ respectively, Lemma~\ref{lem:attr-basin} implies that for any $x\in\PP_{\RR}(V_2)\smallsetminus\xi^{\prime -}(\eta^-_{\gamma})$ we have $\rho_2(\gamma^n)\cdot x\to\xi^{\prime +}(\eta^{+}_{\gamma})$ as $n\to +\infty$.
Since $R$ is closed and $\rho_2(\Gamma)$-invariant, we obtain that either $\PP_{\RR}(R)\subset \xi^{\prime -}(\eta^{-}_{\gamma})$ or $\xi^{\prime +}(\eta_{\gamma}^{+})\in\PP_{\RR}(R)$.
Thus, one of the closed $\Gamma$-invariant subsets $\{ \eta \in \partial_{\infty}\Gamma \mid \PP_{\RR}(R) \subset \xi^{\prime -}(\eta)\}$ or $\{ \eta \in \partial_{\infty}\Gamma \mid \xi^{\prime +}(\eta)\in\PP_{\RR}(R)\}$ is nonempty, hence equal to $\partial_{\infty}\Gamma$ by minimality of the action of the nonelementary group $\Gamma$ on $\partial_{\infty}\Gamma$ (Fact~\ref{fact:dyn_at_infty}.\eqref{item:9}).
This shows that $R=\{0\} $ or $R=V_2$.
Thus $\rho_2$ is irreducible, and in particular $\rho_2 = \rho_{2}^{ss}$.
\end{proof}

\begin{proof}[Proof of Proposition~\ref{prop:Ano_iff_ssAno}]
If $\rho^{ss}$ is $P_{\theta}$-Anosov, then $\rho$ is $P_{\theta}$-Anosov by Proposition~\ref{prop:semisimplification_anosov}.
Conversely, suppose $\rho : \Gamma\to G$ is $P_\theta$-Anosov.
By Proposition~\ref{prop:xi-for-rho-ss}, there exist continuous, dynamics-preserving, transverse boundary maps $\xi^{\prime +} : \partial_{\infty}\Gamma \to G/ P_\theta$ and $\xi^{\prime -} : \partial_{\infty} \Gamma \to G/P_{\theta}^-$ for~$\rho^{ss}$.
By the implication $\eqref{item:lambda1}\Rightarrow\eqref{item:lambda4}$ of Theorem~\ref{thm:char_ano_lambda} applied to~$\rho$, for any $\alpha\in\theta$ we have $\langle\alpha,\lambda(\rho(\gamma))\rangle\to+\infty$ as $\ellinfty{\gamma} \to +\infty$.
The point is that $\lambda\circ\rho^{ss}=\lambda\circ\rho$ (Lemma~\ref{lem:lyapu_not_changed}).
Therefore, for any $\alpha\in\theta$ we have $\langle\alpha,\lambda(\rho^{ss}(\gamma))\rangle\to+\infty$ as $\ellinfty{\gamma} \to +\infty$, and so $\rho^{ss}$ is $P_{\theta}$-Anosov by the implication $\eqref{item:lambda4}\Rightarrow\eqref{item:lambda1}$ of Theorem~\ref{thm:char_ano_lambda} applied to the semisimple representation~$\rho^{ss}$ (this implication has been proved in Section~\ref{subsec:abels-marg-souif}).
\end{proof}

\begin{proof}[Proof of $\eqref{item:lambda3}\Rightarrow\eqref{item:lambda1}$ and $\eqref{item:lambda4}\Rightarrow\eqref{item:lambda1}$ in Theorem~\ref{thm:char_ano_lambda} for general~$\rho$]
By Lemma~\ref{lem:lyapu_not_changed} and Proposition~\ref{prop:xi-for-rho-ss}, if $\rho$ satisfies \eqref{item:lambda3} (\resp \eqref{item:lambda4}) then so does~$\rho^{ss}$.
In Section~\ref{subsec:abels-marg-souif} we proved the implications $\eqref{item:lambda3}\Rightarrow\eqref{item:lambda1}$ and $\eqref{item:lambda4}\Rightarrow\eqref{item:lambda1}$ of Theorem~\ref{thm:char_ano_lambda} for the semisimple representation~$\rho^{ss}$.
By Proposition~\ref{prop:semisimplification_anosov}, if $\rho^{ss}$ satisfies \eqref{item:lambda1}, then so does~$\rho$.
\end{proof}

%%%%%%%%%%%%%%%%%%%%%%%%%%%%%%%%%%%%%%%%%%%%%%%%%%%
\section{Construction of the boundary maps}\label{sec:boundary}

In this section we prove Theorem~\ref{thm:constr-xi} (which implies in particular $\eqref{item:cli}\Rightarrow\eqref{item:Ano}$ in Theorem~\ref{thm:char_ano}) by establishing an explicit version of it, namely Theorem~\ref{thm:constr-xi-explicit}.
We also complete the proof of Theorem~\ref{thm:char_ano} by establishing its implication $\eqref{item:Ano}\Rightarrow\eqref{item:cli}$.
In particular this implication shows that the CLI constants in Theorem~\ref{thm:constr-xi}.\eqref{item:xi-transv_1}, if they exist, must be uniform (see Remark~\ref{rem:constr-xi}.\eqref{item:cli-non-unif}).

To put things into perspective, we start by recalling the notion of limit set in $G/P_{\theta}$.
This notion was first introduced and studied by Guivarc'h \cite{Guivarc'h} for subgroups of $G=\SL_d(\RR)$ acting proximally and strongly irreducibly on~$\RR^d$, then by Benoist \cite{Benoist} for Zariski-dense subgroups of any reductive Lie group~$G$; here we give a definition for arbitrary subgroups of~$G$.

%%%%%%%%%%%%%%%%%%%%%%%%%
\subsection{Limit sets in flag varieties} \label{subsec:limit-set}

Let $G$ be a real reductive Lie group with Cartan decomposition $K(\exp\overline{\aaa}^+)K$ and corresponding Cartan projection $\mu : G\to\overline{\aaa}^+$; we use the notation of Section~\ref{subsec:some-structure-semi}.
For any nonempty subset $\theta\subset\Delta$ of the simple restricted roots of $\aaa$ in~$G$, let $x_{\theta}=eP_{\theta}\in G/P_{\theta}$ be the basepoint of $G/P_{\theta}$.
We define a map $\Xi_{\theta} : G\to G/P_{\theta}$ as follows: for any $g\in G$, we choose $k_g, k'_g \in K$ such that $g =  k_g(\exp\mu(g))k'_g$ and set
\begin{equation} \label{eqn:Xi-theta+-}
\Xi_{\theta}(g) := k_g \cdot x_{\theta} \in G/P_{\theta}.
\end{equation}
This does not depend on the choice of $k_g,k'_g$ as soon as $\langle\alpha,\mu(g)\rangle > 0$ for all $\alpha\in\theta$ (see \cite[Ch.\,IX, Th.\,1.1]{Helgason}). 
If $g$ is proximal in $G/P_{\theta}$ (Definition~\ref{defi:prox-G/P}), then the sequence $(\Xi_{\theta}(g^n))_{n\in\NN}$ converges to the attracting fixed point of $g$ in $G/P_{\theta}$: see Lemma~\ref{lem:xi-preserves-dynamics} below.
We make the following definition.

\begin{definition} \label{defi:limit-set}
Let $\Gamma$ be any discrete group and $\rho : \Gamma\to G$ any representation. 
The \emph{limit set} $\Lambda_{\rho(\Gamma)}^{G/P_{\theta}}$ of $\rho(\Gamma)$ in $G/P_{\theta}$ is the set of all possible limits of sequences $(\Xi_{\theta}\circ\rho(\gamma_n))_{n\in\NN}$ for $(\gamma_n)\in\Gamma^{\NN}$ with $\langle\alpha,\mu(\rho(\gamma_n))\rangle\to +\infty$ for all $\alpha\in\theta$.
\end{definition}

The limit set $\Lambda_{\rho(\Gamma)}^{G/P_{\theta}}$ is closed and $\rho(\Gamma)$-invariant, and does not depend on the choice of Cartan decomposition (Corollary~\ref{cor:limit-set-indep} below).
It is well known (see \cite{Goldsheid_Margulis,Guivarch_Raugi,Benoist_Labourie,Prasad}) that if $\rho(\Gamma)$ is Zariski-dense in~$G$, then $\rho(\Gamma)$ contains elements that are proximal in $G/P_{\theta}$; by \cite[Lem.\,3.6]{Benoist}, in this case $\Lambda_{\rho(\Gamma)}^{G/P_{\theta}}$ is the closure in $G/P_{\theta}$ of the set of attracting fixed points of these elements.

The limit set $\Lambda_{\rho(\Gamma)}^{G/P_{\theta}}$ is nonempty for instance if $\rho$ is \emph{$P_{\theta}$-divergent}, in the sense~that
\[ \forall\alpha\in\theta,\quad\quad \langle\alpha,\mu(\rho(\gamma))\rangle \underset{\gamma\to\infty}{\longrightarrow} +\infty. \]
When $G$ has real rank one, $P_{\theta}$-divergence is equivalent to $\rho$ having finite kernel and discrete image; in this case, we recover the usual limit set of $\rho(\Gamma)$ in $G/P_{\theta}\simeq\partial_{\infty}(G/K)$.

%%%%%%%%%%%%%%%%%%%%%%%%%
\subsection{Constructing boundary maps}

In this Section~\ref{sec:boundary} we prove that if $\Gamma$ is word hyperbolic and if for any $\alpha\in\theta$ we have $\langle \alpha,\mu( \rho(\gamma))\rangle\to +\infty$ fast enough as \({\gamma\to\infty}\), then the map $\Xi^+:=\Xi_{\theta}\circ\nolinebreak\rho :\nolinebreak \Gamma\to G/P_{\theta}$  extends continuously to a $\rho$-equivariant boundary map $\xi^+ : \partial_{\infty}\Gamma\to G/P_{\theta}$ with image the limit set $\Lambda_{\rho(\Gamma)}^{G/P_{\theta}}$.
More precisely, ``fast enough'' is given by the following conditions.

\begin{definition} \label{defi:GSP}
A representation $\rho : \Gamma \to G$ has the \emph{gap summation property} with respect to~$\theta$ if for any $\alpha\in\theta$ and any geodesic ray $(\gamma_n)_{n\in\NN}$ in the Cayley graph of~$\Gamma$,
\begin{equation} \label{eqn:series-conv}
  \sum_{n\in\NN} e^{- \langle \alpha,\,  \mu(\rho(\gamma_n)) \rangle} < + \infty.
  \end{equation}
The representation $\rho$ has the \emph{uniform gap summation property} with respect to~$\theta$ if this series converges uniformly for all geodesic rays $(\gamma_n)_{n\in\NN}$ with $\gamma_0=e$, \ie if for any $\alpha \in \theta$,
  \begin{equation} \label{eqn:series-unif-conv}
  \sup_{\substack{(\gamma_n)_{n\in\NN}\\ \text{geodesic}\\ \text{with }\gamma_0=e}}\ \sum_{n\geq n_0} e^{- \langle \alpha,\, \mu(\rho(\gamma_n)) \rangle} \underset{n_0\to +\infty}{\longrightarrow} 0.
  \end{equation}
\end{definition}

For instance, if there exists $C>0$ such that $\langle\alpha,\mu(\rho(\gamma))\rangle\geq 2\log\ellGamma{\gamma}-C$ for all $\alpha\in\theta$ and $\gamma\in\Gamma$, then $\rho$ has the uniform gap summation property with respect to~$\theta$, and even with respect to $\theta\cup\theta^{\star}$ by \eqref{eqn:opp-inv-mu}.

Recall that the parabolic subgroup $P_{\theta}^-$ of~$G$ is conjugate to~$P_{\theta^{\star}}$, and so $G/P_{\theta}^-$ identifies with $G/P_{\theta^{\star}}$.
The goal of this Section~\ref{sec:boundary} is to prove the following result, which implies Theorem~\ref{thm:constr-xi}.

\begin{theorem}\label{thm:constr-xi-explicit}
Let $\Gamma$ be a word hyperbolic group, $G$ a real reductive Lie group, $\theta \subset \Delta$ a nonempty subset of the simple restricted roots of~$G$, and $\rho : \Gamma \to G$ a representation.
\begin{enumerate}
  \item \label{item:xi-exist}
   If $\rho$ has the gap summation property with respect to $\theta\cup\theta^{\star}$, then the maps 
   \[\left\{ \begin{array}{ll}
   \Xi^+ := \Xi_{\theta} \circ \rho : & \Gamma \to G/P_{\theta},\\
   \Xi^- := \Xi_{\theta^{\star}} \circ \rho : & \Gamma \to G/P_{\theta^{\star}}
   \end{array} \right. \]
   induce $\rho$-equivariant boundary maps
   \[\left\{ \begin{array}{rl}
   \xi^+ : & \partial_{\infty}\Gamma\to G/P_{\theta},\\
   \xi^- : & \partial_{\infty}\Gamma\to G/P_{\theta^{\star}}\simeq G/P_{\theta}^-,
   \end{array} \right. \]
   which are independent of all choices.
   For any $\eta \in \partial_{\infty}\Gamma$, the points $\xi^+(\eta)\in G/P_{\theta}$ and $\xi^-(\eta)\in G/P_{\theta}^-$ are compatible in the sense of Definition~\ref{defi:transverse-compatible}.
   \item\label{item:xi-dyn-preserv}
  If moreover for any $\alpha \in \theta$ and any $\gamma \in \Gamma$ of infinite order,
  \[ \langle \alpha, \mu(\rho(\gamma^n)) \rangle - 2 \log |n| \underset{|n|\to +\infty}{\longrightarrow} + \infty, \]
  then $\xi^+$ and~$\xi^-$ are dynamics-preserving for~$\rho$.
  \item \label{item:xi-cont}
  If $\rho$ has the \emph{uniform} gap summation property with respect to $\theta\cup\theta^{\star}$, then 
  \[\left\{ \begin{array}{rl}
  \Xi^+\sqcup\xi^+ : & \Gamma\cup\partial_{\infty}\Gamma\to G/P_{\theta},\\
  \Xi^-\sqcup\xi^- : & \Gamma\cup\partial_{\infty}\Gamma\to G/P_{\theta^{\star}}\simeq G/P_{\theta}^-
  \end{array} \right. \]
  are continuous, and the images of $\xi^+$ and~$\xi^-$ are the respective limit sets $\Lambda_{\rho(\Gamma)}^{G/P_{\theta}}$ and $\Lambda_{\rho(\Gamma)}^{G/P_{\theta^{\star}}}$ (Definition~\ref{defi:limit-set}).
    \item\label{item:xi-transv}
  If moreover for any $\alpha\in\Sigma^+_{\theta}$ and any geodesic ray $(\gamma_n)_{n\in\NN}$ the sequence $(\langle \alpha,\, \mu( \rho(\gamma_n))\rangle)_{n\in\NN}$ is CLI (Definition~\ref{defi:cli}), then $\xi^+$ and~$\xi^-$ are transverse.
  In particular, $\rho$ is $P_{\theta}$-Anosov and $\xi^+$ and~$\xi^-$ define homeomorphisms between $\partial_{\infty}\Gamma$ and the limit sets $\Lambda_{\rho(\Gamma)}^{G/P_{\theta}}$ and $\Lambda_{\rho(\Gamma)}^{G/P_{\theta^{\star}}}$, respectively.
\end{enumerate}
\end{theorem}

In \eqref{item:xi-exist}, by ``induce'' we mean that $\xi^+(\eta)=\lim_n \Xi^+(\gamma_n)$ and $\xi^-(\eta)=\lim_n \Xi^-(\gamma_n)$ for any quasi-geodesic ray $(\gamma_n)_{n\in\NN}$ in the Cayley graph of~$\Gamma$ with endpoint $\eta\in\partial_{\infty}\Gamma$.

In \eqref{item:xi-transv}, the fact that $\rho$ is $P_{\theta}$-Anosov comes from the implication $\eqref{item:away-from-walls}\Rightarrow\eqref{item:Ano}$ in Theorem~\ref{thm:char_ano}, and the fact that $\xi^+$ and~$\xi^-$ are homeomorphisms onto their image comes from the fact that they are continuous and injective and $\partial_{\infty}\Gamma$ is compact.

\begin{remark} \label{rem:fixed-pt-non-attract}
Let $\gamma\in\Gamma$ be an element of infinite order, with attracting fixed point $\eta^+_{\gamma}\in\partial_{\infty}\Gamma$.
The image of $\eta^+_{\gamma}$ under $\xi^+$ (or under any $\rho$-equivariant map $ \partial_{\infty}\Gamma\to G/P_{\theta}$) is always  a fixed point of $\rho(\gamma)$ in $G/P_{\theta}$; however, this fixed point $\xi^+(\eta^{+}_{\gamma})$ is attracting (in the sense used in Definition~\ref{defi:prox-G/P}) only if $\langle\alpha,\mu(\rho(\gamma)^n)\rangle$ grows faster than $2\log(n)$ for every $\alpha\in\theta$ (Lemma~\ref{lem:prox}).
This shows that the growth assumption in \eqref{item:xi-dyn-preserv} above is optimal.
For instance, here is a case where the assumptions of \eqref{item:xi-exist} and~\eqref{item:xi-cont} are satisfied but the conclusion of~\eqref{item:xi-dyn-preserv}~fails:
\end{remark}

\begin{example}\label{ex:log-growth}
Let $G=\SL_2(\RR)$, with Cartan projection $\mu : G\to\RRp$ obtained by identifying $\overline{\aaa}^+$ with~$\RRp$.
Let $\Gamma$ be a finitely generated Schottky subgroup of~$G$ containing a parabolic element~$u$.
There is a constant $C>0$ such that $\mu( \rho(\gamma)) \geq 2 \log \ellGamma{\gamma} - C$ for all $\gamma\in\Gamma$; in particular, $\rho$ satisfies the uniform gap summation property \eqref{eqn:series-unif-conv}.
However, this growth rate cannot be improved since $\mu(u^n)=2\log n+ O(1)$.
The continuous equivariant boundary map $\xi : \partial_\infty \Gamma \rightarrow \partial_\infty \HH^2$ given by Theorem~\ref{thm:constr-xi-explicit}.\eqref{item:xi-cont} is not dynamics-preserving since the fixed point of~$u$ in $G/P_{\theta}=\partial_\infty \HH^2$ is neither attracting nor repelling; thus the conclusion of Theorem~\ref{thm:constr-xi-explicit}.\eqref{item:xi-dyn-preserv} fails.
The transversality conclusion of Theorem~\ref{thm:constr-xi-explicit}.\eqref{item:xi-transv} also fails since $\xi(\lim_{+\infty} u^n)=\xi(\lim_{+\infty} u^{-n})$ (see Remark~\ref{rem:compatible}.\eqref{item:compat-transv->inj}).
\end{example}

We first discuss the gap summation property (Section~\ref{subsec:gap-summ-prop}), then establish some estimates for the map~$\Xi_{\theta}$ (Section~\ref{subsec:estim-Xi-theta}), which are useful in the proof of Theorem~\ref{thm:constr-xi-explicit}.\eqref{item:xi-exist}--\eqref{item:xi-dyn-preserv}--\eqref{item:xi-cont} (Section~\ref{subsec:exist-equiv-cont}).
The proof of Theorem~\ref{thm:constr-xi-explicit}.\eqref{item:xi-transv} (hence of $\eqref{item:cli}\Rightarrow\eqref{item:Ano}$ of Theorem~\ref{thm:char_ano}) is more delicate, and is the object of Section~\ref{subsec:proof-xi-transv}. 
Finally, in Section~\ref{subsec:proof-contr-Ano} we establish the implication $\eqref{item:Ano}\Rightarrow\eqref{item:cli}$ of Theorem~\ref{thm:char_ano}.

%%%%%%%%%%%%%%%%%%%%%%%%%
\subsection{The gap summation property}
\label{subsec:gap-summ-prop}

The following observation will be useful in the proof of Theorem~\ref{thm:constr-xi-explicit}.
 
\begin{lemma} \label{lem:gap-summ-prop}
For any $c,C>0$, there exists $C_0>0$ such that for any $(c,C)$-quasi-geodesic rays $(\gamma_n)_{n\in\NN}$ and $(\gamma'_n)_{n\in\NN}$ in the Cayley graph of~$\Gamma$, with the same initial point $\gamma_0 = \gamma'_0$ and the same endpoint in $\partial_{\infty}\Gamma$, and for any $\alpha\in\Delta$,
\[ \sum_{n\in\NN} e^{- \langle \alpha, \, \mu(\rho(\gamma'_n)) \rangle} \leq C_0 \, \sum_{n\in\NN} e^{- \langle \alpha,\,  \mu(\rho(\gamma_n)) \rangle}. \]
\end{lemma}

Therefore, in the definition~\ref{defi:GSP} of the gap summation property, it is equivalent to ask for the convergence of the series~\eqref{eqn:series-conv} for all \emph{quasi}-geodesic rays.

\begin{proof}
Since $(\gamma_n)_{n\in\NN}$ and $(\gamma'_n)_{n\in\NN}$ have the same endpoint at infinity in the word hyperbolic group~$\Gamma$, there is a \((c',C')\)-quasi-isometry $\phi : \NN\to\NN$ such that the sequence $(\gamma'_n \gamma_{\phi(n)}^{-1})_{n\in\NN}$ is contained in the \(R\)-ball \(B_e(R)\) centered at~\(e\) in~$\Gamma$.
The constants \(c'\), \(C'\), and~\(R\) depend only on \((c,C)\) (and on the hyperbolicity constant of~$\Gamma$).
Let $M$ be a real number such that $\Vert \mu(\rho(\gamma)) \Vert \leq\nolinebreak M$ for all $\gamma \in B_e(R)$.
By subadditivity of~$\mu$ (Fact~\ref{fact:mu-subadditive}.\eqref{item:mu-strong-subadd}), for any $n\in\NN$,
\begin{align*}
  | \langle \alpha, \mu(\rho(\gamma_{\phi(n)})) - \mu(\rho(\gamma'_n)) \rangle |
  & \leq \big \Vert \mu(\rho(\gamma_{\phi(n)})) - \mu(\rho(\gamma'_n)) \big \Vert \, \Vert \alpha \Vert\\
  & \leq \big\Vert \mu\big(\rho\big(\gamma'_n \gamma_{\phi(n)}^{-1}\big)\big) \big\Vert\, \Vert \alpha \Vert \leq M \Vert \alpha \Vert.
\end{align*}
Moreover,  for any $p\in\NN$, the set $\{ n\in \NN \mid \phi(n) = p\}$ has at most $c' C'+1$ elements.
Thus, for any $n\in\NN$,
\begin{align*}
  \sum_{n\in\NN} e^{- \langle\alpha,\,  \mu(\rho(\gamma'_n)) \rangle} & \leq  e^{M \Vert \alpha \Vert} \sum_{n\in\NN} e^{- \langle\alpha, \, \mu(\rho(\gamma_{\phi(n)}))\rangle}\\
   & \leq  e^{M \Vert \alpha \Vert} (c' C'+1) \sum_{p\in\NN} e^{- \langle\alpha,\,  \mu(\rho(\gamma_p))\rangle}. \qedhere
\end{align*}
\end{proof}

For any $g\in G$, we set
\begin{equation} \label{eqn:Ttheta}
T_{\theta}(g) := \min_{\alpha\in\theta}\ \langle\alpha,\mu(g)\rangle \,\geq 0.
\end{equation}
If the gap summation property holds with respect to~$\theta$, then the series $\sum_n e^{-T_{\theta}(\rho(\gamma_n))}$ converges for every quasi-geodesic ray \((\gamma_n)_{n\in \NN}\).
Here is an immediate consequence of Lemma~\ref{lem:gap-summ-prop}.

\begin{corollary} \label{cor:unif-GSP-q-geod}
If the uniform gap summation property holds with respect to~$\theta$, then for any $c,C>0$,
\[  \sup_{\substack{(\gamma_n)_{n\in\NN}\\ (c,C)\text{-quasi-geodesic}\\ \text{with }\gamma_0=e}}\ \sum_{n\geq n_0} e^{- T_{\theta}(\rho(\gamma_n))} \underset{n_0\to +\infty}{\longrightarrow} 0. \]
\end{corollary}

%%%%%%%%%%%%%%%%%%%%%%%%%
\subsection{Metric estimates for the map $\Xi_{\theta}$} \label{subsec:estim-Xi-theta}

Consider an irreducible, \(\theta\)-proximal linear representation $(\tau,V)$ of~$G$ (see Definition~\ref{defi:theta-compatible} and Lemma~\ref{lem:theta-comp-exists}).
Let $\Vert\cdot\Vert_V$ be a $K$-invariant Euclidean norm on~$V$ for which the weight spaces of~$\tau$ are orthogonal.
It defines a $K$-invariant metric $d_{\PP(V)}$ on $\PP_{\RR}(V)$, given by
\[ d_{\PP(V)}([v],[v']) = |\sin\measuredangle(v,v')| \]
for all nonzero $v,v'\in V$.
By Proposition~\ref{prop:theta_emb}, the space $G/P_{\theta}$ embeds into $\PP_{\RR}(V)$ as the closed $G$-orbit of the point $x_\tau:=\PP_{\RR}(V^{\chi_{\tau}})\in\PP_{\RR}(V)$, and inherits from this a $K$-invariant metric~$d_{\scriptscriptstyle G\!/\!P_{\theta}}$: for any $g,g'\in G$, 
\[ d_{\scriptscriptstyle G\!/\!P_{\theta}}(g\cdot x_\theta, g'\cdot x_\theta) := d_{\PP(V)}(g\cdot x_\tau, g'\cdot x_\tau).\]
Here, and sometimes in the rest of the section, to simplify notation, we omit $\tau$ and write ``$\cdot$'' for the $\tau$-action of $G$ on both $V$ and $\PP_{\RR}(V)$.

The goal of this subsection is to establish the following useful estimates.

\begin{lemma} \label{lem:prefix}
For any compact subset $\mathcal{M}$ of~$G$, there is a constant $\CM>0$ such that for any $g\in G$ and $m\in\mathcal{M}$,
\begin{enumerate}[label=(\roman*),ref=\roman*]
\item\label{item:gm} \quad $d_{\scriptscriptstyle G\!/\!P_{\theta}}(\Xi_{\theta}(gm), \Xi_{\theta}(g)) \leq \CM e^{-T_{\theta}(g)}$,
\item\label{item:mg} \quad $d_{\scriptscriptstyle G\!/\!P_{\theta}}(\Xi_{\theta}(mg), m \cdot \Xi_{\theta}(g)) \leq \CM e^{-T_{\theta}(g)}$.
\end{enumerate}
\end{lemma}

Lemma~\ref{lem:prefix} has the following consequence.

\begin{corollary} \label{cor:limit-set-indep}
Consider a sequence $(g_n)\in G^{\NN}$ with $T_{\theta}(g_n)\to +\infty$.
Suppose that $(\Xi_{\theta}(g_n))_{n\in\NN}$ converges to some $x\in G/P_{\theta}$.
Then 
\begin{enumerate}
  \item\label{item:m-gn} for any $m\in G$, the sequence $(\Xi_{\theta}(mg_n))_{n\in\NN}$ converges to $m\cdot x\in G/P_{\theta}$;
  \item\label{item:indep-of-choices} $x$ does not depend on any choice made in the definition of~$\Xi_{\theta}$: namely, it does not depend on the choice of Cartan decomposition $G=K(\exp\overline{\aaa}^+)K$ (which determines the basepoint $x_{\theta}\in G/P_{\theta}$) nor, given this choice, on the choice of $k_{g_n},k'_{g_n}\in K$.
\end{enumerate}
\end{corollary}

\begin{proof}
For \eqref{item:m-gn}, apply Lemma~\ref{lem:prefix}.\eqref{item:mg} with $g=g_n$.
For \eqref{item:indep-of-choices}, recall that any Cartan decomposition of~$G$ is obtained from $G=K(\exp\overline{\aaa}^+)K$ by conjugating $K$ and $\exp\overline{\aaa}^+$ by some common element $m\in G$; this corresponds to replacing $k_g$ by $mk_{m^{-1}gm}m^{-1}$ for any $g\in G$ and the basepoint $x_{\theta}$ by $m\cdot x_{\theta}$.
Independence from the choice of Cartan decomposition then follows from the existence, given by Lemma~\ref{lem:prefix} applied to \(\mathcal{M}=\{m,m^{-1}\}\), of a constant $\CM>0$ such that for all $g\in G$,
\[ d_{\scriptscriptstyle G\!/\!P_{\theta}}(mk_{m^{-1}gm}\cdot x_{\theta},k_g\cdot x_{\theta})\leq 2\CM\,e^{-T_{\theta}(g)}. \]
Finally, since $T_{\theta}(g_n)>0$ for all large enough~$n$, the choice of $k_{g_n},k'_{g_n}\in K$ has no effect on $\Xi_{\theta}(g_n)$, as seen in Section~\ref{subsec:limit-set}.
\end{proof}

In order to prove Lemma~\ref{lem:prefix}, we make the following observation, where for $g\in G$ we denote by
\[ R(g) := \langle \chi_\tau,\, \mu(g) \rangle = \max_{v\in V\smallsetminus\{ 0\} } \log\biggl(\frac{\Vert g\cdot v\Vert_V}{\Vert v\Vert_V}\biggr) \,\geq 0 \]
the logarithm of the largest singular value of $\tau(g)$.
By Lemma~\ref{lem:tau}, the quantity $T_{\theta}(g)$ of \eqref{eqn:Ttheta} is the logarithm of the ratio of the two largest singular values of $\tau(g)$.

\begin{observation}\label{obs:norm-v}
Let $v\in V$ be a highest-weight vector for~$\tau$ with $\Vert v\Vert_V=1$.
For any $h\in G$ and $a,b\in \exp{\overline{\aaa}^+}$, 
\[ d_{\scriptscriptstyle G\!/\!P_{\theta}}(x_{\theta}, h\cdot x_{\theta}) \leq e^{-T_{\theta}(a)} e^{R(a)-R(b)+R(h^{-1})} \, \Vert a^{-1}hb\cdot v\Vert_V. \]
\end{observation}

\begin{figure}[h!]
\centering
\labellist
\small\hair 2pt
\pinlabel $0$  at -1 2
\pinlabel $v$  at 110 8
\pinlabel $tv$  at 183 -3
\pinlabel $w$  at -4 105
\pinlabel $hv$  at 150 122
\pinlabel $a^{-1}\cdot hv$  at 53 75
\pinlabel $a^{-1}\cdot w$  at -15 75
\endlabellist
\includegraphics[width=7cm]{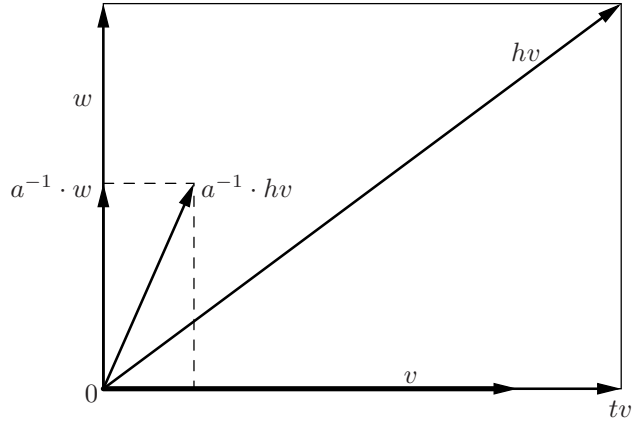}
\caption{Illustration of Observation~\ref{obs:norm-v} when $b=1$: as $a^{-1}$ compresses $e^{T_{\theta}(a)}$ times less strongly
along~$v^\perp$ than in the $v$-direction, we get a lower bound on $\Vert a^{-1}\cdot hv \Vert_V$ in terms of the angle $\measuredangle(v, hv)$ (or its sine $ d_{\scriptscriptstyle G\!/\!P_{\theta}}(x_{\theta}, h\cdot x_{\theta}))$.}
\label{vectors}
\end{figure}

\begin{proof}[Proof of Observation~\ref{obs:norm-v}]
Write $h\cdot v=tv+w$ where $t\in\RR$ and $w$ belongs to the orthogonal $v^{\perp}$ of $v$ in $(V,\Vert\cdot\Vert_V)$ (see Figure~\ref{vectors}). 
Then $\Vert w\Vert_V = \Vert h\cdot v\Vert_V \, |\sin\measuredangle(v,h\cdot v)| = \Vert h\cdot v\Vert_V \, d_{\scriptscriptstyle G\!/\!P_{\theta}}(x_{\theta}, h\cdot x_{\theta})$, and so
\[ \Vert a^{-1}h\cdot v\Vert_V \geq \Vert a^{-1}\cdot w\Vert_V \geq e^{-R(a)+T_{\theta}(a)} \, \Vert w \Vert_V = e^{-R(a)+T_{\theta}(a)} \, \Vert h\cdot v\Vert_V \, d_{\scriptscriptstyle G\!/\!P_{\theta}}(x_{\theta}, h\cdot x_{\theta}).\]
To conclude, note that $1=\Vert h^{-1}h\cdot v\Vert_V\leq e^{R(h^{-1})} \Vert h\cdot v\Vert_V$ and $b\cdot v=\nolinebreak e^{R(b)}v$, so that
\begin{eqnarray*}
\Vert a^{-1}hb\cdot v\Vert_V & = & e^{R(b)}\Vert a^{-1}h\cdot v\Vert_V\\
& \geq & e^{-R(a)+T_{\theta}(a)+R(b)} \, \Vert h\cdot v\Vert_V \, d_{\scriptscriptstyle G\!/\!P_{\theta}}(x_{\theta}, h\cdot x_{\theta})\\
& \geq & e^{-R(a)+T_{\theta}(a)+ R(b)-R(h^{-1})} \, d_{\scriptscriptstyle G\!/\!P_{\theta}}(x_{\theta}, h\cdot x_{\theta}). \hspace{1cm} \qedhere
\end{eqnarray*}
\end{proof}

\begin{proof}[Proof of Lemma~\ref{lem:prefix}]
Let $\mathcal{M}$ be a compact subset of~$G$.
By continuity of~$\mu$, there is a constant $\delta\geq 0$ such that
$\Vert m\cdot v'\Vert_V\leq e^\delta\,\Vert v'\Vert_V$ and
$\Vert\mu(m)\Vert\leq\delta$ for all $m\in\mathcal{M}$ and $v'\in V$, where $\Vert\cdot\Vert$ is the $W$-invariant Euclidean norm on~$\aaa$ from Section~\ref{subsubsec:Cartan-proj}.
Let $v\in V^{\chi_{\tau}}$ be a highest-weight vector for~$\tau$ with $\Vert v\Vert_V=\nolinebreak 1$.
Recall the elements $k_g, k'_g \in K$ defined before \eqref{eqn:Xi-theta+-}.
For any $g\in G$ and $m\in\mathcal{M}$,
\[ \big\Vert \exp(\mu(g))^{-1} (k_g^{-1} k_{gm}) \exp(\mu(gm)) \cdot v \big\Vert_V = \Vert k'_g m k_{gm}^{\prime -1} \cdot v \Vert_V \leq e^{\delta}. \]
By applying Observation~\ref{obs:norm-v} to $(a,b,h)=(\exp({\mu(g)}),\exp({\mu(gm)}),k_g^{-1} k_{gm})$ we obtain
\[ d_{\scriptscriptstyle G\!/\!P_{\theta}}(\Xi_{\theta}(g), \Xi_{\theta}(gm)) = d_{\scriptscriptstyle G\!/\!P_{\theta}}(k_g\cdot x_{\theta}, k_{gm}\cdot x_{\theta}) \leq e^{-T_{\theta}(g)} e^{R(g)-R(gm)} e^\delta. \]
By strong subadditivity of~$\mu$ (Fact~\ref{fact:mu-subadditive}),
\begin{align*}
|R(g)-R(gm)| & =  |\langle\chi_{\tau},\mu(g)-\mu(gm)\rangle| \leq \Vert\mu(g)-\mu(gm)\Vert \, \Vert \chi_\tau\Vert\\
& \leq  \Vert\mu(m)\Vert \, \Vert \chi_\tau\Vert \leq \delta \Vert \chi_\tau\Vert.
\end{align*}
Therefore,
\[ d_{\scriptscriptstyle G\!/\!P_{\theta}}(\Xi_{\theta}(g), \Xi_{\theta}(gm)) \leq e^{\delta (1+\Vert \chi_\tau\Vert)} e^{-T_{\theta}(g)} \]
\ie \eqref{item:gm} holds with $\CM = e^{\delta (1+\Vert \chi_\tau\Vert)}$.

Similarly, we have 
\[ \big\Vert \exp(\mu(mg))^{-1} (k_{mg}^{-1} m k_g) \exp(\mu(g)) \cdot v \big\Vert_V = \Vert k'_{mg} m k_{g}^{\prime -1} \cdot v \Vert_V \leq e^{\delta}. \]
By applying Observation~\ref{obs:norm-v} to $(a,b,h)=(\exp({\mu(mg)}),\exp({\mu(g)}),k_{mg}^{-1} m k_g)$ we obtain
\begin{align*} d_{\scriptscriptstyle G\!/\!P_{\theta}}(\Xi_{\theta}(mg), m \cdot \Xi_{\theta}(g)) 
 &=  d_{\scriptscriptstyle G\!/\!P_{\theta}}(k_{mg}\cdot x_{\theta}, mk_g\cdot x_{\theta}) \\ 
 &\leq e^{-T_{\theta}(g)} e^{R(mg)-R(g)+R(m^{-1})} e^\delta. 
 \end{align*}
As above, $|R(mg)-R(g)| \leq \Vert\mu(m)\Vert \, \Vert \chi_\tau\Vert \leq \delta \Vert \chi_\tau\Vert$, and
\[ R(m^{-1}) \leq \Vert\mu(m^{-1})\Vert \, \Vert\chi_{\tau}\Vert = \Vert\mu(m)\Vert \, \Vert\chi_{\tau}\Vert \leq \delta \, \Vert\chi_{\tau}\Vert, \]
hence
\[ d_{\scriptscriptstyle G\!/\!P_{\theta}}(\Xi_{\theta}(mg), m \cdot \Xi_{\theta}(g)) \leq e^{\delta (1+2\Vert \chi_\tau\Vert)} e^{-T_{\theta}(g)}, \]
\ie \eqref{item:mg} holds with $\CM = e^{\delta (1+2\Vert \chi_\tau\Vert)}$.
\end{proof}

%%%%%%%%%%%%%%%%%%%%%%%%%
\subsection{Existence, equivariance, continuity, and dynamics-preserving property for the boundary maps}
\label{subsec:exist-equiv-cont}

We now give a proof of statements \eqref{item:xi-exist}, \eqref{item:xi-dyn-preserv}, \eqref{item:xi-cont} of Theorem~\ref{thm:constr-xi-explicit}.

\begin{proof}[Proof of Theorem~\ref{thm:constr-xi-explicit}.\eqref{item:xi-exist}]
Let $(\gamma_n)_{n\in\NN}$ be a $(c,C)$-quasi-geodesic ray in the Cayley graph of~$\Gamma$, with endpoint $\eta\in\partial_{\infty}\Gamma$.
The set $\{ \gamma_{n}^{-1} \gamma_{n+1} \,|\, n\in\NN\}$ is contained in the ball $B_e(c+C)$ of radius $c+C$ centered at $e\in\Gamma$.
Applying Lemma~\ref{lem:prefix}.\eqref{item:gm} to $\rho(\gamma_n)$ and to $\mathcal{M} := \rho(B_e(c+C))$, we obtain 
\[ d_{\scriptscriptstyle G\!/\!P_{\theta}}\big(\Xi^+(\gamma_n), \Xi^+(\gamma_{n+1})\big) = d_{\scriptscriptstyle G\!/\!P_{\theta}}\big(\Xi_{\theta} \circ \rho(\gamma_n), \Xi_{\theta} \circ \rho(\gamma_{n+1})\big) \leq \CM \, e^{-T_{\theta}(\rho(\gamma_n))}\]
for all $n\in\NN$.
The gap summation property and Lemma~\ref{lem:gap-summ-prop} imply
\[ \sum_{n\in\NN} e^{-T_{\theta}(\rho(\gamma_n))} < + \infty. \]
Thus $(\Xi^+(\gamma_n))_{n\in\NN}$ is a Cauchy sequence in $G/P_\theta$, and so its limit $\xi^+(\eta)\in G/P_{\theta}$ exists.
More precisely, for any $n_0 \in \NN$,
\begin{equation}
  \label{eq:d_xi_Xi}
  d_{\scriptscriptstyle G\!/\!P_{\theta}}\big(\Xi^+(\gamma_{n_0}), \xi^+ (\eta)\big) \leq \CM \sum_{n\geq n_0} e^{-T_{\theta}(\rho(\gamma_n))}.
\end{equation}
This limit $\xi^+(\eta)$ does not depend on the choice of the quasi-geodesic ray $(\gamma_n)_{n\in\NN}$, because any other quasi-geodesic ray $(\gamma'_n)_{n\in\NN}$ lies within some distance $R$ from $(\gamma_n)$, and so we can apply Lemma~\ref{lem:prefix}.\eqref{item:gm} again (and Lemma~\ref{lem:gap-summ-prop}), taking for $\mathcal{M}$ the image under~$\rho$ of the Cayley $R$-ball centered at~$e$.
Thus $\xi^+ : \partial_{\infty} \Gamma\to G/P_{\theta}$ is well defined.
The independence of $\xi^+$ from all choices then follows from Corollary~\ref{cor:limit-set-indep}.

For the $\rho$-equivariance of~$\xi^+$, consider a quasi-geodesic ray $(\gamma_n)_{n\in\NN}$ with endpoint $\eta\in\partial_{\infty}\Gamma$ and an element $\gamma\in\Gamma$.
Then $(\gamma\gamma_n)_{n\in\NN}$ is a quasi-geodesic ray with endpoint $\gamma\cdot\eta\in\partial_{\infty}\Gamma$, and $\xi^+(\gamma\cdot\eta)=\rho(\gamma)\cdot\xi^+(\eta)$ by Corollary~\ref{cor:limit-set-indep}.\eqref{item:m-gn}.

We argue similarly for $\xi^- : \partial_{\infty}\Gamma\to G/P_{\theta^{\star}}\simeq G/P_{\theta}^-$, replacing $\theta$ with~$\theta^{\star}$ and $\Xi^+$ with~$\Xi^-$.
By construction, for any $\eta \in \partial_{\infty}\Gamma$, the points $\xi^+(\eta)\in G/P_{\theta}$ and $\xi^-(\eta)\in G/P_{\theta}^-$ are compatible in the sense of Definition~\ref{defi:transverse-compatible}.
\end{proof}

Theorem~\ref{thm:constr-xi-explicit}.\eqref{item:xi-dyn-preserv} is based on the following observation.

\begin{lemma}\label{lem:xi-preserves-dynamics}
For any $g\in G$ which is proximal in $G/P_{\theta}$,
\[ \Xi_{\theta}(g^n) \underset{n\to +\infty}{\longrightarrow} \xi_g^+,\]
where $\xi_g^+$ is the attracting fixed point of $g$ in $G/P_{\theta}$. 
\end{lemma}

\begin{proof}
The element~$g$ admits a Jordan decomposition $g=g_hg_eg_u$ where $g_h$ is hyperbolic, $g_e$ is elliptic, $g_u$ is unipotent, and $g_h,g_e,g_u$ commute.
Since $\lambda(g_eg_u)=0$ by definition of~$\lambda$, we have $\Vert\mu(g_e^n g_u^n)\Vert =o(n)$ as $n\to +\infty$, by \eqref{eqn:lambda-lim-mu}.
Let us write $g_h=mzm^{-1}$ where $m\in G$ and $z\in \exp( \overline{\aaa}^+)$.
Then $\xi_g^+=m\cdot P_{\theta}$.
Let $(\tau,V)$ and $\Vert\cdot\Vert_V$ be as in Section~\ref{subsec:estim-Xi-theta}, and let $v\in V^{\chi_{\tau}}$ be a highest-weight vector for~$\tau$ with $\Vert v\Vert_V=1$.
For any $n\in\NN$,
\[ \big\Vert z^{-n} m^{-1} k_{g^n} \exp({\mu(g^n)}) \cdot v \big\Vert_V = \Vert m^{-1} g_h^{-n} g^n k_{g^n}^{\prime -1} \cdot v \Vert_V = \Vert m^{-1} g_e^n g_u^n {k}_{g^n}^{\prime -1} \cdot v\Vert_V. \]
By Observation~\ref{obs:norm-v} applied to $(a,b,h)=(z^n,\exp({\mu(g^n)}),m^{-1} k_{g^n})$, we have
\begin{align*}
 d_{\scriptscriptstyle G\!/\!P_{\theta}}(\xi^{+}_{g}, \Xi_{\theta}({g^n}))= & ~ d_{\scriptscriptstyle G\!/\!P_{\theta}}(m\cdot x_{\theta}, k_{g^n}\cdot x_{\theta}) \ \leq\ C \, d_{\scriptscriptstyle G\!/\!P_{\theta}}(x_{\theta}, m^{-1}k_{g^n}\cdot x_{\theta})\\
 \leq & ~  C e^{-nT_{\theta}(z)} e^{R(z^n)-R(g^n)+R(m)} \, \Vert m^{-1} g_e^n g_u^n {k'}_{g^n}^{-1} \cdot v\Vert_V,
\end{align*}
where $C\geq 0$ is a Lipschitz constant for the action of $m$ on $G/P_{\theta}$.
By strong subadditivity of~$\mu$ (Fact~\ref{fact:mu-subadditive}),
\begin{align*}
|R(z^n)-R(g^n)| & =  |\langle\chi_{\tau},\mu(z^n)-\mu(mz^nm^{-1} g_e^n g_u^n)\rangle|\\
& \leq  \Vert \chi_\tau \Vert \big( 2\Vert\mu(m)\Vert + \Vert\mu(g_e^n g_u^n)\Vert\big) = o(n).
\end{align*}
Similarly,
\[ \Vert m^{-1} g_e^n g_u^n {k}_{g^n}^{\prime -1} \cdot v\Vert_V \leq e^{\langle\chi_{\tau},\, \mu(m^{-1} g_e^n g_u^n)\rangle}  \leq e^{\Vert\chi_{\tau}\Vert (\Vert\mu(m^{-1})\Vert + \Vert\mu(g_e^n g_u^n)\Vert)} = e^{o(n)}. \]
Therefore, $d_{\scriptscriptstyle G\!/\!P_{\theta}}(\xi_{g}^{+}, \Xi_{\theta}({g^n}))=e^{-nT_{\theta}(z)+o(n)}\to 0$ as $n\to\nolinebreak +\infty$.
\end{proof}

\begin{proof}[Proof of Theorem~\ref{thm:constr-xi-explicit}.\eqref{item:xi-dyn-preserv}]
Let $\gamma\in\Gamma$ be an element of infinite order with attracting fixed point $\eta_{\gamma}^+$ in $\partial_{\infty}\Gamma$.\,Suppose that for any $\alpha \in \theta$ we have $\langle \alpha, \mu(\rho(\gamma^n)) \rangle -\nolinebreak 2 \log n \to\nolinebreak\!+\infty$ as $n\to +\infty$.
By Lemma~\ref{lem:prox}, the element $\rho(\gamma)\in G$ is proximal in $G/P_{\theta}$.
By Lemma~\ref{lem:xi-preserves-dynamics}, the sequence $(\Xi^+(\gamma^n))_{n\in\NN}=(\Xi_{\theta}\circ \rho(\gamma^n))_{n\in\NN}$ converges to the attracting fixed point of $\rho(\gamma)$ in $G/P_{\theta}$.
On the other hand, this sequence converges to $\xi^+(\eta_{\gamma}^+)$ by construction of~$\xi^+$.
Thus $\xi^+$ is dynamics-preserving for~$\rho$.
We argue similarly for~$\xi^-$, replacing $\theta$ with~$\theta^{\star}$ and $\Xi^+$ with~$\Xi^-$, and using \eqref{eqn:opp-inv-mu}.
\end{proof}

\begin{proof}[Proof of Theorem~\ref{thm:constr-xi-explicit}.\eqref{item:xi-cont}]
It is sufficient to check the continuity of $\Xi^+\sqcup\xi^+$ and $\Xi^-\sqcup\xi^-$ at any point $\eta\in\partial_{\infty}\Gamma$.
Let $\delta>0$ be the Gromov hyperbolicity constant of the Cayley graph of~$\Gamma$.
By definition of the topology on $\Gamma\cup\partial_{\infty}\Gamma$ (see \cite[Lem.\,3.6]{Bridson_Haefliger}), for any $\eta\in\partial_{\infty}\Gamma$ which is the endpoint of a geodesic ray $(\sigma_n)\in\Gamma^{\NN}$ with $\sigma_0=e$, a basis of neighborhoods of~$\eta$ in $\Gamma\cup\partial_\infty \Gamma$ is given by the family $(\mathcal{V}^{\eta}_{n_0})_{n_0\in \NN}$, where $\mathcal{V}^{\eta}_{n_0}$ is by definition the set of endpoints of geodesic segments or rays $(\gamma_n)_{n\geq 0}$ of length $\geq n_0$ with $\ellGamma{\sigma_n^{-1}\gamma_n}\leq 3\delta$ for all $0\leq n\leq n_0$.
By Corollary~\ref{cor:unif-GSP-q-geod}, under the uniform gap  summation property,
\[ \varepsilon(n_0) := \sup_{\substack{(\gamma_n)_{n\in\NN}\\ \text{geodesic ray}\\ \text{with }\gamma_0=e}}\ \sum_{n\geq n_0} e^{- T_{\theta}(\rho(\gamma_n))} \]
tends to~$0$ as $n_0\to +\infty$.
Let $\mathcal{M} := \rho(B_e(1+3\delta))$ and let $\CM>0$ be given by Lemma~\ref{lem:prefix}.
For any $\eta\in\partial_{\infty}\Gamma$, any $n_0 \in \NN$, and any $y \in \mathcal{V}^{\eta}_{n_0}$ which is the endpoint of a geodesic segment or ray $(\gamma_n)_{n\geq 0}$ as above, we have 
\[ d_{\scriptscriptstyle G\!/\!P_{\theta}}\bigl((\Xi^+\sqcup\xi^+)(y),\xi^+(\eta)\bigr) \leq 2 \, \CM \,\varepsilon(n_0). \]
Indeed, by \eqref{eq:d_xi_Xi} applied to the $(1,3\delta)$-quasi-geodesic ray coinciding with $(\gamma_n)_{n\geq 0}$ for $n\leq n_0$ and with $(\sigma_n)_{n\in\NN}$ for $n>n_0$, we have
\[ d_{\scriptscriptstyle G\!/\!P_{\theta}}\bigl(\Xi^+(\gamma_{n_0}),\xi^+(\eta)\bigr) \leq \CM \,\varepsilon(n_0), \]
and by \eqref{eq:d_xi_Xi} applied to $(\gamma_n)_{n\geq 0}$ we have
\[ d_{\scriptscriptstyle G\!/\!P_{\theta}}\bigl(\Xi^+(\gamma_{n_0}),(\Xi^+\sqcup\xi^+)(y)\bigr) \leq \CM \,\varepsilon(n_0). \]
This shows that $\Xi^+\sqcup\xi^+$ is continuous at any $\eta\in\partial_{\infty}\Gamma$.

By construction, $\xi^+(\partial_\infty\Gamma)$ is contained in the limit set $\Lambda_{\rho(\Gamma)}^{G/P_{\theta}}$.
To check equality, consider a point $x\in\Lambda_{\rho(\Gamma)}^{G/P_{\theta}}$ which is the limit of a sequence $(\Xi^+(\gamma_n))_{n\in\NN}$ for some $(\gamma_n)\in\Gamma^{\NN}$ with $\langle\alpha,\mu(\rho(\gamma_n))\rangle\to +\infty$ for all $\alpha\in\theta$.
Since $\mu$ is a proper map, we have $\gamma_n\to\infty$ in~$\Gamma$.
By compactness of $\partial_{\infty}\Gamma$, up to extracting we may assume that $(\gamma_n)_{n\in\NN}$ converges to some $\eta\in\partial_{\infty}\Gamma$.
Then $x=\xi^+(\eta)$ by continuity of $\Xi^+\sqcup\xi^+$.
This shows that $\xi^+(\partial_{\infty}\Gamma)=\Lambda_{\rho(\Gamma)}^{G/P_{\theta}}$.

We argue similarly for $\xi^-$ and~$\Xi^-$, replacing $\theta$ with~$\theta^{\star}$.
\end{proof}

%%%%%%%%%%%%%%%%%%%%%%%%%
\subsection{Transversality of the boundary maps} \label{subsec:proof-xi-transv}

In order to prove the trans\-versality of the boundary maps $\xi^+,\xi^-$ under the assumption that the sequence $(\langle \alpha,\, \mu( \rho(\gamma_n))\rangle)_{n\in\NN}$ is CLI for any $\alpha\in\Sigma^+_{\theta}$ and any geodesic ray $(\gamma_n)_{n\in\NN}$ (Theorem~\ref{thm:constr-xi-explicit}.\eqref{item:xi-transv}), we first consider the special case where $G=\GL_d(\RR)$ and $P_\theta$ is the stabilizer of a line, \ie $G/P_\theta = \PP^{d-1}(\RR)$.
The general case is treated in Section~\ref{subsubsec:transv-general}: we reduce to this special case using the results of Section~\ref{sec:repr-semis-lie}.

%%%%%%%%
\subsubsection{Transversality in $\GL_d(\RR)$}

Let $G=\GL_d(\RR)$.
As in Example~\ref{ex:roots}, we take $K$ to be $\OO(d)$ and $\overline{\aaa}^+$ to be the set of real diagonal matrices of size $d\times d$ with entries in nonincreasing order; for $1\leq i\leq d$ we denote by $\varepsilon_i\in\aaa^{\ast}$ the evaluation of the $i$-th diagonal entry.
For $\theta=\{ \varepsilon_1-\varepsilon_2\}$, we see $G/P_{\theta}$ as $\PP^{d-1}(\RR)$ and $G/P_{\theta^{\star}}\simeq G/P_{\theta}^-$ as the space of projective hyperplanes in $\PP^{d-1}(\RR)$.
Transversality of $\xi^+$ and~$\xi^-$ in this case is given by the following proposition.

\begin{proposition}\label{prop:transv-SL}
Let $G=\GL_d(\RR)$ and $\theta=\{ \varepsilon_1-\varepsilon_2\}$.
Let $\Gamma$ be a word hyperbolic group and $\rho : \Gamma\to G$ a representation.
Suppose the maps $\xi^+, \xi^-$ of Theorem~\ref{thm:constr-xi-explicit}.\eqref{item:xi-exist} are well defined and $\rho$-equivariant, and that $\Xi^- \sqcup \xi^- = (\Xi_{\theta^{\star}}\circ\rho) \sqcup \xi^- : \Gamma \cup \partial_{\infty}\Gamma \to G/P_{\theta^{\star}}$ is continuous.
Let $(\gamma_n)_{n\in\NN}$ be a quasi-geodesic ray in the Cayley graph of~$\Gamma$, with endpoint $\eta\in\partial_{\infty}\Gamma$.
If the sequences $(\langle \varepsilon_1-\varepsilon_i,\, \mu(\rho(\gamma_n))\rangle)_{n\in\NN}$ are CLI (Definition~\ref{defi:cli}) for all $2\leq i\leq d$, then $\xi^+(\eta)\notin\xi^-(\eta')$.
\end{proposition}

Before proving Proposition~\ref{prop:transv-SL}, let us fix some notation.
Let $(e_1, e_2, \dots, e_d)$ be the canonical basis of~$\RR^d$.
The point $P_{\theta}\in G/P_{\theta}$ corresponds to
\[x_{\theta} := [e_1] \in \PP^{d-1}(\RR).\]
The points $P_{\theta^{\star}}\in G/P_{\theta^{\star}}$ and $P_{\theta}^-\in G/P_{\theta}^-$ correspond respectively to the projective hyperplanes 
\[ X_{\theta} := \PP_{\RR}(\mathrm{span}(e_1,\dots,e_{d-1}))\quad \text{and}\quad Y_{\theta} := \PP_{\RR}(\mathrm{span}(e_2,\dots,e_d)). \]
For any $n\in\NN$, we set $g_n:=\rho(\gamma_n)$ and $a_n:=\exp(\mu(g_n))$ (this is a diagonal matrix with positive entries, in nonincreasing order), and choose elements $k_n,k'_n\in K$ such~that
\[ g_n = k_n a_n k'_n. \]
The crux of the proof of Proposition~\ref{prop:transv-SL} will be to control the elements
\begin{equation}\label{eqn:13}
H_{n}^{m} := k_{n}^{-1} k_{n+m} \, \in K.
\end{equation}
More precisely, our main technical result will be the following.

\begin{lemma}\label{lem:bound-h-n-infty}
With the above notation, if for every $2\leq i\leq d$ the sequence\linebreak $(\langle \varepsilon_1 -\nolinebreak \varepsilon_i, \mu(g_n) \rangle)_{n \in \NN}$ is CLI, then for every $2\leq i\leq d$ the absolute value of the $(i,1)$-th entry of the matrix $a_{n}^{-1} H_{n}^{m} a_n$ is uniformly bounded for $(n,m) \in \NN^2$.
\end{lemma}

Lemma~\ref{lem:bound-h-n-infty} has the following easy consequence. 
We use the $K$-invariant metric $d_{\scriptscriptstyle G\!/\!P_{\theta}}$ of Section~\ref{subsec:estim-Xi-theta} on $G/P_{\theta}=\PP^{d-1}(\RR)$, which is valued in $[0,1]$.

\begin{corollary} \label{cor:limsup-Hnm}
With the above notation, if for every $2\leq i\leq d$ the sequence $(\langle \varepsilon_1 -\nolinebreak \varepsilon_i, \mu(g_n) \rangle)_{n \in \NN}$ is CLI, then there exists $0<\delta\leq 1$ such that for all large enough $n\in\NN$,
  \begin{equation} \label{eqn:gn-inv-kn'}
  d_{\scriptscriptstyle G\!/\!P_{\theta}} \bigl(g_n^{-1} \cdot \xi^+(\eta), k_n^{\prime -1} \cdot x_{\theta} \bigr) \leq 1 - \delta.
  \end{equation}
\end{corollary}

\begin{remark}
Corollary~\ref{cor:limsup-Hnm} states an expansion property for the action of $(\rho(\gamma_n^{-1}))_{n\in\NN}\linebreak =(g_n^{-1})_{n\in\NN}$ on $\PP^{d-1}(\RR)$ at $\xi^+(\eta)$.
Indeed, for any large enough $n\in\NN$, the open set
\[ \mathcal{U}_n := \{ x\in\PP^{d-1}(\RR) ~|~ d_{\scriptscriptstyle G\!/\!P_{\theta}}(a_n^{-1} k_n^{-1} \cdot x, x_{\theta}) < 1 - \delta/2\} \]
is a neighborhood of $\xi^+(\eta)$ by Corollary~\ref{cor:limsup-Hnm} (recall that $d_{\scriptscriptstyle G\!/\!P_{\theta}}$ is $K$-invariant).
There is a constant $C>0$ such that for any~$n$ the element $a_n$ is $Ce^{-\langle\varepsilon_1-\varepsilon_2,\mu(g_n)\rangle}$-Lipschitz on the ball of radius $1-\delta/2$ centered at~$x_{\theta}$.
Therefore $\rho(\gamma_n^{-1})=g_n^{-1}$ is $C^{-1}e^{\langle\varepsilon_1-\varepsilon_2,\mu(g_n)\rangle}$-expanding on~$\mathcal{U}_n$, where $C^{-1}e^{\langle\varepsilon_1-\varepsilon_2,\mu(g_n)\rangle}\to +\infty$.
\end{remark}

\begin{proof}[Proof of Corollary~\ref{cor:limsup-Hnm}]
By construction of~$\xi^+$ and definition \eqref{eqn:13} of~$H_n^m$, for any~$n\in\nolinebreak\NN$,
  \[ \xi^+(\eta) = \lim_{m \to +\infty} k_n H_n^{m} \cdot x_{\theta}. \]
 Since $d_{\scriptscriptstyle G\!/\!P_{\theta}}$ is $K$-invariant and $x_{\theta} = a_{n}\cdot x_{\theta}$, we see that \eqref{eqn:gn-inv-kn'} is equivalent to
\[ \limsup_{m \to +\infty} d_{\scriptscriptstyle G\!/\!P_{\theta}}\bigl( (a_n^{-1} H_n^{m} a_n) \cdot x_{\theta}, x_{\theta} \bigr) \leq 1 - \delta. \]
To check the existence of $0<\delta\leq 1$ such that this last inequality holds for all large enough $n\in\NN$, it is sufficient to see that
\begin{enumerate}[label=(\roman*),ref=\roman*]
  \item\label{item:col-1} the first column of the matrix $a_{n}^{-1} H_{n}^{m} a_n$ is uniformly bounded for $(n,m)\in\NN^2$,
  \item\label{item:entry-1-1} the $(1,1)$-th entry of the matrix $a_{n}^{-1} H_{n}^{m} a_n$, which is also the $(1,1)$-th entry of~$H_n^m$, is uniformly bounded \emph{from below} for $(n,m)\in\NN^2$.
\end{enumerate}
(Here all bounds are meant for the absolute values of the entries.)
Suppose that the sequence $(\langle \varepsilon_1 - \varepsilon_i, \mu(g_n) \rangle)_{n \in \NN}$ is CLI for every $2\leq i\leq d$.
By Lemma~\ref{lem:bound-h-n-infty}, for every $2\leq i\leq d$ the $(i,1)$-th entry of $a_{n}^{-1} H_{n}^{m} a_n$ is uniformly bounded for $(n,m)\in\NN^2$.
The $(1,1)$-th entry is also bounded since $H_n^m\in K=\OO(d)$, hence \eqref{item:col-1} holds.
On the other hand, Lemma~\ref{lem:bound-h-n-infty} implies that for every $2\leq i\leq d$ the $(i,1)$-th entry of $H_{n}^{m}$ is uniformly bounded by a constant multiple of $e^{- \langle \varepsilon_1 - \varepsilon_i,\,  \mu(g_n) \rangle}$, which goes to $0$ as $n\rightarrow +\infty$.
Since $H_{n}^{m}\in K=\OO(d)$, this in turn implies that the absolute value of $(1,1)$-th entry of $H_{n}^{m}$ is close to~$1$ for all large enough~$n$ and all~$m$.
Thus \eqref{item:entry-1-1} holds.
\end{proof}

\begin{proof}[Proof of Proposition~\ref{prop:transv-SL}]
Let $0<\delta\leq 1$ be given by Corollary~\ref{cor:limsup-Hnm}.
We shall prove that for any $\eta^{\prime} \in \partial_{\infty}\Gamma \smallsetminus \{\eta\}$ we can find (infinitely many) $n\in\NN$ such that
\begin{equation} \label{eqn:proof_transvSL}
 \dist_{\scriptscriptstyle G\!/\!P_{\theta}}\big(g_n^{-1} \cdot \xi^+(\eta), g_n^{-1} \cdot \xi^-(\eta')\big) \geq \frac{\delta}{2} \, ,
\end{equation}
where \(\dist_{\scriptscriptstyle G\!/\!P_{\theta}}\) denotes the distance from a point to a set.
In particular, the point $\xi^+(\eta)$ does not belong to the hyperplane $\xi^-( \eta')$.

Applying Fact~\ref{fact:dyn_at_infty}.\eqref{item:4}, we can find a subsequence $(\gamma_{\phi(n)})_{n\in\NN}$ such that $(\gamma_{\phi(n)}^{-1})_{n\in\NN}$ converges to some point $\eta^{\prime \prime} \in \partial_{\infty}\Gamma$ and that $\lim_n \gamma_{\phi(n)}^{-1} \cdot \eta^{\prime} = \eta^{\prime \prime}$ for all $\eta^{\prime} \in \partial_{\infty}\Gamma \smallsetminus \{\eta\}$.
By $\rho$-equivariance and continuity of~$\xi^-$, for all $\eta^{\prime} \in \partial_{\infty}\Gamma \smallsetminus \{\eta\}$,
  \[ \lim_n g_{\phi(n)}^{-1} \cdot \xi^-( \eta') = \xi^-\bigl( \lim_n \gamma_{\phi(n)}^{-1} \cdot \eta'\bigr) = \xi^-( \eta^{\prime \prime}). \]

For any~$n$, the decomposition $g_n = k_n a_n k'_n \in K (\exp\overline{\aaa}^+) K$ induces the decomposition $g_n^{-1} = l_n b_n l_n^{\prime} \in K (\exp\overline{\aaa}^+) K$ with
\[ (l_n, b_n, l_n^{\prime}) = \big(k_n^{\prime -1} \tilde{w}_0, \, \tilde{w}_0^{-1} a_n^{-1} \tilde{w}_0, \, \tilde{w}_0^{-1} k_n^{-1}\big), \]
where $\tilde{w}_0 \in \OO(d)$ is the permutation matrix sending $(e_1,\dots,e_d)$ to $(e_d,\dots,e_1)$, representing the longest element $w_0$ of the Weyl group $W=\mathfrak{S}_d$ (see Example~\ref{ex:opp-inv}). By the definition of~$\xi^-$, the continuity of $\Xi^- \sqcup \xi^-$, and the fact that $l_n\cdot X_{\theta} = k^{\prime -1}_n\cdot Y_{\theta}$, we have
\[ \xi^{-}( \eta^{\prime \prime}) = \lim_n k^{\prime -1}_{\phi(n)} \cdot Y_{\theta}. \]
Thus
\[ \lim_n g_{\phi(n)}^{-1} \cdot \xi^-( \eta') = \lim_n k^{\prime -1}_{\phi(n)} \cdot Y_{\theta}. \]
On the other hand, by definition of~$d_{\scriptscriptstyle G\!/\!P_{\theta}}$ (see Section~\ref{subsec:estim-Xi-theta}), for any $x\in\PP^{d-1}(\RR)$,
  \[ \dist_{\scriptscriptstyle G\!/\!P_{\theta}}(x, Y_{\theta})^2 + d_{\scriptscriptstyle G\!/\!P_{\theta}}(x, x_{\theta})^2 = 1, \]
and so Corollary~\ref{cor:limsup-Hnm} implies the existence of $0<\delta\leq 1$ such that for all large enough~$n$,
\[ \dist_{\scriptscriptstyle G\!/\!P_{\theta}} \bigl(g_n^{-1} \cdot \xi^+(\eta), k_n^{\prime -1} \cdot Y_{\theta} \bigr) \geq \sqrt{1 - (1-\delta)^2} \,\geq\, \delta. \]
In particular, for any $\eta'\in\partial_{\infty}\Gamma\smallsetminus\{\eta\}$ we have, for all large enough~$n$,
\[ \dist_{\scriptscriptstyle G\!/\!P_{\theta}}\big(g_{\phi(n)}^{-1} \cdot \xi^+(\eta), g_{\phi(n)}^{-1} \cdot \xi^-(\eta')\big) \,\geq\, \frac{\delta}{2} \, , \]
proving \eqref{eqn:proof_transvSL}.
\end{proof}

For any $n\in\NN$, let $h_n := k_{n}^{-1} k_{n+1} \in K$, so that $H_{n}^{m} = h_n h_{n+1} \cdots h_{n+m-1}$ for all $n,m\in\NN$.
The rest of the section is devoted to establishing Lemma~\ref{lem:bound-h-n-infty}.
For this we first observe that we have a control on the entries of the matrices~$h_n$.
For $1\leq i,j\leq d$, the $(i,j)$-th entry of a matrix $h\in\GL_d(\RR)$ is denoted by $h(i,j)$.

\begin{lemma}\label{lem:bound_hn}
  There is a constant $r \geq 0$ such that for any $1\leq i,j\leq d$ and any $n\in\NN$, 
  \[| h_{n} (i,j) | \,\leq\, e^{r - |\langle \varepsilon_i - \varepsilon_j,\,  \mu(g_n) \rangle |}.\]
\end{lemma}

\begin{proof}[Proof of Lemma~\ref{lem:bound_hn}]
Since $(\gamma_n)_{n\in\NN}$ is a geodesic ray in the Cayley graph of~$\Gamma$, the sequence $\gamma_{n}^{-1} \gamma_{n+1}$ is bounded (it takes only finitely many values).
Therefore the sequence $\rho(\gamma_n^{-1} \gamma_{n+1}) = g_{n}^{-1} g_{n+1} = k_{n}^{\prime -1} a_{n}^{-1} h_n a_{n+1} k_{n+1}^{\prime}$ is bounded, as well as the sequence $\mu(g_n) - \mu(g_{n+1}) = \log(a_{n+1}^{-1} a_n )$ by Fact~\ref{fact:mu-subadditive}.
  Thus the sequence $a_{n}^{-1} h_n a_n = ( a_{n}^{-1} h_n a_{n+1}) ( a_{n+1}^{-1} a_n)$ is bounded in $\GL_d(\RR)$, as well as its transposed inverse $a_n h_n a_n^{-1}$.
Since \(h_n(i,j) e^{\pm \langle \varepsilon_i - \varepsilon_j, \log a_n \rangle} = ( a_{n}^{\pm 1} h_n a_{n}^{\mp 1})(i,j)\) and $a_n = \exp(\mu(g_n))$, we can take for $e^r$ a bound for the absolute values of the entries of these matrices $(a_{n}^{\pm 1} h_n a_n^{\mp 1})_{n\in \NN}$.
\end{proof}

\begin{proof}[Proof of Lemma~\ref{lem:bound-h-n-infty}]
Let us first explain the strategy.
The goal, rephrased, is to bound, independently of $m$, the entries  $H_n^m(i,1)$ in the first column of the orthogonal matrix  $H_n^m$ by a multiple of \(e^{-\langle \varepsilon_1 - \varepsilon_i,\, \mu(g_n) \rangle}\).
We will develop the entries of $H_{n}^{m}$ in terms of the $(h_k)_{n \leq k < n+m}$.
  In doing so, subproducts of the form 
\[h_k(u_k, u_{k+1}) h_{k+1}( u_{k+1}, u_{k+2}) \cdots h_{k+\ell}( u_{k+\ell}, u_{k+\ell+1})\] 
appear (where \(u_k, \dots, u_{k+\ell+1} \in \{ 1, \dots, d\}\)).
  Controlling the last factor by Lemma~\ref{lem:bound_hn} and bounding all others by $1$, such a product is bounded by (a constant multiple of) \(e^{ \langle \varepsilon_{u_{k}} - \varepsilon_{u_{k+\ell+1}}, \mu(g_{k+\ell})\rangle}\), under the assumption that \(u_{k+\ell+1} < u_{k}\leq u_{k+\ell}\).
  Even more is true: the sum of \emph{all} those products with the triple \((u_k, u_{k+\ell}, u_{k+\ell+1})\) fixed 
  and \( (u_{k+1}, \dots, u_{k+\ell-1})\) varying in \(\{u_k, \dots, d\}^{\ell-1}\) 
  is also bounded by (a constant multiple of) the same estimate \(e^{ \langle \varepsilon_{u_{k}} - \varepsilon_{u_{k+\ell+1}}, \mu(g_{k+\ell})\rangle}\).
  A summation by parts reveals pairings of the form \(\langle \varepsilon_{u_{1}} - \varepsilon_{u_{k}}, \mu(g_k) - \mu(g_{k+\ell})\rangle\) that can be bounded thanks to the CLI hypothesis.
  At that point the coefficient \((a_{n}^{-1} H_{n}^{m} a_n)(i,1)\) is written as a sum of terms each bounded by (a constant multiple of) \(e^{-\kappa m_s}\) for some \(m_s \in \{1, \dots, m\}\).
  The final bound comes from the fact that, for each value of \(m_s\in\NN\), the number of terms in that sum corresponding to \(m_s\) is a polynomial in \(m_s\).
  We now give the details.

  Let us first introduce some notation.
  For any integer $1\leq i\leq d$ and any matrix $h\in\GL_d(\RR)$, we denote by $h[i]\in\GL_{d+1-i}(\RR)$ the matrix obtained by crossing out the $(i-1)$ topmost rows and the $(i-1)$ leftmost columns of~$h$.
  When $h$ is orthogonal, $h[i]$ is the matrix of a $1$-Lipschitz linear transformation.
  Similarly to $H_{n}^{m}$, we define
  \[ {}_i H_{n}^{m} := h_n[i] \, h_{n+1}[i] \, \cdots \, h_{n+m-1}[i] \in \GL_{d+1-i}(\RR). \]
  Then ${}_i H_{n}^{m}$ is the matrix of a $1$-Lipschitz transformation and its entries are bounded by~$1$ in absolute value.
  As above, the $(k,\ell)$-th entry of a matrix $h\in\GL_d(\RR)$ is denoted by $h(k,\ell)$.
  To make things more natural, we index the entries of \(h[i]\) and of 
 ${}_i H_{n}^{m}$ by pairs $(k,\ell)\in\{ i,\dots,d\}^2$.
  
Fix $i\in \{2,\dots, d\}$. Using this notation, we now express the $(i,1)$-th entry of $H_n^m$:
\[ H_{n}^{m}(i,1) = \sum_{\mathbf{u}}\, \mathcal{H}_{\mathbf{u}} \, , \]
where the sum is over all $(m+1)$-tuples $\mathbf{u} = (u_0, u_1, \dots , u_m) \in \{ 1,\dots, d\}^{m+1}$ with $u_0=i$ and $u_m =1$ and we set
  \[ \mathcal{H}_{\mathbf{u}} := h_{n}(u_0, u_1) \, h_{n+1}(u_1, u_2) \cdots h_{n+m-1}(u_{m-1}, u_{m}). \] 
  This sum admits the following decomposition:
  
  \begin{equation}\label{eqn:bound-H-n-m-i-1}
  H_{n}^{m}(i,1) = \sum_{s\in\{ 1,\dots,i-1\} }\ \sum_{\mathbf{i}, \mathbf{j}, \mathbf{m}} \mathcal{H}_{\mathbf{i}, \mathbf{j}, \mathbf{m}} \, ,
  \end{equation}
  where, for fixed $s\in\{ 1,\dots,i-1\}$, the sum is over all $\mathbf{i}, \mathbf{j}, \mathbf{m}$ such that
\begin{itemize}[label=--]
  \item $\mathbf{i} = (i_0, i_1, \dots, i_s) \in \NN^{s+1}$ satisfies $i=i_0 > i_1 > \cdots > i_s =1$, 
  \item $\mathbf{j} = (j_0, j_1, \dots, j_{s-1}) \in \NN^s$ satisfies $d\geq j_k \geq i_k$ for all $0\leq k\leq s-1$,
  \item $\mathbf{m} = ( m_1, \dots, m_s) \in \NN^{s}$ satisfies
    $0< m_1 < m_2 < \cdots <\nolinebreak m_s \leq\nolinebreak
    m$,
\end{itemize}
  and the summand is:
  \[\mathcal{H}_{\mathbf{i}, \mathbf{j}, \mathbf{m}} = \hspace{-1em}
  \begin{array}{l}
    {}_{i_0} H_{n}^{m_1-1} (i_0, j_0) \, h_{n+m_1-1}(j_0, i_1)  \\
    \hspace{1em} {}_{i_1} H_{n+m_1}^{m_2-m_1-1} (i_1, j_1) \, h_{n+m_2-1}(j_1, i_2)
     \cdots  \\
    \hspace{2em} \cdots \, {}_{i_{s-1}} H_{n+ m_{s-1}}^{m_s-m_{s-1}-1} (i_{s-1}, j_{s-1}) \, h_{n+m_s-1}(j_{s-1}, i_s) \\
 \hspace{6em}    {}_{i_s} H_{n+ m_s}^{m-m_s}(i_s, 1).
  \end{array}
  \]
  Indeed, $\mathcal{H}_{\mathbf{i}, \mathbf{j}, \mathbf{m}}$ is the sum of the $\mathcal{H}_{\mathbf{u}}$ over all $\mathbf{u} = (u_0, u_1, \dots , u_m) \in \{ 1,\dots, d\}^{m+1}$ with $u_0=i$ and $u_m =1$ such that:
  \begin{itemize}[label=--]
  \item the smallest index $k$ with $u_k < u_0 = i$ is~$m_1$,
  \item the smallest index $k$ with $u_k < u_{m_1}$ is~$m_2$,
  \item \dots{}
  \item the smallest index $k$ with $u_k < u_{m_{s-2}}$ is $m_{s-1}$,
  \item the smallest index $k$ with $u_k = 1$ is $m_s$, (well defined since $u_m = 1$)
  \item and for any $k \in \{1, \dots, s-1\}$,  $u_{m_k} = i_k$ and $u_{m_{k+1}-1} = j_k$, and \(u_{m_s}=i_s=1\).
  \end{itemize}

From this we see that the subsums $\mathcal{H}_{\mathbf{i}, \mathbf{j}, \mathbf{m}}$, for varying $s$ and $\mathbf{i}, \mathbf{j}, \mathbf{m}$, form a partition of all the $\mathcal{H}_{\mathbf{u}}$.

  Since the entries of ${}_i H_{n}^{m}$ are bounded by $1$ in absolute value, we have
  \[ | \mathcal{H}_{\mathbf{i}, \mathbf{j}, \mathbf{m}} | \leq  | h_{n+m_1-1}(j_0, i_1)| \, | h_{n+m_2-1}(j_1, i_2)| \cdots |h_{n+m_s-1}(j_{s-1}, i_s) |.\]
  We take the convention that $\log(0) = -\infty$.
  By Lemma~\ref{lem:bound_hn}, we then have
  \[ \log | \mathcal{H}_{\mathbf{i}, \mathbf{j}, \mathbf{m}} | \leq s r + \sum_{k=1}^{s}\, \langle \varepsilon_{j_{k-1}} - \varepsilon_{i_k} , \mu(g_{n+m_k-1}) \rangle. \]
  Since $j_{k-1} \geq i_{k-1}$, we have $\langle \varepsilon_{j_{k-1}} , \mu(g) \rangle \leq \langle \varepsilon_{i_{k-1}} , \mu(g) \rangle$ for all $g\in G$, hence
  \begin{equation}\label{eqn:bound-H-i-j-m}
  \log | \mathcal{H}_{\mathbf{i}, \mathbf{j}, \mathbf{m}} | \leq s r + \sum_{k=1}^s\, \langle \varepsilon_{i_{k-1}} - \varepsilon_{i_k} , \mu(g_{n+m_k-1}) \rangle.
  \end{equation}
  We now sum by parts, using the fact that $i_s=1$: the right-hand side of \eqref{eqn:bound-H-i-j-m} is equal to
 \[ s r - \langle \varepsilon_1 - \varepsilon_{i_0}, \mu(g_{n+m_1-1}) \rangle - \sum_{k=1}^{s-1}\, \langle \varepsilon_1 - \varepsilon_{i_k}, \mu(g_{n+m_{k+1}-1}) - \mu(g_{n+m_k-1}) \rangle. \]
 Using the CLI assumption and writing \(\langle \varepsilon_1 -
 \varepsilon_{i_0}, \mu(g_{n+m_1-1}) \rangle = \langle \varepsilon_1 -
 \varepsilon_{i_0}, \mu(g_{n}) \rangle + \langle \varepsilon_1 -
 \varepsilon_{i_0}, \mu(g_{n+m_1-1}) - \mu(g_n) \rangle\), we deduce
 \begin{align*}
 \log | \mathcal{H}_{\mathbf{i}, \mathbf{j}, \mathbf{m}} | \leq &\ 
  s r - \langle \varepsilon_1 - \varepsilon_i, \mu(g_{n}) \rangle - (\kappa(m_1-1) - \kappa') - \sum_{k=1}^{s-1}\, (\kappa (m_{k+1} - m_k) - \kappa')\\
  = &\ s r - \langle \varepsilon_1 - \varepsilon_{i}, \mu(g_n) \rangle + s \kappa' - \kappa m_s + \kappa.
 \end{align*}
Going back to the formula \eqref{eqn:bound-H-n-m-i-1} for $H_{n}^{m}(i,1)$, we obtain
\begin{align*}
   |(a_{n}^{-1} H_{n}^{m} a_n) (i,1)| \,=\,& e^{\langle \varepsilon_1 - \varepsilon_{i},\,  \mu(g_n) \rangle} \, |H_{n}^{m}(i,1)| \\
  \,\leq\,& \sum_{s=1}^{i-1} e^{s(r + \kappa') +\kappa} \sum_{ \mathbf{i}, \mathbf{j}, \mathbf{m}} e^{-\kappa m_s}.
\end{align*}
  Observe now that for a fixed~$q\in\NN$ there is at most a polynomial number $P(q)$ of possible $\mathbf{i}, \mathbf{j}, \mathbf{m}$ with $m_s =q$, hence the above sum is bounded by a multiple of the converging series $\sum_q P(q) e^{-\kappa q}$. The real numbers $(a_{n}^{-1} H_{n}^{m} a_n) (i,1)$ are therefore uniformly bounded for $(n, m) \in \NN^2$.
\end{proof}

%%%%%%%%
\subsubsection{Transversality in general}\label{subsubsec:transv-general}

We now prove Theorem~\ref{thm:constr-xi-explicit}.\eqref{item:xi-transv} in full generality.
Recall the subset $\Sigma^+_{\theta}\subset\Sigma^+$ of the positive restricted roots from \eqref{eqn:T-G/P}.

\begin{proposition}\label{prop:transv}
Let $\Gamma$ be a word hyperbolic group, $G$ a real reductive Lie group, $\theta\subset\Delta$ a nonempty subset of the simple restricted roots of~$G$, and $\rho : \Gamma\to\nolinebreak G$ a representation.
Suppose that the maps $\xi^+,\xi^-$ of Theorem~\ref{thm:constr-xi-explicit}.\eqref{item:xi-exist} are well defined and $\rho$-equivariant, and that $\Xi^- \sqcup \xi^- = (\Xi_{\theta^{\star}}\circ\rho) \sqcup \xi^- : \Gamma \cup \partial_{\infty}\Gamma \to G/P_{\theta^{\star}}$ is continuous.
Let $(\gamma_n)_{n\in\NN}$ be a geodesic ray in the Cayley graph of~$\Gamma$, with endpoint $\eta\in\partial_{\infty}\Gamma$.
If the sequences $(\langle \alpha, \, \mu(\rho(\gamma_n))\rangle)_{n\in\NN}$ for $\alpha\in\Sigma^+_{\theta}$ are CLI, then $\xi^+(\eta)$ and $\xi^-(\eta')$ are transverse for all $\eta'\in\partial_{\infty}\Gamma\smallsetminus\{ \eta\}$.
\end{proposition}

To prove Proposition~\ref{prop:transv} (and also later Proposition~\ref{prop:Ano-CLI}), we shall use the following lemma.

\begin{lemma} \label{lem:reordering}
Let $\kappa,\kappa'>0$.
For $1\leq i\leq D$, let $(x_n^{(i)})_{n\in\NN}$ and $(y_n^{(i)})_{n\in\NN}$ be sequences of real numbers.
Suppose that the $D$ sequences $(x_n^{(i)})_{n\in\NN}$, for $1\leq i\leq D$, are all $(\kappa,\kappa')$-lower CLI, and that for any $n\in\NN$ there exists $\sigma_n\in\mathfrak{S}_D$ such that $y_n^{(i)}=x_n^{(\sigma_n(i))}$ for all $1\leq i\leq D$.
Suppose also that we are in one of the following two cases:
\begin{enumerate}\renewcommand*{\labelenumi}{(\textasteriskcentered)}\renewcommand{\theenumi}{\textasteriskcentered}
   \item\label{item:ast}
     $y^{(1)}_n\geq \dots \geq y^{(D)}_n$ for all $n\in\NN$, or
     \renewcommand*{\labelenumi}{(\textasteriskcentered\textasteriskcentered)}\renewcommand{\theenumi}{\textasteriskcentered\textasteriskcentered}
  \item\label{item:astast}  there exists $M>0$ such that $|y_{n+1}^{(i)} - y_n^{(i)}|\leq M$ for all $n\in\NN$ and~$1\leq\nolinebreak i\leq\nolinebreak D$.
\end{enumerate} 
Then the $D$ sequences $(y_n^{(i)})_{n\in\NN}$, for $1\leq i\leq D$, are all $(\kappa,\tilde{\kappa}')$-lower CLI, where we set
\[ \tilde{\kappa}' := \left\{
\begin{array}{ll}
\kappa' & \text{in case }\eqref{item:ast},\\
\kappa'+D(\kappa+\kappa'+M) & \text{in case }\eqref{item:astast}.
\end{array}
\right. \]
\end{lemma}

\begin{remark}
An analogous statement holds for the \((\kappa^{\prime\prime}, \kappa^{\prime\prime\prime})\)-upper CLI property of sequences $(x_n^{(i)})_{n\in\NN}$ and $(y_n^{(i)})_{n\in\NN}$: this follows from Lemma~\ref{lem:reordering} applied to the sequences $(Rn-x^{(i)}_n)$ and $(Rn-y^{(i)}_n)$, where  $R$ is any real \(\geq \kappa^{\prime\prime}\).
\end{remark}

\begin{proof}[Proof of Lemma~\ref{lem:reordering}]
Suppose we are in case~\eqref{item:ast}.
Then for any $n\in\NN$ and $1\leq i\leq D$ the number $y^{(i)}_{n}$ can be defined as the largest real number $y$ such that at least $i$ of the $D$ numbers $x^{(1)}_n,\dots,x^{(D)}_n$ are $\geq y$.
By the CLI hypothesis, for any $m\geq 0$, at least $i$ of the $D$ numbers $x^{(1)}_{n+m}, \dots, x^{(D)}_{n+m}$ are $\geq y^{(i)}_{n}+\kappa m-\kappa'$, and so $y^{(i)}_{n+m}\geq y^{(i)}_{n}+ \kappa m-\kappa'$.
This proves that $(y_n^{(i)})_{n\in\NN}$ is lower CLI with constants $(\kappa ,\kappa')$.

We now treat case~\eqref{item:astast}.
By case~\eqref{item:ast}, up to reordering the $x^{(i)}_n$ for each $n\in\NN$, we may assume that $x^{(1)}_n\geq \dots \geq x^{(D)}_n$ for all $n\in\NN$.
Fix an integer $1\leq i\leq D$ and an integer $n\in\NN$, and focus on the sequence $(y^{(i)}_{n+m})_{m\geq 0}$.
For any $m\geq 0$, let $r_m:=\sigma_{n+m}(i)\in [1,D]$, so that $y^{(i)}_{n+m}=x^{(r_m)}_{n+m}$ is the $r_m$-th largest number in the family $(y^{(1)}_{n+m},..., y^{(D)}_{n+m})$.
There exist an integer $s\leq D$ and a finite maximal list 
$0 = m_0 < m_1 < \dots < m_{s+1} = +\infty$ such that
\[ r_{m_0} < r_{m_1} < \dots < r_{m_s} \leq D 
\:\:\: \text{ and } \:\:\:
 r_m \leq r_{m_j} \:\text{ for all }\: m < m_{j+1}. \]
For any $m_j\leq m < m_{j+1}$,
\[y^{(i)}_{n+m} - y^{(i)}_{n+m_j} = x^{(r_m)}_{n+m} - x^{(r_{m_j})}_{n+m_j} \geq x^{(r_{m_j})}_{n+m} - x^{(r_{m_j})}_{n+m_j} \geq \kappa (m-m_j) -\kappa', \]
where the first inequality comes from the assumption $x^{(1)}_{n+m}\geq \dots \geq x^{(D)}_{n+m}$ and the second inequality from the CLI hypothesis.
Then for any $m_j\leq m < m_{j+1}$, using the CLI hypothesis again as well as~\eqref{item:astast}, we can bound
\begin{align*}
  y^{(i)}_{n+m} -  y^{(i)}_n
& =  y^{(i)}_{n+m} - y^{(i)}_{n+m_j} + \sum_{k=1}^j \Big( \big(y^{(i)}_{n+m_k} - y^{(i)}_{n+m_k-1}\big) + \big(y^{(i)}_{n+m_k-1} - y^{(i)}_{n+m_{k-1}}\big) \Big)\\
& \geq  \displaystyle \kappa (m-m_j)- \kappa' + \sum_{k=1}^j \big(- M + \kappa(m_k - 1 - m_{k-1}) - \kappa'\big)\\
& =  \kappa m - j(M+\kappa) - (j+1)\kappa'. 
\end{align*}
Since $j \leq s \leq D$, this produces the desired bounds.
\end{proof}

\begin{proof}[Proof of Proposition~\ref{prop:transv}]
Let $(\tau,V)$ be an irreducible, pro\-ximal, $\theta$-compati\-ble linear representation of~$G$ (Lemma~\ref{lem:theta-comp-exists}). 
By Proposition~\ref{prop:theta_emb}.\eqref{item:b_thtmb}, it induces embeddings $\iota^+ :\nolinebreak G/P_{\theta}\to\PP_{\RR}(V)$ and $\iota^- : G/P^{-}_{\theta}\to\PP_{\RR}(V^{\ast})$.
We identify $\GL_{\RR}(V)$ with $\GL_d(\RR)$ for some $d\in\NN$, and denote by $\aaa_{\GL_d}$ the set of diagonal matrices in $\gl_d(\RR)$ and $\overline{\aaa}_{\GL_d}^+$ its subset with entries in nonincreasing order.
Up to conjugating~$\tau$, we may assume that the Cartan decomposition of~$G$ is compatible with the Cartan decomposition of $\GL_d(\RR)$ of Example~\ref{ex:roots}, in the sense that $\tau(K)\subset\OO(d)$ and $\mathrm{d}_e\tau(\aaa)\subset\aaa_{\GL_d}$.
We distinguish the corresponding Cartan projections $\mu : G\to\overline{\aaa}^+$ and $\mu_{\GL_d} : \GL_d(\RR)\to\overline{\aaa}_{\GL_d}^+$.
For any $g\in G$, the matrix $\mathrm{d}_e\tau(\mu(g))$ is diagonal with entries $\langle \chi, \mu(g) \rangle$, where $\chi\in\aaa^{\ast}$ ranges through the weights of~$\tau$.
The matrix $\mu_{\GL_d}(\tau(g))$ is the diagonal matrix with the previous entries ordered: $\langle \varepsilon_1, \mu_{\GL_d}(\tau(g)) \rangle = \langle \chi_\tau, \mu(g) \rangle$, and for any $2\leq i\leq d$ there is a weight $\chi \neq \chi_\tau$ (depending on $g$ and~$i$) such that $\langle \varepsilon_i, \mu_{\GL_d}(\tau(g)) \rangle = \langle \chi, \mu(g) \rangle$.

By Lemma~\ref{lem:tau}, for any weight $\chi\neq\chi_{\tau}$ of~$\tau$, we can write $\chi_{\tau}-\chi=\sum_{\alpha\in\Sigma^+_{\theta}} m_{\alpha}\,\alpha$ where $m_{\alpha}\geq 0$ for all~$\alpha$. 
Using that sums of CLI sequences are again CLI and applying case~\eqref{item:ast} of Lemma~\ref{lem:reordering} to the sequences $x^\chi = ( \langle \chi_{\tau}-\nolinebreak\chi, \mu( \rho( \gamma_n) ) )_{n\in \NN}$ for $\chi$ ranging through the set of weights of~$\tau$ different from \(\chi_\tau\) and 
\[y^{(i)} = ( \langle \varepsilon_1 - \varepsilon_i, \mu_{\GL_d}( \tau \circ \rho( \gamma_n)) \rangle)_{n\in \NN}\]
 for $i$ ranging through $\{ 2,\dots,d\}$, we see that if the sequences $(\langle \alpha, \, \mu( \rho(\gamma_n) )\rangle)_{n\in\NN}$ for $\alpha\in\Sigma^+_{\theta}$ are CLI, then so are the sequences $y^{(i)}$, for all $2\leq i\leq\nolinebreak d$.
By Proposition~\ref{prop:transv-SL}, $\iota^+ \circ\xi^+(\eta)$ and $\iota^-\circ\xi^-(\eta')$ are transverse for all \(\eta'\neq \eta\).
Therefore, $\xi^+(\eta)$ and~$\xi^-(\eta')$ are transverse for all \(\eta'\neq \eta\) by Proposition~\ref{prop:theta_emb}.\eqref{item:b_thtmb}.
\end{proof}

This concludes the proof of Theorem~\ref{thm:constr-xi}, hence also of the implication $\eqref{item:cli} \Rightarrow \eqref{item:Ano}$ of Theorem~\ref{thm:char_ano}.

%%%%%%%%%%%%%%%%%%%%%%%%%
\subsection{Contraction properties for Anosov representations} \label{subsec:proof-contr-Ano}

We now establish the implication $\eqref{item:Ano}\Rightarrow\eqref{item:cli}$ of Theorem~\ref{thm:char_ano}. 

\begin{proposition}\label{prop:Ano-CLI}
Let $\Gamma$ be a word hyperbolic group, $G$ a real reductive Lie group, $\theta \subset \Delta$ a nonempty subset of the simple restricted roots of~$G$, and $\rho : \Gamma \to G$ a $P_{\theta}$-Anosov representation.
Then there exist $\kappa,\kappa'>0$ such that for any $\alpha\in\Sigma^+_{\theta}$ and any geodesic ray $(\gamma_n)_{n\in\NN}$ with $\gamma_0=e$ in the Cayley graph of~$\Gamma$, the sequence $(\langle \alpha, \,\mu( \rho(\gamma_n))\rangle)_{n\in\NN}$ is $(\kappa,\kappa')$-lower CLI.
\end{proposition}

To prove Proposition~\ref{prop:Ano-CLI}, we shall use the following lemma.

\begin{lemma}\label{lem:ellipsoids}
If an ellipsoid $\mathcal{B}$ is contained in an ellipsoid $\mathcal{B}'$ in~$\RR^d$, then for any $1\leq i\leq d$, the $i$-th principal axis of~$\mathcal{B}$ is no longer than the $i$-th principal axis of~$\mathcal{B}'$. 
\end{lemma}

\begin{proof}
It suffices to show that the length $r_i$ of the $i$-th largest axis of $\mathcal{B}$ is the diameter of the largest possible copy of the Euclidean ball of an \(i\)-dimensional subspace (or $i$-ball)  
 of~$\RR^d$ that fits into~$\mathcal{B}$.
Clearly, an $i$-ball of diameter~$r_i$ fits into~$\mathcal{B}$, for instance inside the $i$-dimensional vector space $V_i$ spanned by the $i$ largest axes.
Since the orthogonal $V_{i-1}^\perp$ of $V_{i-1}$ in~$\RR^d$ intersects $\mathcal{B}$ along a $(d-i+1)$-dimensional ellipsoid whose axes are all of length $\leq r_i$, it is also clear that no larger $i$-ball can fit into $\mathcal{B}$, since it would have a diameter in $V_{i-1}^\perp$ by a simple dimension count.
\end{proof}

\begin{proof}[Proof of Proposition~\ref{prop:Ano-CLI}]
By Fact~\ref{fact:Anosov_opp}, we may assume without loss of generality that $\theta=\theta^{\star}$, so that $\Sigma^+_{\theta}=(\Sigma^+_{\theta})^{\star}$.
Let $\tilde{\sigma} : \mathcal{G}_{\Gamma}\to G/L_{\theta}$ be the section 
 associated with the $P_{\theta}$-Anosov representation~$\rho$.
As in Section~\ref{subsec:prelim-result}, we choose a $\rho$-equivariant continuous lift $\tilde{\beta}:\mathcal{G}_\Gamma\rightarrow G/K_\theta$ of~$\tilde{\sigma}$ and a $\rho$-equivariant set-theoretic lift $\breve{\beta}:\mathcal{G}_\Gamma\rightarrow G$ of~$\tilde{\beta}$.
For any $(t,v)\in \RR \times \mathcal{G}_\Gamma$ the element \( l_{t,v} =\breve{\beta}(v)^{-1} \breve{\beta}(\varphi_t \cdot v) \) belongs to $ L_\theta$.

By Proposition~\ref{prop:QIflowbis}, Lemma~\ref{lem:mu_mutheta}, and Proposition~\ref{prop:mu-theta-mu}, there are a compact subset $\mathcal{D}$ of~$\mathcal{G}_{\Gamma}$ and constants $\mathcal{K}, n_0, \kappa_0, \kappa'_0>0$ with the following property: for any geodesic ray $(\gamma_n)_{n\in\NN}$ with $\gamma_0=e$ in the Cayley graph of~$\Gamma$, there exist $v \in \mathcal{D}$ and a $(\kappa_0, \kappa'_0)$-lower CLI sequence $(t_n)\in\RR^{\NN}$  such that for all $n\geq n_0$,
\[ \| \mu(\rho(\gamma_n)) - \mu_\theta(l_{t_n,v})  \| \leq \mathcal{K}. \]
Therefore it is enough to prove the existence of $\kappa,\tilde{\kappa}'>0$ such that for any $\alpha\in\Sigma^+_{\theta}$, any $t,s\in\RR$, and any $v\in\mathcal{D}$,
\begin{equation}\label{eqn:suff-cli}
\bigl \langle \alpha, \, \mu_{\theta}(l_{t+s,v}) - \mu_{\theta}(l_{t,v})\rangle \geq \kappa s - \tilde{\kappa}'.
\end{equation}
In fact, we only need to establish \eqref{eqn:suff-cli} for $t,s\in\NN$.
Indeed, $l_{t,v}$ is uniformly close to~$l_{\lfloor t\rfloor,v}$ for $t\in\RR$ and $v\in\mathcal{D}$, by the cocycle property \eqref{eqn:l-cocycle} of the map $(t,v)\mapsto l_{t,v}$, the $\Gamma$-invariance property \eqref{eqn:l-inv}, and the fact that the set $\{ l_{s,v} \,|\, s\in [0,1],\ v\in\mathcal{D}\}$ is bounded in~$G$ since $\{ \breve\beta(\varphi_s\cdot v) \,|\, s\in [0,1],\ v\in\mathcal{D}\}$ is by Lemma~\ref{lem:mu_mutheta}.

Recall the subbundle $E^+$ of $T(G/L_{\theta})$ from \eqref{eqn:E+E-}.
Since \(\breve{\beta} : \mathcal{G}_{\Gamma}\to G\) is a lift of $\tilde{\sigma} : \mathcal{G}_{\Gamma}\to G/L_{\theta}$, we have \(E^{+}_{\tilde{\sigma}(v)} = \breve{\beta}(v) \cdot T_{x_\theta}(G/P_\theta)\) for all \(v\in\mathcal{G}_\Gamma\), where we write ${\breve{\beta}}(v)$ for the derivative of the action of $\breve{\beta}(v)$ by left translation, and $x_{\theta}=eP_{\theta}\in G/P_{\theta}$.
Fix a $K_{\theta}$-invariant Euclidean norm $\Vert\cdot\Vert_0$ on $T_{x_\theta} (G/P_\theta) \simeq \mathfrak{u}_{\theta}^-$.
For any $v\in\mathcal{G}_{\Gamma}$,
\begin{equation}\label{eqn:def-norm-v}
w \longmapsto \Vert w\Vert_v := \Vert\breve{\beta}(v)^{-1} \cdot w\Vert_0
\end{equation}
defines a Euclidean norm on~$E^{+}_{\tilde{\sigma}(v)}$.
The norm \(\Vert\cdot\Vert_v\) does not depend on the lift
\(\breve{\beta}(v)\in G\) of \(\tilde{\beta}(v)\in G/K_\theta\) since $\Vert\cdot\Vert_0$ is $K_{\theta}$-invariant.
As \(\tilde{\beta}\) is continuous, the family $(\Vert\cdot\Vert_v)_{v\in\mathcal{G}_{\Gamma}}$ is continuous and $\rho$-equivariant.
By Definition~\ref{defi:ano1}.\eqref{item:E+} of a $P_{\theta}$-Anosov representation, there exist $\kappa,\kappa'>0$ such that
\[ \Vert w\Vert_v \leq e^{-\kappa s+\kappa'} \Vert w\Vert_{\varphi_s\cdot v}\]
for all $v\in\mathcal{G}_{\Gamma}$ and all $w\in E^{+}_{\tilde{\sigma}(v)}$.
In other words, for any $v\in\mathcal{G}_{\Gamma}$, the unit ball $\mathcal{B}_{t+s,v}$ of $\Vert\cdot\Vert_{\varphi_{t+s}\cdot v}$ in $E^+_{\tilde{\sigma}(v)}$ is contained in $e^{-\kappa s+\kappa'}$ times the unit ball $\mathcal{B}_{t,v}$ of $\Vert\cdot\Vert_{\varphi_t\cdot v}$, for all $t,s\geq 0$.
We now apply Lemma~\ref{lem:ellipsoids} to $E^+_{\tilde{\sigma}(v)}$ endowed with the Euclidean norm  $\Vert\cdot\Vert_v$: for any $1\leq i\leq\dim G/P_{\theta}$, the length of the $i$-th principal axis of $\mathcal{B}_{t+s,v}$ is at most $e^{-\kappa s+\kappa'}$ times that of the $i$-th principal axis of $\mathcal{B}_{t,v}$. 
By construction, the lengths of the principal axes of $\mathcal{B}_{t,v}$ are the $e^{-\langle\alpha,\, \mu_\theta(l_{t,v}^{-1})\rangle}$ for $\alpha$ ranging through~$\Sigma^+_{\theta}$, or in other words the $e^{-\langle\alpha,\, \mu_\theta(l_{t,v})\rangle}$ for $\alpha$ ranging through $(\Sigma^+_{\theta})^{\star}=\Sigma^+_{\theta}$, see \eqref{eqn:opp-inv-mu}.
Since $ \| \mu_\theta(l_{t+1,v})-\mu_\theta(l_{t,v}) \|$ is bounded, we may apply case~\eqref{item:astast} of Lemma~\ref{lem:reordering}: there exists $\tilde{\kappa}'>0$ such that for any $\alpha\in\Sigma^+_{\theta}$, any $t,s\in\NN$, and any $v\in\mathcal{D}$,
\[ \bigl \langle \alpha, \,\mu_\theta(l_{{t+s},v})  - \, \mu_\theta(l_{t,v}) \rangle \geq \kappa s - \tilde{\kappa}'. \]
Thus \eqref{eqn:suff-cli} holds, which completes the proof.
\end{proof}

%%%%%%%%%%%%%%%%%%%%%%%%%%%%%%%%%%%%%%%%%%%%%%%%%%%
\section{Anosov representations and proper actions}
\label{sec:gen-proper}

In view of the properness criterion of Benoist \cite{Benoist_properness} and Kobayashi \cite{TKobayashi_proper} (see Section~\ref{subsec:intro-Ano-implies-proper}), our characterizations of Anosov representations $\rho: \Gamma \to G$ in terms of the Cartan projection $\mu$ (Theorem~\ref{thm:char_ano}) provide a direct link with the properness of the action of $\Gamma$ (via~$\rho$) on homogeneous spaces of~$G$, and immediately imply Corollary~\ref{cor:link-proper-Ano}.
In this section we illustrate the proofs of Corollaries \ref{cor:Hitchin} and~\ref{cor:max} in some cases of Tables~\ref{table1} and~\ref{table2}, and provide a proof of Corollary~\ref{cor:standard-Ano}.

%%%%%%%%%%%%%%%%%%%%%%%%%
\subsection{Hitchin representations}
\label{sec:hitch-repr}

We first consider Corollary~\ref{cor:Hitchin}, which concerns Hitchin representations into split real semisimple Lie groups~$G$.
The point is that any Hitchin representation \(\rho: \pi_1(\Sigma) \to G\) is \(P_\Delta\)-Anosov \cite{Labourie_anosov}, \cite[Th.\,1.15]{Fock_Goncharov}.
Thus, in order to deduce Corollary~\ref{cor:Hitchin} from Corollary~\ref{cor:link-proper-Ano}, it is sufficient to check that \(\mu(H) \subset \bigcup_{\alpha\in \Delta} \Ker(\alpha)\) for every pair \((G,H)\) in Table~\ref{table1}.
Let $K_H$ be a maximal compact subgroup of~$H$ and $\overline{\aaa}_H^+$ a closed Weyl chamber in a Cartan subspace $\aaa_H$ of~$\h$ such that $H=K_H(\exp\overline{\aaa}_H^+)K_H$.
We may assume that the maximal compact subgroup \(K\) of \(G\) contains \(K_H\) and that the Cartan subspace \(\aaa\) of~$\g$ contains~\(\aaa_H\) \cite{Karpelevich,Mostow_dec}.
Let $\mu : G\to\overline{\aaa}^+$ and $\mu_H : H\to\overline{\aaa}_H^+$ be the corresponding Cartan projections.

\begin{lemma} \label{lem:aH-a-Hitchin}
In order to prove that $\pi_1(\Sigma)$ acts sharply on $G/H$ for any Hitchin representation $\rho: \pi_1(\Sigma) \to G$, it is sufficient to verify that
\[ \aaa_H \subset \bigcup_{\alpha\in\Sigma} \Ker(\alpha) \subset \aaa. \]
\end{lemma}

\begin{proof}
For any $h\in H$, the element $\mu(h)\in\overline{\aaa}^+$ is the unique $W$-translate of $\mu_H(h)\in\nolinebreak\overline{\aaa}_H^+$ contained in~$\overline{\aaa}^+$.
Thus, if $\aaa_H\subset\bigcup_{\alpha\in\Sigma} \Ker(\alpha)$, then $\mu(H)\subset \overline{\aaa}^+\cap\bigcup_{\alpha\in\Sigma} \Ker(\alpha) = \overline{\aaa}^+\cap\bigcup_{\alpha\in \Delta} \Ker(\alpha)$.
We conclude using Corollary~\ref{cor:link-proper-Ano}.
\end{proof}

For the pairs \((G,H)\) in Table~\ref{table1}, the embedding \(H\hookrightarrow G\) is the natural one, \ie \(H\) is the stabilizer of a decomposition of the standard \(G\)-module or of an endomorphism of~it.
For instance, in case~(iii), the identification \(\CC^d \simeq \RR^{2d}\) gives an embedding \(j: \GL_d(\CC) \to \GL_{2d}(\RR)\) which restricts to an embedding \(\SL_d(\CC) \times \U(1) \hookrightarrow \SL_{2d}(\RR)\) since \(\det(j(A)) = |\det(A)|^2\) for any \(A\in \GL_d(\CC)\).

As an example, let us exhibit $\aaa_H\subset\aaa$ as in Lemma~\ref{lem:aH-a-Hitchin} for case~(vi) of Table~\ref{table1}, where \((G,H)= (\SO(d,d), \GL_k(\RR))\).
Let 
\[\aaa = \{ \mathrm{diag}(a_1, \dots, a_d, -a_1, \dots, -a_d) \mid a_1, \dots , a_d\in \RR \}\]
be the set of diagonal matrices contained in~\(\g\), where $G=\SO(d,d)$ is defined by the quadratic form $x_1x_{d+1} + \dots + x_dx_{2d}$.
Let
\[ \aaa_H = \{ \mathrm{diag}(b_1, \dots, b_k) \mid b_1, \dots, b_k \in \RR\} \]
be the set of diagonal matrices contained in \(\h=\gl_k(\RR)\).
The action of \(\GL_d(\RR)\) on the vector space \(\RR^d \oplus (\RR^d)^*\) preserves the natural pairing between \(\RR^d\) and \((\RR^d)^*\), which is a symmetric bilinear form of signature \((d,d)\).
This defines an embedding $\GL_d(\RR)\hookrightarrow G=\SO(d,d)$; composing it with the natural inclusion \(H=\GL_k(\RR)\hookrightarrow \GL_d(\RR)\) gives the embedding \(H\hookrightarrow G\).
The corresponding embedding \(\aaa_H\hookrightarrow\aaa\) is given by
\[ \mathrm{diag}(b_1, \dots, b_k) \longmapsto \mathrm{diag}( b_1, \dots, b_k, 0, \dots, 0, -b_1, \dots, -b_k, 0,\dots, 0). \]
The pair \((\aaa_H,\aaa)\) satisfies the condition of Lemma~\ref{lem:aH-a-Hitchin} since the restricted root
\[ \alpha : \mathrm{diag}(a_1, \dots, a_d, -a_1, ,\dots, -a_d) \longmapsto a_{d-1}-a_d \]
of \(\aaa\) in~\(G\) is zero on~\(\aaa_H\) under the condition \(k<d-1\).

\begin{remark}
In all examples of Table~\ref{table1} except (i), (vi), and~(vii), the homogeneous space $G/H$ is an affine symmetric space.
We obtained these examples by checking the condition of Lemma~\ref{lem:aH-a-Hitchin} in Okuda's classification \cite[App.\,A]{Okuda13} of affine symmetric spaces admitting a properly discontinuous action by a discrete group which is not virtually abelian.
Examples (i), (vi), and~(vii) of Table~\ref{table1} were obtained by checking the condition of Lemma~\ref{lem:aH-a-Hitchin} in a list of examples of nonsymmetric $G/H$ admitting properly discontinuous actions by nonvirtually abelian groups given in \cite{Bochenski_Jastrzebski_Tralle}.
\end{remark}

%%%%%%%%%%%%%%%%%%%%%%%%%
\subsection{Maximal representations}
\label{sec:maxim-repr}

We now address Corollary~\ref{cor:max}. 
Let \(G\) be a simple Lie group of Hermitian type, of real rank $d\geq 1$.
Then the restricted root system of~\(G\) is of type \(C_d\) or \(BC_d\).
This means that there is a system of simple restricted roots \(\Delta =\{ \alpha_1, \dots, \alpha_d\}\) such that \((\alpha_i,
\alpha_j)\neq 0\) \(\Leftrightarrow\) \(|i-j|\leq 1\) and
\(\|\alpha_d\| > \|\alpha_1\| = \cdots = \|\alpha_{d-1}\|\).
The restricted root system is of type \(BC_d\) if and only if \(2\alpha_1\) is a root.

Let $\Sigma$ be a closed hyperbolic surface.
The point is that any maximal representation \(\rho: \pi_1(\Sigma) \to G\) is \(P_{\alpha_d}\)-Anosov \cite{Burger_Iozzi_Labourie_Wienhard, Burger_Iozzi_Wienhard_anosov}.
Similarly to Lemma~\ref{lem:aH-a-Hitchin}, we thus only need to check the following for each pair $(G,H)$ of Table~\ref{table2}. As above we assume that the maximal compact subgroup \(K_H\) of \(H\) is contained in \(K\) and that the Cartan subspace \(\aaa_H\) of \(H\) is contained in \(\aaa\).

\begin{lemma} \label{lem:aH-a-max}
In order to prove that $\pi_1(\Sigma)$ acts sharply on $G/H$ for any maximal representation $\rho: \pi_1(\Sigma) \to G$, it is sufficient to check that 
\[ \aaa_H \subset \bigcup_{w\in W} \Ker(w\cdot\alpha_d) = \bigcup_{1\leq i\leq d} \Ker(\alpha_i + \dots + \alpha_d) \subset \aaa. \]
\end{lemma}

As an example, let us exhibit $\aaa_H\subset\aaa$ as in Lemma~\ref{lem:aH-a-max} for case~(ii) in Table~\ref{table2}, where $(G,H)=(\Sp(2d,\RR),\U(d-k,k)$.
By symmetry, we may assume \(k<d-k\).
Let
\[\aaa = \{ \mathrm{diag}(a_1, \dots, a_d, -a_1, \dots, -a_d) \mid a_1, \dots , a_d\in \RR \} \]
be the set of diagonal matrices contained in~\(\g\), where $G=\Sp(2d,\RR)$ is defined by the symplectic form $x_1\wedge x_{d+1} + \dots + x_d\wedge x_{2d}$.
Let
\[ \aaa_H =\{ \mathrm{diag}( b_1, \dots, b_k, 0, \dots, 0, -b_k, \dots, -b_1) \mid b_1, \dots, b_k \in \RR\} \]
be the set of real diagonal matrices contained in \(\h\), where $H=\U(d-k,k)$ is defined by the Hermitian form $z_1\overline{z_d} + \dots + z_k\overline{z_{d-k+1}} + |z_{k+1}|^2 + \dots + |z_{d-k}|^2$.
The embedding of $H$ into~$G$ is given by identifying $\CC^d$ with~$\RR^{2d}$ and observing that the imaginary part of the Hermitian form above is symplectic.
The corresponding embedding \(\aaa_H\hookrightarrow\aaa\) is given by
\begin{multline*}
 \mathrm{diag}( b_1, \dots, b_k, 0, \dots, 0, -b_k, \dots, -b_1)\\
 \longmapsto  \mathrm{diag}( b_1, \dots, b_k, 0, \dots,0, -b_k, \dots, -b_1, -b_1, \dots, -b_k, 0, \dots,0, b_k, \dots, b_1).
\end{multline*}
Since \(k<d-k\), the restricted root
\[ \alpha_{k+1}+\dots+\alpha_d : \mathrm{diag}(a_1, \dots, a_d, -a_1, \dots, -a_d) \longmapsto a_{k+1} \] 
is zero on~\(\aaa_H\), and so the condition of Lemma~\ref{lem:aH-a-max} is satisfied.

%%%%%%%%%%%%%%%%%%%%%%%%%
\subsection{Proper actions of rank-$1$ reductive groups yield Anosov representations} \label{subsec:proof-cor-intro}

In order to prove Corollary~\ref{cor:standard-Ano}, we first establish the following easy consequence of Theorem~\ref{thm:char_ano}; see also \cite[Prop.\,3.1]{Labourie_anosov} and \cite[Prop.\,4.7]{Guichard_Wienhard_DoD} for similar criteria. 

\begin{corollary} \label{cor:cc-G1-Ano}
Let $G$ be a real reductive Lie group, $\theta \subset \Delta$ a nonempty subset of the simple restricted roots of $\aaa$ in~$G$, and $G_1$ a reductive subgroup of~$G$ of real rank~$1$.
Suppose that $\aaa\cap\g_1$ is a Cartan subspace of~$\g_1$ satisfying $(\aaa\cap\g_1)\cap\nolinebreak\Ker(w\cdot\nolinebreak\alpha)=\nolinebreak\{0\}$ for all $w\in W$ and $\alpha\in\theta$.
Then for any convex cocompact subgroup $\Gamma$ of~$G_1$, the inclusion of $\Gamma$ into~$G$ is $P_{\theta}$-Anosov.
\end{corollary}

\begin{example}
Let $\Gamma$ be a convex cocompact subgroup of $\SL_2(\RR)$ (for instance the image of a Fuchsian representation $\pi_1(\Sigma)\to\SL_2(\RR)$ for a closed hyperbolic surface~$\Sigma$).
For $k\geq 1$, let $\iota_k : \SL_2(\RR)\hookrightarrow\SL_2(\RR)^k$ be the diagonal embedding.
For $d\geq 2k$, if we see $\SL_2(\RR)^k$ as a subgroup of $G=\SL_d(\RR)$ by embedding it in the upper left corner of $\SL_d(\RR)$, then $\iota_k : \Gamma\to G$ is $P_{\varepsilon_k-\varepsilon_{k+1}}$-Anosov.
\end{example}

\begin{proof}[Proof of Corollary~\ref{cor:cc-G1-Ano}]
Up to conjugation, we may assume that $G_1$ admits the decom\-position
\begin{equation} \label{eqn:Cartan-G1}
G_1 = (K\cap G_1) (\exp(\aaa)\cap G_1) (K\cap G_1) .
\end{equation}
The set $\mu(G_1)=\mu(\exp(\aaa)\cap G_1)=\overline{\aaa}^+\cap W\cdot (\aaa\cap\g_1)$ is then a union of two (possibly equal) rays $\mathcal{L}_1,\mathcal{L}_2$ starting at~$0$.
By assumption, for any $i\in\{ 1,2\}$ and $\alpha\in\theta$ we have $\mathcal{L}_i\cap\Ker(\alpha)=\{ 0\}$; this is also true for any $\alpha\in\Sigma^+_{\theta}\smallsetminus\theta$ since $\overline{\aaa}^+\cap\Ker(\alpha)=\{ 0\}$.
Therefore, for any $\alpha\in\Sigma^+_{\theta}$ there is a constant $C_{\alpha,i}>0$ such that $\langle\alpha,Y\rangle = C_{\alpha,i}\,\Vert Y\Vert$ for all $Y\in\mathcal{L}_i$.

Let $\Gamma$ be a convex cocompact subgroup of~$G_1$.
Let us prove that the natural inclusion of $\Gamma$ in~$G$ is $P_{\theta}$-Anosov.

We first assume that $G_1$ is semisimple of real rank~$1$, so that there is only one ray \(\mathcal{L}_1=\mathcal{L}_2\).
Let $\mu_{G_1}$ be a Cartan projection for~$G_1$ associated with the decomposition \eqref{eqn:Cartan-G1}.
Since $G_1$ has real rank~$1$, we may see $\mu_{G_1}$ as a map from $G_1$ to~$\RRp$, and there is a constant $C>0$ such that $\Vert\mu(g)\Vert = C \, \mu_{G_1}(g)$ for all $g\in G_1$; thus, for all $g\in G_1$,
\begin{equation}\label{eqn:mu-G1}
\langle\alpha, \mu(g)\rangle = C_{\alpha,1} C \, \mu_{G_1}(g).
\end{equation}
Since $\Gamma$ is convex cocompact in~$G_1$, its inclusion in~$G_1$ is Anosov with respect to a minimal parabolic subgroup of~$G_1$ (Remark~\ref{rem:rank1ccquasi}).
By Theorem~\ref{thm:char_ano} (or Corollary~\ref{cor:CLI-rk1}), for any geodesic ray $(\gamma_n)_{n\in\NN}$ in the Cayley graph of~$\Gamma$, the sequence $(\mu_{G_1}(\gamma_n))_{n\in\NN}$ is CLI, and the CLI constants are uniform over all geodesic rays $(\gamma_n)_{n\in\NN}$ with $\gamma_0=e$.
By \eqref{eqn:mu-G1}, the sequence $(\langle\alpha,\mu(\gamma_n)\rangle)_{n\in\NN}$ is CLI as well for all $\alpha\in\Sigma^+_{\theta}$, with uniform CLI constants, and so the inclusion of $\Gamma$ in~$G$ is $P_{\theta}$-Anosov by Theorem~\ref{thm:char_ano}.

We now assume that $G_1$ is not semisimple.
It is then a central extension of the compact group $K\cap G_1$ by the one-parameter group $\exp(\aaa)\cap G_1$, and $\Gamma$ is virtually the cyclic group generated by some element $\gamma\in G_1 \smallsetminus K$.
Let $\mu_{G_1}$ be a Cartan projection for~$G_1$ associated with the decomposition \eqref{eqn:Cartan-G1}.
We may see $\mu_{G_1}$ as a map from $G_1$ to~$\RR$, and there is a constant $C>0$ such that $\Vert\mu(\gamma^n)\Vert = C \, |n| \, |\mu_{G_1}(\gamma)|$ for all $n\in\ZZ$.
For any $\alpha\in\Sigma^+_{\theta}$, the estimates $\langle\alpha,Y\rangle = C_{\alpha,i}\,\Vert Y\Vert$ for $Y\in \mathcal{L}_i$ imply that the sequences $(\langle \alpha, \mu(\gamma^n)\rangle)_{n\in\NN}$ and $(\langle \alpha, \mu(\gamma^{-n})\rangle)_{n\in\NN}$ are CLI, and so the inclusion of $\Gamma$ in~$G$ is $P_{\theta}$-Anosov by Theorem~\ref{thm:char_ano}.
\end{proof}

\begin{proof}[Proof of Corollary~\ref{cor:standard-Ano}]
Up to conjugation, we may assume that $G_1$ admits the decomposition \eqref{eqn:Cartan-G1}.
Then $\aaa\cap\g_1$ is a Cartan subspace of~$\g_1$ and $\mu(G_1)=\mu(\exp(\aaa)\cap\nolinebreak G_1)=\overline{\aaa}^+\cap W\cdot (\aaa\cap\g_1)$.
Since $G_1$ acts properly on $G/H$ and $\mu(H)\supset \bigcup_{\alpha\in\theta} \Ker(\alpha) \cap \overline{\aaa}^+$ by assumption, the properness criterion of Benoist and Kobayashi of Section~\ref{subsec:intro-Ano-implies-proper} (see also the earlier paper \cite{TKobayashi}) implies that $(\aaa\cap\g_1)\cap\nolinebreak\Ker(w\cdot\nolinebreak\alpha)=\nolinebreak\{0\}$ for all $w\in W$ and $\alpha\in\theta$.
By Corollary~\ref{cor:cc-G1-Ano}, for any convex cocompact subgroup $\Gamma$ of~$G_1$, the natural inclusion of $\Gamma$ in~$G$ is $P_{\theta}$-Anosov.
\end{proof}

%%%%%%%%%%%%%%%%%%%%%%%%%%%%%%%%%%%%%%%%%%%%%%%%%%%
\section{Proper actions on group manifolds}
\label{sec:proper}

In this section, we consider properly discontinuous actions on group
manifolds and deduce Theorems~\ref{thm:proper-GxG-rk1}  and~\ref{thm:Anosov=proper} and Corollary~\ref{cor:proper-open-GxG} from Theorems \ref{thm:char_ano} and~\ref{thm:char_ano_lambda_intro}.

%%%%%%%%%%%%%%%%%%%%%%%%%
\subsection{Proper actions on group manifolds, uniform domination, and Anosov representations}

Before stating our main Theorem~\ref{thm:complete-proper-GxG}, we introduce some useful notation and terminology.

%%%%%%%%
\subsubsection{Uniform $P_{\theta}$-domination}

Let $\Gamma$ be a discrete group, $G$ a real reductive Lie group, and $\theta\subset\Delta$ a nonempty subset of the simple restricted roots of $\aaa$ in~$G$.
Recall that $\omega_\alpha$ denotes the corresponding fundamental weight \eqref{eqn:9} of a simple root $\alpha\in\Delta$.
We adopt the following terminology.

\begin{definition} \label{def:dominate}
A representation $\rho_L\in\Hom(\Gamma,G)$ \emph{uniformly $P_{\theta}$-dominates} a representation $\rho_R\in\Hom(\Gamma,G)$ if there exists $c<1$ such that for all $\alpha\in\theta$ and $\gamma\in\Gamma$,
\[ \langle\omega_{\alpha},\lambda(\rho_R(\gamma))\rangle \leq c \, \langle\omega_{\alpha},\lambda(\rho_L(\gamma))\rangle. \]
  If $G$ has real rank~$1$, then $\theta=\Delta$ is a singleton and we say simply that $\rho_L$ \emph{uniformly dominates}~$\rho_R$.
\end{definition}
  
Note that uniform $P_{\theta}$-domination is equivalent to uniform $P_{\theta\cup\theta^{\star}}$-domination by \eqref{eqn:opp-inv-lambda}.
We shall use uniform $P$-domination both in~$G$ and in the setting of the next paragraph.

%%%%%%%%
\subsubsection{Automorphism groups of bilinear forms} \label{subsubsec:Aut(b)}

Let $\KK$ be $\RR$, $\CC$, or the ring $\bH$ of quaternions.
Let $V$ be a $\KK$-vector space (where $\KK$ acts on the right in the case of~$\bH$), and let $b : V \otimes_{\RR} V \to \KK$ be a nondegenerate $\RR$-bilinear form which is symmetric or antisymmetric (if $\KK=\RR$ or~$\CC$), or Hermitian or anti-Hermitian (if $\KK=\CC$ or~$\bH$).
Let $\Aut_{\KK}(b)$ be the subgroup of $\GL_{\KK}(V)$ preserving~$b$; we shall always assume it is noncompact.
Table~\ref{table4} gives a list of all possible examples.

\begin{table}[h!]
\centering
\begin{tabular}{|c|c|c|c|c|}
\hline
$\Aut_{\KK}(b)$ & $\KK$ & $\dim_{\KK}(V)$ &
Description of $b$ \tabularnewline
\hline
$\OO(p,q)$ & $\RR$& $p+q$ &
symmetric \tabularnewline
$\U(p,q)$ & $\CC$& $p+q$ &
Hermitian \tabularnewline
$\Sp(p,q)$ & $\bH$ & $p+q$ &
Hermitian \tabularnewline
$\OO(d,\CC)$ & $\CC$ & $d$ &
symmetric \tabularnewline
$\Sp(2d,\RR)$ & $\RR$ & $2d$ & antisymmetric \tabularnewline
$\Sp(2d,\CC)$ & $\CC$ & $2d$ & antisymmetric \tabularnewline
$\OO^*(2d)$ & $\bH$ & $2d$ & anti-Hermitian \tabularnewline
\hline
\end{tabular}
\vspace{0.2cm}
\caption{In these examples, $p,q,d$ are any integers $\geq 1$.}
\label{table4}
\end{table}

We denote by $Q_0(b) \subset \Aut_{\KK}(b)$ the stabilizer of a $b$-isotropic line of~$V$.
It is a maximal proper parabolic subgroup, 
 and $\mathcal{F}_0(b) := \Aut_{\KK}(b)/Q_0(b)$ identifies with the set of $b$-isotropic $\KK$-lines of~$V$ inside the projective space $\PP_{\KK}(V) = (V\smallsetminus\{ 0\})/\KK^{\ast}$.
As in Section~\ref{subsec:theta-compat}, we say that a representation $\tau : G\to\Aut_{\KK}(b)$ is \emph{proximal} if some element of $\tau(G)$ has an attracting fixed point in $\PP_{\KK}(V)$.

\begin{remark} \label{rem:Ano-nonconnected}
We have assumed since Section~\ref{subsec:some-structure-semi} that the reductive group~$G$ acts on~\(\g\) by inner automorphisms, to simplify the description of the parabolic subgroups of~$G$ and ensure that the Cartan projection $\mu : G\to \overline{\aaa}^+$ is well defined.
However, some of the groups $\Aut_{\KK}(b)$ in Table~\ref{table4}, namely $\OO(p,q)$ for $p=q$ and \(\OO(d,\CC)\) for even~\(d\), have elements acting on~\(\g\) via outer automorphisms.
For these groups $G=\Aut_{\KK}(b)$, we can still define notions of
\begin{itemize}
  \item $P_{\theta}$-Anosov representation $\rho : \Gamma\to G$;
  \item uniform $P_{\theta}$-domination of $\rho_R : \Gamma\to G$ by $\rho_L : \Gamma\to G$;
  \item sharpness (see \eqref{eqn:sharp}) for the action of $\Gamma$ via $\rho : \Gamma\to G$ on some homogeneous space $G/H$.
\end{itemize}
Indeed, for any representation $\rho : \Gamma\to G$ there is a finite-index subgroup $\Gamma'$ of~$\Gamma$ such that $\rho(\Gamma')$ is contained in the identity component of~$G$, which acts on~\(\g\) by inner automorphisms.
We define the properties above to hold when they hold for the restriction to~$\Gamma'$.
This does not depend on the finite-index subgroup~$\Gamma'$: see property~\eqref{item:f-i-subgroup} of Section~\ref{subsubsec:examples-properties} for Anosov representations, and Fact~\ref{fact:mu-subadditive} for sharpness.
\end{remark}

%%%%%%%%
\subsubsection{A useful normalization}
\label{sec:useful-normalization}

We shall use the following normalization to avoid having to switch \(\rho_L\) and~\(\rho_R\) in Theorem~\ref{thm:complete-proper-GxG}.

Let $\tau : G\to\Aut_{\KK}(b)$ be an irreducible representation of~$G$ with highest weight~$\chi_{\tau}$.
For any $\rho_L,\rho_R\in\Hom(\Gamma,G)$, the following always holds up to switching $\rho_L$ and~$\rho_R$:
\begin{equation} \label{eqn:tautology}
\sup_{\substack{\gamma\in\Gamma\\ \text{of infinite order}}}\ \langle\chi_{\tau},\lambda(\rho_L(\gamma)) - \lambda(\rho_R(\gamma))\rangle \geq 0.
\end{equation}

If $G$ is semisimple of real rank~$1$, then $\theta=\Delta$ is a singleton $\{\alpha\}$ and (assuming $\tau$ to be nonzero) the inequality \eqref{eqn:tautology} is equivalent to
\[ \sup_{\substack{\gamma\in\Gamma\\ \text{of infinite order}}}\ \langle\omega_{\alpha},\lambda(\rho_L(\gamma)) - \lambda(\rho_R(\gamma))\rangle \geq 0. \]
For $G$ of arbitrary real rank, \eqref{eqn:tautology} always holds when $\rho_L$ uniformly $P_{\theta}$-dominates~$\rho_R$ and \(\tau\) is \(\theta\)-compatible (Definition~\ref{defi:theta-compatible}).

%%%%%%%%
\subsubsection{The main theorem}

Here is the main result of this section. 

\begin{theorem} \label{thm:complete-proper-GxG}
Let $\Gamma$ be a discrete group, $G$ a real reductive Lie group, and $\theta\subset\Delta$ a nonempty subset of the simple restricted roots of~$G$, with $\theta = \theta^\star$.
For $\KK = \RR$, $\CC$, or $\bH$, let $V$ be a $\KK$-vector space and $\tau : G\to\GL_{\KK}(V)$ an irreducible, $\theta$-proximal representation preserving a nondegenerate $\RR$-bilinear form $b : V \otimes_{\RR} V \to \KK$ which is symmetric, antisymmetric, Hermitian, or anti-Hermitian over~$\KK$.
Let \(b'\) be a nonzero real multiple of \(b\).
For a pair $(\rho_L,\rho_R)\in\Hom(\Gamma,G)^2$ with the normalization \eqref{eqn:tautology}, consider the following conditions:
\begin{enumerate}
  \item\label{item:p1} $\Gamma$ is word hyperbolic and $\rho_L$ is $P_\theta$-Anosov and uniformly $P_\theta$-dominates~$\rho_R$;
  \item\label{item:p2} $\Gamma$ is word hyperbolic, $\rho_L$ is $P_\theta$-Anosov, and $\tau\circ\rho_L : \Gamma\to\Aut_{\KK}(b)$ uniformly $Q_0(b)$-dominates $\tau\circ\rho_R$;
  \item\label{item:p3} $\Gamma$ is word hyperbolic and $\tau\circ\rho_L:\Gamma \to \Aut_{\KK}(b)$ is $Q_0(b)$-Anosov and uniformly $Q_0(b)$-dominates $\tau\circ\rho_R:\Gamma \to \Aut_{\KK}(b)$;
  \item\label{item:p4} $\Gamma$ is word hyperbolic and $\tau\circ\rho_L \oplus \tau\circ\rho_R : \Gamma\to\Aut_{\KK}(b\oplus b')$ is $Q_0(b\oplus b')$-Anosov;
  \item\label{item:p7} $\rho_L$ is a quasi-isometric embedding and the action of $\Gamma$ on $(G\times\nolinebreak G)/\Diag(G)$ via $(\rho_L,\rho_R)$ is sharp (see~\eqref{eqn:sharp}); 
  \item\label{item:p8} $\rho_L$ is a quasi-isometric embedding and the action of $\Gamma$ on $(G\times\nolinebreak G)/\Diag(G)$ via $(\rho_L,\rho_R)$ is properly discontinuous;
  \item\label{item:p9} $(\rho_L,\rho_R) : \Gamma\to G\times G$ is a quasi-isometric embedding and the action of $\Gamma$ on $(G\times G)/\Diag(G)$ via $(\rho_L,\rho_R)$ is properly discontinuous.
\end{enumerate}

The following implications always hold:
\[ \eqref{item:p1} \Longrightarrow 
\eqref{item:p2} \Longleftrightarrow 
\eqref{item:p3} \Longleftrightarrow 
\eqref{item:p4} \Longrightarrow  
\eqref{item:p7} \Longrightarrow 
\eqref{item:p8} \Longrightarrow
\eqref{item:p9}. \]
If $\theta$ is a singleton, then $\eqref{item:p2} \Rightarrow \eqref{item:p1}$ holds as well.
If $G$ has real rank~$1$, then all conditions are equivalent.
\end{theorem}

\begin{remarks}
\begin{enumerate}[label=(\alph*),ref=\alph*]
  \item An irreducible, $\theta$-proximal representation $\tau : G\to\Aut_{\KK}(b)$ as in conditions \eqref{item:p2}, \eqref{item:p3}, \eqref{item:p4} always exists when $\theta = \theta^\star$: see Proposition~\ref{prop:exists-tau-O(b)} below.
  On the other hand, conditions \eqref{item:p1}, \eqref{item:p7}, \eqref{item:p8}, \eqref{item:p9} do not involve~$\tau$.
  \item The condition $\theta = \theta^{\star}$ is not restrictive: see Fact~\ref{fact:Anosov_opp}.
  \item In condition~\eqref{item:p4} there are essentially two choices for~$b'$: either \(b'=b\) or \(b'=-b\).
The groups \(\Aut_\KK(b\oplus b)\) and \(\Aut_\KK(b\oplus (-b))\) both sit inside \(\GL_\KK(V\oplus V)\) and their intersection is \(\Aut_\KK(b) \times \Aut_\KK(b)\).
These groups are not isomorphic in general: see Example~\ref{ex:SO(1,d)}.
Even when they are isomorphic, the corresponding two embeddings of $\Gamma$ via $(\rho_L,\rho_R)$ tend to be of a quite different nature: this is the case for instance when \(b\) is a real symplectic form.
\end{enumerate}
\end{remarks}

\begin{example} \label{ex:SO(1,d)}
For $G=\OO(1,d)$, which has real rank~$1$, the set $\theta$ is necessarily equal to~$\Delta$, which is a singleton.
We may take $\Aut_{\KK}(b)$ to be~$G$ and $\tau : G\to\Aut_{\KK}(b)$ to be the identity map.
Theorem~\ref{thm:complete-proper-GxG} then states the equivalence of the following conditions, for a discrete group $\Gamma$ and a pair $(\rho_L,\rho_R)\in\Hom(\Gamma,G)^2$ of representations:
\begin{itemize}
  \item[\eqref{item:p1}] $\Gamma$ is word hyperbolic and one of the representations $\rho_L$ or~$\rho_R$ is convex cocompact and uniformly dominates the other;
  \item[(\ref{item:p4}\(^+\))] $\Gamma$ is word hyperbolic and $\rho_L \oplus \rho_R : \Gamma\to\OO(2,2d)$ is Anosov with respect to the stabilizer of an isotropic line in~$\RR^{2,2d}$;
  \item[(\ref{item:p4}\(^-\))] $\Gamma$ is word hyperbolic and $\rho_L \oplus \rho_R : \Gamma\to\OO(d+1,d+1)$ is Anosov with respect to the stabilizer of an isotropic line in~$\RR^{d+1,d+1}$;
  \item[\eqref{item:p7}] one of the representations $\rho_L$ or~$\rho_R$ is a quasi-isometric embedding and the action of $\Gamma$ on $(G\times G)/\Diag(G)$ via $(\rho_L,\rho_R)$ is sharp;
  \item[\eqref{item:p8}] one of the representations $\rho_L$ or~$\rho_R$ is a quasi-isometric embedding and the action of $\Gamma$ on $(G\times G)/\Diag(G)$ via $(\rho_L,\rho_R)$ is properly discontinuous;
  \item[\eqref{item:p9}] $(\rho_L,\rho_R) : \Gamma\to G\times G$ is a quasi-isometric embedding and the action of $\Gamma$ on $(G\times G)/\Diag(G)$ via $(\rho_L,\rho_R)$ is properly discontinuous.
\end{itemize}
Here we see $\OO(2,2d)$ (\resp $\OO(d+1,d+1)$) as the stabilizer in $\GL_{2d+2}(\RR)$ of the quadratic form $x_0^2 - x_1^2 - \dots - x_d^2 + y_0^2 - y_1^2 - \dots - y_d^2$ (\resp $x_0^2 - x_1^2 - \dots - x_d^2 - y_0^2 + y_1^2 + \dots + y_d^2$).
Similar equivalences are true after replacing
\[ \big(\OO(1,d), \OO(2,2d), \OO(d+1,d+1), \RR^{2d+2}\big) \]
with $(\U(1,d), \U(2,2d), \U(d+1,d+1), \CC^{2d+2})$ or with $(\Sp(1,d), \Sp(2,2d), \Sp(d+\nolinebreak 1,d+\nolinebreak 1),\linebreak \bH^{2d+2})$, or after taking compact extensions of these groups.
\end{example}

We refer to \cite{Goldman_NonStddLortz, Ghys_HomoSL2C, TKobayashi_deformation, Salein, Kassel_these, Gueritaud_Kassel, Gueritaud_Kassel_Wolff, Deroin_Tholozan, Tholozan_domin, Danciger_Gueritaud_Kassel} for examples of discrete subgroups of $\OO(1,d)\times\OO(1,d)$ satisfying the equivalent conditions of Example~\ref{ex:SO(1,d)}.

Theorems \ref{thm:proper-GxG-rk1} and~\ref{thm:Anosov=proper} are contained in Theorem~\ref{thm:complete-proper-GxG}: namely, conditions \eqref{item:proper-GxG-rk1}, \eqref{item:sharp-GxG-rk1}, \eqref{item:Gamma-j-rho} of Theorem~\ref{thm:proper-GxG-rk1} correspond to conditions \eqref{item:p9}, \eqref{item:p7}, \eqref{item:p1} of Theorem~\ref{thm:complete-proper-GxG} (see Example~\ref{ex:SO(1,d)} for classical~$G$), while conditions \eqref{item:Gamma-j-rho-Ano}, \eqref{item:Q0bb-Ano}, \eqref{item:Q0b-b-Ano} of Theorem~\ref{thm:Anosov=proper} correspond to conditions \eqref{item:p1}, \eqref{item:p4} with $(\tau,b')=(\mathrm{id},b)$, and \eqref{item:p4} with $(\tau,b')=(\mathrm{id},-b)$ of Theorem~\ref{thm:complete-proper-GxG}. 

\begin{remark} \label{rem:GxG-false-higher-rk}
When $G$ has higher real rank, the implication $\eqref{item:p7} \Rightarrow \eqref{item:p4}$ of Theorem~\ref{thm:complete-proper-GxG} is false.
Indeed, if $\Gamma$ is quasi-isometrically embedded in $G\times G$ and acts sharply on $(G\times G)/\Diag(G)$, it does not need to be word hyperbolic: for instance, for $G = \OO(2,2d)$, any discrete subgroup of $\OO(1,2d)\times\U(1,d) \subset G\times G$ acts sharply on $(G\times G)/\Diag(G)$.
The implication is actually false even if we assume $\Gamma$ to be word hyperbolic: for instance, take $\rho_L$ to be any quasi-isometric embedding which is not $P_{\theta}$-Anosov (see \eg Appendix~\ref{app:discr-quasi-isom}) and $\rho_R$ to be the constant representation.
\end{remark}

%%%%%%%%
\subsubsection{A complement to the main theorem}

In Theorem~\ref{thm:complete-proper-GxG}, we may replace the notion of Anosov representation into $\Aut_{\KK}(b)$ with that of Anosov representation into $\GL_{\KK}(V)$, as follows.

\begin{proposition} \label{prop:complement-GxG}
In the setting of Theorem~\ref{thm:complete-proper-GxG}, let $P_{\varepsilon_1-\varepsilon_2}(V)$ be the stabilizer in $\GL_{\KK}(V)$ of a line of~$V$ and $P_{\varepsilon_1-\varepsilon_2}(V\oplus V)$ the stabilizer in $\GL_{\KK}(V\oplus\nolinebreak V)$ of a line of $V\oplus V$.
Condition~\eqref{item:p3} of Theorem~\ref{thm:complete-proper-GxG} is equivalent to
\begin{enumerate}\renewcommand{\theenumi}{\ref{item:p3}'}\renewcommand{\labelenumi}{(\ref{item:p3}')}
  \item\label{item:p3prime} $\Gamma$ is word hyperbolic and $\tau\circ\rho_L: \Gamma \to \GL_\KK(V) $ is $P_{\varepsilon_1-\varepsilon_2}(V)$-Anosov and uniformly $P_{\varepsilon_1-\varepsilon_2}(V)$-dominates $\tau\circ\rho_R$.
\end{enumerate}
Condition~\eqref{item:p4} of Theorem~\ref{thm:complete-proper-GxG} is equivalent to
\begin{enumerate}\renewcommand{\theenumi}{\ref{item:p4}'}\renewcommand{\labelenumi}{(\ref{item:p4}')}
  \item\label{item:p4prime} $\Gamma$ is word hyperbolic and $\tau\circ\rho_L \oplus \tau\circ\rho_R : \Gamma\to \GL_{\KK}(V\oplus V)$ is $P_{\varepsilon_1-\varepsilon_2}(V\oplus V)$-Anosov.
\end{enumerate}
\end{proposition}

%%%%%%%%%%%%%%%%%%%%%%%%%
\subsection{Linear representations into automorphism groups of bilinear forms}
\label{subsec:orth-repr}

Before proving Theorem~\ref{thm:complete-proper-GxG} and Proposition~\ref{prop:complement-GxG}, we make a few useful observations and fix some notation.

%%%%%%%%
\subsubsection{Existence of representations}

The following proposition justifies the assumptions in Theorem~\ref{thm:complete-proper-GxG}.

\begin{proposition}\label{prop:exists-tau-O(b)}
Let $G$ be a noncompact real reductive Lie group and $\theta \subset \Delta$ a nonempty subset of the simple restricted roots of~$G$.
For $\KK=\RR$, $\CC$, or~$\bH$, there exists an irreducible, $\theta$-proximal representation $\tau : G\to \GL_{\KK}(V)$ preserving a nondegenerate $\RR$-bilinear form $b : V\otimes_\RR V\to \KK$ if and only if $\theta = \theta^{\star}$.
\end{proposition}

Note that in this case the group $\Aut_{\KK}(b)$ is necessarily noncompact since it contains an element which is proximal in $\PP_{\KK}(V)$.

One implication of Proposition~\ref{prop:exists-tau-O(b)} is given by the following observation.

\begin{lemma} \label{lem:preserv-B} 
For $\KK=\RR$, $\CC$, or~$\bH$, let $\tau : G\to \GL_{\KK}(V)$ be an irreducible representation with highest weight~$\chi_{\tau}$.
If the group $\tau(G)$ preserves a nondegenerate $\RR$-bilinear form $b:V\otimes_\RR V\to \KK$, then $\chi_{\tau} = \chi_{\tau}^{\star}$; moreover, $b$ is unique up to scale.
When $\KK=\RR$ and \(\tau\) is proximal, the converse also holds.
\end{lemma}

\begin{proof}[Proof of Lemma~\ref{lem:preserv-B}]
The dual representation $(\tau^*, V^*=\Hom_{\KK}(V,\KK))$ has highest weight~$\chi_{\tau}^{\star}$.
Therefore, if there exists a $(\tau,\tau^*)$-equivariant isomorphim $\psi : V\to V^*$ then $\chi_{\tau} = \chi_{\tau}^{\star}$; in this case $\psi$ is unique up to scale by the Schur lemma.
We then note that the space of nondegenerate $\tau(G)$-invariant bilinear forms $b : V\otimes_{\KK} V\to \KK$ identifies with the space of $(\tau,\tau^*)$-equivariant isomorphims $V\to V^*$ by sending $b$ to the isomorphism $v\mapsto b(v,\cdot)$.
This treats the case of a symmetric or antisymmetric~form.

For the case of Hermitian and anti-Hermitian forms (where \(\KK=\CC\) or~\(\bH\)), we observe that the (real vector) space of forms \(b: V \otimes_\RR V \to\KK\) that are \(\KK\)-linear in the second variable and antilinear in the first variable identifies with the space of $(\tau,\bar{\tau}^*)$-equivariant homomorphims $V\to \bar{V}^*$ where \(\bar{\tau}^*\) is the representation of \(G\) on the space \(\bar{V}^*\) of antilinear forms \(u: V\to \KK\), \ie \(u(vz) = \bar{z} u(v)\) for all \(v\in V\) and \(z\in\KK\).
The highest weight of~\(\bar{V}^*\) is also~\(\chi_{\tau}^{\star}\).

When $\KK=\RR$ and $\tau$ is proximal, the equality $\chi_{\tau} = \chi_{\tau}^{\star}$ implies the existence of an equivariant isomorphism $V\to V^*$, hence of a nondegenerate invariant bilinear~form. 
\end{proof}

\begin{proof}[Proof of Proposition~\ref{prop:exists-tau-O(b)}]
Suppose there exists an irreducible, $\theta$-proximal representation $\tau : G\to \GL_{\KK}(V)$ preserving a nondegenerate $\RR$-bilinear form $b:V\otimes_\RR V\to \KK$.
By Lemma~\ref{lem:preserv-B}, the highest weight $\chi_{\tau}$ of~$\tau$ satisfies $\chi_{\tau} = \chi_{\tau}^{\star}$.
By definition of $\theta$-compatibility, $\theta$ is the set of $\alpha \in \Delta$ such that $(\chi_\tau, \alpha) > 0$.
Since the $W$-invariant scalar product $(\cdot,\cdot)$ on~$\aaa^{\ast}$ is invariant under $\alpha\mapsto\alpha^{\star}$, we conclude that $\theta = \theta^{\star}$.

Conversely, suppose $\theta = \theta^{\star}$.
By Lemma~\ref{lem:theta-comp-exists}, we can find an irreducible proximal real representation $(\tau,V)$ of~$G$ with highest weight~$\chi_{\tau} \in \sum_{\alpha\in \theta} \NN^*\omega_\alpha$ satisfying \(\chi_{\tau} = \chi_\tau^{\star}\); it is \(\theta\)-compatible by definition.
By Lemma~\ref{lem:preserv-B}, the group $\tau(G)$ preserves a nondegenerate real bilinear form.
Tensoring with $\KK$ gives an irreducible,  $\theta$-proximal $\KK$-representation~$V$ together with an invariant bilinear form $b : V\otimes_\RR V \to \KK$.
\end{proof}

\begin{remark}
For $\KK=\RR$ or \(\CC\), we can always assume $b$ to be symmetric up to replacing $V$ with the irreducible representation of highest weight $2\chi_\tau$, which is a subrepresentation of $\mathrm{Sym}^2(V)$.
\end{remark}

%%%%%%%%
\subsubsection{Cartan and Lyapunov projections for $G$, $\Aut_{\KK}(b)$, and $\GL_{\KK}(V)$}

As in Theorem~\ref{thm:complete-proper-GxG}, let $b: V \otimes_\RR V \to \KK$ be a nondegenerate $\RR$-bilinear form on a $\KK$-vector space~$V$ and $\tau : G\to\Aut_{\KK}(b)\subset\GL_{\KK}(V)$ an irreducible, $\theta$-proximal representation.
We identify $\GL_{\KK}(V)$ with $\GL_d(\KK)$ where $d=\dim_{\KK}(V)$, and use the notation of Example~\ref{ex:roots}.
Up to conjugating, we may assume that the real Lie groups $G$, $\Aut_{\KK}(b)$, and $\GL_{\KK}(V)$ have compatible Cartan decompositions, in the sense of inclusion of the corresponding maximal compact subgroups and inclusion of the corresponding Cartan subspaces (see Remark~\ref{rem:Cartan_comp_rep}).
We denote the corresponding Cartan projections by $\mu$, $\mu_{b}$, $\mu_{\GL_{\KK}(V)}$, and the corresponding Lyapunov projections by $\lambda$, $\lambda_{b}$, $\lambda_{\GL_{\KK}(V)}$.
Let $\alpha_0(b)$ be a simple restricted root of $\Aut_{\KK}(b)$, determining the parabolic subgroup $Q_0(b)$ of Section~\ref{subsubsec:Aut(b)}; then $\alpha_0(b)=\alpha_0(b)^{\star}$.
Let $\omega_{\alpha_0(b)}$ be the corresponding fundamental weight.
We use similar notation for $b\oplus b'$ on $V\oplus V$.
Then the following equalities hold.

\begin{lemma}\label{fact:tau-O(b)}
Let $\nu$ be either the Cartan projection $\mu$ or the Lyapunov projection~$\lambda$.
For any $g\in G$,
\begin{enumerate}
  \item\label{item:omega-alpha-0} $\langle \omega_{\alpha_0(b)}, \nu_{b}(\tau(g)) \rangle  = \langle \varepsilon_1, \nu_{\GL_{\KK}(V)}(\tau(g)) \rangle = \langle\chi_{\tau}, \nu(g)\rangle$,
  \item\label{item:alpha-0} $\langle \alpha_0(b), \nu_{b}(\tau(g)) \rangle = \langle \varepsilon_1 - \varepsilon_2, \,\nu_{\GL_{\KK}(V)}(\tau(g)) \rangle$.
\end{enumerate}
For any $g,g'\in G$ with $\langle\chi_{\tau},\nu(g)\rangle \geq \langle\chi_{\tau},\nu(g')\rangle$,
\begin{enumerate}
\setcounter{enumi}{2}
  \item \label{item:alpha-0-b+b} $\langle \alpha_0(b\oplus b'), \nu_{b\oplus b'}(\tau(g) \oplus \tau(g')) \rangle$\\
  ${}\ \ = \langle \varepsilon_1 - \varepsilon_2, \,\nu_{\GL_{\KK}(V\oplus V)}(\tau(g) \oplus \tau(g')) \rangle$\\
  \noindent ${}\ \ = \min\big\{ \langle \alpha_0(b), \nu_{b}(\tau(g)) \rangle, \langle \omega_{\alpha_0(b)}, \,\nu_{b}(\tau(g)) - \nu_{b}(\tau(g'))\rangle \big\}$.
 \end{enumerate}
\end{lemma}

The space $\mathcal{F}_0(b) = \Aut_{\KK}(b)/Q_0(b)$ identifies with the subset of $\PP_{\KK}(V)$ consisting of $b$-isotropic lines, and similarly for $\mathcal{F}_0(b\oplus b')$ inside $\PP_{\KK}(V\oplus V)$.
The embedding  $V \simeq V \oplus \{0\} \hookrightarrow V\oplus V$ induces a natural embedding $\mathcal{F}_0(b) \hookrightarrow \mathcal{F}_0(b\oplus b')$.

Similarly to Lemma~\ref{fact:tau-O(b)}.\eqref{item:alpha-0-b+b} for~$\lambda$, the following holds:

\begin{remark} \label{rem:attract-f-p-g-g'}
Let $g,g'\in G$ satisfy $\langle\chi_{\tau}, \lambda(g) - \lambda(g')\rangle > 0$.
Then the element $\tau(g) \oplus {\tau}(g') \in \Aut_{\KK}(b\oplus b')$ is proximal in $\mathcal{F}_0(b\oplus b')$ if and only if $\tau(g) \in \Aut_{\KK}(b)$ is proximal in~$\mathcal{F}_0(b)$.
In this case the attracting fixed point of $\tau(g) \oplus {\tau}(g')$ in $\mathcal{F}_0(b\oplus b')$ is the image of the attracting fixed point of $\tau(g)$ under the natural embedding $\mathcal{F}_0(b) \hookrightarrow \mathcal{F}_0(b\oplus b')$, and the same holds with $(g^{-1}, {g'}^{-1})$ instead of $(g,g')$, by \eqref{eqn:opp-inv-lambda} and the fact that $\chi_{\tau} = \chi_{\tau}^{\star}$ (Lemma~\ref{lem:preserv-B}).
\end{remark}

%%%%%%%%
\subsubsection{The properness criterion of Benoist and Kobayashi for group manifolds} \label{subsubsec:prop-crit-GxG}

If $\mu : G\to\overline{\aaa}^+$ is a Cartan projection for~$G$ as above, then
\begin{equation}
 \mu\times\mu : G\times G \longrightarrow \overline{\aaa}^+\times\overline{\aaa}^+\label{eq:muGxG}
\end{equation}
is a Cartan projection for $G\times G$.
It sends $\Diag(G)$ to the diagonal of $\overline{\aaa}^+\times\nolinebreak\overline{\aaa}^+$.
Let $\Vert\cdot\Vert$ be a $W$-invariant Euclidean norm on~$\aaa$ as in Section~\ref{subsubsec:Cartan-proj}.
In this setting the properness criterion of Benoist and Kobayashi (see Section~\ref{subsec:intro-Ano-implies-proper}) can be expressed as follows.

\smallskip 
\noindent
{\bf Properness criterion for group manifolds} \cite{Benoist_properness, TKobayashi_proper}: 
{\em A discrete subgroup $\Gamma'$ of $G\times G$ acts properly discontinuously on $(G\times G)/\Diag(G)$ if and only if}
\[ \Vert \mu(\gamma'_1) - \mu(\gamma'_2) \Vert \underset{\scriptscriptstyle \gamma'=(\gamma'_1,\gamma'_2)\to\infty}{-\!\!\!-\!\!\!-\!\!\!-\!\!\!\longrightarrow} + \infty, \]
where $\gamma'\to\infty$ means that $\gamma'$ exits every finite subset of~$\Gamma'$.

The action of $\Gamma'$ on $(G\times G)/\Diag(G)$ is \emph{sharp}, in the sense of \eqref{eqn:sharp}, if and only if there exist $c,C>0$ such that for any $\gamma' = (\gamma'_1,\gamma'_2) \in \Gamma'$,
\[ \Vert \mu(\gamma'_1) - \mu(\gamma'_2) \Vert \geq c \, \big( \Vert \mu(\gamma'_1) \Vert  + \Vert \mu(\gamma'_2) \Vert \big) - C. \]

%%%%%%%%%%%%%%%%%%%%%%%%%
\subsection{Proof of Theorem~\ref{thm:complete-proper-GxG}} \label{subsec:proof-proper-GxG}

In Theorem~\ref{thm:complete-proper-GxG}, the implication $\eqref{item:p7} \Rightarrow \eqref{item:p8}$  is immediate from the definition \eqref{eqn:sharp} of sharpness and the properness criterion of Benoist and Kobayashi.
Remark~\ref{rem:qimu}, with~\eqref{eq:muGxG}, yields the implication $\eqref{item:p8}\Rightarrow \eqref{item:p9}$.

We now prove the other implications in Theorem~\ref{thm:complete-proper-GxG}, using the notation of Section~\ref{subsec:orth-repr}.
Note that our proofs of $\eqref{item:p8}\Rightarrow \eqref{item:p4}$ for $G$ of real rank~$1$ and $\eqref{item:p3} \Rightarrow \eqref{item:p4}$ rely on our characterizations of Anosov representations given by Theorems \ref{thm:char_ano}.\eqref{item:away-from-walls} and~\ref{thm:char_ano_lambda_intro}.\eqref{item:away-from-walls-lambda}, while $\eqref{item:p4} \Rightarrow \eqref{item:p7}$ and $\eqref{item:p4} \Rightarrow \eqref{item:p3}$ rely on Theorems \ref{thm:char_ano}.\eqref{item:lin-away-from-walls} and~\ref{thm:char_ano_lambda_intro}.\eqref{item:lin-away-from-walls-lambda}.

\begin{proof}[Proof of $\eqref{item:p1}\Rightarrow \eqref{item:p2}$ in Theorem~\ref{thm:complete-proper-GxG}]
By definition of $\theta$-compatibility, we can write $\chi_{\tau} = \sum_{\alpha\in\theta} n_{\alpha}\,\omega_{\alpha}$ where $n_{\alpha}>0$ for all $\alpha\in\theta$.
Lemma~\ref{fact:tau-O(b)}.\eqref{item:omega-alpha-0} for the Lyapunov projection~$\lambda$ then yields, for any $\gamma\in\Gamma$,
\begin{align*}
\langle \omega_{\alpha_0(b)}, \lambda_{b}(\tau\circ\rho_L(\gamma)) \rangle & =  \sum_{\alpha\in\theta} n_{\alpha} \, \langle \omega_{\alpha}, \lambda(\rho_L(\gamma)) \rangle,\\
\langle \omega_{\alpha_0(b)}, \lambda_{b}(\tau\circ\rho_R(\gamma)) \rangle & =  \sum_{\alpha\in\theta} n_{\alpha} \, \langle \omega_{\alpha}, \lambda(\rho_R(\gamma)) \rangle.
\end{align*}
Therefore, if $\rho_L$ uniformly $P_{\theta}$-dominates~$\rho_R$, then $\tau\circ\rho_L$ uniformly $Q_0(b)$-dominates $\tau\circ\nolinebreak\rho_R$.
\end{proof}

If $\theta$ is a singleton (\eg if $G$ is semisimple of real rank~$1$), then the previous proof shows that the uniform $P_{\theta}$-domination of $\rho_R$ by~$\rho_L$ is equivalent to the uniform $Q_0(b)$-domination of $\tau\circ\rho_R$ by $\tau\circ\rho_L$, \ie $\eqref{item:p2}\Rightarrow \eqref{item:p1}$ holds as well.

\begin{proof}[Proof of $\eqref{item:p2}\! \Leftrightarrow\! \eqref{item:p3prime}\! \Leftrightarrow\! \eqref{item:p3}$ and $\eqref{item:p4}\! \Leftrightarrow\! \eqref{item:p4prime}$ in Theorem~\ref{thm:complete-proper-GxG} and Proposition~\ref{prop:complement-GxG}] 
Suppose $\Gamma$ is word hyperbolic.
The natural inclusion \(j:\Aut_\KK(b) \hookrightarrow \GL_\KK(V)\) is \(\alpha_0(b)\)-proximal, hence Proposition~\ref{prop:theta-comp-Anosov}.\eqref{item:theta-comp-Ano} applies: a representation \(\tau \circ \rho_L:\Gamma \to \Aut_\KK(b)\) is \(Q_0(b)\)-Anosov if and only if \(j\circ (\tau\circ \rho_L):\Gamma\to \GL_\KK(V)\) is \(P_{\varepsilon_1 - \varepsilon_2}(V)\)-Anosov.
Moreover, Lemma~\ref{fact:tau-O(b)}.\eqref{item:omega-alpha-0} for the Lyapunov projection~$\lambda$ yields that $\tau\circ\rho_L$ uniformly $Q_0(b)$-dominates $\tau\circ\rho_R$ 
if and only if $j\circ (\tau\circ\rho_L)$ uniformly $P_{\varepsilon_1-\varepsilon_2}(V)$-dominates $j\circ (\tau\circ\rho_R)$.
Thus $\eqref{item:p3} \Leftrightarrow \eqref{item:p3prime}$ holds.

The equivalence $\eqref{item:p4} \Leftrightarrow \eqref{item:p4prime}$ follows from the same argument, using $b\oplus b'$ on $V\oplus V$ instead of $b$ on~$V$.

Similarly, $\rho_L$ is $P_\theta$-Anosov if and only if $(j\circ \tau)\circ\rho_L$ is $P_{\varepsilon_1-\varepsilon_2}(V)$-Anosov by Proposition~\ref{prop:theta-comp-Anosov}.\eqref{item:theta-comp-Ano}; thus $\eqref{item:p2} \Leftrightarrow \eqref{item:p3prime}$ holds.
\end{proof}

In order to prove the equivalence $\eqref{item:p3} \Leftrightarrow \eqref{item:p4}$ in Theorem~\ref{thm:complete-proper-GxG}, we first establish the following. 

\begin{proposition} \label{prop:3<->4partly}
In the setting of Theorem~\ref{thm:complete-proper-GxG}, suppose $\Gamma$ is word hyperbolic.
Then the following are equivalent:
\begin{enumerate}[label=(\roman*),ref=\roman*]
\setcounter{enumi}{2}
  \item\label{item:p3bis} there exists a continuous, $(\tau\circ\rho_L)$-equivariant, transverse, dynamics-preserving map $\xi_V : \partial_{\infty}\Gamma \to \mathcal{F}_0(b)$ and for all $\gamma\in\Gamma$,
  \begin{equation} \label{eqn:weak-domin}
  \langle\omega_{\alpha_0(b)},\lambda_b(\tau\circ\rho_R(\gamma))\rangle < \langle\omega_{\alpha_0(b)},\lambda_b(\tau\circ\rho_L(\gamma))\rangle \, ;
  \end{equation}
  \item\label{item:p4bis} there exist a continuous, $(\tau\circ\rho_L\oplus\tau\circ\rho_R)$-equivariant, transverse, dynamics-preserving map $\xi_{V\oplus V} : \partial_{\infty}\Gamma \to \mathcal{F}_0(b\oplus b')$.
\end{enumerate}
\end{proposition}

\begin{proof}
Suppose \eqref{item:p3bis} holds.
By Remark~\ref{rem:compatible}.\eqref{item:Ano-prox}, for any $\gamma \in \Gamma$ of infinite order, $\tau\circ\rho_L(\gamma)$ is proximal in $\mathcal{F}_0(b)$.
Remark~\ref{rem:attract-f-p-g-g'} and \eqref{eqn:weak-domin} imply that $(\tau\circ\rho_L \oplus {\tau}\circ\rho_R)(\gamma)$ is proximal in $\mathcal{F}_0(b\oplus b')$ and $\xi_V$ sends the attracting fixed point of $\gamma$ in $\partial_{\infty}\Gamma$ to the attracting fixed point of $(\tau\circ\rho_L \oplus {\tau}\circ\rho_R)(\gamma)$ in $\mathcal{F}_0(b\oplus b')$, after embedding $\mathcal{F}_0(b)$ into $\mathcal{F}_0(b\oplus b')$.
Thus we obtain a continuous, $(\tau\circ\rho_L\oplus\tau\circ\rho_R)$-equivariant, transverse, dynamics-preserving maps $\xi_{V\oplus V} : \partial_{\infty}\Gamma \to \mathcal{F}_0(b\oplus b')$ by postcomposing $\xi_V$ with the natural inclusion $\mathcal{F}_0(b) \hookrightarrow \mathcal{F}_0(b\oplus b')$, and \eqref{item:p4bis} holds.

Conversely, suppose \eqref{item:p4bis} holds.
By Remark~\ref{rem:compatible}.\eqref{item:Ano-prox}, for any $\gamma \in \Gamma$ of infinite order, $(\tau\circ\rho_L \oplus {\tau}\circ\rho_R)(\gamma)$ is proximal in $\mathcal{F}_0(b\oplus b')$.
By the normalization \eqref{eqn:tautology}, Lemma~\ref{fact:tau-O(b)}.\eqref{item:omega-alpha-0}, and Remark~\ref{rem:attract-f-p-g-g'}, there exists $\gamma\in\Gamma$ of infinite order for which $\tau\circ\rho_L(\gamma)$ is proximal in~$\mathcal{F}_0(b)$ and the attracting fixed point of $(\tau\circ\rho_L \oplus {\tau}\circ\rho_R)(\gamma)$ in $\mathcal{F}_0(b\oplus b')$ is the image of the attracting fixed point of $\tau(\rho_L(\gamma))$ under the natural embedding $\mathcal{F}_0(b)\hookrightarrow\mathcal{F}_0(b\oplus b')$; the same holds for~$\gamma^{-1}$ instead of~$\gamma$.
In particular, the closed $\Gamma$-invariant set
\[ \{\eta \in \partial_{\infty} \Gamma \mid \xi_{V\oplus V}(\eta) \in \mathcal{F}_0(b) \subset \mathcal{F}_0(b\oplus b')\} \]
contains the attracting fixed points $\eta_{\gamma}^+, \eta_{\gamma^{-1}}^+$ of $\gamma$ and~$\gamma^{-1}$, hence is nonempty.
This set is equal to $\partial_{\infty}\Gamma$, by minimality of the action of $\Gamma$ on $\partial_{\infty}\Gamma$ if $\Gamma$ is nonelementary, and by the fact that $\partial_{\infty}\Gamma = \{ \eta_{\gamma}^+, \eta_{\gamma^{-1}}^+\}$ if $\Gamma$ is elementary.
Therefore, $\xi_{V\oplus V}$ defines a continuous map from $\partial_{\infty}\Gamma$ to $\mathcal{F}_0(b)$ which is equivariant and dynamics-preserving for $\tau\circ\rho_L$.
Moreover, \eqref{eqn:weak-domin} holds by Lemma~\ref{fact:tau-O(b)}.\eqref{item:alpha-0-b+b} for the Lyapunov projection~$\lambda$.
Thus \eqref{item:p3bis} holds.
\end{proof}

\begin{proof}[Proof of $\eqref{item:p3} \Leftrightarrow \eqref{item:p4}$ in Theorem~\ref{thm:complete-proper-GxG}]
Suppose condition \eqref{item:p3} of Theorem~\ref{thm:complete-proper-GxG} holds, \ie $\Gamma$ is word hyperbolic and $\tau\circ\rho_L$ is $Q_0(b)$-Anosov and uniformly $Q_0(b)$-dominates $\tau\circ\rho_R$.
By Proposition~\ref{prop:3<->4partly}, there exists a continuous, $(\tau\circ\rho_L\oplus\tau\circ\rho_R)$-equivariant, transverse, dynamics-preserving map $\xi_{V\oplus V} : \partial_{\infty}\Gamma \to \mathcal{F}_0(b\oplus -b)$.
Since $\tau\circ\rho_L$ is $Q_0(b)$-Anosov, Theorem~\ref{thm:char_ano_lambda_intro}.\eqref{item:away-from-walls-lambda} implies
\[ \big\langle \alpha_0(b), \lambda_{b}(\tau\circ\rho_L(\gamma)) \big\rangle \underset{\ellinfty{\gamma}\to +\infty}{\longrightarrow} + \infty. \]
Moreover, uniform $Q_0(b)$-domination implies
\[ \big\langle \omega_{\alpha_0(b)}, \lambda_{b}(\tau\circ\rho_L(\gamma)) - \lambda_{b}(\tau\circ\rho_R(\gamma)) \big\rangle \underset{\ellinfty{\gamma}\to +\infty}{\longrightarrow} + \infty. \]
  By Lemma~\ref{fact:tau-O(b)}.\eqref{item:alpha-0-b+b} for the Lyapunov projection~$\lambda$,
  \[ \big\langle \alpha_0(b\oplus b'), \lambda_{b\oplus b'}\big((\tau\circ\rho_L \oplus {\tau}\circ\rho_R)(\gamma)\big) \big\rangle \underset{\ellinfty{\gamma}\to +\infty}{\longrightarrow} + \infty. \]
Therefore $\tau\circ\rho_L\oplus {\tau}\circ\rho_R$ is $Q_0(b\oplus b')$-Anosov by Theorem~\ref{thm:char_ano_lambda_intro}.\eqref{item:away-from-walls-lambda}, \ie condition \eqref{item:p4} of Theorem~\ref{thm:complete-proper-GxG} holds.

Conversely, suppose condition \eqref{item:p4} of Theorem~\ref{thm:complete-proper-GxG} holds, \ie $\Gamma$ is word hyperbolic and $\tau\circ\rho_L \oplus {\tau}\circ\rho_R : \Gamma\to\Aut_{\KK}(b\oplus b')$ is $Q_0(b\oplus b')$-Anosov.
By Proposition~\ref{prop:3<->4partly}, there exists a continuous, $(\tau\circ\rho_L)$-equivariant, transverse, dynamics-preserving map $\xi_V : \partial_{\infty}\Gamma \to \mathcal{F}_0(b)$.
Since $\tau\circ\rho_L \oplus {\tau}\circ\rho_R$ is $Q_0(b\oplus b')$-Anosov, Theorem~\ref{thm:char_ano_lambda_intro}.\eqref{item:away-from-walls-lambda} implies
  \[ \big\langle \alpha_0(b\oplus b'), \lambda_{b\oplus b'}((\tau\circ\rho_L \oplus {\tau}\circ\rho_R)(\gamma)) \big\rangle \underset{\ellinfty{\gamma}\to +\infty}{\longrightarrow} + \infty. \]
  By Lemma~\ref{fact:tau-O(b)}.\eqref{item:alpha-0-b+b} for the Lyapunov projection~$\lambda$,
  \[ \big\langle \alpha_0(b), \lambda_{b}(\tau\circ\rho_L(\gamma)) \big\rangle \underset{\ellinfty{\gamma}\to +\infty}{\longrightarrow} + \infty. \]
  Therefore $\tau\circ\rho_L$ is $Q_0(b)$-Anosov by Theorem~\ref{thm:char_ano_lambda_intro}.\eqref{item:away-from-walls-lambda}.
  On the other hand, using Theorem~\ref{thm:char_ano_lambda_intro}.\eqref{item:lin-away-from-walls-lambda} and Lemma~\ref{fact:tau-O(b)}.\eqref{item:alpha-0-b+b}, we see that there exist $c,C>0$ such that for any $\gamma \in \Gamma$,
\[ \langle \omega_{\alpha_0(b)}, \lambda_{b}(\tau\circ\rho_L(\gamma)) - \lambda_{b}(\tau\circ\rho_R(\gamma)) \rangle \geq c \, \langle \omega_{\alpha_0(b)}, \lambda_{b}(\tau\circ\rho_L(\gamma)) \rangle - C. \]
Applying this to $\gamma^n$, dividing by~$n$, and taking the limit, we obtain
\[ \langle \omega_{\alpha_0(b)}, \lambda_{b}(\tau\circ\rho_R(\gamma)) \rangle \leq (1-c) \, \langle \omega_{\alpha_0(b)}, \lambda_{b}(\tau\circ\rho_L(\gamma)) \rangle. \]
Thus $\tau\circ\rho_L(\gamma)$ uniformly $Q_0(b)$-dominates $\tau\circ\rho_R(\gamma)$, \ie condition \eqref{item:p3} of Theorem~\ref{thm:complete-proper-GxG} holds.
\end{proof}

\begin{proof}[Proof of $\eqref{item:p2}, \eqref{item:p4} \Rightarrow \eqref{item:p7}$ in Theorem~\ref{thm:complete-proper-GxG}]
Suppose that \eqref{item:p2} and \eqref{item:p4} hold (we have seen that they are equivalent).
Since $\rho_L$ is $P_{\theta}$-Anosov, it is a quasi-isometric embedding (see Section~\ref{subsubsec:examples-properties}).
Since $\tau\circ\rho_L \oplus {\tau}\circ\rho_R$ is $Q_0(b\oplus b')$-Anosov, Theorem~\ref{thm:char_ano}.\eqref{item:lin-away-from-walls}, Lemma~\ref{fact:tau-O(b)}.\eqref{item:omega-alpha-0}--\eqref{item:alpha-0-b+b} for the Cartan projection~$\mu$, and the normalization \eqref{eqn:tautology} show that there exist $c,C>0$ such that for any $\gamma \in \Gamma$,
\[ \langle \chi_{\tau}, \mu(\rho_L(\gamma)) - \mu(\rho_R(\gamma)) \rangle \geq c \, \ellGamma{\gamma} - C. \]
Using \eqref{eqn:mu-leq-klength}, we see that there exist $c', C'>0$ such that for any $\gamma \in \Gamma$, 
\[ \Vert \mu(\rho_L(\gamma)) - \mu(\rho_R(\gamma)) \Vert \geq c' \, \big( \Vert\mu(\rho_L(\gamma))\Vert + \Vert\mu(\rho_R(\gamma))\Vert \big) - C', \]
where $\Vert\cdot\Vert$ is the $W$-invariant Euclidean norm on~$\aaa$ from Section~\ref{subsubsec:Cartan-proj}.
Thus the action of $\Gamma$ on $(G\times G)/\Diag(G)$ via $(\rho_L,\rho_R)$ is sharp (see Section~\ref{subsubsec:prop-crit-GxG}).
\end{proof}

Note that any subgroup of \(G\times G\) is always of the form
\begin{equation}
\Gamma^{\rho_L, \rho_R} =\{ (\rho_{L}(\gamma), \rho_{R}(\gamma) ) \mid \gamma \in \Gamma\} ,\label{eq:subgrp_GxG}
\end{equation}
where $\Gamma$ is a group and $\rho_L, \rho_R \in\Hom(\Gamma,G)$ two representations, corresponding to the two projections of $G\times G$ onto~$G$.

In the case that $G$ is semisimple of real rank~$1$, the implication $\eqref{item:p9}\Rightarrow \eqref{item:p8}$ of Theorem~\ref{thm:complete-proper-GxG} is an immediate consequence of the following result.
(Since $G$ has real rank~$1$, we identify $\overline{\aaa}^+$ with~$\RRp$ and see $\lambda$ as a function $G\to\RRp$.)

\begin{theorem}[\cite{Kassel_corank1}]\label{thm:kas-cork1}
Let $G$ be a semisimple Lie group of real~rank~$1$. 
Then any discrete subgroup of $G \times G$ acting properly discontinuously on $(G \times G) / \Diag(G)$ is of the form \( \Gamma^{\rho_L, \rho_R} \) as in \eqref{eq:subgrp_GxG} where, under the normalization \eqref{eqn:tautology}, the representation $\rho_{L}$ has finite kernel and discrete image and $\mu(\rho_R(\gamma))<\mu(\rho_L(\gamma))$ for almost all $\gamma\in\Gamma$, and
\begin{equation}\label{eqn:lambda-rhoLR}
\lambda(\rho_R(\gamma)) < \lambda(\rho_L(\gamma))~\text{ for all $\gamma\in\Gamma$ of infinite order.}
\end{equation}
In particular, if $\Gamma^{\rho_L,\rho_R}$ is finitely generated and quasi-isometrically embedded in $G\times G$, then \(\rho_L:\Gamma \to G\) is a quasi-isometric embedding and \(\Gamma\) is word hyperbolic (see Remark~\ref{rem:rank1ccquasi}).
\end{theorem}

We now use this result to prove the implication $\eqref{item:p8}\Rightarrow \eqref{item:p4}$ of Theorem~\ref{thm:complete-proper-GxG}.

\begin{proof}[Proof of $\eqref{item:p8}\Rightarrow \eqref{item:p4}$ in Theorem~\ref{thm:complete-proper-GxG} for $G$ semisimple of real rank~$1$]
\hspace{-0.2cm} Suppose $G$ is semi\-simple of real rank~$1$ and \eqref{item:p8} holds.
By Remark~\ref{rem:rank1ccquasi} the group~\(\Gamma\) is word hyperbolic and $\rho_L$ is $P_{\theta}$-Anosov. 
The boundary map of~$\rho_L$ induces a boundary map $\xi : \partial_{\infty}\Gamma\to \mathcal{F}_0(b \oplus b')$ that is continuous, $(\tau\circ\rho_L\oplus {\tau}\circ\rho_R)$-equivariant, and transverse.
By \eqref{eqn:lambda-rhoLR} in Theorem~\ref{thm:kas-cork1} and by Remark~\ref{rem:attract-f-p-g-g'}, the map $\xi$ is dynamics-preserving.
By the properness criterion of Benoist and Kobayashi (Section~\ref{subsubsec:prop-crit-GxG}), we have
\[ \Vert \mu(\rho_L(\gamma)) - \mu(\rho_R(\gamma)) \Vert \underset{\gamma\to\infty}{\longrightarrow} + \infty. \]
Since $G$ is semisimple of real rank~$1$, this also holds if we replace the norm $\Vert\cdot\Vert$ on~$\aaa$ with $\langle\chi_{\tau},\cdot\rangle$.
Using Lemma~\ref{fact:tau-O(b)}.\eqref{item:alpha-0-b+b} for the Cartan projection~$\mu$, as well as the fact that $\rho_L$ is a quasi-isometric embedding and Remark~\ref{rem:qimu}, we deduce that
\[ \langle \alpha_0(b\oplus b'), \mu_{b\oplus b'}((\tau\circ\rho_L\oplus {\tau}\circ\rho_R)(\gamma)) \rangle \underset{\gamma\to\infty}{\longrightarrow} + \infty. \]
Therefore, $\tau\circ\rho_L\oplus{\tau}\circ\rho_R$ is $Q_0(b\oplus b')$-Anosov by Theorem~\ref{thm:char_ano}.\eqref{item:away-from-walls}.
\end{proof}

%%%%%%%%%%%%%%%%%%%%%%%%%
\subsection{Proofs of Corollaries \ref{cor:sharpness-conj-GxG-rk1} and~\ref{cor:proper-open-GxG}}

We use the following classical cohomological arguments (see \eg \cite[Corollary\,5.5]{TKobayashi}):
\begin{enumerate}[label=(\roman*),ref=\roman*]
  \item\label{item:cohom-GxG} When a torsion-free discrete subgroup of $G\times G$ acts properly discontinuously on $(G\times G)/\Diag(G)$, it acts cocompactly on $(G\times G)/\Diag(G)$ if and only if its cohomological dimension is equal to $\dim_{\RR}(G/K)$.
  \item\label{item:cohom-G} A torsion-free discrete subgroup of~$G$ is a uniform lattice in~$G$ if and only if its cohomological dimension is equal to $\dim_{\RR}(G/K)$.
\end{enumerate}

\begin{proof}[Proof of Corollary~\ref{cor:sharpness-conj-GxG-rk1}]
Suppose the action of $\Gamma$ on $(G\times G)/\Diag(G)$ is properly discontinuous and cocompact.
Then $\Gamma$ is finitely generated (using the Milnor--\v{S}varc lemma), and so up to passing to a finite-index subgroup we may assume that $\Gamma$ is torsion-free (using the Selberg lemma).
By Theorem~\ref{thm:kas-cork1}, one of the projections of $\Gamma$ onto~$G$ is injective and discrete; its image is a uniform lattice of~$G$ by \eqref{item:cohom-GxG} and~\eqref{item:cohom-G}, hence it is a quasi-isometric embedding.
We conclude using the implication $\eqref{item:p8}\Rightarrow\eqref{item:p7}$ of Theorem~\ref{thm:complete-proper-GxG} (or the implication $\eqref{item:proper-GxG-rk1}\Rightarrow\eqref{item:sharp-GxG-rk1}$ of Theorem~\ref{thm:proper-GxG-rk1}).
\end{proof}

\begin{proof}[Proof of Corollary~\ref{cor:proper-open-GxG}]
The first statement (openness) follows from the equivalence $\eqref{item:p4}\Leftrightarrow\eqref{item:p9}$ of Theorem~\ref{thm:complete-proper-GxG} and from the fact that being Anosov is an open property \cite{Labourie_anosov,Guichard_Wienhard_DoD}.
For the second statement (compactness), note that up to passing to a finite-index subgroup we may again assume $\Gamma$ to be torsion-free, by the Selberg lemma; there is a neighborhood $\mathcal{U}\subset\Hom(\Gamma,G\times G)$ consisting of injective and discrete representations and we use \eqref{item:cohom-GxG} above.
\end{proof}

%%%%%%%%%%%%%%%%%%%%%%%%%
\subsection{Example of a non-Anosov representation with nice boundary maps}

We now construct an example of a representation $\rho : \Gamma\to G$ which admits continuous, dynamics-preserving, transverse boundary maps $\xi^+ : \partial_{\infty}\Gamma\to G/P_{\theta}$ and $\xi^- : \partial_{\infty}\Gamma\to\nolinebreak G/P_{\theta}^-$, but which is not $P_{\theta}$-Anosov.

\begin{example} \label{ex:nice-maps-not-Ano}
Let $\Gamma$ be a finitely generated discrete group and $\rho_L,\rho_R : \Gamma\to G=\SO(1,2)=\Aut_{\RR}(b)$ two representations such that $\rho_L$ is convex cocompact, $\lambda(\rho_R(\gamma))<\lambda(\rho_L(\gamma))$ for all $\gamma\in\Gamma$ of infinite order, but
\begin{equation} \label{eqn:sup=1}
\sup_{\substack{\gamma\in\Gamma\\ \text{of infinite order}}}\ \frac{\lambda(\rho_R(\gamma))}{\lambda(\rho_L(\gamma))} = 1.
\end{equation}
By Proposition~\ref{prop:3<->4partly} with $\tau=\mathrm{id}$, the representation $\rho := \rho_L \oplus \rho_R : \Gamma\to\Aut_{\RR}(b\oplus b)$ admits a continuous, dynamics-preserving, transverse boundary map $\xi : \partial_{\infty}\Gamma\to\nolinebreak\mathcal{F}_0(b\oplus\nolinebreak b)$.
However, $\rho_L$ does not uniformly $Q_0(b)$-dominate~$\rho_R$, and so Theorem~\ref{thm:complete-proper-GxG} shows that $\rho$ is \emph{not} $Q_0(b\oplus b)$-Anosov.
\end{example}

Here is one construction of a pair $(\rho_L,\rho_R)$ as in Example~\ref{ex:nice-maps-not-Ano}
 for $\Gamma$ a free group on two generators, following a key idea of \cite{Goldman_Labourie_Margulis_Minsky}.
Let $S$ be a hyperbolic one-holed torus with infinite area and compact convex core~$S_0$, and let $\Upsilon\subset S_0$ be a transversely measured geodesic lamination with irrational support. Then $\Upsilon$ intersects every nonperipheral closed curve of $S$, and every half-leaf of $\Upsilon$ is dense in $\Upsilon$.

Let $(\ell(t))_{t\in\RR}$ be an injectively immersed geodesic of $S$ parameterized by arc length, with $\ell(0)\in\partial S_0$ and $\ell(t)$ spiralling asymptotically to $\Upsilon$ as $t\rightarrow +\infty$. We can find a one-parameter family of (distinct) small deformations $\{\ell_s\}_{s\in[-\varepsilon,\varepsilon]}$ of $\ell$, with each $\ell_s$ an injectively immersed geodesic of $S$ parameterized by arc length, satisfying $\lim_{t\rightarrow +\infty} d(\ell_s(t),\ell(t))=0$. The map $\mathbf{\ell}:(s,t)\mapsto \ell_s(t)$ from $[-\varepsilon,\varepsilon]\times \RR$ to $S$ is then injective. Let $S'$ be the hyperbolic surface obtained from $S$ by collapsing each arc $\mathbf{\ell}([-\varepsilon, \varepsilon]\times\{t\})$ to a point. The collapsing map $\psi:S\rightarrow S'$ is $1$-Lipschitz and allows us to identify the fundamental groups of both $S$ and $S'$ with $\Gamma$. 

We can take for $\rho_L$ and $\rho_R$ holonomy representations of $S$ and $S'$ respectively.
For any $\gamma\in\Gamma\smallsetminus\{e\}$, the inequality $\lambda(\rho_R(\gamma))<\lambda(\rho_L(\gamma))$ follows from the fact that the geodesic representative of the loop in~$S$ associated with~$\gamma$ crosses the collapsing region $\mathbf{\ell}([-\varepsilon,\varepsilon]\times \RR)$ (because it crosses $\Upsilon$), hence is taken by $\psi$ to a shorter rectifiable loop.
The supremum \eqref{eqn:sup=1} is approached when $\gamma$ follows $\Upsilon$ for most of its length.

%%%%%%%%%%%%%%%%%%%%%%%%%%%%%%%%%%%%%%%%%%%%%%%%%%%
\appendix

%%%%%%%%%%%%%%%%%%%%%%%%%%%%%%%%%%%%%%%%%%%%%%%%%%%
\section{An unstable quasi-isometrically embedded subgroup}
\label{app:discr-quasi-isom}

In this appendix we give an example (Proposition~\ref{prop:an-unstable-quasi}) of a quasi-isometric embedding $\rho_0$ of a free group $\Gamma$ into a semisimple Lie group of higher real rank which is not stable under small deformations; in particular $\rho_0$ is \emph{not} an Anosov representation of~$\Gamma$.
This example was first described in~\cite{Guichard_these}.
We use it to construct a non-Anosov representation of $\Gamma$ into $\SL_6(\RR)$ with interesting properties (Example~\ref{ex:xi-not-dyn-preserv}).

\begin{proposition} \label{prop:an-unstable-quasi}
Let $\Gamma$ be a free group on two generators.
Then there is a continuous family $\{\rho_t\}_{t\in [0,1]}$ of representations $\Gamma\to\SL_2(\RR) \times \SL_2(\RR)$ such that
  \begin{itemize}
  \item $\rho_0$ is a quasi-isometric embedding;
  \item for any $t\notin \QQ$, the group $\rho_t(\Gamma)$ is dense in $\SL_2(\RR) \times \SL_2(\RR)$ (for the real topology).
  \end{itemize}
\end{proposition}

In order to prove Proposition~\ref{prop:an-unstable-quasi}, we consider a free generating subset $\{ a,b\}$ of~$\Gamma$.
For any $\gamma = a^{m_1}b^{n_1} \cdots a^{m_N} b^{n_N}\in\Gamma$ with $m_i\neq 0$ for all $i>1$ and $n_i\neq 0$ for all $i<N$, we set
\[\left\{\begin{array}{l}
\ellFF{\gamma}{a} = |m_1| + \cdots + |m_N|,\\
\ellFF{\gamma}{b} \hspace{0.05cm} = |n_1| + \cdots + |n_N|.
\end{array}\right.\]
Then the word length function $\ellFF{\cdot}{\Gamma} : \Gamma\to\NN$ with respect to $\{ a,b\}$ satisfies
\[\ellFF{\gamma}{\Gamma} = \ellFF{\gamma}{a}+\ellFF{\gamma}{b}\]
for all $\gamma\in\Gamma$.
We identify the Weyl chamber $\overline{\aaa}^+$ with~$\RRp$, so that the Cartan projection $\mu : \SL_2(\RR)\to\overline{\aaa}^+$ of Section~\ref{subsubsec:Cartan-proj} takes values in~$\RRp$.
With this notation, Proposition~\ref{prop:an-unstable-quasi} is an easy consequence of the following.

\begin{proposition} \label{prop:an-unstable-quasi-2}
Let $\Gamma$ be a free group on two generators $a,b$ and let \(\kappa>0\).
Then there is a continuous family $\{\rho_{\alpha,t}\}_{t\in [0,1]}$ of representations $\Gamma \to \SL_2(\RR)$ such that
\begin{itemize}
  \item $\mu(\rho_{\alpha,0}(\gamma)) \geq \kappa \ellFF{\gamma}{a}$ for all $\gamma \in \Gamma$;
  \item for any $t\notin\QQ$, the group $\rho_{\alpha,t}(\Gamma)$ is dense in $\SL_2(\RR)$ (for the real topology), the element $\rho_{\alpha,t}(a)$ is hyperbolic, and $\rho_{\alpha,t}(b)$ is elliptic.
\end{itemize}
\end{proposition}

\begin{proof}[Proof of Proposition~\ref{prop:an-unstable-quasi} using Proposition~\ref{prop:an-unstable-quasi-2}]
Let $\{\rho_{\alpha,t}\}_{t\in [0,1]}$ be the continuous family of representations $\Gamma\to\SL_2(\RR)$ given by Proposition~\ref{prop:an-unstable-quasi-2} (for \(\kappa=1\) say), and let $\{\rho_{\beta,t}\}_{t\in [0,1]}$ be defined by $\rho_{\beta,t}=\rho_{\alpha,t}\circ\varsigma$ where $\varsigma$ is the automorphism of~$\Gamma$ switching $a$ and~$b$.
For any $t\in [0,1]$, consider the representation
\[ \rho_t=(\rho_{\alpha,t},\rho_{\beta,t}) : \Gamma \longrightarrow \SL_2(\RR)\times\SL_2(\RR).\]
Then $\rho_0$ is a quasi-isometric embedding because
\[ \mu(\rho_{\alpha, 0}(\gamma))+\mu(\rho_{\beta, 0}(\gamma)) \geq \kappa(\ellFF{\gamma}{a}+\ellFF{\gamma}{b})=\kappa \ellFF{\gamma}{\Gamma} \]
for all $\gamma\in\Gamma$.
Let $G_t$ be the closure of $\rho_t(\Gamma)$ in $\SL_2(\RR)\times\SL_2(\RR)$.
If $t\notin\QQ$, then the two projections of $G_t$ to $\SL_2(\RR)$ are equal to the full group $\SL_2(\RR)$.
In that case, by Goursat's lemma (see \eg \cite{Petrillo}), either $G_t = \SL_2(\RR) \times\nolinebreak\SL_2(\RR)$ or $G_t = \{ (h,ghg^{-1}) \mid h\in \SL_2(\RR)\}$ for some $g$ in $\GL_2(\RR)$.
The second case cannot occur since it would imply that the hyperbolic element $\rho_{\alpha,t}(a)$ is conjugate to the elliptic element $\rho_{\beta,t}(a)$.
\end{proof}

\begin{proof}[Proof of Proposition~\ref{prop:an-unstable-quasi-2}]
We define $\rho_{\alpha,t}(a)$ to be a hyperbolic element $A\in\SL_2(\RR)$, independent of~$t$, whose properties will be specified in a moment.
For $t>0$, we also define
\[ \rho_{\alpha,t}(b) = \begin{pmatrix} \cos \pi t & \frac{1}{\pi t}\sin \pi t \\ -\pi t \sin \pi t& \cos \pi t \end{pmatrix},\]
and extend this by continuity to $\rho_{\alpha,0}(b) = B := (\begin{smallmatrix} 1 & 1 \\ 0 & 1 \end{smallmatrix})$.

For $t$ in $[0,1]\smallsetminus\QQ$, the element $\rho_{\alpha,t}(b)$ is conjugate to an irrational rotation, hence the closure of the group spanned by $\rho_{\alpha,t}(b)$ is a conjugate of $\SO(2)$; since $\SL_2(\RR)$ is generated by any hyperbolic element and $\SO(2)$, we conclude that $\rho_{\alpha,t}(\Gamma)$ is dense in $\SL_2(\RR)$.

We now show that, for some appropriate choice of the hyperbolic element~$A$, we have $\mu(\rho_{\alpha,0}(\gamma)) \geq \kappa \ellFF{\gamma}{a}$ for all $\gamma \in \Gamma$.
Endow $\PP^1(\RR) = \RR \cup \{\infty\}=\partial_\infty\mathbb{H}^2$ with the round metric centered at $\sqrt{-1}\in\HH^2$.
The parabolic element $B=\rho_{\alpha,0}(b)$ fixes the point $\infty\in\PP^1(\RR)$, and the two compact intervals $V_+:=[\frac{1}{2},\infty]$ and $V_-:=[\infty, \frac{-1}{2}]$ of $\PP^1(\RR)$ (intersecting only at the point~$\infty$) satisfy:
\begin{itemize}
  \item $B( \PP^1(\RR) \smallsetminus V_-) = V_+$ and $B|_{\PP^1(\RR) \smallsetminus V_-}$ is $1$-Lipschitz;
  \item $B^{-1}( \PP^1(\RR) \smallsetminus V_+) = V_-$ and $B^{-1}|_{\PP^1(\RR) \smallsetminus V_+}$ is $1$-Lipschitz.
\end{itemize}

\begin{figure}[h!]
\centering
\labellist
\small\hair 2pt
\pinlabel $-\frac{1}{2}$  at 12 32
\pinlabel $\frac{1}{2}$  at 167 32
\pinlabel $\infty$  at 92 187
\pinlabel $A$  at 90 28
\pinlabel $B$  at 90 91
\pinlabel $\HH^2$  at 80 60
\pinlabel $U_-$  at 52 2
\pinlabel $U_+$  at 133 2
\pinlabel $V_-$  at 13 150
\pinlabel $V_+$  at 170 150
\endlabellist
\includegraphics[width=6cm]{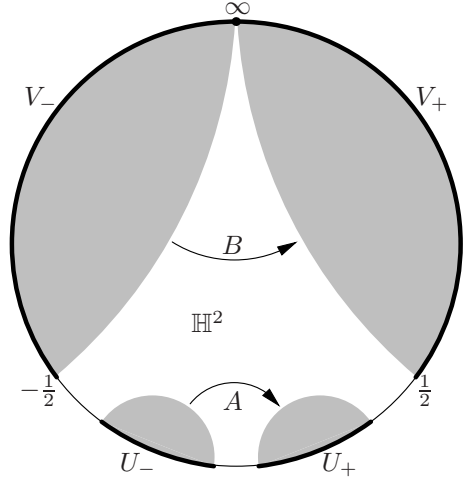}
\caption{The parabolic element $B$ fixes $\infty$ and takes $-1/2$ to $1/2$. For $A$ we can choose any hyperbolic element with large enough translation length whose translation axis contains the shortest segment connecting the hyperbolic half-plane bordered by~$U_-$ to the one bordered by~$U_+$.}
\label{domain}
\end{figure}

Choose two disjoint compact intervals $U_-$ and $U_+$ in $\PP^1(\RR) \smallsetminus (V_- \cup V_+)$ (see Figure~\ref{domain}).  
There exists a hyperbolic element $A= \rho_{\alpha,0}(a)$ such that
\begin{itemize}
  \item $A( \PP^1(\RR) \smallsetminus U_-) \subset U_+$ and $A|_{\PP^1(\RR) \smallsetminus U_-}$ is $e^{-\kappa}$-contracting;
  \item $A^{-1}( \PP^1(\RR) \smallsetminus U_+) \subset U_-$ and $A^{-1}|_{\PP^1(\RR) \smallsetminus U_+}$ is $e^{-\kappa}$-contracting.
\end{itemize}
Then for any $\gamma \in \Gamma$, the element $\rho_{\alpha,0}(\gamma)$ is $e^{-\kappa\ellFF{\gamma}{a}}$-contracting at every point in $\PP^1(\RR) \smallsetminus (V_- \cup V_+ \cup U_- \cup
U_+)$.
We obtain that $\mu(\rho_{\alpha,0}(\gamma)) \geq \kappa \ellFF{\gamma}{a}$ for all $\gamma \in \Gamma$ from the following lemma.
\end{proof}

\begin{lemma}
For any $g\in\SL_2(\RR)$ and $\ell\geq 0$, if $g$ is $e^{-\ell}$-contracting at some point of $\PP^1(\RR)$, then $\mu(g) \geq \ell$. 
\end{lemma}

\begin{proof}
The assumptions do not change if we multiply $g$ by elements of $\SO(2)$ on either side.
Thus we can assume that $g =
  \bigl(\begin{smallmatrix}
    e^{-s/2} & 0 \\ 0 & e^{s/2}
  \end{smallmatrix}\bigr)
$ with $s\geq 0$ (so that $\mu(g) =s$).
Let $x_0 \in \PP^1(\RR)$ be a point where the differential $\mathrm{d}g$ is $e^{-\ell}$-contracting.
Then $x_0 \neq \infty$.
Let $\partial/ \partial u$ be the translation-invariant vector field on~$\RR$.
Denoting by $\| \cdot\|_x$ the Riemannian norm (for the round metric) of tangent vectors at a point $x\in \PP^1(\RR)$, the norm $\Vert \partial/\partial u \Vert_x$ is a decreasing function of $|x|$, and the contraction of $\mathrm{d}g$ at~$x_0$ is
\[\frac{ \| \mathrm{d}g \cdot \partial /\partial u\|_{g\cdot x_0}}{ \| \partial
  /\partial u\|_{x_0}} =  \frac{ \|  e^{-s} \partial /\partial u\|_{g\cdot x_0}}{ \| \partial
  /\partial u\|_{x_0}} \leq e^{-\ell}.\]
Since $\frac{ \|  \partial /\partial u\|_{g\cdot x_0}}{ \| \partial /\partial u\|_{x_0}}\geq 1$, the conclusion $\mu(g)= s\geq \ell$ follows.
\end{proof}

\begin{remark}
   For \(t\in \QQ \cap (0,1]\), the representation \(\rho_t\) has a nontrivial kernel: a power \(b^n\) of \(b\) is in the kernel of \(\rho_{\alpha,t}\) and \(a^n\) is in the kernel of \(\rho_{\beta,t}\) and therefore the commutator \(a^n b^n a^{-n} b^{-n}\) is in the kernel of~\(\rho_t\).
   Thus \(\rho_0\) is the endpoint of a continuous family of representations, all of them being nondiscrete or nonfaithful.
\end{remark}

\begin{remark} \label{rem:no-cont-dyn-preserv-map}
The representation $\rho : \Gamma\to\SL_2(\RR)\times\SL_2(\RR)$, seen as a representation into $\SL_4(\RR)$, does not admit any continuous, dynamics-preserving boundary map $\xi^+ : \partial_{\infty}\Gamma\to\PP^3(\RR)$.
Indeed, suppose by contradiction that such a map~$\xi^+$ exists.
By construction, the closed $\rho(\Gamma)$-invariant sets $\{ \eta \in \partial_{\infty}\Gamma \mid \xi^+(\eta)\subset\PP(\RR^2\oplus\{0\})\}$ and $\{ \eta \in \partial_{\infty}\Gamma \mid \xi^+(\eta)\in\PP_{\RR}(\{0\}\oplus\RR^2)\}$ are both nonempty: this contradicts the minimality of the action of $\Gamma$ on $\partial_{\infty}\Gamma$.
\end{remark}

We can use the representation $\rho_0$ of Proposition~\ref{prop:an-unstable-quasi} to construct an example of a representation $\rho : \Gamma\to G$ and a set $\theta\subset\Delta$ with the following properties:
\begin{itemize}
  \item there exist continuous, $\rho$-equivariant, transverse maps $\xi^+ : \partial_{\infty}\Gamma\to G/P_{\theta}$ and $\xi^- : \partial_{\infty}\Gamma\to G/P_{\theta}^-$ such that for any $\eta \in \partial_{\infty}\Gamma$, the points $\xi^+(\eta)\in G/P_{\theta}$ and $\xi^-(\eta)\in G/P_{\theta}^-$ are compatible in the sense of Definition~\ref{defi:transverse-compatible};
  \item no such maps can be dynamics-preserving for~$\rho$. 
\end{itemize}

\begin{example} \label{ex:xi-not-dyn-preserv}
Let $\Gamma$ be a free group on two generators and $\rho' : \Gamma\to\nolinebreak\SL_2(\RR)$ a convex cocompact representation.
We see $\SL_2(\RR)$ as a subgroup of $G=\GL_6(\RR)$ by embedding it into the lower right corner of~$G$, and use the notation of Example~\ref{ex:roots} for~$G$.
By \eqref{eqn:mu-leq-klength}, there exists $k>0$ such that
\[ \langle \varepsilon_1, \mu(\rho'(\gamma)) \rangle \leq k \ellGamma{\gamma} \]
for all $\gamma\in\Gamma$.
On the other hand, if we choose $\kappa>k$ in Proposition~\ref{prop:an-unstable-quasi-2}, then the representation $\rho_0 : \Gamma\to\SL_2(\RR)\times\SL_2(\RR)$ constructed in Proposition~\ref{prop:an-unstable-quasi}, seen as a representation into $G=\GL_6(\RR)$ by embedding $\SL_2(\RR)\times\SL_2(\RR)$ into the upper left corner of~$G$, satisfies 
\[ \langle \varepsilon_1, \mu(\rho_0(\gamma)) \rangle \geq \kappa \ellGamma{\gamma} \]
for all $\gamma\in\Gamma$.
In particular, using \eqref{eqn:lambda-lim-mu}, we see that $\rho_0$ uniformly $P_{\{\varepsilon_1-\varepsilon_2\}}$-dominates $\rho'$ as representations into~$G$: there is a constant \(c<1\) such that \( \langle \varepsilon_1, \lambda(\rho'(\gamma)) \rangle \leq c \, \langle \varepsilon_1, \lambda(\rho_0(\gamma)) \rangle \) for all \(\gamma\in \Gamma\).
Consider the representation
\[ \rho := (\rho_0, \rho') : \Gamma \longrightarrow \big(\SL_2(\RR)\times\SL_2(\RR)\big) \times \SL_2(\RR) \longhookrightarrow G. \]
It admits continuous, $\rho$-equivariant, transverse boundary maps $\xi^+: \partial_\infty \Gamma \to \PP(\RR^6) = G/P_{\{\varepsilon_1-\varepsilon_2\}}$ and $\xi^-: \partial_\infty \Gamma \to \PP((\RR^6)^{\ast}) = G/P_{\{\varepsilon_1-\varepsilon_2\}}^-$, obtained by composing the boundary maps of the Anosov representation $\rho' : \Gamma\to\SL_2(\RR)$ (Example~\ref{subsubsec:examples-properties}.\eqref{item:2}) with the inclusion of $\PP(\RR^2)\simeq\PP(\{ 0\} \times \RR^2)$ into $\PP(\RR^6)$.
However, $\rho$ does not admit any continuous, dynamics-preserving boundary map, since $\rho_0$ uniformly $P_{\{\varepsilon_1-\varepsilon_2\}}$-dominates $\rho'$ and $\rho_0$ does not admit any continuous, dynamics-preserving boundary map (Remark~\ref{rem:no-cont-dyn-preserv-map}).
\end{example}

%%%%%%%%%%%%%%%%%%%%%%%%%%%%%%%%%%%%%%%%%%%%%

\providecommand{\bysame}{\leavevmode\hbox to3em{\hrulefill}\thinspace}

\end{document}